\documentclass{article}
\usepackage{appendix}
\usepackage{mathrsfs}
\usepackage{amsmath}
\usepackage{amssymb, amsmath}
\usepackage{graphicx}
\usepackage{fancyhdr}
\usepackage [latin1]{inputenc}
\usepackage{enumerate}
\usepackage{geometry}
\catcode`\@=11 \@addtoreset{equation}{section}

\catcode`\@=12
\usepackage{color}

\usepackage[colorlinks,
            linkcolor=blue,
            anchorcolor=blue,
            citecolor=red,
            ]{hyperref}

\DeclareFontFamily{U}{mathx}{}
\DeclareFontShape{U}{mathx}{m}{n}{ <-> mathx10 }{}
\DeclareSymbolFont{mathx}{U}{mathx}{m}{n}
\DeclareFontSubstitution{U}{mathx}{m}{n}

\DeclareMathAccent{\widecheck}{0}{mathx}{"71}

\makeatletter
\newcommand{\wcheck}[1]{\mathpalette\wcheck@{#1}}
\newcommand{\wcheck@}[2]{%
  \begingroup
  \edef\wcheck@font{\the
    \ifx#1\displaystyle\textfont\else\ifx#1\textstyle\textfont
    \else\ifx#1\scriptstyle\scriptfont\else\scriptscriptfont\fi\fi\fi\@ne
  }%
  \sbox\z@{\wcheck@font\mbox{#2}\mbox{\char\the\skewchar\font}}%
  \sbox\tw@{\wcheck@font#2\char\the\skewchar\font}%
  \dimen@=\dimexpr\wd\tw@-\wd\z@\relax
  {\,\kern2\dimen@\widecheck{\!\kern-2\dimen@#2\!}\,}%
  \endgroup
}
\makeatother

\newtheorem{Theorem}{Theorem}[section]
\newtheorem{Proposition}{Proposition}[section]
\newtheorem{Lemma}{Lemma}[section]
\newtheorem{Corollary}{Corollary}[section]

\newtheorem{Remark}{Remark}[section]

\newcommand{\newcom}{\newcommand}
\newcommand{\bTheorem}[1]{
\begin{Theorem} \label{T#1} }
\newcommand{\eT}{\end{Theorem}}

\newcommand{\bProposition}[1]{
\begin{Proposition} \label{P#1}}
\newcommand{\eP}{\end{Proposition}}

\newcommand{\bLemma}[1]{
\begin{Lemma} \label{L#1} }
\newcommand{\eL}{\end{Lemma}}

\newcommand{\bCorollary}[1]{
\begin{Corollary} \label{C#1} }
\newcommand{\eC}{\end{Corollary}}

\newcommand{\beq}{\begin{equation}}
\newcommand{\eeq}{\end{equation}}
\newcom{\ben}{\begin{eqnarray}}
\newcom{\een}{\end{eqnarray}}
\newcom{\beno}{\begin{eqnarray*}}
\newcom{\eeno}{\end{eqnarray*}}
\newcom{\bali}{\begin{aligned}}
\newcom{\eali}{\end{aligned}}

\newcommand{\bFormula}[1]{
\begin{equation} \label{#1}}
\newcommand{\eF}{\end{equation}}

\newcommand{\les}{ \lesssim }

\newcommand{\Mp}[1]{\left\{#1\right\}}
\renewcommand{\sp}[1]{\left(#1\right)}
\newcommand{\n}[1]{{\left\|#1\right\|}}
\newcommand{\tsgn}{\text{sgn}}
\newcommand{\Zd}{\widetilde{\mathbb{Z}}^d}
\newcommand{\De}{\Delta}
\newcommand{\f}{\frac}
\newcommand{\df}{\dfrac}

\newcommand{\p}{\partial}

\newcommand{\vae}{a_\ep}
\newcommand{\vbe}{b_\ep}
\newcommand{\Td}{\mathbb{T}_\mathbf{a}^d}
\newcommand{\T}{|\mathbb{T}_\mathbf{a}^d|}
\newcommand{\vthe}{\vartheta_\ep}

\newcommand{\vue}{\vu_\ep}
\newcommand{\vUe}{\vU_\ep}
\newcommand{\vVe}{\vV_\ep}
\newcommand{\vWe}{\vW_\ep}
\newcommand{\vZe}{\vZ_\ep}

\newcommand{\vA}{\vc{A}}
\newcommand{\vB}{\vc{B}}

\newcommand{\vE}{\vc{E}}
\newcommand{\vF}{\vc{F}}
\newcommand{\vu}{\vc{u}}

\newcommand{\vZ}{\vc{Z}}
\newcommand{\vU}{\vc{U}}

\newcommand{\vv}{\vc{v}}
\newcommand{\vV}{\vc{V}}

\newcommand{\vw}{\vc{w}}
\newcommand{\vW}{\vc{W}}

\newcommand{\vc}[1]{{\boldsymbol #1}}

\newcommand{\Div}{{\rm div\,} }
\newcommand{\Grad}{\nabla}

\newcommand{\dx}{\,{\rm d} x}

\newcommand{\dt}{\,{\rm d} t }
\newcommand{\dta}{\,{\rm d} \tau }

\newcommand{\intTd}[1]{\int_{\mathbb{T}_\mathbf{a}^d} #1  \,\dx}

\newcommand{\bProof}{{\bf Proof: }}

\newcommand{\ep}{\varepsilon}

\newcommand\Cbox[2]{%
    \newbox\contentbox%
    \newbox\bkgdbox%
    \setbox\contentbox\hbox to \hsize{%
        \vtop{
            \kern\columnsep
            \hbox to \hsize{%
                \kern\columnsep%
                \advance\hsize by -2\columnsep%
                \setlength{\textwidth}{\hsize}%
                \vbox{
                    \parskip=\baselineskip
                    \parindent=0bp
                    #2
                }%
                \kern\columnsep%
            }%
            \kern\columnsep%
        }%
    }%
    \setbox\bkgdbox\vbox{
        \color{#1}
        \hrule width  \wd\contentbox %
               height \ht\contentbox %
               depth  \dp\contentbox
        \color{black}
    }%
    \wd\bkgdbox=0bp%
    \vbox{\hbox to \hsize{\box\bkgdbox\box\contentbox}}%
    \vskip\baselineskip%
}

\usepackage{graphicx} % Required for inserting images

\title{\bf{Low Mach number limit for a compressible two-fluid model with algebraic closure and ill-prepared initial data in critical Besov spaces} }

\author{
Sai Li\thanks{Department of Applied Mathematics, Nanjing Forestry University, Nanjing--210037, People's Republic of China, E-mail: \texttt{lsmath@njfu.edu.cn}}
\qquad  \qquad
Yang Li\thanks{School of Mathematical Sciences, Anhui University, Hefei--230601, People's Republic of China, E-mail: \texttt{lynjum@163.com}} 
}

\begin{document}

\maketitle 
\abstract{
In this paper, we study the low Mach number limit for a compressible two-fluid model with algebraic closure in the $d$-dimensional torus with $d \geq 2$. For large and ill-prepared initial data in critical Besov spaces, we prove that, provided that the Mach number is sufficiently small, the rescaled compressible two-fluid flow exists in critical Besov spaces for any finite time not exceeding the lifespan of the incompressible flow. Moreover, the rescaled compressible two-fluid flow converges to the incompressible Navier--Stokes flow as the Mach number tends to zero. The proof is based on a high-middle-low frequency analysis of the densities and velocity field, combined with a filtering technique involving wave operators. The main novelty is the derivation of new a priori estimates for the high-middle frequency part of the solution to the compressible two-fluid model, depending explicitly on time, the frequency parameter and the Mach number. 
To the best of our knowledge, this is the first work that proves (almost) global convergence for large and ill-prepared initial data in the low Mach number limit for compressible two-fluid model in critical framework. 
}

\medskip

\vspace{2mm}


\tableofcontents

%%%%%%%%%%%%%%%%%%%%%%%%%%%%%%%%%%%%%%%%%%%%%%%%%%%%%%%%%%%%%%%%%%%%%%%%%%%
\section{Introduction}

\subsection{Background}
In this paper, we consider the evolution of two immiscible compressible viscous fluids in the $d$-dimensional periodic domain with $d\geq 2$. We assume that the two fluids share the same velocity field and admit the algebraic closure of pressure. In Eulerian coordinates, the governing equations read as
\begin{equation}\label{eq:TF-orig}
\left\{\begin{aligned}
& \partial_t (\alpha_{\pm} \varrho_{\pm})
        + \Div (\alpha_{\pm} \varrho_{\pm} \vu)
        = 0, \\
& \p_t[ (\alpha_{+} \varrho_{+} +\alpha_{-} \varrho_{-}  )   \vu ] 
         + \Div[ (\alpha_{+} \varrho_{+} +\alpha_{-} \varrho_{-}  )   \vu \otimes \vu  ]
         + \nabla p= \mu \Delta \vu +(\mu+\lambda) \nabla \Div \vu 
         , \\
& \alpha_{+} + \alpha_{-}=1, \quad \alpha_{\pm}  \geq 0
         , \\
        &  p=p_{+}=p_{-}. 
\end{aligned}\right.
\end{equation}
Here, $\alpha_\pm=\alpha_\pm(t,x)$ denote the volume fractions of the two phases, $\varrho_\pm=\varrho_\pm(t,x)$ their mass densities, $\vu =\vu (t,x)$ the common velocity field, and $p$ the common pressure. The constants $\mu$ and $\lambda$ are the shear and bulk viscosity coefficients, respectively. For more physical explanations about the two-fluid model above, we refer to \cite{Bouchut,IsHi} among many others.

Following Bresch et al. \cite{BMZ19}, it is convenient to introduce the conservative variables
\begin{align}
 R=\alpha_{+} \varrho_{+},\qquad 
 Q=\alpha_{-} \varrho_{-} , \qquad 
 Z =  \varrho_{+}.
\end{align}
Then system \eqref{eq:TF-orig} can be rewritten as
\begin{equation}
\left\{\begin{aligned}
& \partial_t R + \Div (R \vu  )
        = 0, \\
& \partial_t Q + \Div (Q \vu )
        = 0, \\
& \partial_t [(R+Q) \vu ] + \Div [(R+Q) \vu \otimes \vu ]
        +\nabla p= \mu \Delta \vu  +(\mu+\lambda) \nabla \Div \vu .
\end{aligned}\right.
\end{equation}
The algebraic pressure closure means that the phase pressures coincide. Assuming both fluids are isentropic, we have
\begin{align}
p_{+}=( \varrho_{+})^{ \gamma_{+} } =
p_{-}=( \varrho_{-})^{ \gamma_{-} }, \quad \gamma_\pm>1.
\end{align}
As a consequence, the common pressure can be expressed as
\begin{align}
    p(Z)=Z^{\gamma_{+}},
\end{align}
where $Z$ is determined implicitly by $(R,Q)$ through
\begin{align}\label{eq:press-strc}
       Q=\sp{ 1-\f{R} {Z }  } Z^{\gamma}
       , \qquad \gamma=\frac{ \gamma_{+}  }{ \gamma_{-}  }, \qquad 
       R \leq Z. 
\end{align}
Bresch et al. \cite{BMZ19} showed that \eqref{eq:press-strc} uniquely determines $Z$, and so there exists some function $\mathcal{Z}(\cdot,\cdot)$ such that $Z=\mathcal{Z}(R,Q)$.

Now we recall some previous results closely related to \eqref{eq:TF-orig}. In a semi-stationary Stokes regime, Bresch et al. \cite{BMZ19} proved the global existence of finite energy weak solutions with periodic boundary conditions. For a general compressible viscous two-fluid system, including in particular \eqref{eq:TF-orig}, Novotn\'{y} and Pokorn\'{y} \cite{NM20} proved the global existence of finite energy weak solutions. For strong solutions, Piasecki and Zatorska \cite{PZ1} proved the local well-posedness for large initial data and global well-posedness for small initial data. For the weak--strong uniqueness principle of \eqref{eq:TF-orig}, Li and Zatorska \cite{LZ22} obtained a conditional result; by introducing a new relative entropy functional, this result was recently improved to an unconditional one \cite{LLPZ26}. For the inviscid two-fluid model, i.e., \eqref{eq:TF-orig} with $\mu=\lambda=0$, Li and Zatorska \cite{LiZa} proved the existence of infinitely many weak solutions in a $3$D bounded domain. In the simplified $1$D case, Li et al. \cite{LSZ1} proved the existence, uniqueness and large time behavior of global weak solutions. We emphasize that there are many fruitful results on other types of two-fluid models, and we refer to \cite{LLZ26} for more discussions.

Very recently, there are some progress on the low Mach number limit of the compressible two-fluid system \eqref{eq:TF-orig}. For well-prepared initial data and a bounded smooth domain of $\mathbb{R}^3$, Lebot \cite{Leb26} considered the rescaled compressible two-fluid system \eqref{eq:TF-orig}, and identified the inhomogeneous incompressible Navier--Stokes equations as the limit system. Moreover, as remarked in \cite{Leb26}, it is possible to extend their results to periodic domains. However, the convergence rates are not clear in \cite{Leb26}. For well-prepared initial data and the $3$D periodic domain, Li et al. \cite{LLZ26} identified the incompressible Navier--Stokes equations as the limit system and further obtained the explicit convergence rates. The analysis for ill-prepared initial data appears more delicate due to the acoustic waves and the complicated structure of the two-fluid system \eqref{eq:TF-orig}. \emph{As far as we know, there is no previous study for low Mach number limit of the rescaled compressible two-fluid system \eqref{eq:TF-orig} with ill-prepared initial data}. We stress that the consideration of ill-prepared initial data was listed as one of the open problems in Lebot \cite{Leb25}. This is the motivation of the present work.

\subsection{Heuristic analysis}  
According to the previous discussions, we consider the low Mach number limit of the two-phase flow on the torus $\Td \, (d\geq 2)$ with periodic parameter $\mathbf{a}=(\mathbf{a}^1, \mathbf{a}^2, \cdots, \mathbf{a}^d)$, where $\mathbf{a}^j > 0$, $1 \leq j\leq d$. The rescaled two-fluid model reads as 
\begin{equation}
\left\{\begin{aligned}
&\partial_t R_\ep + \Div (R_\ep \vue) = 0, \\
&\partial_t Q_\ep + \Div (Q_\ep \vue) = 0, \\
&\partial_t [(R_\ep + Q_\ep) \vue] + \Div [(R_\ep + Q_\ep) \vue \otimes \vue] + \dfrac{\nabla p(Z_\ep)}{\ep^2} = \mu_\ep \Delta \vue + (\mu_\ep + \lambda_\ep) \nabla \Div\vue.
\end{aligned}\right.
\end{equation}
Here, $R_\ep (t,x)$ and $Q_\ep (t,x)$ are densities of the two phases, $\vue (t,x)$ is the velocity field, all of which are $2\pi\mathbf{a}^j$-periodic in $x_j$ for $1\leq j\leq d$.
The pressure is given by $p(Z_\ep) = Z_\ep^{\gamma_+}$, and $Z_\ep=\mathcal{Z} (R_\ep, Q_\ep)$ through 
\begin{align}\label{eq:press}
    Q_\ep = \left(1 - \frac{R_\ep}{Z_\ep}\right) Z_\ep^\gamma, \qquad 
\gamma = \frac{\gamma_+}{\gamma_-}, \qquad 
R_\ep \le Z_\ep. 
\end{align}
For simplicity, throughout this paper we assume that \(\mu_\ep = \mu\) and \(\lambda_\ep = \lambda\) are constants satisfying
\[
\mu > 0, \qquad \nu:=2\mu + \lambda > 0.
\] 
To proceed, we consider the following equivalent system
\begin{equation}\label{eq:res-two-fluid}
\left\{\begin{aligned}
&
\partial_t R_\ep 
+ 
\Div (R_\ep \vue) 
= 0, \\[4pt]
&
\partial_t Q_\ep 
+ 
\Div  (Q_\ep \vue)
= 0, \\[4pt]
&
\partial_t \vue
+ 
\vue \cdot \nabla \vue
+ 
\dfrac{\nabla p(Z_\ep)}{\ep^2(R_\ep + Q_\ep)} 
= 
\dfrac{\mu}{R_\ep + Q_\ep} \Delta \vue
+ \dfrac{\mu + \lambda}{R_\ep + Q_\ep} \nabla \Div \vue.
\end{aligned} \right.
\end{equation}
We set the equilibrium states of $R_\ep$ and $Q_\ep$ as positive constants $R^0$ and $Q^0$; hence we consider the initial data
\[
(R_{\ep}, Q_{\ep}, \vu_{\ep})(0,\cdot)=(R^0+\ep a_{0, \ep}, Q^0+\ep b_{0,\ep}, \vu_{0, \ep}).
\]
Correspondingly, we introduce the unknowns
\[
a_\ep:=\dfrac{R_\ep-R^0}{\ep}, \qquad 
b_\ep:=\dfrac{Q_\ep-Q^0}{\ep}.
\]
Then, we see that $(a_\ep, b_\ep, \vu_\ep)$ should solve 
\begin{equation}\label{1.1.1}
\left\{\begin{aligned}
&\partial_t a_\ep + \dfrac{R^0 \Div \vu_\ep}{\ep}=-\Div(a_\ep \vu_\ep), \\
&\partial_t b_\ep + \dfrac{Q^0 \Div \vu_\ep}{\ep}=-\Div(b_\ep \vu_\ep), \\
&\partial_t \vu_\ep + \dfrac{\nabla p(Z_\ep)}{\ep^2(R_\ep + Q_\ep)} = \mathcal{A}\vu_\ep+I(\ep a_\ep+\ep b_\ep)\mathcal{A}\vu_\ep- \vu_\ep \cdot \nabla \vu_\ep,
\end{aligned}\right.
\end{equation}
where
\[
\mathcal{A}\mathbf{v}:=\dfrac{1} {R^0+Q^0}\mu\Delta \mathbf{v}+\dfrac{1} {R^0+Q^0}(\mu+\lambda)\Grad \Div \mathbf{v}, \qquad 
I(c):=\dfrac{R^0+Q^0} {R^0+Q^0+c}-1. 
\]

In a periodic domain with ill-prepared initial data, $a_\ep$ and $b_\ep$
will undergo rapid oscillations. Therefore, to determine the asymptotic limit of solutions to \eqref{1.1.1}, we need a higher-order Taylor expansion of the pressure term. Observe that
\[
 \dfrac{\nabla p(Z_\ep)}{\ep^2(R_\ep + Q_\ep)}=\dfrac{p'(Z_\ep) \p_R{Z_\ep}}{R_\ep + Q_\ep} \dfrac{\Grad a_\ep}{\ep}+\dfrac{p'(Z_\ep) \p_Q{Z_\ep}}{R_\ep + Q_\ep} \dfrac{\Grad b_\ep}{\ep}.
\]
We define the functions $\widetilde{p}(\cdot, \cdot)$ and $\widecheck{p}(\cdot, \cdot)$ on $\mathbb{R}^2$ as 
\begin{align*}
&
\widetilde{p}(x, y)
:=
\f
{p'( \mathcal{Z} (x+ R^0, y+Q^0)) \, \p_R \mathcal{Z}  (x+R^0, y+Q^0) }
{R^0+x + Q^0+y},
\\
&
\widecheck{p}(x, y)
:=
\f
{p'( \mathcal{Z} (x+ R^0, y+Q^0)) \, \p_Q \mathcal{Z} (x+R^0, y+Q^0)  }
{R^0+x + Q^0+y}.
\end{align*}
Then we expand the pressure as 
\begin{align}\label{Taylor}
\dfrac{\nabla p(Z_\ep)}{\ep^2(R_\ep + Q_\ep)}
&
=
c_1 \dfrac{\Grad a_\ep}{\ep}
+
c_2\dfrac{\Grad b_\ep}{\ep}
+
c_3 a_\ep\Grad a_\ep+c_4 b_\ep\Grad a_\ep
+
c_5 a_\ep\Grad b_\ep+c_6 b_\ep\Grad b_\ep
\\
&
\quad
+
a_\ep \Grad a_\ep
\int_0^1
(\p_x \widetilde{p})(t\ep a_\ep, t\ep b_\ep)-(\p_x \widetilde{p})(0, 0)\dt
+
b_\ep \Grad a_\ep
\int_0^1
(\p_y \widetilde{p})(t\ep a_\ep, t\ep b_\ep)-(\p_y \widetilde{p})(0, 0)\dt
\nonumber
\\
&
\quad
+
a_\ep \Grad b_\ep
\int_0^1
(\p_x \widecheck{p})(t\ep a_\ep, t\ep b_\ep)-(\p_x \widecheck{p})(0, 0)\dt
+
b_\ep \Grad b_\ep
\int_0^1
(\p_y \widecheck{p})(t\ep a_\ep, t\ep b_\ep)-(\p_y \widecheck{p})(0, 0)\dt
\nonumber,
\end{align}
where
\begin{align*}
&
c_1
:=
\widetilde{p}(0, 0),
\qquad 
c_2
:=
\widecheck{p}(0, 0),
\qquad
c_3
:=
(\p_x \widetilde{p})(0, 0),
\\[4pt]
&
c_4
:=
(\p_y \widetilde{p})(0, 0),
\qquad
c_5
:=
(\p_x \widecheck{p})(0, 0),
\qquad
c_6
:=
(\p_y \widecheck{p})(0, 0).
\end{align*}
Using the expansion of the pressure, we rewrite \eqref{1.1.1} as 
\begin{equation}\label{1.1.3}
\left\{\begin{aligned}
&
\partial_t a_\ep 
+ 
\dfrac{R^0 \Div \vu_\ep}{\ep}
=
-
\operatorname{div}(a_\ep \vu_\ep), 
\\
&
\partial_t b_\ep
+
\dfrac{Q^0 \Div \vu_\ep}{\ep}
=
-
\operatorname{div}(b_\ep \vu_\ep), 
\\
&
\partial_t \vu_\ep
+
c_1
\dfrac{\Grad a_\ep}{\ep}
+
c_2
\dfrac{\Grad b_\ep}{\ep}
+
c_3 
a_\ep\Grad a_\ep
+
c_4
b_\ep\Grad a_\ep
+
c_5 
a_\ep\Grad b_\ep
+
c_6
b_\ep\Grad b_\ep
\\
&
\qquad  
= 
\mathcal{A}\vu_\ep
+
I(\ep a_\ep + \ep b_\ep)\mathcal{A}\vu_\ep
- 
\vu_\ep \cdot \nabla \vu_\ep
\\
&
\qquad \qquad \qquad 
-
\left[
\dfrac{\nabla p(Z_\ep)}{\ep^2(R_\ep + Q_\ep)}
-
\left(
c_1 \dfrac{\Grad a_\ep}{\ep}+c_2\dfrac{\Grad b_\ep}{\ep}
+
c_3 
a_\ep\Grad a_\ep
+
c_4 
b_\ep\Grad a_\ep
+
c_5 
a_\ep\Grad b_\ep
+
c_6 
b_\ep\Grad b_\ep
\right)
\right].
\end{aligned}\right.
\end{equation}

We define the linear operator $L$ as 
\[
L\left(
\begin{array}{l}
a\\
b\\
\vu
\end{array}
\right)
=\left(
\begin{array}{l}
R^0\Div\vu\\
Q^0\Div\vu\\
c_1\Grad a+c_2\Grad b\\
\end{array}
\right).
\]
Direct calculation gives the image and kernel spaces of $L$ as 
\begin{align*}
& 
\text{Im} \,L=\left\{ \sp{g,\f{Q^0}{R^0}g,\Grad f}^T: \int_{\Td}  g \dx=0 \right\},  \\
& 
\text{Ker} \, L=\left\{  \sp{ g+c^*,-\f{c_1}{c_2}g+c^{**},\Phi }^T: \int_{\Td}  g  \dx=0, \quad \Div \Phi=0 \right\},
\end{align*} 
where $c^*$ and $c^{**}$ are arbitrary constants. Inspired by \cite{CGHJ24}, we construct an equivalent inner product $\langle \cdot, \cdot \rangle_\mathbb{H}$ on $L^2(\Td)$ such that the operator $L$ is skew-symmetric with respect to this new inner product, i.e.,
\begin{align}\label{1.1.2}
\langle L(a,b,\vu)^T, (\tilde{a},\tilde{b},\tilde{\vu}) \rangle_\mathbb{H} = - \langle (a,b,\vu)^T, L(\tilde{a},\tilde{b},\tilde{\vu})  \rangle_\mathbb{H}. 
\end{align}
We define
\[
\langle (a,b,\vu)^T, (\tilde{a},\tilde{b},\tilde{\vu}) \rangle_\mathbb{H} 
:= \lambda_1 \int_{\Td} a \bar{\tilde{a}} \dx 
+ \lambda_2 \int_{\Td} b \bar{\tilde{b}} \dx 
+ \lambda_3 \int_{\Td} \vu \bar{\tilde{\vu}} \dx,
\]
where the positive constants $\lambda_1, \lambda_2, \lambda_3$ will be determined later. Notice that
\begin{align*}
&
\langle 
L(a,b,\vu)^T, 
(\tilde{a},\tilde{b},\tilde{\vu}) 
\rangle_\mathbb{H} 
\\
& \quad
=
\lambda_1 \int_{\Td} R^0 \Div \vu \bar{\tilde{a}} \dx
+ 
\lambda_2 \int_{\Td} Q^0\Div \vu \bar{\tilde{b}} \dx
+ \lambda_3 \int_{\Td} (c_1\Grad a+c_2\Grad b) \bar{\tilde{\vu}} \dx
\\
&\quad
=
-\lambda_1  R^0 
\int_{\Td}   \vu \Grad\bar{\tilde{a}} \dx
-\lambda_2 Q^0
\int_{\Td}  \vu \Grad\bar{\tilde{b}} \dx
-\lambda_3 c_1
\int_{\Td}  a\Div\bar{\tilde{\vu}} \dx
-\lambda_3 c_2
\int_{\Td}  b\Div\bar{\tilde{\vu}} \dx,
\\[4pt]
&
\langle 
(a,b,\vu)^T, 
L(\tilde{a},\tilde{b},\tilde{\vu}) 
\rangle_\mathbb{H}
\\
&\quad
= 
\lambda_1 R^0 
\int_{\Td}  a \Div\bar{\tilde{\vu}} \dx
+
\lambda_2 Q^0
\int_{\Td}  b \Div\bar{\tilde{\vu}} \dx
+
\lambda_3 c_1
\int_{\Td} \vu \Grad \bar{\tilde{a}}  \dx.
+
\lambda_3 c_2
\int_{\Td} \vu \Grad \bar {\tilde{b}} \dx.
\end{align*}
For the purpose of \eqref{1.1.2}, it suffices to assume that $\lambda_1, \lambda_2, \lambda_3$ satisfy
\[
\lambda_3 c_1 = \lambda_1 R^0, \qquad \lambda_3 c_2 = \lambda_2 Q^0.
\]
The solutions $\lambda_1, \lambda_2, \lambda_3$ obtained here are not unique. However, to ensure that they are all positive so that we can construct the equivalent inner product  $\langle \cdot, \cdot \rangle_\mathbb{H}$, we need $c_1, c_2>0$.
From $\mathcal{Z} (R^0, Q^0)\geq R^0$ and
\[
\p_{R} \mathcal{Z} (R^0, Q^0) 
 =
\f{ \mathcal{Z} (R^0, Q^0)^{\gamma-1} }
{
\gamma \mathcal{Z} (R^0, Q^0)^{\gamma-1} - R^0 (\gamma-1) \mathcal{Z} (R^0, Q^0)^{\gamma-2}},
\quad  
\p_{Q} \mathcal{Z} (R^0, Q^0) 
=
\f{ 1 }
{
\gamma \mathcal{Z} (R^0, Q^0)^{\gamma-1} - 
R^0 (\gamma-1) 
\mathcal{Z} (R^0, Q^0)^{\gamma-2}
} ,
\]
we conclude that $c_1, c_2>0$. Hence, we take
\[
\lambda_1 = c_1 Q^0, \qquad \lambda_2 = c_2 R^0, \qquad \lambda_3 = R^0 Q^0.
\]
From the skew-symmetry of $L$ with respect to the inner product $\langle \cdot, \cdot \rangle_{\mathbb{H}}$, we obtain the following direct sum decomposition
\[
L^2(\Td)= \text{Im} \, L \oplus \text{Ker}\,  L.
\]
For any function $f$, we set $\underline{f}:=f-\f{\widehat{f}_0}{\sqrt{|\Td|}}$ and 
$\De^{-1}f
:=
-\sum\limits_{k\neq0}\df{\widehat{f}_k}{|k|^2} \df{\text{e}^{ik\cdot x}}{\sqrt{|\Td|}}
$,
where $\widehat{f}_k$ are the Fourier coefficients (see the definition in \eqref{fourier}). Let $\mathbb{P}^\perp$ and $\mathbb{P}$ be the projection operators onto $\operatorname{Im} L$ and $\operatorname{Ker} L$ respectively. Then, for any $(a,b,\vu)$ we have
\begin{align*}
\mathbb{P}^\perp\left(
\begin{array}{l}
a
\\
b
\\
\vu
\end{array}
\right)
=\left(
\begin{array}{l}
\dfrac{c_1\underline{a}+c_2\underline{b}}{c_1+c_2\f{Q^0}{R^0}}
\\[15pt]
\dfrac{Q^0}{R^0}\dfrac{c_1\underline{a}+c_2\underline{b}}{c_1+c_2\f{Q^0}{R^0}}
\\[15pt]
\mathcal{P}^\perp \vu
\\
\end{array}
\right), 
\qquad 
\mathbb{P}\left(
\begin{array}{l}
a
\\
b
\\
\vu
\end{array}
\right)
=
\left(
\begin{array}{l}
\dfrac{Q^0\underline{a}-R^0\underline{b}}{Q^0+R^0\f{c_1}{c_2}}
\\[15pt]
-\dfrac{c_1}{c_2}\dfrac{Q^0\underline{a}-R^0\underline{b}}{Q^0+R^0\f{c_1}{c_2}}
\\[15pt]
\mathcal{P}\underline{\vu}
\\
\end{array}
\right)+\mathbb{P}_0\left(
\begin{array}{l}
a
\\[15pt]
b
\\[15pt]
\vu
\\
\end{array}
\right).
\end{align*}
Here, we denote by 
\begin{align*}
&
\mathbb{P}_0(a, b, \vu)^T
:=
\sp{ \f{\widehat{a}_0}{\sqrt{|\Td|}}, \f{\widehat{b}_0}{\sqrt{|\Td|}}, \f{\widehat{\vu}_0}{\sqrt{|\Td|}} }^T, 
\qquad 
\mathcal{P}^\perp\vu
:=\Grad\Delta^{-1}\Div \vu, 
\qquad 
\mathcal{P}\vu:=\vu-\Grad\Delta^{-1}\Div \vu. 
\end{align*}
Applying the operator $\mathbb{P}^\perp$ to \eqref{1.1.3} gives 
\begin{align}\label{1.1}
\p_t \mathbb{P}^\perp(a_\ep,\vbe,\vue)^T 
+
\dfrac{L} {\ep}\mathbb{P}^\perp(a_\ep,\vbe,\vue)^T
+
\mathbb{P}^\perp
\mathcal{Q}\left((a_\ep,\vbe,\vue)^T, (a_\ep,\vbe,\vue)^T\right)
- 
\mathcal{D}(\mathbb{P}^\perp(a_\ep,\vbe,\vue)^T)
=
\mathbb{P}^\perp
r_\ep , 
\end{align}
where we defined for $\vA=(a, b, \vu)^T$ that 
\begin{align*}
\mathcal{Q}(\vA,\vA)&=
\begin{pmatrix}
\Div (a\vu)  \\
\Div(b \vu) \\
\vu \cdot \nabla \vu+c_3 a\Grad a+c_4 b\Grad a+c_5 a\Grad b+c_6 b\Grad b
\end{pmatrix}
, \\
\mathcal{D}(\vA)& 
=
(0,0,\mathcal{A}\vu)^{T}, \\ 
r_\ep
& 
=
(0,0,I(\ep a_\ep+\ep b_\ep)\mathcal{A}\vue )^{T} 
\\
& 
\qquad 
- 
\left( 
0,0, \dfrac{\nabla p(Z_\ep)}{\ep^2(R_\ep + Q_\ep)}-\left(c_1 \dfrac{\Grad a_\ep}{\ep}+c_2\dfrac{\Grad b_\ep}{\ep}+c_3 a_\ep\Grad a_\ep+c_4 b_\ep\Grad a_\ep+c_5 a_\ep\Grad b_\ep+c_6 b_\ep\Grad b_\ep\right)
\right)^{T}
\\
&
=
(0, 0, r^3_\ep)^T.
\end{align*}
To proceed, we set
\[
\mathcal{Q}(\vA_1,\vA_2)
=
\dfrac{1}{2}\Big(
Q(\vA_1+\vA_2, \vA_1+\vA_2)-Q(\vA_1,\vA_1)-Q(\vA_2,\vA_2)
\Big),
\qquad 
\vUe=\mathbb{P}(a_\ep,\vbe,\vue)^T.
\]
Applying $\mathbb{P}$ to \eqref{1.1.3}, we deduce\footnote{
Here, we use the fact that $\mathcal{L}\left(\df{t}{\ep}\right)\mathcal{L}\left(-\df{t}{\ep}\right) \vA= \mathcal{L}\left(-\df{t}{\ep}\right)\mathcal{L}\left(\df{t}{\ep}\right) \vA=\vA$, for any $\vA=(a, b, \vu)^T$ and $t\in \mathbb{R}$.
}
\begin{align}\label{1.2}
\p_t \vUe
&
+
\mathbb{P}\mathcal{Q}(\underline{\vUe}, \underline{\vUe})
+
2\mathbb{P}\mathcal{Q}(\mathbb{P}_0\vUe, \underline{\vUe})
- 
\mathcal{D}(\vUe) 
\nonumber
\\
&
=
\mathbb{P} r_\ep
-
2\mathbb{P}\mathcal{Q}\left(\underline{\vUe}, \mathcal{L}\left(\df{t}{\ep}\right) \vVe\right)
-
2\mathbb{P}\mathcal{Q}\left(\mathbb{P}_0\vUe, \mathcal{L}\left(\df{t}{\ep}\right) \vVe\right),
\end{align}
where we used  
$
\mathbb{P}\mathcal{Q}\left(\mathcal{L}\left(\df{t}{\ep}\right)\vVe, \mathcal{L}\left(\df{t}{\ep}\right)\vVe\right)=\mathcal{Q}(\mathbb{P}_0\vUe, \mathbb{P}_0\vUe)=0.
$
Let $\mathcal{L}(t)=\text{e}^{-tL}$ be the semigroup of solutions associated with $L$ and set  
$\vVe(t)=\mathcal{L}\left(-\dfrac{t}{\ep}\right)\mathbb{P}^\perp(a_\ep,\vbe,\vue)^T$.
Applying $\mathcal{L}\left(-\dfrac{t}{\ep}\right)$ to \eqref{1.1} yields 
\begin{align}\label{1.6}
\p_t \vVe 
+
&
\mathcal{Q}_1^\ep(\vVe, \vVe)+\mathcal{Q}_2^\ep(\underline{\vUe}, \vVe)
+
\mathcal{Q}_2^\ep(\mathbb{P}_0\vUe, \vVe)
- 
\mathcal{D}^\ep(\vVe)
\nonumber
\\
&
= 
\mathcal{L}\left(-\dfrac{t}{\ep}\right) \mathbb{P}^\perp r_\ep
-
\mathcal{L}\left(-\dfrac{t}{\ep}\right) \mathbb{P}^\perp\mathcal{Q}(\underline{\vUe}, \underline{\vUe})
-
2\mathcal{L}\left(-\dfrac{t}{\ep}\right) \mathbb{P}^\perp\mathcal{Q}(\mathbb{P}_0 \vUe, \underline{\vUe}),
\end{align}
where we set for $\vA, \vB\in \text{Im}\,L$ and $\vE\in \text{Ker}\, L$ 
\begin{align*}
& \underline{\vE}
:= \vE-\mathbb{P}_0\vE,
\qquad 
\mathcal{Q}_1^\ep(\vA, \vB)
:=
2\mathcal{L}\left(-\dfrac{t}{\ep}\right) 
\mathbb{P}^\perp
\mathcal{Q}
\left(\mathcal{L}\left(\dfrac{t}{\ep}\right)\vA, \mathcal{L}\left(\dfrac{t}{\ep}\right) \vB\right), \\
&
\mathcal{Q}_2^\ep(\vE, \vA)
:=
2\mathcal{L}\left(-\dfrac{t}{\ep}\right) 
\mathbb{P}^\perp
\mathcal{Q}\left(\vE, \mathcal{L}\left(\dfrac{t}{\ep}\right) \vA\right), 
\qquad
\mathcal{D}^\ep(\vA)=\mathcal{L}\left(-\dfrac{t}{\ep}\right) \mathcal{D}\left(\mathcal{L}\left(\dfrac{t}{\ep}\right) \vA\right).
\end{align*}

If $(a_{0,\ep}, b_{0, \ep}, \vu_{0, \ep})\rightarrow (a_0, b_0, \vu_0)$ in some sense, we can expect that $\vVe \rightarrow \vV$ and $\vUe  \rightarrow \vU$,  where $\vV\in \text{Im} \, L$ and $\vU \in \text{Ker} \, L$. Here, $\vU$ should solve 
\begin{equation}\label{1.3}
\left\{\begin{aligned} 
&
\p_t \vU
+ 
\mathbb{P}\mathcal{Q}(\underline{\vU}, \underline{\vU})
+
2\mathbb{P}\mathcal{Q}(\mathbb{P}_0\vU, \underline{\vU})
-
\mathcal{D}\vU
=0,
\\   
&
\vU(0)=\mathbb{P} (a_0, b_0, \vu_0)^T.
\end{aligned}
\right.
\end{equation}
Writing $\vU=(\vU^1, \vU^2, \vU^3)$, we see that
$\vU^1$ and $\vU^2$ satisfy the transport equations, while $\vU^3$ satisfies the incompressible Navier--Stokes equations, see \eqref{5.3}. Indeed, from \eqref{5.2} we know that $\vU$ takes 
\begin{align*}
\vU
=
(\vU^1, \vU^2, \vU^3)
=
\sp{ \rho, -\df{c_1}{c_2}\rho, \vw }^T
+
\sp{ 
\df{\widehat{(a_0)}_0}{\sqrt{|\Td|}},
\df{\widehat{(b_0)}_0}{\sqrt{|\Td|}},
\df{\widehat{(\vu_0)}_0}{\sqrt{|\Td|}}
}^T,    
\end{align*}
with  $\widehat{\rho}_0=\widehat{\vw}_0=0$ and $\Div \vw=0$. On the other hand, $\vV$ should solve 
\begin{equation}\label{1.4}
\left\{\begin{aligned}
&
\p_t \vV 
+
\mathcal{Q}_1(\vV, \vV)
+
\mathcal{Q}_2(\underline{\vU}, \vV)
+
\mathcal{Q}_2(\mathbb{P}_0 \vU, \vV)
- 
\overline{\mathcal{D}}(\vV)=0,
\\
&
\vV(0)=\mathbb{P}^\perp (a_0, b_0, \vu_0)^T,
\end{aligned}
\right.
\end{equation}
where $\vU$ satisfies \eqref{1.3} and $\mathcal{Q}_1$, $\mathcal{Q}_2$ and  $\overline{\mathcal{D}}$ are the limit operators of $\mathcal{Q}^\ep_1, \mathcal{Q}^\ep_2$ and $\mathcal{D}^\ep$, respectively. The limits of these operators will be derived in Section \ref{decay}.

\subsection{The main theorem}

For brevity, we assume that the initial data $(a_{0,\ep}, b_{0, \ep}, \vu_{0, \ep})$ are independent of $\ep$, i.e., $(a_{0,\ep}, b_{0, \ep}, \vu_{0, \ep})\equiv (a_0, b_0, \vu_0)$. We emphasize that similar result holds if $(a_{0,\ep}, b_{0, \ep}, \vu_{0, \ep})$ converge to $(a_0, b_0, \vu_0)$ suitably.

We introduce some notations. It follows from \eqref{1.1.3}$_1$ and \eqref{1.1.3}$_2$ that
\begin{align*}
\widehat{(\vae)}_0 \equiv \widehat{(a_0)}_0,
\qquad 
\widehat{(\vbe)}_0 \equiv \widehat{(b_0)}_0.
\end{align*}
By setting
\begin{align*}
\vthe
:=
\dfrac{c_1\underline{a_\ep}+c_2\underline{b_\ep}}{c_1+c_2\f{Q^0}{R^0}},
\qquad 
\rho_\ep
:=
\dfrac{Q^0\underline{a_\ep}-R^0\underline{b_\ep}}{Q^0+R^0\f{c_1}{c_2}},
\end{align*}
we obtain
\begin{align}
&
\mathbb{P}(a_\ep, b_\ep, \vue)^T
= 
\sp{ \rho_\ep, -\df{c_1}{c_2}\rho_\ep, \mathcal{P}\underline{\vue}  }^T
+
\mathbb{P}_0(a_\ep, b_\ep, \vue)^T
=
\sp{ \rho_\ep, -\df{c_1}{c_2}\rho_\ep, \mathcal{P}\underline{\vue}  }^T
+
\sp{ \df{\widehat{(a_0)}_0}{\sqrt{|\Td|}}, \df{\widehat{(b_0)}_0}{\sqrt{|\Td|}}, \df{\widehat{(\vue)}_0}{\sqrt{|\Td|}}  }^T,
\label{decom1}
\\
& 
\mathbb{P}^\perp(a_\ep, b_\ep, \vue)^T
=
\sp{ \vthe, \df{Q^0}{R^0}\vthe, \mathcal{P}^\perp \vue }^T.
\label{decom2}
\end{align}
Our main theorem is stated as follows, where the definition of the truncated norm is given in Section \ref{preli} and $\beta_0$ is the positive constant from Proposition \ref{Pro6.2}.

\begin{Theorem}\label{main}
Let $0<\theta<1$ and $(a_0, b_0, \vu_0)\in B^{\f d 2}_{2, 1} \times B^{\f d 2}_{2, 1}\times B^{\f d 2-1}_{2, 1}$. Then there exists $0<T^*_0\leq \infty$ such that \eqref{1.3} admits a unique solution $\vU \in \text{Im}\, L$ on $[0, T^*_0]$ satisfying
\begin{align*}
\vU
\in 
\widetilde{C}_{T^*_0}(B^{\f d 2}_{2, 1})
\times
\widetilde{C}_{T^*_0}(B^{\f d 2}_{2, 1})
\times
\sp{
\widetilde{C}_{T^*_0}(B^{\f d 2-1}_{2, 1})
\cap
L_{T^*_0}^1(\underline{B}^{\f d 2+1}_{2, 1})
},
\end{align*}
and \eqref{1.4} admits a unique solution $\vV \in \text{Ker} \,L$ on $[0, T^*_0]$ satisfying
\begin{align*}
\vV \cap \widetilde{C}_{T^*_0}(B^{\f d 2-1}_{2, 1})\cap L^1_{T^*_0}(B^{\f d 2+1}_{2, 1}).
\end{align*}
For any $0<T_0< \infty$ with $T_0\leq T^*_0$, there exists $\ep_0=\ep_0( T_0, a_0, b_0, \vu_0, \vV, \vU)>0$ such that for any $0<\ep \leq \ep_0$, \eqref{1.1.3} admits a unique solution on $[0, T_0]$ satisfying 
\begin{align*}
(\vae, \vbe, \vue) \in
\widetilde{C}_{T_0}(B^{\f d 2}_{2, 1})
\times
\widetilde{C}_{T_0}(B^{\f d 2}_{2, 1})
\times
\sp{
\widetilde{C}_{T_0}(B^{\f d 2-1}_{2, 1})
\cap
L_{T_0}^1(B^{\f d 2+1}_{2, 1})
}.
\end{align*}
Moreover, we have
\begin{enumerate}[(1)]
\item {\bf{Uniform Bounds:}} %The following estimates hold: 
\begin{align}\label{uni1}
&
\ep
\n{\vthe}^{h;\f{\beta_0}{\ep}}_{\widetilde{L}^\infty_{T_0}(B^{\f d 2}_{2, 1})}
+
\df{1}{\ep}
\n{\vthe}^{h;\f{\beta_0}{\ep}}_{L^1_{T_0}(B^{\f d 2}_{2, 1})}
+
\n{\vthe}^{l;\f{\beta_0}{\ep}}_{\widetilde{L}^\infty_{T_0}(B^{\f d 2-1}_{2, 1})\cap L^1_{T_0}(B^{\f d 2+1}_{2, 1})}
+
\n{\rho_\ep}_{\widetilde{L}^\infty_{T_0}(B^{\f d 2}_{2, 1})}
+
\n{\vue}_{\widetilde{L}^\infty_{T_0}(\underline{B}^{\f d 2-1}_{2, 1})\cap L^1_{T_0}(\underline{B}^{\f d 2+1}_{2, 1})}
\nonumber
\\
&
\qquad
+
|\widehat{(a_0)}_0|
+
|\widehat{(b_0)}_0|
+
\n{\widehat{(\vue)}_0}_{L^\infty(0, T_0)}
\\
&
\quad
\leq
C
\sp{
\n{\vV}_{\widetilde{L}_{T^*_0}(B^{\f d 2-1}_{2, 1})
\cap
L_{T^*_0}^1(B^{\f d 2+1}_{2, 1})}
+
\n{\vw}_{\widetilde{L}_{T^*_0}(B^{\f d 2-1}_{2, 1})
\cap
L_{T^*_0}^1(B^{\f d 2+1}_{2, 1})}
+
\n{\rho}_{\widetilde{L}_{T^*_0}(B^{\f d 2}_{2, 1})}
+
|\widehat{(a_0)}_0|
+
|\widehat{(b_0)}_0|
+
|\widehat{(\vu_0)}_0|
},
\nonumber
\\
&
\f{R^0}{2}\leq R^0+\ep\vae(t, x) \leq \f{3R^0}{2},
\quad 
\f{Q^0}{2}\leq Q^0+\ep\vbe(t, x) \leq \f{3Q^0}{2},
\quad 
\text{ for all } (t,x)\in[0,T_0] \times \Td.
\nonumber
\end{align}

\item {\bf{Decay Properties:}} %The following decay estimates hold:
\begin{align*}
&
\df{
\n{\vVe-\vV}_{
\widetilde{L}^\infty_{T_0}(B^{\f d 2-1-\theta}_{2, 2})
\cap
L^2_{T_0}(B^{\f d 2-\theta}_{2, 2})
}
}
{\tau^*(\ep)}
\\
&
\qquad
+
\df{\n{\rho_\ep-\rho}_{
\widetilde{L}^\infty_{T_0}(B^{\f d 2-1-\theta}_{2, 1})}
+
\n{\mathcal{P}\vue-\vw}_{
\widetilde{L}^\infty_{T_0}(\underline{B}^{\f d 2-1-\theta}_{2, 1})
\cap
L^1_{T_0}(\underline{B}^{\f d 2+1-\theta}_{2, 1})
}
+
\n{\widehat{(\vue)}_0-\widehat{(\vu_0)}_0}_{L^\infty(0, T_0)}
}
{\ep^{\f{\theta}{3+\theta}}}
\\
& \quad 
\leq
C(a_0, b_0, \vu_0, T_0, \vV, \vU),
\end{align*}
where $\tau^*(\ep)$ increases monotonically in $\ep$ and satisfies $\lim\limits_{\ep \to 0^{+}} \tau^*(\ep)=0$.

\item {\bf{Critical Convergence:}} %The following critical convergence result holds: 
\begin{align*}
&
\lim\limits_{\ep \to 0^+}
\ep
\sp{
\n{a_\ep}_{\widetilde{L}^\infty_T(B^{\f d 2}_{2,1})}
+
\n{b_\ep}_{\widetilde{L}^\infty_T(B^{\f d 2}_{2,1})}
}
+
\n{\vVe-\vV}_{
\widetilde{L}^\infty_{T_0}(B^{\f d 2-1}_{2, 1})
\cap
\widetilde{L}^2_{T_0}(B^{\f d 2}_{2, 1})
}
\\
&
\quad
+
\n{\rho_\ep-\rho}_{
\widetilde{L}^\infty_{T_0}(B^{\f d 2}_{2, 1})}
+
\n{\mathcal{P}\vue-\vw}_{
\widetilde{L}^\infty_{T_0}(\underline{B}^{\f d 2-1}_{2, 1})
\cap
L^1_{T_0}(\underline{B}^{\f d 2+1}_{2, 1})
}
+
\n{\widehat{(\vue)}_0-\widehat{(\vu_0)}_0}_{L^\infty(0, T_0)}
=
0.
\end{align*}

\end{enumerate}

\end{Theorem}

Several remarks are in order. 
\begin{Remark} 
Li et al. \cite{LLZ26} considered the low Mach number limit of \eqref{eq:res-two-fluid} with well-prepared initial data: 
\begin{align}\label{data}
R_{0, \ep}=1+O(\ep^2), 
\qquad
Q_{0, \ep}=1+O(\ep^2),
\qquad
\vu_{0, \ep}=\vu_0+O(\ep) \quad \text{ with } \quad \Div \vu_0=0, 
\end{align}
and the $3$D incompressible Navier--Stokes equations was identified as the limit system; in particular, the quantities $\vV$ and
$\rho$ do not appear. Now we explain the compatibility with the present work. We infer from \eqref{data} that 
\begin{align*}
a_{0, \ep}=O(\ep), 
\quad 
b_{0, \ep}=O(\ep),
\quad 
\mathbb{P}^\perp (a_{0, \ep}, b_{0, \ep}, \vu_{0, \ep})
=
(O(\ep), O(\ep), O(\ep)),
\quad 
\mathbb{P} (a_{0, \ep}, b_{0, \ep}, \vu_{0, \ep})
=
(O(\ep), O(\ep), \vu_{0}+O(\ep)), 
\end{align*}
which ensures 
\begin{align*}
\lim\limits_{\ep \to 0^+}
\mathbb{P}^\perp (a_{0, \ep}, b_{0, \ep}, \vu_{0, \ep})
=
(0, 0, 0),
\qquad
\mathbb{P}^\perp (a_{0, \ep}, b_{0, \ep}, \vu_{0, \ep})
=
(0, 0, \vu_0).
\end{align*}
Thus, the initial data for \eqref{1.4} is $(0, 0, 0)$, which implies $\vV \equiv 0$, while the initial data for \eqref{1.3} is $(0,0,\vu_0)$, which implies that the solution takes the form $\vU=(\vU^1,\vU^2,\vU^3)=(0,0,\vw+\widehat{(\vu_0)}_0/\sqrt{\T})$ with $\widehat{\vw}_0=0$ and $\Div \vw=0$. In particular, since $\vU^1 \equiv 0$, we also have $\rho \equiv 0$. Consequently, only the solution to the $3$D incompressible Navier--Stokes equations is observed in the limit.
\end{Remark}

\begin{Remark}
It seems difficult to give an explicit expression of $\tau^*(\ep)$, as the coefficients $\widetilde{C}_M$ and $\widecheck{C}_M$ (defined in \eqref{xishu1} and \eqref{xishu2}, respectively) cannot be determined in general. These coefficients depend on the periodic parameter $\mathbf{a}$, and this may cause the decay rate of $\tau^*(\ep)$ slower than any exponential decay in $\ep$. For certain suitable choices of $\mathbf{a}$, such as $\mathbf{a}=(1, 1, \cdots, 1)$, the coefficients $\widetilde{C}_M$ and $\widecheck{C}_M$ can be explicitly determined, and consequently $\tau^*(\ep)$ can be calculated explicitly. However, the loss of decay rate of $\tau^*(\ep)$ does not have any influence on the validity of Theorem \ref{main}. The only property that matters is $\lim\limits_{\ep \to 0^{+}} \tau^*(\ep)=0$. 
\end{Remark}

\begin{Remark}
By \eqref{decom1} and \eqref{decom2}, we have
\begin{align*}
(\vae, \vbe, \vue)^T
-
\vU
&
=
\mathbb{P}^\perp (a_\ep, b_\ep, \vue)^T
+
\mathbb{P}(a_\ep, b_\ep, \vue)^T
-
\vU
\\
&
=
\mathcal{L}\left(\dfrac{t}{\ep}\right)
\sp{
\vVe
-
\vV
}
+
\mathcal{L}\left(\dfrac{t}{\ep}\right)
\vV
+
\sp{ \rho_\ep-\rho, -\df{c_1}{c_2}(\rho_\ep-\rho), \mathcal{P}\underline{\vue}-\vw }^T
+
\sp{0, 0,  \df{\widehat{(\vue)}_0-\widehat{(\vu_0)}_0}{\sqrt{|\Td|}} }^T
\\
&
=:
I_1+I_2+I_3+I_4.
\end{align*}
It follows from Theorem \ref{main} that the terms $I_1$, $I_3$ and $I_4$ converge strongly to zero, whereas $I_2$ converges only weakly to zero. Consequently, $(\vae, \vbe, \vue)^T$  converges weakly to 
$\vU$. In the special case $\vV=0$, we have $I_2=0$; thus the convergence becomes strong. This explains why the strong convergence result can be obtained in case of well-prepared initial data, see \cite{LLZ26}. 
\end{Remark}

\begin{Remark}
The approach developed in this paper is robust enough to treat the low Mach number limit of other two-fluid models with periodic boundary conditions and ill-prepared initial data, such as the fluid-particle two-phase model \cite{Kwo-Li-19} and the liquid-gas two-phase model \cite{Yao-Zhu-Zi-12}. 
\end{Remark}

\begin{Remark}
The estimates for $\vae$ and $\vbe$ can be obtained from Theorem \ref{main} and the following equivalent relations:
\begin{align*}
&
\n{a_\ep}_{\widetilde{L}^q_T(B^{s}_{p,r})} 
\approx \n{a_\ep}_{\widetilde{L}^q_T(\underline{B}^{s}_{p,r})}
+
T^{\f 1 q}|\widehat{(a_0)}_0|,
\ \ 
\n{b_\ep}_{\widetilde{L}^q_T(B^{s}_{p,r})} 
\approx
\n{b_\ep}_{\widetilde{L}^q_T(\underline{B}^{s}_{p,r})}
+
T^{\f 1 q}|\widehat{(b_0)}_0|,    
\\
&
\n{a_\ep}_{\widetilde{L}^q_T(\underline{B}^{s}_{p,r})}+\n{b_\ep}_{\widetilde{L}^q_T(\underline{B}^{s}_{p,r})}
\approx \n{\rho_\ep}_{\widetilde{L}^q_T(B^{s}_{p,r})}
+
\n{\vthe}_{\widetilde{L}^q_T(B^{s}_{p,r})},
\end{align*}
where $0<T<\infty$, $1\leq p, q, r\leq \infty$ and $s\in \mathbb{R}$.
\end{Remark}

\subsection{Difficulties, ideas and remarks  on \texorpdfstring{$T_0=T^*_0=\infty$}{}}
In the proof of Theorem \ref{main}, three main difficulties arise: the size of the initial data, the growth of the linear estimate, and the growth of the nonlinear estimate. These growth issues are fundamentally driven by $(\widehat{(a_0)}_0, \widehat{(b_0)}_0) \neq (0, 0)$ and the term $\rho_\ep$. Since the limit of  $\rho_\ep$
is $\rho$, and $\rho$ satisfies only the transport equation (see \eqref{5.3}$_1$), the only uniform bound available for $\rho$ is in $\widetilde{L}^\infty_t(B^{\f d 2}_{2, 1})$. Consequently, the same is true for $\rho_\ep$: only a uniform bound in $\widetilde{L}^\infty_t(B^{\f d 2}_{2, 1})$ can be established. Hence, for any $\widetilde{L}^q_t(B^{\f d 2-s}_{2, 1})$ norm  with $1\leq q<\infty$ and $0\leq s<\infty$, the estimate on $\rho_\ep$ suffers from a 
$t^{\f 1 q}$-growth:
\begin{align*}
\n{\rho_\ep}_{
\widetilde{L}^q_t(B^{\f d 2-s}_{2, 1})}
\les
t^{\f 1 q}
\n{\rho_\ep}_{
\widetilde{L}^\infty_t(B^{\f d 2}_{2, 1})}.
\end{align*}
In the linear estimate, the term
\begin{align*}
L_1:=
\sp{
\f{c_3 \widehat{(a_0)}_0}{\sqrt{\T}}
+
\f{c_4 \widehat{(b_0)}_0}{\sqrt{\T}}
+
\f{c_5 Q^0 \widehat{(a_0)}_0}{R^0\sqrt{\T}}
+
\f{c_6 Q^0\widehat{(b_0)}_0}{R^0\sqrt{\T}}
}
\Grad \rho_\ep
\end{align*}
arises, and we are forced to estimate $\n{L_1}^{m;\zeta;\f{\beta_0}{\ep}}_{L^1_t(B^{\f d 2-1}_{2, 1})}$, which introduces a $t$-growth.
In the nonlinear estimate, the term
\begin{align*}
N_1:=
\sp{
c_3+\f{c_6 c^2_1}{c^2_2}-\f{c_1c_4}{c_2}-\f{c_1c_5}{c_2}
}
\rho_\ep
\Grad \rho_\ep
\end{align*}
arises, and we are forced to estimate $\n{N_1}^{m;\zeta;\f{\beta_0}{\ep}}_{L^1_t(B^{\f d 2-1}_{2, 1})}$, which also introduces a $t$-growth.

However, we must establish a uniform  estimate for $E_t^\ep$, with respect to both $t$ and $\ep$, which is precisely the left-hand side of \eqref{uni1}. To this end, we work within the high-middle-low frequency decomposition method developed by Fujii \cite{F24} for single phase Navier--Stokes system in $\mathbb{R}^d$, and follow the idea of Li \cite{L26}, establishing two fundamental properties: a vanishing property for the high-middle-frequency parts, and a decay property for the low-frequency part. Guided by this idea, we design the functional $E^{\ep, \zeta}_t$, whose definition is given in Section \ref{construct}. Specifically, we employ the filtering technique to derive the decay estimates of $\vZe$ and $\vWe$ (see Section \ref{construct} for their definitions), as stated in Propositions \ref{Pro4.1} and \ref{Pro4.2}. Then, by virtue of \eqref{new8.16}, a decay property for the low-frequency part is established. For the high-middle-frequency parts of the solution to the compressible two-fluid model, we establish new a priori estimates that depend explicitly on time, the frequency parameter, and the Mach number, see Proposition \ref{Pro6.4}. As discussed above, these a priori estimates inevitably grow in time, yet they still imply the vanishing property of the high-middle-frequency parts when $T_0<\infty$. When establishing the a priori estimates for the high-middle-frequency parts, two main difficulties arise: one is the coupling with the low-frequency part, the other being the composite function structure of the pressure term; these are discussed in detail in Subsection \ref{nolin}. With the aid of \eqref{cons3} and delicate computations that fully exploit the high-middle-frequency energy structure, we are able to overcome these difficulties. Eventually, combining \eqref{cons5} with $T_0<\infty$ yields the desired uniform estimate for $E^\ep_t$.

Finally, in the case that $T_0=T^*_0=\infty$, both the growth of linear estimate and nonlinear estimate tend to infinity. Of course, one may impose the condition $(\widehat{(a_0)}_0, \widehat{(b_0)}_0) = (0, 0)$ on the initial data to circumvent the difficulty caused by the growth of linear estimate. However, the growth of nonlinear estimate persists and cannot be handled by our method. On the other hand, one might ask whether, without adopting the approach of the present paper, it is possible to eliminate the nonlinear growth by exploiting the specific structure of the equations. This appears a delicate issue, since the pressure term does not admit an explicit expression due to the algebraic closure.

%\subsection{Notations}
Before ending this section, we introduce some notations. Throughout this paper, the letter $C$ denotes a generic positive constant changing from line to line. We write $C(a_1,a_2,\cdots,a_n)$ to
indicate a positive constant that depends only on $a_1,a_2,\cdots,a_n$.  The notation $a_1\lesssim  a_2$  means that $a_1\leq Ca_2$, and  $a_1\approx a_2$ indicates that $C^{-1}a_1\leq a_2\leq Ca_1$. For two Banach spaces $B_1$ and $B_2$, the norm on their intersection is defined by $\|\cdot\|_{B_1\cap B_2}:=\|\cdot\|_{B_1}+\|\cdot\|_{B_2}$. For two operators $A$ and $B$, the commutator is defined as $[A, B]f:=(AB)f-B(Af)$.

%%%%%%%%%%%%%%%%%%%%%%%%%%%%%%%%%%%%%%%%%%%%%%%%%%%%%%%%%%%%%%%%%%%%%%%%%%%
\section{Preliminaries}\label{preli}

\subsection{Function Spaces}
Let $\varphi: \mathbb{R} \rightarrow [0,1]$ be a radial smooth function supported in $\{\xi\in\mathbb{R}:3/4\leq|\zeta|\leq 8/3\}$ satisfying 
%the dyadic partition of unity condition:
\[
\sum_{j\in\mathbb{Z}}\varphi(2^{-j}\xi)=1 \quad 
\text{ for any } \xi\in\mathbb{R}\backslash\{0\}.
\]
The construction of such $\varphi$ is classical, see for instance \cite{BCD11}. For $g\in\mathcal{S}'$, $k \in \widetilde{\mathbb{Z}}^d$ and $j\in\mathbb{Z}$, we define
\begin{align}\label{fourier} 
 \widehat{g}_k := \left< g,\f{\text{e}^{-i k \cdot x}} {\sqrt{|\Td|}} \right>_{\mathcal{S}'(\Td) \times \mathcal{S}(\Td)}, 
 \quad 
 \Delta_j g  := \sum\limits_{k \in \widetilde{\mathbb{Z}}^d}\varphi(2^{-j}|k|)\widehat{g}_k\f{\text{e}^{i k \cdot x}} {\sqrt{|\Td|}}, 
 \quad 
 S_j g := \f{\widehat{g}_0}{\sqrt{|\Td|}}+\sum\limits_{j'\leq j-1}\Delta_{j'} g,  
\end{align}
where  $\mathcal{S}$ and $\mathcal{S}'$ denote the space of smooth functions which are $2\pi \mathbf{a}^j$-periodic ($1\leq j\leq d$) in the $j$th variable and its dual space, $\widetilde{\mathbb{Z}}^d:=\mathbb{Z}/\mathbf{a}^1\times \mathbb{Z}/\mathbf{a}^2\times\cdots\times\mathbb{Z}/\mathbf{a}^d$ and $|\Td|:=2\pi \mathbf{a}^1\cdot 2\pi \mathbf{a}^2\cdots 2\pi \mathbf{a}^d$. This yields the decomposition:
\[
g=\sum\limits_{k \in \widetilde{\mathbb{Z}}^d} {{\widehat g}_k}\,\f{e^{ ik \cdot x}}{\sqrt{\T}}=\f{\widehat{g}_0}{\sqrt{|\Td|}}+\sum\limits_{j\in \mathbb{Z}} \Delta_j g.
\]
We set $j_\mathbf{a}:=\max\{j\in \mathbb{Z}:2^{-j}\cdot \min\{\f1 {\mathbf{a}^1},\cdots,\f1 {\mathbf{a}^d}\}\geq \f 8 3\}$ and observe that 
\[
\varphi(2^{-j}|k|)=0,\quad  \text{ if } j\leq j_\mathbf{a},\quad  k \in \widetilde{\mathbb{Z}}^d.
\]
It follows that
\beq\label{zero}
\Delta_j g\equiv0,\quad  \text{ if } j\leq j_\mathbf{a}.
\eeq

 For $1\leq p,r\leq \infty $ and $s\in\mathbb{R}$, we introduce
\begin{align*}
\intTd g &:= \int_{[0,2\pi \mathbf{a}_1)\times [0,2\pi \mathbf{a}_2)\times\cdots\times [0,2\pi \mathbf{a}_d)}g  \, \dx,
\qquad 
\|g\|_{L^p(\Td)}:=\left(\intTd {|g|^p}\right)^{\f1 p},\\
\|g\|_{B_{p,r}^s(\Td)}  &
:=
\left(\sum\limits_{j\in\mathbb{Z}}2^{jsr}\|\Delta_{j} g\|^{r}_{L^p(\Td)}+\left\|\f{\widehat{g}_0}{\sqrt{\T}}\right\|^{r}_{L^p(\Td)}\right)^{\f1 r},\\
 B_{p,r}^s(\Td)& :=
 \left\{g\in\mathcal{S}'(\Td):\|g\|_{B_{p,r}^s(\Td)}<\infty \right\}, \qquad   \underline{B}_{p,r}^s(\Td):=B_{p,r}^s(\Td)\cap \underline{\mathcal{S}}'(\Td),
\\
 H^s(\Td) & :=
 \left\{g\in\mathcal{S}'(\Td):\|g\|_{H^s(\Td)}:\left(|{\widehat{g}_0}|^2+\sum\limits_{k\neq0}|k|^{2s}|{\widehat{g}_k}|^2\right)^{\f1 2}<\infty \right\}.
\end{align*}
For any $s\in\mathbb{R}$, it is well-known that $B_{2,2}^s(\Td)=H^s(\Td)$ in the sense of norm equivalence. For $1\leq q,p,r\leq\infty$, $s\in\mathbb{R}$, $0<T\leq\infty$ and a time-dependent function $\Psi$ taking values in Besov spaces, we define
\begin{align*}
&L^q_T(B_{p,r}^s):=
\Mp{
\Psi:
\|\Psi\|_{L^q_T(B_{p,r}^s)} :=
\left\|\|\Psi(t,\cdot)\|_{B_{p,r}^s}\right\|_{L^q(0,T)}
<\infty
},
\\ 
&
\widetilde{L}^{q}_T(B_{p,r}^s):=
\Mp{\Psi:
\|\Psi\|_{\widetilde{L}^{q}_T(B_{p,r}^s)}:=\left(\sum\limits_{j\in\mathbb{Z}}2^{jsr}\|\Delta_{j} \Psi\|^{r}_{L^q(0,T;L^p)}+\left\|\f{\widehat{\Psi}_0}{\sqrt{\T}}\right\|^{r}_{L^q(0,T;L^p)}\right)^{\f1 r}
<\infty
}.
\end{align*}
The second norm was first introduced by Chemin and Lerner \cite{CL95}. The corresponding zero-mean versions are defined as:
\begin{align*}
&L^q_T(\underline{B}_{p,r}^s):=
\Mp{
\Psi:
\|\Psi\|_{L^q_T(\underline{B}_{p,r}^s)} :=
\|\underline{\Psi}\|_{L^q_T(B_{p,r}^s)}
<\infty
},
\\ 
&
\widetilde{L}^{q}_T(\underline{B}_{p,r}^s):=
\Mp{\Psi:
\|\Psi\|_{\widetilde{L}^{q}_T(B_{p,r}^s)}:=
\|\underline{\Psi}\|_{\widetilde{L}^{q}_T(B_{p,r}^s)}
<\infty
}.
\end{align*}
Due to Minkowski inequality, 
\begin{align}\label{minski}
\|\Psi\|_{L^q_T(B_{p,r}^s)}\leq\|\Psi\|_{\widetilde{L}^{q}_T(B_{p,r}^s)}, \quad \text{if}\ \ r\leq q;
\qquad 
\|\Psi\|_{L^q_T(B_{p,r}^s)}\geq\|\Psi\|_{\widetilde{L}^{q}_T(B_{p,r}^s)}, \quad \text{if}\ \ r\geq q.
\end{align}
Thanks to \eqref{zero}, we have for $1\leq p,q,r_1,r_2\leq\infty$ and $-\infty<s_1< s_2<+\infty$ that 
\begin{align}
 &\max\{\|\Psi\|_{L^q_T(B_{p,r_1}^{s_1})},\|\Psi\|_{\widetilde{L}^{q}_T(B_{p,r_1}^{s_1})}\}
 \lesssim\min\{\|\Psi\|_{L^{q}_T(B_{p,r_2}^{s_2})},\|\Psi\|_{\widetilde{L}^{q}_T(B_{p,r_2}^{s_2})}\} \label{minski2}
\end{align}
For $0\leq\zeta < \eta<\infty$, we define the high, middle and low frequency norms as  
\begin{align*}
& 
\|g\|^{h;\eta}_{B_{p,r}^s}:=\left(\sum\limits_{2^j\geq\eta}2^{jsr}\|\Delta_{j} g\|^{r}_{L^p(\Td)}\right)^{\f1 r},\qquad \|g\|^{m;\zeta,\eta}_{B_{p,r}^s}:=\left(\sum\limits_{\zeta\leq2^j<\eta}2^{jsr}\|\Delta_{j} g\|^{r}_{L^p(\Td)}\right)^{\f1 r},
\\
&
\|g\|^{l;\zeta}_{B_{p,r}^s}:=\left(\sum\limits_{2^j<\zeta}2^{jsr}\|\Delta_{j} g\|^{r}_{L^p(\Td)}+\left\|\f{\widehat{g}_0}{\sqrt{\T}}\right\|^{r}_{L^p(\Td)}\right)^{\f1 r},
\\
&
\|\Psi\|^{h;\eta}_{L^q_T(\dot{B}_{p,r}^s)}:=\left\|\|\Psi\|^{h;\eta}_{B_{p,r}^s}\right\|_{L^q(0,T)},
\\
&
\|\Psi\|^{m;\zeta,\eta}_{L^q_T(B_{p,r}^s)}:=\left\|\|\Psi\|^{m;\zeta,\eta}_{B_{p,r}^s}\right\|_{L^q(0,T)}, \qquad  \|\Psi\|^{l;\zeta}_{L^q_T(B_{p,r}^s)}:=\left\|\|\Psi\|^{l;\zeta}_{B_{p,r}^s}\right\|_{L^q(0,T)},
\\
&
\|\Psi\|^{h;\eta}_{\widetilde{L}^q_T(\dot{B}_{p,r}^s)}:=\left(\sum\limits_{2^j\geq\eta}2^{jsr}\|\Delta_{j} \Psi\|^{r}_{L^q(0,T;L^p)}\right)^{\f1 r},
\\
&
\|\Psi\|^{m;\zeta,\eta}_{\widetilde{L}^q_T(\dot{B}_{p,r}^s)}:=\left(\sum\limits_{\zeta\leq2^j<\eta}2^{jsr}\|\Delta_{j} \Psi\|^{r}_{L^q(0,T;L^p)}\right)^{\f1 r},
\\
& 
\|\Psi\|^{l;\zeta}_{\widetilde{L}^q_T(B_{p,r}^s)}
:=
\left(\sum\limits_{2^j<\zeta}2^{jsr}\|\Delta_{j} \Psi\|^{r}_{L^q(0,T;L^p)}+\left\|\f{\widehat{\Psi}_0}{\sqrt{\T}}\right\|^{r}_{L^q(0,T;L^p)}\right)^{\f1 r}.
\end{align*}
The corresponding zero-mean versions are defined as:
\[
\|g\|^{l;\zeta}_{\underline{B}_{p,r}^s}:=\|\underline{g}\|^{l;\zeta}_{B_{p,r}^s},\quad  
\|\Psi\|^{l;\zeta}_{L^q_T(\underline{B}_{p,r}^s)}:=\|\underline{\Psi}\|^{l;\zeta}_{L^q_T(B_{p,r}^s)},
\quad
\|\Psi\|^{l;\zeta}_{\widetilde{L}^q_T(\underline{B}_{p,r}^s)}:=\|\underline{\Psi}\|^{l;\zeta}_{\widetilde{L}^q_T(B_{p,r}^s)}.
\]
%It is clear that
%$\|g\|^{h;\eta}_{B_{p,r}^s}\approx \|g^{h;\eta}\|_{B_{p,r}^s}$,  $\|g\|^{m;\zeta;\eta}_{B_{p,r}^s}\approx \|g^{m;\zeta,\eta}\|_{B_{p,r}^s}$, $\|g\|^{l;\zeta}_{B_{p,r}^s}\approx \|g^{l;\zeta}\|_{B_{p,r}^s}$, $\|\Psi\|^{h;\eta}_{\widetilde{L}^q_T(\dot{B}_{p,r}^s)}\approx\|\Psi^{h;\eta}\|_{\widetilde{L}^q_T(\dot{B}_{p,r}^s)}$, $\|\Psi\|^{m;\zeta,\eta}_{\widetilde{L}^q_T(\dot{B}_{p,r}^s)}\approx\|\Psi^{m;\zeta,\eta}\|_{\widetilde{L}^q_T(\dot{B}_{p,r}^s)}$,
%$\|\Psi\|^{l;\zeta}_{\widetilde{L}^q_T(\dot{B}_{p,r}^s)}\approx\|\Psi^{l;\zeta}\|_{\widetilde{L}^q_T(\dot{B}_{p,r}^s)}$.

Finally, we introduce Bony's para-product decomposition. For $g,h\in\mathcal{S}'(\Td)$, we have
\begin{align}
gh=T_{g}h+T_{h}g+R(g,h)+\f{\widehat{g}_0}{\sqrt{|\Td|}}\cdot\f{\widehat{h}_0}{\sqrt{|\Td|}}
\nonumber
\end{align}
where
\begin{align}
\quad T_{g}h:=\sum\limits_{j\in\mathbb{Z}}S_{j-2} g\Delta_{j} h,\qquad
R(g,h):=\sum\limits_{|j-j'|\leq2}\Delta_{j} g\Delta_{j'} h.
\nonumber
\end{align}
We additionally define $T'_{g}h:=T_{h}g+R(g,h)$. The following almost orthogonality properties are well-known:
\begin{align}
& 
\Delta_{j'}\Delta_{j} g\equiv0, \quad \text{ if }\, |j'-j|\geq2, \qquad   \Delta_{j'}( S_{j-2}h\Delta_{j} g)\equiv0, \quad \text{ if }\, |j'-j|\geq3, \label{ao1}  \\
&
\Delta_{j''}(\Delta_{j} g\Delta_{j'} h)\equiv0, \quad \text{ if } j''-j\geq5,\,|j'-j|\leq2. \label{ao2}
\end{align}

\begin{Remark} We introduce the norms $\|\cdot\|_{L^q_T(\underline{B}_{p,r}^s)}$, $\|\cdot\|_{\widetilde{L}^{q}_T(\underline{B}_{p,r}^s)}$ and $\|\cdot\|^{l;\zeta}_{\widetilde{L}^q_T(\underline{B}_{p,r}^s)}$ with the primary aim of eliminating the influence of the zero mode. As evident from their definitions, these norms satisfy the following relations:
\begin{align*}
& 
\|\Psi\|_{L^q_T(\underline{B}_{p,r}^s)}=\left\|\left(\sum\limits_{j\in\mathbb{Z}}2^{jsr}\|\Delta_{j} \Psi(\cdot)\|^{r}_{L^p(\Td)}\right)^{\f1 r}\right\|_{L^q(0,T)},
\\
& 
\|\Psi\|_{\widetilde{L}^{q}_T(\underline{B}_{p,r}^s)}=\left(\sum\limits_{j\in\mathbb{Z}}2^{jsr}\|\Delta_{j} \Psi\|^{r}_{L^q(0,T;L^p)}\right)^{\f1 r},  \qquad 
\|\Psi\|^{l;\zeta}_{\widetilde{L}^q_T(\underline{B}_{p,r}^s)}=\left(\sum\limits_{2^j<\zeta}2^{jsr}\|\Delta_{j} \Psi\|^{r}_{L^q(0,T;L^p)}\right)^{\f1 r}.
\end{align*} 
\end{Remark}

\subsection{Basic product estimates on \texorpdfstring{$\Td$}{}}

In this subsection, we introduce the product estimates in the torus. Compared with $\mathbb{R}^d$, the distinction lies mainly in the handling of the zero and non-zero modes.

\begin{Lemma}\label{le4.1}
Let $f, g \in\mathcal{S}'(\Td) $, $1\leq p, p_1, p_2, r, r_1, r_2\leq \infty$, $-\infty<s, s_1, s_2<\infty$. Then we have
\begin{align}
&
\widehat{(T_g f)}_0=\widehat{(T_f g)}_0=0\label{4.6},
\\
&
|\widehat{(R(f, g))}_0|
\les
\n{f}_{\underline{B}^{s_1}_{p_1, r_1}}
\n{g}_{\underline{B}^{s_2}_{p_2, r_2}},  \qquad
\text{ if } 1 \leq \f1 {r_1}+ \f1 {r_2}, \ \  1 = \f1 {p_1}+ \f1 {p_2}, \ \ s_1+s_2\geq0,\label{4.1}
\\
&
|\widehat{(fg)}_0|
\les
|\widehat{f}_0|
|\widehat{g}_0|
+
\n{f}_{\underline{B}^{s_1}_{p_1, r_1}}
\n{g}_{\underline{B}^{s_2}_{p_2, r_2}} ,  \qquad
\text{ if } 1 \leq \f1 {r_1}+ \f1 {r_2}, \ \  1 = \f1 {p_1}+ \f1 {p_2}, \ \ s_1+s_2\geq0,\label{4.2}
\\
&
\n{T_f g}_{B^{s}_{p, r}}
\les 
\n{f}_{L^\infty}
\n{g}_{\underline{B}^{s}_{p, r}},\label{4.3}
\\
&
\n{T_f g}_{B^{s}_{p, r}}
\les 
\n{f}_{B^{s_1}_{p_1, r_1}}
\n{g}_{\underline{B}^{s_2}_{p_2, r_2}},
\qquad
\text{ if } \f 1 r\leq \f1 {r_1}+ \f1 {r_2}, \ \ \f 1 p= \f1 {p_1}+ \f1 {p_2}, \ \ s=s_1+s_2, \ \ s_1<0,\label{4.4}
\\
&
\n{R(f, g)}_{\underline{B}^{s}_{p, r}}
\les
\n{f}_{\underline{B}^{s_1}_{p_1, r_1}}
\n{g}_{\underline{B}^{s_2}_{p_2, r_2}}, \qquad
\text{ if } \f 1 r\leq \f1 {r_1}+ \f1 {r_2}, \ \ \f 1 p= \f1 {p_1}+ \f1 {p_2}, \ \ 0<s=s_1+s_2, \label{4.5}
\\
&
\n{fg}_{B^{s}_{2, r}} 
\les
\n{f}_{B^{s_1}_{2, r}}
\n{g}_{B^{s_2+\f d 2}_{2, 1}},
\qquad
\text{ if }  -\f d 2<s=s_1+s_2, \ \ s_1< \f d 2,\ \ s_2\leq 0,
\label{4.9}
\\
&
\n{fg}_{B^{s}_{2, r}} 
\les
\n{f}_{B^{\f d 2}_{2, 1}}
\n{g}_{B^{s_2+\f d 2}_{2, 1}},
\qquad
\text{ if }  -\f d 2<s=\f d 2+s_2, \ \ s_2\leq 0,
\label{4.8}
\end{align}
provided that the right-hand sides are finite. 
\end{Lemma}
\bProof
We first prove \eqref{4.6}. Notice that
\begin{align*}
\widehat{(T_g f)}_0
&=
\df{1}{\sqrt{|\Td|}}
\sum\limits_{j\in\mathbb{Z}}
\int_{\Td}S_{j-2} g\Delta_{j} f\dx
\\
&=
\df{1}{\sqrt{|\Td|}}
\sp{
\sum\limits_{j\in\mathbb{Z}}
\int_{\Td} \df{\widehat{g}_0}{\sqrt{\Td}}\Delta_{j} f\dx
+
\sum\limits_{j\in\mathbb{Z}}
\sum\limits_{j'\leq j-3}
\int_{\Td}\Delta_{j'} g \Delta_{j} f\dx
}
\\
&=
\df{1}{\sqrt{|\Td|}}
\sum\limits_{j\in\mathbb{Z}}
 \sum\limits_{j'\leq j-3}
 \int_{\Td}g \Delta_{j'}\Delta_{j} f\dx
=0,
\end{align*}
where we used \eqref{ao1}. Next we prove \eqref{4.1}. Without loss of generality, we assume 
$s_1\ge 0$. Since $s_1+s_2\ge 0$, it follows that $s_2\ge -s_1$. Set $\f 1 {r'_1}:=1-\f1 r_1$.  From $1\leq \f1 r_1+\f1 r_2$ we see that $r_2\leq r'_1$. Hence,
\begin{align*}
|\widehat{(R(f, g))}_0|
&
\les
\sum\limits_{|j-j'|\leq2}
\n{\Delta_{j} f}_{L^{p_1}}
\n{\Delta_{j'} g}_{L^{p_2}}
\\
&
\les
\n{f}_{\underline{B}^{s_1}_{p_1, r_1}}
\n{g}_{\underline{B}^{-s_1}_{p_2, r'_1}}
\les
\n{f}_{\underline{B}^{s_1}_{p_1, r_1}}
\n{g}_{\underline{B}^{-s_1}_{p_2, r_2}}
\les
\n{f}_{\underline{B}^{s_1}_{p_1, r_1}}
\n{g}_{\underline{B}^{s_2}_{p_2, r_2}} . 
\end{align*}
Next we show \eqref{4.2}. By Bony's para-product decomposition,
\[
gf=T_{g}f
+T_{f}g
+R(g,f)
+\f{\widehat{g}_0}{\sqrt{|\Td|}}\cdot\f{\widehat{f}_0}{\sqrt{|\Td|}},
\]
which gives \eqref{4.2} by \eqref{4.6} and \eqref{4.1}. By \eqref{4.6}, we have 
$
\n{T_f g}_{B^{s}_{p, r}}=\n{T_f g}_{\underline{B}^{s}_{p, r}};
$
whence \eqref{4.3} and \eqref{4.4} follow immediately from Theorem 2.47 in
\cite{BCD11}, while \eqref{4.5} is a direct consequence of Theorem 2.52 in \cite{BCD11}. Finally, we show \eqref{4.9}. From $s>-\f d 2$ , $s_1< \f d 2$, $s_2 \leq 0$, \eqref{4.2}, \eqref{4.3}, \eqref{4.4} and \eqref{4.5}, we get
\begin{align*}
\n{fg}_{B^{s}_{2, r}}
&
\leq
|\widehat{(fg)}_0|
+
\n{fg}_{\underline{B}^{s}_{2, r}}
\\
&
\leq
\widehat{(fg)}_0
+
\n{T_f g}_{\underline{B}^{s}_{2, r}}
+
\n{T_g f}_{\underline{B}^{s}_{2, r}}
+
\n{R(f, g)}_{\underline{B}^{s}_{2, r}}
\\
&
\les
\widehat{(fg)}_0
+
\n{f}_{B^{s_1-\f d 2}_{\infty, r}}
\n{g}_{B^{s_2+\f d 2}_{2, 1}}
+
\n{g}_{B^{s_2}_{\infty, 1}}
\n{f}_{B^{s_1}_{2, r}}
+
\n{R(f, g)}_{\underline{B}^{s+\f d 2}_{1, r}}
\\
&
\les
\n{f}_{B^{s_1}_{2, r}}
\n{g}_{B^{s_2+\f d 2}_{2, 1}}.
\end{align*}
%where we used $\n{g}_{L^\infty}\leq \n{g}_{B^{0}_{\infty, 1}}$, when $s_2=0$. 
\eqref{4.8} is verified similarly. This completes the proof.   \ \  $\Box$

\begin{Remark}
In Lemma \ref{le4.1}, we carefully distinguished between $B^{s}_{p, r}$ and $\underline{B}^{s}_{p, r}$. We briefly explain the motivation. Consider the heat equation on $\Td$,
\begin{equation*}
\left\{\begin{aligned}
&\p_t \vv-\De \vv=0, 
\\
&\vv(0)=\vv_0.
\end{aligned}
\right.
\end{equation*}
For $1\leq p,r\leq \infty$, $1\leq q<\infty$, $-\infty<s<\infty$, it is standard to obtain the maximal regularity estimate: 
\begin{align}\label{4.7}
\n{\vv}_{\widetilde{L}^\infty_T(B^{s}_{p, r})}
+
\n{\vv}_{\widetilde{L}^q_T(\underline{B}^{s+\f 2 q}_{p, r})}
\les
\n{\vv_0}_{B^{s}_{p, r}}.
\end{align}
On the other hand, since $\widehat{\vv}_0\equiv\widehat{(\vv_0)}_0$, we have
\begin{align}\label{r4.8}
\n{\vv}_{\widetilde{L}^q_T(B^{s+\f 2 q}_{p, r})}
\ge
\n{\df{\widehat{\vv}_0}
{\sqrt{\T}}
}_{L^q_T(L^p)}
=
\n{\df{\widehat{(\vv_0)}_0}
{\sqrt{\T}}
}_{L^q_T(L^p)}
=
T^{\f1 q}
\n{\df{\widehat{(\vv_0)}_0}
{\sqrt{\T}}
}_{L^p} . 
\end{align}
If $\widehat{(\vv_0)}_0\neq 0$, the estimates \eqref{4.7} and \eqref{r4.8} differ significantly, which forces us to distinguish between
$B^{s}_{p, r}$ and $\underline{B}^{s}_{p, r}$.
\end{Remark}

\section{Construction of \texorpdfstring{$E^\ep_T$ and $E^{\ep, \zeta}_T$}{}}\label{construct}
This section is devoted to the construction of the quantities 
$E^\ep_T$ and  $E^{\ep, \zeta}_T$, where $E^\ep_T$
denotes the energy functional associated with \eqref{1.1.3}, and 
$E^{\ep, \zeta}_T$ is obtained by replacing the low-frequency part with 
$\vVe-\vV$ and  $\vUe-\vU$, with 
$\vU$ and $\vV$ as in Theorem \ref{Th6.1} and Theorem \ref{Th7.2}, respectively. Then, we set $\vZe:=\vVe-\vV, \,\vWe:= \vUe-\vU$ and recall that 
\begin{align*}
&
\vU
=
(\vU^1, \vU^2, \vU^3)
=
\sp{\rho, -\df{c_1}{c_2}\rho, \vw}^T
+
\sp{
\df{\widehat{(a_0)}_0}{\sqrt{|\Td|}},
\df{\widehat{(b_0)}_0}{\sqrt{|\Td|}},
\df{\widehat{(\vu_0)}_0}{\sqrt{|\Td|}}
}^T,
\\
&
\mathbb{P}(a_\ep, b_\ep, \vue)^T
= 
\sp{\rho_\ep, -\df{c_1}{c_2}\rho_\ep, \mathcal{P}\underline{\vue}}^T
+
\sp{
\df{\widehat{(a_0)}_0}{\sqrt{|\Td|}}, \df{\widehat{(b_0)}_0}{\sqrt{|\Td|}}, \df{\widehat{(\vue)}_0}{\sqrt{|\Td|}}  }^T, 
\qquad 
\mathbb{P}^\perp(a_\ep, b_\ep, \vue)^T
=
\sp{ \vthe, \df{Q^0}{R^0}\vthe, \mathcal{P}^\perp \vue }^T.
\end{align*}
Then, for $0<\theta< 1$, $1\leq\zeta<\beta_0 /\ep$, $0<T\leq T^*_0$, where 
$T^*_0$ is given in Theorem \ref{Th6.1}, we define
\begin{align*}
&
\vZ^{\ep}_{T, \theta}
:=
\n{\vVe-\vV}_{\widetilde{L}^\infty_T(B^{\f d 2-1-\theta}_{2, 2})\cap L^2_T(B^{\f d 2-\theta}_{2, 2})} , 
\\[4pt]
&
\vW^{\ep}_{T, \theta}
:=
\n{\rho_\ep-\rho}_{\widetilde{L}^\infty_T(B^{\f d 2-1-\theta}_{2, 1})}
+
\n{\mathcal{P}\vue-\vw}_{\widetilde{L}^\infty_T(\underline{B}^{\f d 2-1-\theta}_{2, 1})
\cap
L^1_T(\underline{B}^{\f d 2+1-\theta}_{2, 1})}
+
\n{\widehat{(\vue)}_0-\widehat{(\vu_0)}_0}_{L^\infty(0, T)} , 
\\[4pt]
&
E^{\ep}_{T}
:=  
\ep
\n{\vthe}^{h;\f{\beta_0}{\ep}}_{\widetilde{L}^\infty_T(B^{\f d 2}_{2, 1})}
+
\df{1}{\ep}
\n{\vthe}^{h;\f{\beta_0}{\ep}}_{L^1_T(B^{\f d 2}_{2, 1})}
+
\n{\vthe}^{l; \f{\beta_0}{\ep}}_{\widetilde{L}^\infty_T(B^{\f d 2-1}_{2, 1})\cap L^1_T(B^{\f d 2+1}_{2, 1})}
+
\n{\rho_\ep}_{\widetilde{L}^\infty_T(B^{\f d 2}_{2, 1})}
+
\n{\vue}_{\widetilde{L}^\infty_T(\underline{B}^{\f d 2-1}_{2, 1})\cap L^1_T(\underline{B}^{\f d 2+1}_{2, 1})} 
\\
&
\qquad
\quad
+
|\widehat{(a_0)}_0|
+
|\widehat{(b_0)}_0|
+
\n{\widehat{(\vue)}_0}_{L^\infty(0, T)} , 
\\[4pt]
&
E^{\ep, \zeta}_{T}
:= 
\ep
\n{\vthe}^{h;\f{\beta_0}{\ep}}_{\widetilde{L}^\infty_T(B^{\f d 2}_{2, 1})}
+
\df{1}{\ep}
\n{\vthe}^{h;\f{\beta_0}{\ep}}_{L^1_T(B^{\f d 2}_{2, 1})}
+
\n{\vthe}^{m; \zeta, \f{\beta_0}{\ep}}_{\widetilde{L}^\infty_T(B^{\f d 2-1}_{2, 1})\cap L^1_T(B^{\f d 2+1}_{2, 1})}
+
\n{ \mathcal{P}^\perp\vue}^{h;\zeta}_{\widetilde{L}^\infty_T(B^{\f d 2-1}_{2, 1})\cap L^1_T(B^{\f d 2+1}_{2, 1})}
\\
&
\qquad
\quad
+
\n{\rho_\ep}^{h;\zeta}_{\widetilde{L}^\infty_T(B^{\f d 2}_{2, 1})}
+
\n{ \mathcal{P}\vue}^{h;\zeta}_{\widetilde{L}^\infty_T(B^{\f d 2-1}_{2, 1})\cap L^1_T(B^{\f d 2+1}_{2, 1})}
+
\n{\vVe-\vV}^{l; \zeta}_{\widetilde{L}^\infty_T(B^{\f d 2-1}_{2, 1})\cap L^1_T(B^{\f d 2+1}_{2, 1})}
\\[4pt]
&
\qquad
\quad
+
\n{ \mathcal{P}\vue-\vw}^{l;\zeta}_{\widetilde{L}^\infty_T(\underline{B}^{\f d 2-1}_{2, 1})\cap L^1_T(\underline{B}^{\f d 2+1}_{2, 1})}
+
\n{\rho_\ep-\rho}^{l;\zeta}_{\widetilde{L}^\infty_T(B^{\f d 2}_{2, 1})}
+
\n{\widehat{(\vue)}_0-\widehat{(\vu_0)}_0}_{L^\infty(0, T)},
\\[4pt]
&
\vV_{T^*_0}
:=
\n{\vV}_{\widetilde{L}^\infty_{T^*_0}(B^{\f d 2-1}_{2, 1})\cap L^1_{T^*_0}(B^{\f d 2+1}_{2, 1})},
\\
&
\vU_{T^*_0}
:=
\n{\rho}_{\widetilde{L}^\infty_{T^*_0}(B^{\f d 2}_{2, 1})}
+
\n{\vw}_{\widetilde{L}^\infty_{T^*_0}(B^{\f d 2-1}_{2, 1})\cap L^1_{T_0}(B^{\f d 2+1}_{2, 1})}
+
|\widehat{(a_0)}_0|
+
|\widehat{(b_0)}_0|
+
|\widehat{(\vu_0)}_0|,
\\[4pt]
&
E_0
:=
\n{a_0}_{B^{\f d 2}_{2, 1}}
+
\n{b_0}_{B^{\f d 2}_{2, 1}}
+
\n{\vu_0}_{B^{\f d 2-1}_{2, 1}}.
\end{align*}

Let $[T]:=\max\{1, T\}$. We have 
\begin{Lemma}\label{cons} Let $1\leq p,r,q\leq \infty$, $2\leq q_1 \leq \infty$, $ 0<\ep \leq 1$, $-\infty<s<\infty$, and $0<T\leq \infty$ with $T\leq T^*_0$. Then there exists a positive constant $C$ such that
\begin{align}
&
C^{-1}
\sp{
\n{a_\ep}_{\widetilde{L}^q_T(\underline{B}^{s}_{p,r})}
+
\n{b_\ep}_{\widetilde{L}^q_T(\underline{B}^{s}_{p,r})}
}
\leq
\n{\rho_\ep}_{\widetilde{L}^q_T(B^{s}_{p,r})}
+
\n{\vthe}_{\widetilde{L}^q_T(B^{s}_{p,r})}
\leq
C
\sp{
\n{a_\ep}_{\widetilde{L}^q_T(\underline{B}^{s}_{p,r})}
+
\n{b_\ep}_{\widetilde{L}^q_T(\underline{B}^{s}_{p,r})}
},\label{cons1}
\\
&
\ep
\n{\vthe}_{\widetilde{L}^\infty_T(B^{\f d 2}_{2,1})}
+
\n{\vthe}_{\widetilde{L}^\infty_T(B^{\f d 2-1}_{2,1})}
+
\n{\vthe}_{\widetilde{L}^2_T(B^{\f d 2}_{2,1})}
+
\n{ \mathbb{P}^\perp \vue}_{\widetilde{L}^\infty_T(B^{\f d 2-1}_{2, 1})\cap \widetilde{L}^2_T(B^{\f d 2}_{2, 1})}   \nonumber
\\
& \qquad 
+
\n{\vVe}_{\widetilde{L}^\infty_T(B^{\f d 2-1}_{2, 1})\cap \widetilde{L}^2_T(B^{\f d 2}_{2, 1})}
\leq
C
E^\ep_T,\label{cons2}
\\
&
\ep
\sp{
\n{a_\ep}_{\widetilde{L}^\infty_T(B^{\f d 2}_{2,1})}
+
\n{b_\ep}_{\widetilde{L}^\infty_T(B^{\f d 2}_{2,1})}
}
\leq
C
\min\left\{
E^{\ep}_{T},
\ \
\ep
\n{\vthe}^{h;\f{\beta_0}{\ep}}_{\widetilde{L}^\infty_T(B^{\f d 2}_{2, 1})}
+
\n{\vthe}^{m; \zeta, \f{\beta_0}{\ep}}_{\widetilde{L}^\infty_T(B^{\f d 2-1}_{2, 1})}
+
\zeta\ep E^{\ep}_{T}
\right\},    \label{cons3}
\\
&
\ep
\sp{
\n{a_\ep}_{\widetilde{L}^\infty_T(B^{\f d 2-\theta}_{2,1})}
+
\n{b_\ep}_{\widetilde{L}^\infty_T(B^{\f d 2-\theta}_{2,1})}
}
\leq
C
\ep^\theta
E^{\ep}_{T},        \label{cons15}
\\
&
\n{a_\ep}_{\widetilde{L}^2_T(B^{\f d 2}_{2,1})}
+
\n{b_\ep}_{\widetilde{L}^2_T(B^{\f d 2}_{2,1})}
\leq
C
[T]^\f1 2
E^{\ep}_{T}  ,  \label{cons16}
\\
&
\n{\vU}_{\widetilde{L}^q_T(\underline{B}^{\f d 2-1+\f 2 q}_{2,1})}
\leq
C
[T]^{\f 1 q}\vU_{T^*_0},     \label{cons4}
\\
&
\n{\vUe}_{\widetilde{L}^q_T(\underline{B}^{\f d 2-1+\f 2 q}_{2,1})}
\leq
C
[T]^{\f 1 q}
E^{\ep}_{T}, 
\label{cons18}
\\
&
\n{\vZe}^{l;\zeta}_{
\widetilde{L}^\infty_T(B^{\f d 2-1}_{2, 1})
\cap
L^1_T(B^{\f d 2 +1}_{2, 1})
} 
+
\n{\mathcal{P}\vue-\vw}^{l;\zeta}_{
\widetilde{L}^\infty_T(\underline{B}^{\f d 2-1}_{2, 1})
\cap
L^1_T(\underline{B}^{\f d 2 +1}_{2, 1})
} 
+
\n{\rho_\ep-\rho}^{l;\zeta}_{
\widetilde{L}^\infty_T(B^{\f d 2}_{2, 1})} 
+
\n{\widehat{(\vue)}_0-\widehat{(\vu_0)}_0}_{L^\infty(0, T)}
\nonumber
\\
&
\qquad
\leq
C
[T]^{\f1 2}
\zeta^{1+2\theta}
\sp{
\vZ^{\ep}_{T, \theta}
+
\vW^{\ep}_{T, \theta}
},\label{new8.16}
\\
&
E^{\ep, \zeta}_T
\leq
C
\left(
\ep
\n{\vthe}^{h;\f{\beta_0}{\ep}}_{\widetilde{L}^\infty_T(B^{\f d 2}_{2, 1})}
+
\df{1}{\ep}
\n{\vthe}^{h;\f{\beta_0}{\ep}}_{L^1_T(B^{\f d 2}_{2, 1})}
+
\n{\vthe}^{m; \zeta, \f{\beta_0}{\ep}}_{\widetilde{L}^\infty_T(B^{\f d 2-1}_{2, 1})\cap L^1_T(B^{\f d 2+1}_{2, 1})}
+
\n{ \mathcal{P}^\perp\vue}^{h;\zeta}_{\widetilde{L}^\infty_T(B^{\f d 2-1}_{2, 1})\cap L^1_T(B^{\f d 2+1}_{2, 1})}
\right.
\nonumber
\\
&
\qquad
\qquad
\left.
+
\n{\rho_\ep}^{h;\zeta}_{\widetilde{L}^\infty_T(B^{\f d 2}_{2, 1})}
+
\n{ \mathcal{P}\vue}^{h;\zeta}_{\widetilde{L}^\infty_T(B^{\f d 2-1}_{2, 1})\cap L^1_T(B^{\f d 2+1}_{2, 1})}
+
[T]^{\f1 2}
\zeta^{1+2\theta}
\sp{
\vZ^{\ep}_{T, \theta}
+
\vW^{\ep}_{T, \theta}
}
\right)\label{cons20},
\\
&
E^{\ep}_{T}
\leq
C
\sp{
E^{\ep, \zeta}_T
+
[T]^{\f1 2}
\zeta^{1+2\theta}
\sp{
\vZ^{\ep}_{T, \theta}
+
\vW^{\ep}_{T, \theta}
}
+
\vV_{T^*_0}
+
\vU_{T^*_0}
},\label{cons5}
\\
&
\n{\vVe-\vV}_{
\widetilde{L}^\infty_{T}(B^{\f d 2-1}_{2, 1})
\cap
\widetilde{L}^2_{T}(B^{\f d 2}_{2, 1})
}
+
\n{\rho_\ep-\rho}_{
\widetilde{L}^\infty_{T}(B^{\f d 2}_{2, 1})}
+
\n{\mathcal{P}\vue-\vw}_{
\widetilde{L}^\infty_{T}(\underline{B}^{\f d 2-1}_{2, 1})
\cap
L^1_{T_0}(\underline{B}^{\f d 2+1}_{2, 1})
}
+
\n{\widehat{(\vue)}_0-\widehat{(\vu_0)}_0}_{L^\infty(0, T)}
\nonumber
\\
&
\qquad
\leq 
C
\sp{
E^{\ep, \zeta}_{T}
+
\n{\vV}^{h;\zeta}_{
\widetilde{L}^\infty_{T^*_0}(B^{\f d 2-1}_{2, 1})
\cap
L^1_{T^*_0}(B^{\f d 2+1}_{2, 1})
}
+
\n{\rho}^{h;\zeta}_{
\widetilde{L}^\infty_{T^*_0}(B^{\f d 2}_{2, 1})}
+
\n{\vw}^{h;\zeta}_{
\widetilde{L}^\infty_{T^*_0}(B^{\f d 2-1}_{2, 1})
\cap
L^1_{T^*_0}(B^{\f d 2+1}_{2, 1})
}
}
.\label{cons21}
\end{align}
\end{Lemma}
\bProof \eqref{cons1} follows from the definitions of $\rho_\ep$ and $\vthe$. \eqref{cons2} is obtained by combining the interpolation inequality, \eqref{3.1.3} and 
\begin{align}
\ep
\n{\vthe}_{\widetilde{L}^\infty_T(B^{\f d 2}_{2,1})}
&
\leq
\ep
\sp{
\n{\vthe}^{h; \f{\beta_0}{\ep}}_{\widetilde{L}^\infty_T(B^{\f d 2}_{2,1})}
+
\df{\beta_0}{\ep}
\n{\vthe}^{m; \zeta, \f{\beta_0}{\ep}}_{\widetilde{L}^\infty_T(B^{\f d 2-1}_{2,1})}
+
\zeta
\n{\vthe}^{l; \zeta}_{\widetilde{L}^\infty_T(B^{\f d 2-1}_{2,1})}
}
\les
E^{\ep}_{T},\label{cons6}
\\
\n{\vthe}_{\widetilde{L}^\infty_T(B^{\f d 2-1}_{2,1})}
&
\leq
\sp{\df{\beta_0}{\ep}}^{-1}\n{\vthe}^{h; \f{\beta_0}{\ep}}_{\widetilde{L}^\infty_T(B^{\f d 2}_{2,1})}
+
\n{\vthe}^{m; \zeta, \f{\beta_0}{\ep}}_{\widetilde{L}^\infty_T(B^{\f d 2-1}_{2,1})}
+
\n{\vthe}^{l; \zeta}_{\widetilde{L}^\infty_T(B^{\f d 2-1}_{2,1})}
\les
E^{\ep}_{T},\nonumber
\end{align}
where we used $\ep\zeta<\beta_0$ in \eqref{cons6}. Notice that $\widehat{(a_\ep)}_0 \equiv \widehat{(a_0)}_0$ and $\widehat{(b_\ep)}_0 \equiv \widehat{(b_0)}_0$. Hence  
\begin{align}
\ep
\sp{
\n{a_\ep}_{\widetilde{L}^\infty_T(B^{\f d 2}_{2,1})}
+
\n{b_\ep}_{\widetilde{L}^\infty_T(B^{\f d 2}_{2,1})}
}
=
\ep
\sp{
\n{a_\ep}_{\widetilde{L}^\infty_T(\underline{B}^{\f d 2}_{2,1})}
+
\n{b_\ep}_{\widetilde{L}^\infty_T(\underline{B}^{\f d 2}_{2,1})}
+
|\widehat{(a_0)}_0|
+
|\widehat{(b_0)}_0|
}.\label{cons7}
\end{align}
From \eqref{cons6} and the definition of $E^{\ep}_{T}$, we get
\begin{align}
\ep
\n{\vthe}_{\widetilde{L}^\infty_T(B^{\f d 2}_{2,1})}
\les
\ep
\n{\vthe}^{h;\f{\beta_0}{\ep}}_{\widetilde{L}^\infty_T(B^{\f d 2}_{2, 1})}
+
\n{\vthe}^{m; \zeta,  \f{\beta_0}{\ep}}_{\widetilde{L}^\infty_T(B^{\f d 2-1}_{2, 1})}
+
\zeta\ep E^{\ep}_{T}, \quad 
\ep
\sp{
\n{\rho_\ep}_{\widetilde{L}^\infty_T(B^{\f d 2}_{2,1})}
+
|\widehat{(a_0)}_0|
+
|\widehat{(b_0)}_0|
}
\leq
\ep\zeta E^{\ep}_{T}\label{cons8}.
\end{align}
\eqref{cons3} follows from \eqref{cons1} and \eqref{cons7}-\eqref{cons8}. By \eqref{minski2} and \eqref{cons1} , we see 
\begin{align*}
\ep
\sp{
\n{a_\ep}_{\widetilde{L}^\infty_T(B^{\f d 2-\theta}_{2,1})}
+
\n{b_\ep}_{\widetilde{L}^\infty_T(B^{\f d 2-\theta}_{2,1})}
}
&
\les
\ep
\sp{
\n{\vthe}_{\widetilde{L}^\infty_T(B^{\f d 2-\theta}_{2,1})}
+
\n{\rho_\ep}_{\widetilde{L}^\infty_T(B^{\f d 2-\theta}_{2,1})}
+
|\widehat{(a_0)}_0|
+
|\widehat{(b_0)}_0|
}
\\
&
\les
\ep
\sp{\df{\beta_0}{\ep}}^{-\theta}\n{\vthe}^{h; \f{\beta_0}{\ep}}_{\widetilde{L}^\infty_T(B^{\f d 2}_{2,1})}
+
\ep
\sp{\df{\beta_0}{\ep}}^{1-\theta}
\n{\vthe}^{l; \f{\beta_0}{\ep}}_{\widetilde{L}^\infty_T(B^{\f d 2-1}_{2,1})}
+
\ep
E^{\ep}_{T}
\\
&
\les
\ep^{\theta} 
E^{\ep}_{T},
\end{align*}
which yields \eqref{cons15}. We infer from \eqref{cons1}-\eqref{cons2} that
\begin{align*}
\n{a_\ep}_{\widetilde{L}^2_T(B^{\f d 2}_{2,1})}
+
\n{b_\ep}_{\widetilde{L}^2_T(B^{\f d 2}_{2,1})}
&
\les
\n{\vthe}_{\widetilde{L}^2_T(B^{\f d 2}_{2,1})}
+
\n{\rho_\ep}_{\widetilde{L}^2_T(B^{\f d 2}_{2,1})}
+
[T]^{\f1 2}
\sp{
|\widehat{(a_0)}_0|
+
|\widehat{(b_0)}_0|
}
\\
&
\les
[T]^{\f1 2} 
E^{\ep}_{T}
\end{align*}
which gives \eqref{cons16}. By the interpolation inequality, $2\leq q_1\leq \infty$ and \eqref{minski2}, we obtain
\begin{align*}
&
\n{\vw}_{\widetilde{L}^{q_1}_{T}(B^{\f d 2-1+\f 2 q}_{2, 1})}
\leq
\n{\vw}_{\widetilde{L}^\infty_{T}(B^{\f d 2-1}_{2, 1})\cap L^1_{T}(B^{\f d 2+1}_{2, 1})}
\leq
\vU_{T^*_0} , 
\\
&
\n{\rho}_{\widetilde{L}^{q_1}_{T}(B^{\f d 2-1+\f 2 q}_{2, 1})}
\les
\n{\rho}_{\widetilde{L}^{q_1}_{T}(B^{\f d 2}_{2, 1})}
\leq
T^{\f1 {q_1}}
\n{\rho}_{\widetilde{L}^\infty_{T}(B^{\f d 2}_{2, 1})}
\leq
(T)^{\f1 {q_1}} \vU_{T^*_0},
\end{align*}
which yields \eqref{cons4}. The same argument as in \eqref{cons4} yields \eqref{cons18}.

Due to \eqref{minski2}, 
\begin{align}
&
\n{\vZe}^{l;\zeta}_{\widetilde{L}^\infty_T(B^{\f d 2-1}_{2, 1})} 
\leq
\zeta^{2\theta}
\n{\vZe}^{l;\zeta}_{\widetilde{L}^\infty_T(B^{\f d 2-1-2\theta}_{2, 1})}
\les
\zeta^{2\theta}
\vZ^{\ep}_{T, \theta}\label{cons10},
\\
&
\n{\vZe}^{l;\zeta}_{L^1_T(B^{\f d 2+1}_{2, 1})} 
\leq
[T]^{\f1 2}
\n{\vZe}^{l;\zeta}_{L^2_T(B^{\f d 2+1}_{2, 1})} 
\leq
[T]^{\f1 2}
\zeta^{1+2\theta}
\n{\vZe}^{l;\zeta}_{L^2_T(B^{\f d 2-2\theta}_{2, 1})} 
\les
[T]^{\f1 2}
\zeta^{1+2\theta}
\vZ^{\ep}_{T, \theta}\label{cons19},
\end{align}
Seeing that 
\begin{align}
&
\n{\mathcal{P}\vue-\vw}^{l,\zeta}_{\widetilde{L}^\infty_T(\underline{B}^{\f d 2-1}_{2, 1})
\cap
L^1_T(\underline{B}^{\f d 2+1}_{2, 1})}
\leq
\zeta^{\theta}
\n{\mathcal{P}\vue-\vw}^{l,\zeta}_{\widetilde{L}^\infty_T(\underline{B}^{\f d 2-1-\theta}_{2, 1})
\cap
L^1_T(\underline{B}^{\f d 2+1-\theta}_{2, 1})}
\leq
\zeta^{\theta}
\vW^{\ep}_{T, \theta},\label{cons11}
\\
&
\n{\rho_\ep-\rho}^{l,\zeta}_{\widetilde{L}^\infty_T(B^{\f d 2}_{2, 1})}
\leq
\zeta^{1+\theta}
\n{\rho_\ep-\rho}^{l,\zeta}_{\widetilde{L}^\infty_T(B^{\f d 2-1-\theta}_{2, 1})}
\leq
\zeta^{1+\theta}
\vW^{\ep}_{T, \theta},\label{cons12}
\\
&
\n{\widehat{(\vue)}_0-\widehat{(\vu_0)}_0}_{L^\infty(0, T)}
\leq
\zeta^{1+2\theta}
\n{\widehat{(\vue)}_0-\widehat{(\vu_0)}_0}_{L^\infty(0, T)}
\leq
\zeta^{1+2\theta}
\vW^{\ep}_{T, \theta},\label{cons14}
\end{align}
where we used $\zeta\geq 1$ in \eqref{cons14}. \eqref{new8.16} follows from  \eqref{cons10}-\eqref{cons14}. Combining \eqref{cons10}-\eqref{cons14} with the definition of 
$E^{\ep, \zeta}_T$, we arrive at \eqref{cons20}.

By \eqref{3.1.3}, we obtain
\begin{align}
E^{\ep}_{T}
&
\les
E^{\ep, \zeta}_{T}
+
\vV_{T^*_0}
+
\vU_{T^*_0}
+
\n{\vZe}^{l;\zeta}_{\widetilde{L}^\infty_T(B^{\f d 2-1}_{2, 1})
\cap 
\widetilde{L}^1_T(B^{\f d 2+1}_{2, 1})}
+
\n{\mathcal{P}\vue-\vw}^{l,\zeta}_{\widetilde{L}^\infty_T(\underline{B}^{\f d 2-1}_{2, 1})
\cap
L^1_T(\underline{B}^{\f d 2+1}_{2, 1})}\nonumber
\\
&\quad
+
\n{\rho_\ep-\rho}^{l,\zeta}_{\widetilde{L}^\infty_T(B^{\f d 2}_{2, 1})}
+
\n{\widehat{(\vue)}_0-\widehat{(\vu_0)}_0}_{L^\infty(0, T)}.\label{cons9}
\end{align}
Inserting \eqref{cons10}-\eqref{cons14} into \eqref{cons9},  we get \eqref{cons5}.

\eqref{cons21} follows from \eqref{3.1.3} and the definition of $E_{T}^{\ep,\zeta}$.
\ \ $\Box$

\section{Computation of eigenvalues, eigenvectors and operators}

\subsection{The nonzero eigenvalues of \texorpdfstring{$L$}{} and the associated eigenvectors}

We now compute the nonzero eigenvalues and the associated eigenvectors of the operator $L$. By the skew-symmetry of $L$, its eigenvalues are purely imaginary.
Suppose that $(g, \df{Q^0}{R^0}g, i\Grad f)\in \text{ Im } L$, where $\widehat{f}_0=0$,  is an eigenvector of the operator $L$ corresponding to the eigenvalue $\lambda i$ with
$\lambda\in\mathbb{R}\setminus\{0\}$. Then
\[
L\left(
\begin{array}{l}
g
\\[10pt]
\df{Q^0}{R^0}g
\\[10pt]
i\Grad f
\end{array}
\right)
=\left(
\begin{array}{l}
iR^0\De f
\\[10pt]
iQ^0\De f
\\[10pt]
\left(c_1+c_2\df{Q^0}{R^0}\right)\Grad g
\\
\end{array}
\right)
=\lambda i\left(
\begin{array}{l}
g
\\[10pt]
\df{Q^0}{R^0}g
\\[10pt]
i\Grad f
\end{array}
\right).
\]
It follows that
\[
 R^0\De f
 =
 \lambda g,
 \qquad 
 \sp{ c_1+c_2\f{Q^0}{R^0} }g
 =
 - \lambda f.
\]
Let 
$
\lambda
=
\sqrt{(c_1R^0+c_2Q^0)}\lambda^*
$. 
Then 
\[
-\De g
=
(\lambda^*)^2 g.
\]
Consequently, the Fourier coefficients satisfy
\[
|k|^2 \widehat{g}_k
=
(\lambda^*)^2\widehat{g}_k, \quad 
\text{ for any } 
k\in \widetilde{\mathbb{Z}}^d\setminus\{0\}, 
\]
the solutions of which are given by
\[
\lambda^*
=
\pm |k^*|, \qquad 
\widehat{g}_k=0 \text{ when } 
k\neq k^*, \qquad 
g
=
\pm
|k^*| \widehat{g}_{k^*} 
\df{\text{e}^{ik^*\cdot x }}
{\sqrt{\T}}.
\]
Hence, the nonzero eigenvalues of $L$ are given by 
 $
 \lambda i
 =
 \pm i\sqrt{c_1R^0+c_2Q^0}|k|, 
 |k|\neq 0
 $, with the corresponding eigenvector
\[
\left(
1, \df{Q^0}{R^0}, \pm\df{\sqrt{c_1R^0+c_2Q^0} k}{R^0 |k|}
\right)^T
\df{\text{e}^{ikx}}{\sqrt{|\Td|}}.
\]
Therefore, an orthonormal basis of $\text{ Im }L$ with respect to the inner product 
$
\left<\cdot, \cdot\right>_{\mathbb{H}}
$ is given by 
$\Phi_k^\alpha $, where $k\in\Zd\setminus\{0\}, \alpha\in\{1, -1\}$ and 
\begin{align}\nonumber
& \Phi^\alpha_k
=
\left(
1, \df{Q^0}{R^0}, \df{c_{0}}{R^0} 
\df{\alpha\, \tsgn(k) k}{|k|}
\right)^T 
\df{\text{e}^{ik\cdot x}}{c_\mathbf{a}},\\
\nonumber
& c_{0}
:=
\sqrt{c_1R^0+c_2Q^0},
\qquad 
c_{\mathbf{a}}
:=
c_{0}\sqrt{\df{2Q^0}{R^0}|\Td|},
\end{align}
and $\tsgn(k)$ stands for a generalized sign function on $\widetilde{\mathbb{Z}}^d\setminus\{0\}$: its value is $1$ if and only if the first nonzero component of $k$ is positive, $-1$ elsewhere.
Each $\Phi_k^\alpha $ corresponds to the eigenvalue $ic_0\lambda_k^\alpha$ with $\lambda_k^\alpha:=\alpha \,\tsgn(k) |k|$. For any $\vA\in \text{Im }\, L$, we have 
\[
\vA
=
\sum\limits_{k,\alpha} 
\widehat{\vA}^{\alpha}_k 
\Phi^\alpha_k, 
\qquad 
\widehat{\vA}^{\alpha}_k
=
\left<\vA, \Phi^\alpha_k\right>_{\mathbb{H}},
\]
and 
$
\mathcal{L}(t)\vA
=
\sum\limits_{k,\alpha}
\text{e}^{-itc_0\lambda_k^\alpha}
\widehat{\vA}^{\alpha}_k 
\Phi^\alpha_k
$ with $t\in\mathbb{R}$.

\subsection{Computation of \texorpdfstring{$
\mathcal{Q}_1^\ep,
\mathcal{Q}_2^\ep, 
\mathcal{D}^\ep
$}{} and their limits}

We now calculate
$
\mathcal{Q}_1^\ep,
\mathcal{Q}_2^\ep, 
\mathcal{D}^\ep
$ and their limits 
$
\mathcal{Q}_1, \mathcal{Q}_2, \overline{\mathcal{D}}
$. 
Let 
$
\vA
=
\sum\limits_{k,\alpha} 
\widehat{\vA}^{\alpha}_k
\Phi^\alpha_k
\in \text{Im} \, L
$, $\vB
=
\sum\limits_{k,\alpha} 
\widehat{\vB}^{\alpha}_k 
\Phi^\alpha_k
\in \text{Im} \,L$,
$
\vE=(\vE^1, \vE^2, \vE^3)^T
\in \text{Ker}\, L
$. 
Indeed, there exist
\[
\rho^*
=
\sum\limits_{k\neq 0}
\widehat{\rho}^*_k
\df{\text{e}^{ik\cdot x}}
{\sqrt{\Td}},
\qquad 
\vw^*
=
\sum\limits_{k\neq 0} 
\widehat{\vw}^*_k
\df{\text{e}^{ik\cdot x}}
{\sqrt{\Td}},
\qquad 
\Div \vw^*=0, 
\]
such that $\vE$ can be written as
$\vE=(\rho^*, -\df{c_1}{c_2}\rho^*, \vw^*)^T+\mathbb{P}_0\vE$. Direct calculation gives 
\begin{align*}
    \mathcal{Q}_1^\ep(\vA, \vB)
    &
    =
    2\mathcal{L}\left(-\dfrac{t}{\ep}\right) \mathbb{P}^\perp
    \mathcal{Q}\left(
    \mathcal{L}\left(\dfrac{t}{\ep}\right)\vA, 
    \mathcal{L}\left(\dfrac{t}{\ep}\right)\vB
    \right)
    \\
    &
    =
    2\sum\limits_{k,\alpha}
    \text{e}^{
    \f{itc_0\lambda_k^\alpha}{\ep}
    }
    \left<
    \mathcal{Q}\left(
    \mathcal{L}\left(\dfrac{t}{\ep}\right)\vA,
    \mathcal{L}\left(\dfrac{t}{\ep}\right)\vA
    \right),
    \Phi^\alpha_k
    \right>_{\mathbb{H}}
    \Phi^\alpha_k
    \\
    &
    =
    2\sum\limits_{k,\alpha}
    \sum\limits_{\gamma, \beta, m+l=k}
    \left<
    \mathcal{Q}\left(\Phi_l^\beta, \Phi_m^\gamma\right), 
    \Phi_k^\alpha
    \right>_\mathbb{H}
    \widehat{\vA}_l^\beta 
    \widehat{\vB}_m^\gamma 
    \text{e}^{
    \f{itc_0}{\ep}
    (\lambda_k^\alpha-\lambda_l^\beta-\lambda_m^\gamma)
    }
    \Phi_k^\alpha
    \\
    &
    =
    \sum\limits_{k,\alpha} 
    \sum\limits_{\gamma, \beta, m+l=k} 
    \left(
    \df{c_0}{R^0}
    \left(
    \df{\gamma \text{sgn}(m) (m\cdot k)}{|m|}
    +
    \df{\beta \tsgn(l) (l \cdot k)}{|l|}
    \right)
    \right.
    \\
    &\quad\left.
    +
    \lambda_k^\alpha
    \left(
    \df{c_0 \beta \gamma \tsgn(m)\tsgn(l)(l\cdot m)}{R^0|m||l|}
    +
    \df{c_3R^0}{c_0}
    +\df{(c_4+c_5)Q^0}{c_0}
    +\df{c_6(Q^0)^2}{c_0R^0}
    \right)
    \right)
    \\
    &\quad\cdot
    i\widehat{\vA}_l^\beta 
    \widehat{\vB}_m^\gamma 
    \text{e}^{
    \f{itc_0}{\ep}
    (\lambda_k^\alpha-\lambda_l^\beta-\lambda_m^\gamma)
    }
    \df{\Phi_k^\alpha}{2c_\mathbf{a}},
    \\[10pt]
    \mathcal{Q}_2^\ep(\underline{\vE}, \vA)
    &
    =
    2\mathcal{L}\left(-\dfrac{t}{\ep}\right) \mathbb{P}^\perp
    \mathcal{Q}\left(
    (\rho^*, -\df{c_1}{c_2}\rho^*, \vw^*)^T,
    \mathcal{L}\left(\dfrac{t}{\ep}\right) \vA
    \right)
    \\
    &
    =
    2\sum\limits_{k,\alpha}
    \sum\limits_{\gamma,l+m=k} 
    \left<
    \mathcal{Q}\left(
    (\widehat{\rho}^*_l, -\df{c_1}{c_2}\widehat{\rho}^*_l, \widehat{\vw}^*_l)^T 
    \df{\text{e}^{il\cdot x}}{\sqrt{|\Td|}}, 
    \Phi_m^\gamma
    \right), 
    \Phi^\alpha_k
    \right>_{\mathbb{H}}
    \widehat{\vA}_m^\gamma
    \text{e}^{
    \f{itc_0}{\ep}
    (\lambda_k^\alpha-\lambda_m^\gamma)
    }
    \Phi_k^\alpha
    \\
&
=
\sum\limits_{k,\alpha}
\sum\limits_{\gamma,l+m=k}
\Bigg(
(\widehat{\vw}^*_l\cdot k)
\left(
1
+
\df{\alpha\gamma\text{sgn}(k)\text{sgn}(m)(m\cdot (k+l))}
{|m||k|}
\right)
\\
&\quad
+\df{\widehat{\rho}^*_l\alpha\tsgn(k)}{c_0|k|}
\left(
|k|^2\left(c_3R^0-\df{c_6c_1Q^0}{c_2}\right)
+
Q^0(m\cdot k)(c_4+c_5)
-
\df{c_1R^0}{c_2}(l\cdot k)(c_4+c_5)
\right)
\Bigg)
\\
&\quad \cdot
i{\vA}_m^\gamma
\text{e}^{
\f{itc_0}{\ep}
(\lambda_k^\alpha-\lambda_m^\gamma)
}
\df{\Phi_k^\alpha}{2\sqrt{|\Td|}},
\end{align*}
\begin{align*}
\mathcal{Q}_2^\ep(\mathbb{P}_0\vE, \vA)
&
=
2\mathcal{L}\left(-\dfrac{t}{\ep}\right) \mathbb{P}^\perp
\mathcal{Q}\left(
\left(\df{\widehat{(\vE^1)}_0}{\sqrt{|\Td|}}, 
\df{\widehat{(\vE^2)}_0}{\sqrt{|\Td|}}, 
\df{\widehat{(\vE^3)}_0}{\sqrt{|\Td|}}
\right), 
\mathcal{L}\left(\dfrac{t}{\ep}\right) \vA
\right)
\\
&
=
2\sum\limits_{k,\alpha} 
\sum\limits_{\beta, l=k} 
\left<
\mathcal{Q}\left(
\left(
\df{\widehat{(\vE^1)}_0}{\sqrt{|\Td|}}, 
\df{\widehat{(\vE^2)}_0}{\sqrt{|\Td|}}, 
\df{\widehat{(\vE^3)}_0}{\sqrt{|\Td|}}
\right)^T, 
\Phi_l^\beta
\right), 
\Phi^\alpha_k
\right>_{\mathbb{H}}
\widehat{\vA}_l^\beta
\text{e}^{
\f{itc_0}{\ep}
(\lambda_k^\alpha-\lambda_l^\beta)
}
\Phi_k^\alpha
\\
&=
\sum\limits_{k,\alpha}
\left(
\widehat{(\vE^1)}_0
\left(c_1+c_3R^0+c_5Q^0\right)
+
\widehat{(\vE^2)}_0
\left(c_2+c_4R^0+c_6Q^0\right)
\right)
i\df{\lambda_k^\alpha\widehat{\vA}^\alpha_k}{c_0}
\df{\Phi_k^\alpha}{2\sqrt{|\Td|}}
\\
&\quad
+
\sum\limits_{k,\alpha} k\cdot \widehat{(\vE^3)}_0
i\df{\widehat{\vA}_k^\alpha \Phi_k^\alpha}{\sqrt{|\Td|}}
+
i\sum\limits_{k,\alpha}
\left(
\widehat{(\vE^1)}_0
\left(c_3R^0+c_5Q^0-c_1\right)
+
\widehat{(\vE^2)}_0
\left(c_4R^0+c_6Q^0-c_2\right)
\right)
\\
&
\qquad
\cdot
i\df{\lambda_k^\alpha\widehat{\vA}^{-\alpha}_k}{c_0}
\text{e}^{\f{2itc_0\lambda_k^\alpha}{\ep}}
\df{\Phi_k^\alpha}{2\sqrt{|\Td|}},
\\[10pt]
\mathcal{D}^\ep(\vA)
&
=
\mathcal{L}\left(-\dfrac{t}{\ep}\right) \mathcal{D}\left(
\mathcal{L}\left(\dfrac{t}{\ep}\right) 
\vA
\right)
\\
&
=
\sum\limits_{k,\alpha}
\sum\limits_{\beta, l=k} 
\left<
\mathcal{D}(\Phi_l^\beta), \Phi_k^\alpha
\right>_\mathbb{H}
\widehat{\vA}_l^\beta
\text{e}^{\f{itc_0}{\ep}
(\lambda_k^\alpha-\lambda_l^\beta)
}
\Phi_k^\alpha
\\
&
=
\sum\limits_{k,\alpha}
\df{-\nu|k|^2\widehat{\vA}^\alpha_k}{2(R^0+Q^0)}
\Phi^\alpha_k
+
\sum\limits_{k,\alpha}
\df{\nu|k|^2\widehat{\vA}^{-\alpha}_k}{2(R^0+Q^0)}
\text{e}^{\f{2itc_0\lambda_k^\alpha}{\ep}}
\Phi^\alpha_k
\\
&
=
\df{\nu}{2(R^0+Q^0)}
\De \vA
+
\sum\limits_{k,\alpha}
\df{\nu|k|^2\widehat{\vA}^{-\alpha}_k}{2(R^0+Q^0)}
\text{e}^{\f{2itc_0\lambda_k^\alpha}{\ep}}
\Phi^\alpha_k. 
\end{align*}
By the Riemann-Lebesgue lemma, i.e., for 
$
\lambda\in\mathbb{R}\setminus \{0\}, 
0<T\leq \infty, f\in L^1(0, T)
$, 
\[
\lim\limits_{\ep \to 0} \int_0^T 
\text{e}^{\f{i\lambda t}{\ep}} f \dt=0,
\]
we deduce in the sense of distributions that
\begin{align}
\mathcal{Q}_1(\vA,\vB)
&
:=
\lim\limits_{\ep \to 0}
\mathcal{Q}^\ep_1(\vA,\vB)
\nonumber
\\
&
=
\sum_{
\substack{
k, l, m, \alpha, \beta, \gamma
\\
k=l+m
\\ 
\lambda_k^\alpha-\lambda_l^\beta-\lambda_m^\gamma=0
}
}
\left(
\df{c_0}{R^0}
\left(\df{\gamma \text{sgn}(m) (m\cdot k)}{|m|}
+\df{\beta \tsgn(l) (l \cdot k)}{|l|}
\right)
\right.
\nonumber
\\
&\quad\left.
+
\lambda_k^\alpha 
\left(
\df{c_0 \beta \gamma \tsgn(m)\tsgn(l)(l\cdot m)}{R^0|m||l|}
+
\df{c_3R^0}{c_0}
+
\df{(c_4+c_5)Q^0}{c_0}
+
\df{c_6(Q^0)^2}{c_0R^0}
\right)
\right)
i\widehat{\vA}_l^\beta
\widehat{\vB}_m^\gamma 
\df{\Phi_k^\alpha}{2c_\mathbf{a}}
\nonumber
\\
&
=
\sum_{
\substack{
k, l, m, \alpha, \beta, \gamma
\\ 
k=m+l
\\ 
\lambda_k^\alpha-\lambda_l^\beta-\lambda_m^\gamma=0
}
}
\left(
\df{3c_0}{R^0}
+
\df{c_3R^0}{c_0}
+
\df{(c_4+c_5)Q^0}{c_0}
+
\df{c_6(Q^0)^2}{c_0R^0}
\right)
i\lambda_k^\alpha
\widehat{\vA}_l^\beta 
\widehat{\vB}_m^\gamma 
\df{\Phi_k^\alpha}{2c_\mathbf{a}},
\label{Q_1}
\end{align}
\begin{align}
\mathcal{Q}_2(\underline{\vE},\vA)
&
:=
\lim\limits_{\ep \to 0}
\mathcal{Q}^\ep_2(\underline{\vE},\vA)
\nonumber
\\
&
=
\sum_{
\substack{
k, m, l, \alpha, \gamma
\\ 
k=l+m
\\ 
\lambda_k^\alpha=\lambda_m^\gamma
}
}
\Bigg(
(\widehat{\vw}^*_l\cdot k)
\left(
1
+
\df{\alpha\gamma\text{sgn}(k)\text{sgn}(m)(m\cdot (k+l))
}
{|m||k|}
\right)
+
\df{\widehat{\rho}^*_l\alpha\tsgn(k)}{c_0|k|}
\left(
|k|^2\left(c_3R^0-\df{c_6c_1Q^0}{c_2}\right)
\right.
\nonumber
\\
&\quad\left.
+
Q^0(m\cdot k)(c_4+c_5)
-
\df{c_1R^0}{c_2}(l\cdot k)(c_4+c_5)
\right)
\Bigg)
\cdot 
i\widehat{\vA}_m^\gamma
\df{\Phi_k^\alpha}{2\sqrt{|\Td|}}
\nonumber
\\
&
=
\sum_{
\substack{k, m, l, \alpha, \gamma
\\ 
k=l+m
\\ \lambda_k^\alpha=\lambda_m^\gamma
}
}
\Bigg(
\df{2(\widehat{\vw}^*_l\cdot k)(m\cdot k)}{|m||k|}
+
\df{\widehat{\rho}^*_l\alpha\tsgn(k)}{c_0|k|}
\left(
|k|^2\left(c_3R^0-\df{c_6c_1Q^0}{c_2}\right)
\right.
\nonumber
\\
&\quad\left.
+
Q^0(m\cdot k)(c_4+c_5)
-
\df{c_1R^0}{c_2}(l\cdot k)(c_4+c_5)
\right)
\Bigg) 
\cdot 
i \widehat{\vA}_m^\gamma
\df{\Phi_k^\alpha}{2\sqrt{|\Td|}},
\label{Q_2}
\\[10pt]
\mathcal{Q}_2(\mathbb{P}_0 \vE,\vA)
&
:=
\lim\limits_{\ep \to 0}
\mathcal{Q}^\ep_2(\mathbb{P}_0 \vE,\vA)
\nonumber
\\
&
=
\sum\limits_{k,\alpha}
\left(
\widehat{(\vE^1)}_0
\left(c_1+c_3R^0+c_5Q^0\right)
+
\widehat{(\vE^2)}_0
\left(c_2+c_4R^0+c_6Q^0\right)
\right)
i\df{\lambda_k^\alpha\widehat{\vA}^\alpha_k}{c_0}
\df{\Phi_k^\alpha}{2\sqrt{|\Td|}}
\nonumber
\\
&
\quad
+
\sum\limits_{k,\alpha} 
k\cdot \widehat{(\vE^3)}_0
i\df{\widehat{\vA}_k^\alpha \Phi_k^\alpha}{\sqrt{|\Td|}},
\label{Q_21}
\\[10pt]
\overline{\mathcal{D}}(\vA)
&
:=
\lim\limits_{\ep \to 0}
\mathcal{D}^\ep(\vA)
=
\df{\nu}{2(R^0+Q^0)}
\De \vA.
\nonumber
\end{align}
Here, in the computation of 
$\mathcal{Q}_1(\vA, \vB)$, we used the fact that 
$k=l+m$ and $\lambda_k^\alpha=\lambda_l^\beta+\lambda_m^\gamma$ imply that $k, l$ and $m$ are collinear. In the computation of $\mathcal{Q}_2 (\underline{\vE}, \vB) $, we used the fact that $k=l+m$ and $\lambda_k^\alpha=\lambda_m^\gamma $ imply 
\[
1+\df{\alpha\gamma\text{sgn}(k)\text{sgn}(m)(m\cdot (k+l))}{|m||k|}=\df{2m\cdot k}{|k| |m|}.
\]
Next, for  $\vE, \vF=(\vF^1, \vF^2, \vF^3)\in \text{Ker}\, L$, we calculate $\mathbb{P}^\perp\mathcal{Q}(\mathbb{P}_0 \vF, \underline{\vE})$. Indeed, we have
\begin{align}\label{new1}
\mathbb{P}^\perp\mathcal{Q}(\mathbb{P}_0 \vF, \underline{\vE})
&
=
\sum\limits_{k,\alpha}
\sum\limits_{k=l} 
\left<
\mathcal{Q}\left(
\left(
\df{\widehat{(\vF^1)}_0}{\sqrt{|\Td|}}, 
\df{\widehat{(\vF^2)}_0}{\sqrt{|\Td|}}, 
\df{\widehat{(\vF^3)}_0}{\sqrt{|\Td|}}
\right)^T, 
(\widehat{\rho}^*_l, -\df{c_1}{c_2}\widehat{\rho}^*_l, \widehat{\vw}^*_l)^T
\df{\text{e}^{il\cdot x}}{\sqrt{\T}}
\right), 
\Phi^\alpha_k
\right>_{\mathbb{H}}
\Phi^\alpha_k
\nonumber
\\
&
=
\sum_{k, \alpha}
\left(
\widehat{(\vF^1)}_0\left(c_3-\df{c_1}{c_2}c_5\right)
+
\widehat{(\vF^2)}_0
\left(c_4-\df{c_1}{c_2}c_6\right)
\right)
\df{i\lambda_k^\alpha c_\mathbf{a}R^0\widehat{\rho}^*_k}
{4c_0|\Td|}
\Phi^\alpha_k.
\end{align}

\section{Decay properties of \texorpdfstring{$\vZe$ and $\vWe$}{}}\label{decay}

\subsection{Auxiliary lemmas}
In this subsection, we prepare some auxiliary lemmas that will be used in the sequel. First, we introduce some basic facts about $\mathbb{P}$, $\mathbb{P}^\perp$ and $\mathcal{L}(t)$. $\mathbb{P}$ and $\mathbb{P}^\perp$ are bounded on $L^2$-type spaces, and $\mathcal{L}(t)$ is norm-equivalent on $L^2$-type spaces. More precisely,   for $0<T\leq \infty, -\infty<s, \tau<\infty $, $1\leq q,r\leq \infty$, we have
\begin{align}\label{3.1.3}
&
\n{\mathbb{P}\vA}_{B_{2,r}^s }
\les
\n{\vA}_{B_{2,r}^s} ,
\quad
\n{\mathbb{P}^\perp \vA}_{B_{2,r}^s }
\les
\n{\vA}_{\underline{B}_{2,r}^s},
\quad
\|\mathcal{L}(\tau)
\vA
\|_{B_{2,r}^s}
\approx 
\| \vA\|_{B_{2,r}^s},\nonumber
\quad 
\text{ if } \vA\in B_{2,r}^s,
\\
&
\|\mathcal{L}\left(\pm\df{t}{\ep}\right)
\vB(t,x)
\|_{L_T^q(B_{2,r}^s)}
\approx 
\|\vB(t,x)\|_{L_T^q(B_{2,r}^s)},
\quad 
\text{ if }
\vB\in L_T^q(B_{2,r}^s),
\\
 &
\|\mathcal{L}\left(\pm\df{t}{\ep}\right)
\vB(t,x)
\|_{\widetilde{L}_T^q(B_{2,r}^s)}
\approx 
\|\vB(t,x)\|_{\widetilde{L}_T^q(B_{2,r}^s)}
\quad 
\text{ if }
\vB\in \widetilde{L}_T^q(B_{2,r}^s).
\nonumber
\end{align}

Next, we prepare some lemmas.
\begin{Lemma}\label{le8.2}
Let $0<\theta<1$, $s\in \mathbb{R}$, $1\leq r \leq \infty$, and $A, B \in \sp{\mathcal{S}'(\Td)}^{d+2}$, then
\begin{align*}
&
\n{\mathcal{Q}(A, B)}_{B^{\f d 2-1}_{2, 1}} 
\les 
\n{A}_{B^{\f d 2}_{2, 1}}
\n{B}_{B^{\f d 2}_{2, 1}},
\\
&
\n{\mathcal{Q}(A, B)}_{B^{\f d 2-1-\theta}_{2, r}} 
\les
\min\left\{
\n{B}_{B^{\f d 2}_{2, 1}}
\n{A}_{B^{\f d 2-\theta}_{2, r}},\ \
\n{A}_{B^{\f d 2}_{2, 1}}
\n{B}_{B^{\f d 2-\theta}_{2, r}}
\right\},
\\
&
\n{\mathcal{Q}(\mathbb{P}_0 A, B)}_{B^{s}_{2, r}} 
\les 
|\mathbb{P}_0 A|
\n{B}_{B^{s+1}_{2, r}},
\end{align*}
provided that the norms on the right-hand side are finite.
\end{Lemma}
\bProof 
Using the definition of $\mathcal{Q}(A, B)$, together with \eqref{4.9} and \eqref{4.8},  we obtain the estimates for $\n{\mathcal{Q}(A, B)}_{B^{\f d 2-1}_{2, 1}}$ and $\n{\mathcal{Q}(A, B)}_{B^{\f d 2-1-\theta}_{2, r}}$. The estimate for $\n{\mathcal{Q}(\mathbb{P}_0 A, B)}_{B^{s}_{2, r}} $ is straightforward.              \ \ $\Box$

\begin{Lemma}\label{le4.2}
Let $\vA, \vB\in \text{Im} \, L$, $\vE, \vF=(\vF^1, \vF^2, \vF^3)\in \text{Ker} \,L$, $0<\theta<1$, $1\leq r \leq \infty$, and $-\infty< s, \tau<\infty$. Then 
\begin{align*}
\n{\mathcal{Q}^\ep_1(\vB, \vA)}_{B^{\f d 2-1}_{2, 1}}
&
\les 
\n{\vA}_{B^{\f d 2}_{2, 1}}
\n{\vB}_{B^{\f d 2}_{2, 1}},
\\
\n{\mathcal{Q}^\ep_2(\underline{\vE}, \vA)}_{B^{\f d 2-1}_{2, 1}}
&
\les 
\n{\vE}_{\underline{B}^{\f d 2}_{2, 1}}
\n{\vA}_{B^{\f d 2}_{2, 1}},
\\
\n{
\mathcal{L}(\tau)
\sp{
\mathbb{P}^\perp
\mathcal{Q}(\underline{\vE}, \underline{\vF})
}
}_{B^{\f d 2 -1}_{2, 1}}
&
\les
\n{\vE}_{\underline{B}^{\f d 2}_{2, 1}}
\n{\vF}_{\underline{B}^{\f d 2}_{2, 1}},
\\
\n{\mathcal{Q}^\ep_1(\vB, \vA)}_{B^{\f d 2-1-\theta}_{2, r}}
&
\les 
\min\left\{
\n{\vB}_{B^{\f d 2}_{2, 1}}
\n{\vA}_{B^{\f d 2-\theta}_{2, r}},\ \
\n{\vA}_{B^{\f d 2}_{2, 1}}
\n{\vB}_{B^{\f d 2-\theta}_{2, r}}
\right\},
\\
\n{\mathcal{Q}^\ep_2(\underline{\vE}, \vA)}_{B^{\f d 2-2-\theta}_{2, r}}
&\les 
\min\left\{
\n{\vE}_{\underline{B}^{\f d 2}_{2, 1}}
\n{\vA}_{B^{\f d 2-1-\theta}_{2, r}},\ \
\n{\vA}_{B^{\f d 2}_{2, 1}}
\n{\vE}_{\underline{B}^{\f d 2-1-\theta}_{2, r}}
\right\},
\\
\n{
\mathcal{L}(\tau)
\sp{
\mathbb{P}^\perp
\mathcal{Q}(\underline{\vE}, \underline{\vF})
}
}_{B^{\f d 2 -1-\theta}_{2, r}}
&\les
\min\left\{
\n{\vE}_{\underline{B}^{\f d 2}_{2, 1}}
\n{\vF}_{\underline{B}^{\f d 2-1-\theta}_{2, r}}, \ \
\n{\vF}_{\underline{B}^{\f d 2}_{2, 1}}
\n{\vE}_{\underline{B}^{\f d 2-1-\theta}_{2, r}}
\right\},
\\
\n{\mathcal{Q}^\ep_2(\mathbb{P}_0\vE, \vA)}_{B^{s}_{2, r}}
&\les 
|\mathbb{P}_0\vE|
\n{\vA}_{B^{s+1}_{2, r}},
\\
\n{
\mathcal{L}(\tau)
\sp{
\mathbb{P}^\perp
\mathcal{Q}(\mathbb{P}_0\vF, \underline{\vE})
}
}_{B^{s}_{2, r}}
&\les
\sp{
|\widehat{(\vF^1)}_0|
+
|\widehat{(\vF^2)}_0|
}
\n{\vE}_{\underline{B}^{s+1}_{2, r}},
\end{align*}
provided that the norms on the right-hand side are finite.
\end{Lemma}
\bProof 
Using \eqref{3.1.3} and Lemma \ref{le8.2}, we obtain the estimates for 
$\n{\mathcal{Q}^\ep_1(\vB, \vA)}_{B^{\f d 2-1}_{2, 1}}$,
$\n{\mathcal{Q}^\ep_2(\underline{\vE}, \vA)}_{B^{\f d 2-1}_{2, 1}}$,
$
\n{
\mathcal{L}(\tau)
\sp{
\mathbb{P}^\perp
\mathcal{Q}(\underline{\vE}, \underline{\vF})
}
}_{B^{\f d 2-1}_{2, 1}},
$
$\n{\mathcal{Q}^\ep_1(\vB, \vA)}_{B^{\f d 2-1-\theta}_{2, r}}$, 
$\n{\mathcal{Q}^\ep_2(\underline{\vE}, \vA)}_{B^{\f d 2-2-\theta}_{2, r}}$,
$
\n{
\mathcal{L}(\tau)
\sp{
\mathbb{P}^\perp
\mathcal{Q}(\underline{\vE}, \underline{\vF})
}
}_{B^{\f d 2-1-\theta}_{2, r}}
$
and
$\n{\mathcal{Q}^\ep_2(\mathbb{P}_0\vE, \vA)}_{B^{s}_{2, r}}$. The last estimate follows from \eqref{new1} and \eqref{3.1.3}.
 \ \ $\Box$

\begin{Lemma}\label{le4.4} 
Let $-\infty<s, s_1<\infty$, $s_2, s_3, s_4, s_5\geq 0$ and 
\begin{align*}
    f=
\sum\limits_{k\neq0}
\widehat{f}_k
\df{\text{e}^{ik\cdot x}}{\sqrt{\T}},
\qquad 
g=
\sum\limits_{k\neq0}
\widehat{g}_k
\df{\text{e}^{ik\cdot x}}{\sqrt{\T}},
\qquad
h=
\sum\limits_{k\neq0}
\widehat{h}_k
\df{\text{e}^{ik\cdot x}}{\sqrt{\T}}
\end{align*}
with
$
\widehat{h}_k=\sum\limits_{k=l+m}C(k,l,m)\widehat{f}_l\widehat{g}_m.
$
If $s+s_1\geq 0$ and there exists a positive constant $K$ such that for any $(k,l,m)\in \sp{
\widetilde{\mathbb{Z}}^d\setminus\{0\}}^3$,
\begin{align*}
|C(k,l,m)|\leq K |k|^{s_1}|l|^{s_2}|m|^{s_3}.
\end{align*}
Then, 
\begin{align*}
\n{h}_{H^s}
\les
K
\sp{
\n{g}_{H^{s+s_1+s_2+s_3+s_4}}
\n{|k|^{-s_4}\widehat{f}_k}_{\ell^1
\sp{
\widetilde{\mathbb{Z}}^d\setminus\{0\}}}
+
\n{f}_{H^{s+s_1+s_2+s_3+s_5}}
\n{|k|^{-s_5}\widehat{g}_k}_{\ell^1
\sp{
\widetilde{\mathbb{Z}}^d\setminus\{0\}}
}
},
\end{align*}
provided that the right-hand side is finite.
\end{Lemma}
\bProof
The proof is similar to that of Lemma 2.4 in \cite{D02}, so we only give a sketch of proof. Notice that 
\begin{align*}
\n{h}^2_{H^s}
&
\leq
\sum_{k\neq 0}
|k|^{2s}
\sp{
\sum_{k=m+l}
|C(k,m,l)|
|\widehat{f}_l|
|\widehat{g}_m|
}^2
\\
&
\les
\sum_{k\neq 0}
|k|^{2s}
\sp{
\sum_{k=m+l, |k|>2|l|}
|C(k,m,l)|
|\widehat{f}_l|
|\widehat{g}_m|
}^2
+
\sum_{k\neq 0}
|k|^{2s}
\sp{
\sum_{k=m+l, |k|\leq 2|l|}
|C(k,m,l)|
|\widehat{f}_l|
|\widehat{g}_m|
}^2
\\
&
\les
K^2\sum_{k\neq 0}
\sp{
\sum_{k=m+l, |k|>2|l|}
|m|^{s+s_1+s_2+s_3+s_4}
|\widehat{g}_m|
|l|^{-s_4}|\widehat{f}_l|
}^2
\\
&
\quad
+
K^2
\sum_{k\neq 0}
\sp{
\sum_{k=m+l, |k|\leq 2|l|}
|l|^{s+s_1+s_2+s_3+s_5}
|\widehat{f}_l|
|m|^{-s_5}|\widehat{g}_m|
}^2
\\
&
\les
K^2
\sp{
\n{g}^2_{H^{s+s_1+s_2+s_3+s_4}}
\n{|k|^{-s_4}\widehat{f}_k}^2_{\ell^1
\sp{
\widetilde{\mathbb{Z}}^d\setminus\{0\}}
}
+
\n{f}^2_{H^{s+s_1+s_2+s_3+s_5}}
\n{|k|^{-s_5}\widehat{g}_k}^2_{\ell^1
\sp{
\widetilde{\mathbb{Z}}^d\setminus\{0\}}
}
}.
\end{align*}
This completes the proof. \ \ $\Box$

\begin{Lemma}\label{Le8.6} 
Let $R, R_1, R_2, R_3>0$ such that $R_j<R$ for $j=1,2,3$. Suppose that $f$ is smooth on $(-R, R)$ with $f(0)=0$ and $g$ is smooth on $(-R, R)\times (-R, R))$ with $g(0, 0)=0$. For $0<s\leq \f d 2$ and $\n{u}_{L^\infty}\leq R_1$, $\n{v}_{L^\infty}\leq R_2$, $\n{w}_{L^\infty}\leq R_3$, there exist positive constants 
$C(R_1, f)$ and $C(R_1, R_2, g)$ such that
\begin{align*}
\n{f(u)}_{B^{s}_{2,1}}
&
\leq
C(R_1, f)
\n{u}_{B^{s}_{2,1}},
\\
\n{g(v, w)}_{B^{s}_{2,1}}
&
\leq
C(R_2, R_3, g)
\sp{
\n{v}_{B^{s}_{2,1}}
+
\n{w}_{B^{s}_{2,1}}
}.
\end{align*}
\end{Lemma}
Lemma \ref{Le8.6} is proved by repeating the proof of Theorem 2.61 in \cite{BCD11}. We omit the details.

\subsection{Decay properties of \texorpdfstring{$\vZe$}{}}
This subsection is devoted to the decay properties of $\vZe$ with respect to the Mach number $\ep$.
\begin{Proposition}\label{Pro4.1}
Let $T^*_0$ be as in Theorem \ref{Th6.1},
$0<T< \infty$ with $T\leq T^*_0$, $0<\theta<1$ and $(a_0, b_0, \vu_0)\in B^{\f d 2}_{2, 1} \times B^{\f d 2}_{2, 1} \times B^{\f d 2-1}_{2, 1}$. Assume that $(a_\ep, b_\ep, \vue)$ is a regular solution to \eqref{1.1.3} on $[0, T]$,  satisfying
\begin{align*}
\ep 
\sp{
\n{\vae}_{L^\infty_T(L^\infty)}
+
\n{\vbe}_{L^\infty_T(L^\infty)}
}
\leq
\f{\min\{R^0, Q^0\}}{2}.
\end{align*}
Then 
\begin{align*}
\vZ_{T, \theta}^\ep 
&
\leq
C
\exp{ 
\left[
C
\sp{
\vV_{T^*_0}^2
+ 
[T]
\sp{E^{\ep}_T}^2
}
\right]
}
\Bigg[
\sp{
\vW^\ep_{T, \theta}
}^{\f{\theta}{1+\theta}}
\sp{
\vU_{T^*_0}+E^{\ep}_T
}^{\f{1}{1+\theta}}
\vV_{T^*_0}
+
\vW^\ep_{T, \theta}
\vV_{T^*_0}
\\
&
\qquad 
+
\tau(\ep)
[T]^{\f 3 2}
\sp{
E_0
+
E^2_0
+
\vV_{T^*_0}
+
\vU_{T^*_0}
+
E^{\ep}_T
+1
}
\sp{
\vV^2_{T^*_0}
+
\vU^2_{T^*_0}
+
\sp{
\sp{
E^{\ep}_T
}^2
+
E^{\ep}_T
+
1}^2
}
\Bigg],
\end{align*}
where $\tau(\ep)$ is monotonically increasing in $\ep$ and satisfies $\lim\limits_{\ep \to 0^{+}} \tau(\ep)=0$.
\end{Proposition}
\bProof
Subtracting \eqref{1.4} from \eqref{1.6}, we obtain
\begin{align}\label{3.1.1}
\p_t \vZe
&
+
\mathcal{Q}_1^\ep(\vV+\vVe, \vZe)
+
\mathcal{Q}_2^\ep(\underline{\vUe}, \vZe)
+
\mathcal{Q}_2^\ep(\mathbb{P}_0 \vUe, \vZe)
-
\overline{\mathcal{D}}(\vZe)
\\
&
=
\mathcal{L}\left(-\df{t}{\ep}\right)\mathbb{P}^\perp r_\ep
-
\mathcal{Q}_2^\ep(\underline{\vWe}, \vV)
-
\mathcal{Q}_2^\ep(\mathbb{P}_0\vWe, \vV)
+
\sum_{j=1}^{6}R_{j,\ep},
\nonumber
\end{align}
where
\begin{align*}
R_{1,\ep}
&
=
\mathcal{Q}_1(\vV, \vV)
-
\mathcal{Q}^\ep_1(\vV, \vV),
\qquad 
R_{2,\ep}
=
\mathcal{Q}_2(\underline{\vU}, \vV)
-
\mathcal{Q}^\ep_2(\underline{\vU}, \vV),
\\
R_{3,\ep}
&
=\mathcal{Q}_2(\mathbb{P}_0 \vU, \vV)
-
\mathcal{Q}^\ep_2(\mathbb{P}_0 \vU, \vV),
\qquad 
R_{4,\ep}
=
-2\mathcal{L}\left(-\df{t}{\ep}\right)\mathbb{P}^\perp
\mathcal{Q}(\mathbb{P}_0 \vUe, \underline{\vUe}) , 
\\
R_{5,\ep}
&
=-\mathcal{L}\left(-\df{t}{\ep}\right)\mathbb{P}^\perp
\mathcal{Q}(\underline{\vUe}, \underline{\vUe}),
\qquad 
R_{6, \ep}
=
-(\overline{\mathcal{D}}(\vVe)-\mathcal{D}^\ep(\vVe)). 
\end{align*}
For our purposes, a delicate treatment of $R_{j, \ep} \, (1\leq j\leq 6)$
is required. Following the approach in \cite{D02,S94}, we employ the negative time regularity method. For any $\vA=(a,b, \vu)^T$ and $M>0$, we define the truncation of $\vA$ as 
$
\vA_M:=\sum\limits_{|k|\leq M}
(\widehat{a}_k, \widehat{b}_k, \widehat{\vu}_k)^T
\df{\text{e}^{ik\cdot x}}{\sqrt{\Td}}
$. Moreover, we set
$
\vA^M:=\vA-\vA_M
$. 
We recall the expression
$
\vU
=
(\rho, -\df{c_1}{c_2}\rho, \vw)^T
+
(
\df{\widehat{(a_0)}_0}{\sqrt{|\Td|}}
, 
\df{\widehat{(b_0)}_0}{\sqrt{|\Td|}}
, 
\df{\widehat{(\vu_0)}_0}{\sqrt{|\Td|}}
)^T
$ and set
\begin{align*}
R_{1, \ep, M}
:=
&
\mathcal{Q}_1(\vV_M, \vV_M)
-
\mathcal{Q}^\ep_1(\vV_M, \vV_M), 
\quad 
R^M_{1, \ep}
:=
R_{1, \ep}
-
R_{1, \ep, M},
\\[10pt]
\widetilde{R}_{1,\ep, M}
:=
&
-\sum_{
\substack{k, l, m, \alpha, \beta, \gamma
\\ 
k=m+l, |m|\leq M, |l|\leq M
\\ 
\lambda_k^\alpha-\lambda_l^\beta-\lambda_m^\gamma\neq0
}
}
\left(
\df{c_0}{R^0}
\left(
\df{\gamma \text{sgn}(m) (m\cdot k)}{|m|}
+
\df{\beta \tsgn(l) (l \cdot k)}{|l|}
\right)
\right.
\\
&
\quad\left.
+
\lambda_k^\alpha 
\left(
\df{c_0 \beta \gamma \tsgn(m)\tsgn(l)(l\cdot m)}
{ R^0|m||l|}
+
\df{c_3R^0}{c_0}
+
\df{(c_4+c_5)Q^0}{c_0}
+
\df{c_6(Q^0)^2}{c_0R^0}
\right)
\right)
\\
&
\qquad \qquad  \cdot
\df{
\text{e}^{ \f{itc_0}{\ep}
(\lambda_k^\alpha - \lambda_l^\beta-\lambda_m^\gamma)
}
} 
{c_0(\lambda_k^\alpha-\lambda_l^\beta-\lambda_m^\gamma) }
\widehat{\vV}_l^\beta 
\widehat{\vV}_m^\gamma 
\df{\Phi_k^\alpha}{2c_\mathbf{a}},
\\[10pt]
\widetilde{R}^t_{1,\ep, M}
:=
&
-\sum_{
\substack
{
k, l, m, \alpha, \beta, \gamma
\\ 
k=m+l, |m|\leq M, |l|\leq M
\\ 
\lambda_k^\alpha-\lambda_l^\beta-\lambda_m^\gamma\neq0
}
}
\left(
\df{c_0}{R^0}
\left(
\df{\gamma \text{sgn}(m) (m\cdot k)}{|m|}
+
\df{\beta \tsgn(l) (l \cdot k)}{|l|}
\right)
\right.
\\
&
\quad\left.
+
\lambda_k^\alpha 
\left(
\df{c_0 \beta \gamma \tsgn(m)\tsgn(l)(l\cdot m)}{R^0|m||l|}
+
\df{c_3R^0}{c_0}
+
\df{(c_4+c_5)Q^0}{c_0}
+
\df{c_6(Q^0)^2}{c_0R^0}
\right)
\right)
\\
&
\qquad \qquad \cdot 
\df{
\text{e}^{
\f{itc_0}{\ep}
(\lambda_k^\alpha-\lambda_l^\beta-\lambda_m^\gamma)
}
}
{
c_0(\lambda_k^\alpha-\lambda_l^\beta-\lambda_m^\gamma)
}
\left(
\p_t\widehat{\vV}_l^\beta \widehat{\vV}_m^\gamma
+
\widehat{\vV}_l^\beta \p_t\widehat{\vV}_m^\gamma
\right) 
\df{\Phi_k^\alpha}{2c_\mathbf{a}},   
\end{align*}
\begin{align*}
R_{2,\ep, M}:=
&
\mathcal{Q}_2(\underline{\vU_M}, \vV_M)
-
\mathcal{Q}^\ep_2(\underline{\vU_M}, \vV_M),
\quad 
R^M_{2,\ep}:=R_{ 2,\ep}-R_{2,\ep, M},
\\
\widetilde{R}_{2, \ep, M}:=
&
-\sum_{
\substack{
k, m, l, \alpha, \gamma
\\ 
k=l+m, |l|\leq M, |m|\leq M
\\ 
\lambda_k^\alpha\neq\lambda_m^\gamma
}
}
\Bigg(
(\widehat{\vw}_l\cdot k)
\left(
1
+
\df{\alpha\gamma\text{sgn}(k)\text{sgn}(m)(m\cdot (k+l))}
{|m||k|}
\right)
\\
&
+
\df{\widehat{\rho}_l\alpha\tsgn(k)}{c_0|k|}
\left(
|k|^2\left(c_3R^0 -\df{c_6c_1Q^0}{c_2} \right)
+
Q^0(m\cdot k)(c_4+c_5)
-
\df{c_1R^0}{c_2}
(l\cdot k)(c_4+c_5)
\right)
\Bigg)
\\
& \qquad \qquad 
\cdot 
\df{
\text{e}^{\f{itc_0}{\ep}(\lambda_k^\alpha-\lambda_m^\gamma)}
}
{c_0(\lambda_k^\alpha-\lambda_m^\gamma)}
\widehat{\vV}_m^\gamma
\df{\Phi_k^\alpha}{2\sqrt{|\Td|}},
\\[10pt]
\widetilde{R}^t_{2, \ep, M}:=
&
-\sum_{
\substack{
k, m, l, \alpha, \gamma
\\ 
k=l+m, |l|\leq M, |m|\leq M
\\
\lambda_k^\alpha\neq\lambda_m^\gamma
}
}
\Bigg(
(\p_t\widehat{\vw}_l\cdot k)
\left(
1+
\df{\alpha\gamma\text{sgn}(k)\text{sgn}(m)(m\cdot (k+l))}
{|m||k|}
\right)
\\
&
+
\df{\p_t\widehat{\rho}_l\alpha\tsgn(k)}
{c_0|k|}
\left(
|k|^2\left(c_3R^0-\df{c_6c_1Q^0}{c_2}\right)
+
Q^0(m\cdot k)(c_4+c_5)
-
\df{c_1R^0}{c_2}(l\cdot k)(c_4+c_5)
\right)
\Bigg)
\\
& \qquad  
\cdot
\df{
\text{e}^{\f{itc_0}{\ep}(\lambda_k^\alpha-\lambda_m^\gamma)}
}
{c_0(\lambda_k^\alpha-\lambda_m^\gamma)}
\widehat{\vV}_m^\gamma
\df{\Phi_k^\alpha}{2\sqrt{|\Td|}}
-
\sum_{
\substack{
k, m, l, \alpha, \gamma
\\
k=l+m, |l|\leq M, |m|\leq M
\\
\lambda_k^\alpha\neq\lambda_m^\gamma
}
}
\Bigg(
(\widehat{\vw}_l\cdot k)
\left(
1+
\df{\alpha\gamma\text{sgn}(k)\text{sgn}(m)(m\cdot (k+l))}{|m||k|}
\right)
\\
&
+
\df{\widehat{\rho}_l\alpha\tsgn(k)}{c_0|k|}
\left(
|k|^2\left(c_3R^0-\df{c_6c_1Q^0}{c_2}\right)
+
Q^0(m\cdot k)(c_4+c_5)
-
\df{c_1R^0}{c_2}(l\cdot k)(c_4+c_5)
\right)
\Bigg)
\\
& \qquad 
\cdot 
\df{
\text{e}^{\f{itc_0}{\ep}(\lambda_k^\alpha-\lambda_m^\gamma)}
}
{c_0(\lambda_k^\alpha-\lambda_m^\gamma)}
\p_t\widehat{\vV}_m^\gamma
\df{\Phi_k^\alpha}{2\sqrt{|\Td|}},
\end{align*}        
\begin{align*}
R_{3,\ep, M}:=
&
\mathcal{Q}_2(\mathbb{P}_0 \vU, \vV_M)
-
\mathcal{Q}^\ep_2(\mathbb{P}_0 \vU, \vV_M), 
\quad   
R^M_{3,\ep}
:=  R_{3,\ep}
-  
R_{3,\ep, M},
\\[10pt]
\widetilde{R}_{3,\ep, M}:=
&
-\sum\limits_{|k|\leq M,\alpha}
\left(
\widehat{(a_0)}_0
\left(c_3R^0+c_5Q^0-c_1\right)
+
\widehat{(b_0)}_0
\left(c_4R^0+c_6Q^0-c_2\right)
\right)
\df{\lambda_k^\alpha\widehat{\vV}^{-\alpha}_k}
{c_0}
\df{\text{e}^{\f{2itc_0\lambda_k^\alpha}{\ep}}}
{2c_0\lambda_k^\alpha}
\df{\Phi_k^\alpha}
{2\sqrt{|\Td|}},
\\[10pt]
\widetilde{R}^t_{3,\ep, M}:=
&
-\sum\limits_{|k|\leq M,\alpha}
\left(
\widehat{(a_0)}_0
\left(c_3R^0+c_5Q^0-c_1\right)
+
\widehat{(b_0)}_0
\left(c_4R^0+c_6Q^0-c_2\right)
\right)
\df{\lambda_k^\alpha\p_t\widehat{\vV}^{-\alpha}_k}
{c_0}
\df{\text{e}^{\f{2itc_0\lambda_k^\alpha}{\ep}}}
{2c_0\lambda_k^\alpha}
\df{\Phi_k^\alpha}
{2\sqrt{|\Td|}},
\end{align*}
\begin{align*}
 R_{4,\ep, M}:=
 &
 -2\mathcal{L}\left(-\df{t}{\ep}\right)\mathbb{P}^\perp
 \mathcal{Q}(\mathbb{P}_0 \vUe, \underline{\vU_{\ep, M}}),
 \quad 
 R^M_{4,\ep}:=R_{4,\ep}-R_{4,\ep, M},
 \\[10pt]
\widetilde{R}_{4,\ep, M}:=
&
-\sum_{|k|\leq M, \alpha}
\left(
\widehat{(a_0)}_0
\left(c_3-\df{c_1}{c_2}c_5\right)
+
\widehat{(b_0)}_0
\left(c_4-\df{c_1}{c_2}c_6\right)
\right)
\df{\lambda_k^\alpha c_\mathbf{a}R^0\widehat{\rho}_{\ep, k}}
{2c_0|\Td|}
\df{
\text{e}^{\f{it c_o\lambda_k^\alpha}{\ep}}
}
{c_0\lambda_k^\alpha}
\Phi_k^\alpha,
\\[10pt]
\widetilde{R}^t_{4,\ep, M}:=
&
-\sum_{|k|\leq M, \alpha}
\left(
\widehat{(a_0)}_0\left(c_3-\df{c_1}{c_2}c_5\right)
+
\widehat{(b_0)}_0
\left(c_4-\df{c_1}{c_2}c_6\right)
\right)
\df{\lambda_k^\alpha c_\mathbf{a}R^0\p_t\widehat{\rho}_{\ep, k}}
{2c_0|\Td|}
\df{
\text{e}^{\f{itc_o\lambda_k^\alpha}{\ep}}
}
{c_0\lambda_k^\alpha}
\Phi_k^\alpha , 
\\[10pt]
R_{5,\ep, M}:=
&
-\mathcal{L}\left(-\df{t}{\ep}\right)\mathbb{P}^\perp
\mathcal{Q}(\underline{\vU_{\ep, M}}, \underline{\vU_{\ep, M}}),
\quad 
R^M_{5,\ep}:=R_{5,\ep}-R_{5,\ep, M},
\\[10pt]
\widetilde{R}_{5,\ep, M}:=
&
-\sum_{
\substack{
k, m, l, \alpha
\\
k=l+m, |l|\leq M, |m|\leq M
}
}
\Bigg[
\df{c_\mathbf{a}R^0\alpha\tsgn(k)}{2c_0|\Td||k|}
(\widehat{(\mathcal{P}^\perp \underline{\vue})}_l\cdot k)
(\widehat{(\mathcal{P}^\perp \underline{\vue})}_m\cdot k)
\\
&
\quad
+
\df{\lambda_k^\alpha c_\mathbf{a}R^0\widehat{\rho}_{\ep, l}\widehat{\rho}_{\ep, m}}{4c_0|\Td|}
\left(
c_3-\f{c_1}{c_2}(c_4+c_5)
+c_6\left(\df{c_1}{c_2}\right)^2
\right)
\Bigg]
\df{
\text{e}^{\f{itc_o\lambda_k^\alpha}{\ep}}
}
{c_0\lambda_k^\alpha}
\Phi_k^\alpha,
\\[10pt]
\widetilde{R}^t_{5,\ep, M}:=
&
-\sum_{
\substack{
k, m, l, \alpha
\\ 
k=l+m, |l|\leq M, |m|\leq M
}
}
\Bigg[
\df{c_\mathbf{a}R^0\alpha\tsgn(k)}{2c_0|\Td||k|}
\left(
(\p_t\widehat{(\mathcal{P}^\perp \underline{\vue})}_l\cdot k)
(\widehat{(\mathcal{P}^\perp \underline{\vue})}_m\cdot k)
\right.
\\
&
\left.
\quad
+
(\widehat{(\mathcal{P}^\perp \underline{\vue})}_l\cdot k)
(\p_t\widehat{(\mathcal{P}^\perp \underline{\vue})}_m\cdot k)
\right)
+
\df{
\lambda_k^\alpha c_\mathbf{a}R^0
\left(
\p_t\widehat{\rho}_{\ep, l}\widehat{\rho}_{\ep, m}
+
\widehat{\rho}_{\ep, l}\p_t\widehat{\rho}_{\ep, m}
\right)
}
{4c_0|\Td|}
\\
&
\qquad 
\left(
c_3-\f{c_1}{c_2}(c_4+c_5)
+c_6\left(\df{c_1}{c_2}\right)^2
\right)
\Bigg]
\df{
\text{e}^{\f{itc_o\lambda_k^\alpha}{\ep}}
}
{c_0\lambda_k^\alpha}
\Phi_k^\alpha,
\\[10pt]
R_{6, \ep, M}:=
 &
 -(\overline{\mathcal{D}}(\vV_{\ep, M})-\mathcal{D}^\ep(\vV_{\ep, M})) ,
 \quad 
 R^M_{6, \ep}:=R_{6, \ep}-R_{6, \ep, M} ,
 \\
\widetilde{R}_{6,\ep, M}:=
&
\sum\limits_{|k|\leq M,\alpha}
\df{\nu|k|^2\widehat{\vV}^{-\alpha}_{\ep, k}}
{2(R^0+Q^0)}
\df{\text{e}^{\f{2itc_0\lambda_k^\alpha}{\ep}}}
{2ic_0\lambda_k^\alpha}
\Phi^\alpha_k,
\\
\widetilde{R}^t_{6,\ep, M}:=
&
\sum\limits_{|k|\leq M,\alpha}
\df{\nu|k|^2\p_t\widehat{\vV}^{-\alpha}_{\ep, k}}
{2(R^0+Q^0)}
\df{
\text{e}^{\f{2itc_0\lambda_k^\alpha}{\ep}}
}
{2ic_0\lambda_k^\alpha}
\Phi^\alpha_k,
\\[10pt]
\end{align*}

It can be readily verified that $R_{j,\ep, M}$, $\widetilde{R}_{j,\ep, M}$ and $\widetilde{R}^t_{j,\ep, M}$ satisfy 
\begin{align}\label{3.1.2}
\p_t\ep \widetilde{R}_{j,\ep, M}
-
\ep \widetilde{R}^t_{j,\ep, M}
=
R_{j,\ep,M}
, \ \ 1\leq j \leq 6.
\end{align}
Next, we set $\widetilde{\vZ}_{\ep, M}:=\vZe-\sum\limits_{j=1}^{6}\ep \widetilde{R}_{j,\ep, M}$. Subtracting \eqref{3.1.2} from \eqref{3.1.1}, we find that $\widetilde{\vZ}_{\ep, M}$ satisfies 
\begin{align}\label{3.1.4}
\p_t
&
\widetilde{\vZ}_{\ep, M}
+
\mathcal{Q}_1^\ep(\vV+\vVe, \widetilde{\vZ}_{\ep, m})
+
\mathcal{Q}_2^\ep(\underline{\vUe}, \widetilde{\vZ}_{\ep, M})
+
\mathcal{Q}_2^\ep(\mathbb{P}_0 \vUe, \widetilde{\vZ}_{\ep, M})
-
\overline{\mathcal{D}}(\widetilde{\vZ}_{\ep, M})
\nonumber
\\
&
=
\mathcal{L}\left(-\df{t}{\ep}\right)
\mathbb{P}^\perp r_\ep
-
\mathcal{Q}_2^\ep(\underline{\vWe}, \vV)
-
\mathcal{Q}_2^\ep(\mathbb{P}_0{\vWe}, \vV)
+
\sum_{j=1}^{6}R^M_{j,\ep}
\\
&
\quad
-
\ep\sum_{j=1}^{6}
\sp{
\widetilde{R}^t_{j, \ep, M}
+
\mathcal{Q}_1^\ep(\vV+\vVe, \widetilde{R}_{M, j, \ep})
+
\mathcal{Q}_2^\ep(\underline{\vUe}, \widetilde{R}_{j, \ep, M})
+
\mathcal{Q}_2^\ep(\mathbb{P}_0 \vUe, \widetilde{R}_{j, \ep, M})
-
\overline{\mathcal{D}}(\widetilde{R}_{j, \ep, M})
}\nonumber.
\end{align}
Applying Lemma \ref{le4.3} to \eqref{3.1.4} and using Lemma \ref{cons}, we obtain
\begin{align}\label{3.1.5}
&
\vZ^{\ep}_{T, \theta}
\leq
C
e^{C
\sp{
\vV_{T^*_0}^2
+ 
[T]
\sp{E^{\ep}_T}^2
}
}
\Bigg(
\n{\widetilde{\vZ}_{\ep, M}(0)}_{\underline{B}^{\f d 2 -1-\theta}_{2, 2}}
+
\n{r_\ep}_{L_T^1(\underline{B}^{\f d 2 -1-\theta}_{2, 2})}
+
\n{\mathcal{Q}^\ep_2(\vWe, \vV)}_{L_T^1(\underline{B}^{\f d 2 -1-\theta}_{2, 2})}
\nonumber
\\
& \quad 
+
\n{\mathcal{Q}^\ep_2(\mathbb{P}_0 \vWe, \vV)}_{L_T^1(\underline{B}^{\f d 2 -1-\theta}_{2, 2})}
+
\sum_{j=1}^{5}
\n{R^M_{j, \ep}}_{L_T^1(\underline{B}^{\f d 2 -1-\theta}_{2, 2})}\nonumber
+
\n{R^M_{6, \ep}}_{L_T^2(\underline{B}^{\f d 2 -2-\theta}_{2, 2})}
\\
&\quad 
+
\ep
\sum_{1}^{6}
\left(
\n{\widetilde{R}^t_{j, \ep, M}}_{L_T^1(\underline{B}^{\f d 2 -1-\theta}_{2, 2})}
+
\n{\mathcal{Q}_1^\ep(\vV+\vVe, \widetilde{R}_{j, \ep, M} )}_{L_T^1(\underline{B}^{\f d 2 -1-\theta}_{2, 2})}
+
\n{\mathcal{Q}_2^\ep(\underline{\vUe}, \widetilde{R}_{j, \ep, M} )}_{L_T^1(\underline{B}^{\f d 2 -1-\theta}_{2, 2})}
\right.
\nonumber
\\
&
\left. \quad 
+
\n{\mathcal{Q}_2^\ep(\mathbb{P}_0\vUe, \widetilde{R}_{j, \ep, M} )}_{L_T^2(\underline{B}^{\f d 2 -2-\theta}_{2, 2})}
+
\n{\overline{\mathcal{D}}(\widetilde{R}_{j, \ep, M} )}_{L_T^2(\underline{B}^{\f d 2 -2-\theta}_{2, 2})}
+
\n{\widetilde{R}_{j, \ep, M} }_{
\widetilde{L}_T^\infty (B^{\f d 2-1-\theta}_{2, 2})
\cap 
L_T^2 (B^{\f d 2-\theta}_{2, 2})}
\right).
\end{align}

Next, we choose $M$ sufficiently large, which will be determined later, to estimate the terms on the right-hand side of \eqref{3.1.5}. To this end, we introduce some basic estimates for the truncations $\vA_M$
and $\vA^M$ of $\vA$. For $1\leq  r \leq \infty$, $1\leq M$, $-\infty<s<\infty$, $s_1\geq0$ and $s_2>0$, we have
\begin{align}
&
\n{\vA_M}_{B^{s+s_1}_{2, r}}
\les
M^{s_1}
\n{\vA}_{B^{s}_{2, r}},
\qquad 
\n{\vA^M}_{B^{s-s_1}_{2, r}}
\les
M^{-s_1}
\n{\vA}_{\underline{B}^{s}_{2, r}},
\label{3.1.6}
\\
&
\n{|k|^{-s_2}\widehat{(\vA_M)}_k}_{\ell^1\sp{
\widetilde{\mathbb{Z}}^d\setminus\{0\}
}
}
\les
M^{s_1}
\n{\vA}_{\underline{B}^{\f d 2-s_1}_{2, 2}},
\qquad 
\n{|k|^{-s_2}\widehat{(\vA^M)}_k}_{\ell^1\sp{
\widetilde{\mathbb{Z}}^d\setminus\{0\}
}
}
\les
M^{-s_2}
\n{\vA}_{\underline{B}^{\f d 2}_{2, 2}}.\label{3.1.7}
\end{align}
We set
\begin{align}
\widetilde{B}(M):=
&
\Mp{(k, l, m, \alpha, \beta, \gamma): k=l+m, |m|\leq M, |l|\leq M, \lambda_k^\alpha-\lambda_l^\beta-\lambda_m^\gamma \neq 0},
\nonumber
\\
\widecheck{B}(M):=
&
\Mp{(k, l, m, \alpha, \gamma): k=l+m, |m|\leq M, |l|\leq M, \lambda_k^\alpha-\lambda_m^\gamma \neq 0},
\nonumber
\\
\widetilde{C}_M:=
&
\sup
\Mp{|\lambda_k^\alpha-\lambda_l^\beta-\lambda_m^\gamma|^{-1}: 
(k, l, m, \alpha, \beta, \gamma) \in \widetilde{B}(M)},
\label{xishu1}
\\
\widecheck{C}_M:=
&
\sup
\Mp{|\lambda_k^\alpha-\lambda_m^\gamma|^{-1}: 
(k, l, m, \alpha, \gamma) \in \widecheck{B}(M)}.
\label{xishu2}
\end{align}

\noindent $\bullet$ \textbf{Estimates for}  $\widetilde{R}_{j, \ep, M}$ and $\widetilde{\vZ}_{\ep, M}(0)$. 

From \eqref{minski}, \eqref{minski2}, \eqref{3.1.6}, \eqref{3.1.7}, Lemma \ref{cons}  and Lemma \ref{le4.4}, we obtain\footnote{In this paper, we shall frequently use the fact that $\ell^{r_1}(\mathbb{Z}) \hookrightarrow \ell^{r_2}(\mathbb{Z}), 1\leq r_1\leq r_2\leq\infty$, without further mention.}
\begin{align*}
&
\n{\widetilde{R}_{1, \ep, M} }_{
\widetilde{L}_T^\infty (B^{\f d 2-1-\theta}_{2, 2})
}
\les
\n{\widetilde{R}_{1, \ep, M} }_{
L_T^\infty (B^{\f d 2-1-\f{\theta} {2}}_{2, 2})
}
\les
\widetilde{C}_M
\n{|k|^{-\f{\theta}{2}}\widehat{(\vV_M)}_k}_{L_T^\infty
\sp{
\ell^1
\sp{
\widetilde{\mathbb{Z}}^d\setminus\{0\}
}
}
}
\n{\vV_M}_{L^\infty_T(B^{\f d 2}_{2, 2})}
\\
&
\qquad
\qquad
\qquad
\qquad
\quad
\les
\widetilde{C}_M
M^2
 \n{\vV}^2_{
L_T^\infty (B^{\f d 2-1}_{2, 2})
}
\les 
\widetilde{C}_M
M^2
\vV_{T^*_0}^2,
\\[4pt]
&
\n{\widetilde{R}_{1, \ep, M} }_{
L_T^2 (B^{\f d 2-\theta}_{2, 2})
}
\les
\widetilde{C}_M
M^2
 \n{\vV}_{
L_T^\infty (B^{\f d 2-1}_{2, 2})
}
 \n{\vV}_{
L_T^2 (B^{\f d 2}_{2, 2})
}
\les
\widetilde{C}_M
M^2
\vV_{T^*_0}^2 , 
\\[4pt]
&
 \n{\widetilde{R}_{2, \ep, M} }_{
\widetilde{L}_T^\infty (B^{\f d 2-1-\theta}_{2, 2})
}
\les
\widecheck{C}_M
M^2
 \n{\vV}_{
L_T^\infty (B^{\f d 2-1}_{2, 2})
}
 \n{\vU}_{
L_T^\infty (\underline{B}^{\f d 2-1}_{2, 2})
}
\leq
\widecheck{C}_M
M^2
\vV_{T^*_0}
\vU_{T^*_0},
\\[4pt]
&
\n{\widetilde{R}_{2, \ep, M} }_{
L_T^2 (B^{\f d 2-\theta}_{2, 2})
}
\les
\widecheck{C}_M
M^2
 \n{\vU}_{
L_T^\infty (\underline{B}^{\f d 2-1}_{2, 2})
}
 \n{\vV}_{
L_T^2 (B^{\f d 2}_{2, 2})
}
\les
\widecheck{C}_M
M^2
\vV_{T^*_0}
\vU_{T^*_0},
\\[4pt]
&
\n{\widetilde{R}_{3, \ep, M} }_{
\widetilde{L}_T^\infty (B^{\f d 2-1-\theta}_{2, 2})
}
\les
\sp{|\widehat{(a_0)}_0|+ |\widehat{(b_0)}_0 |} 
\n{\vV}_{
L_T^\infty (B^{\f d 2-1}_{2, 2})
}
\les
\vU_{T^*_0}
\vV_{T^*_0},
\\[4pt]
&
\n{\widetilde{R}_{3, \ep, M} }_{
L_T^2 (B^{\f d 2-\theta}_{2, 2})
}
\les
\sp{|\widehat{(a_0)}_0|+ |\widehat{(b_0)}_0 |} 
\n{\vV}_{
L_T^2 (B^{\f d 2}_{2, 2})
}
\les
\vU_{T^*_0}
\vV_{T^*_0},
\\[4pt]
&
\n{\widetilde{R}_{4, \ep, M} }_{
\widetilde{L}_T^\infty (B^{\f d 2-1-\theta}_{2, 2})
}
\les
\sp{|\widehat{(a_0)}_0|+ |\widehat{(b_0)}_0 |} 
\n{\vUe}_{
L_T^\infty (\underline{B}^{\f d 2-1}_{2, 2})
}
\les
\sp{E^{\ep}_T}^2,
\\[4pt]
&
\n{\widetilde{R}_{4, \ep, M} }_{
L_T^2 (B^{\f d 2-\theta}_{2, 2})
}
\les
\sp{|\widehat{(a_0)}_0|+ |\widehat{(b_0)}_0 |} 
\n{\vU}_{
L_T^2 (\underline{B}^{\f d 2}_{2, 2})
}
\les
[T]^{\f1 2}
\sp{E^{\ep}_T}^2,
\\[4pt]
&
\n{\widetilde{R}_{5, \ep, M} }_{
\widetilde{L}_T^\infty (B^{\f d 2-1-\theta}_{2, 2})
}
\les
\n{\widetilde{R}_{5, \ep, M} }_{
L_T^\infty (B^{\f d 2-1}_{2, 2})
}
\les
\n{|k|^{-1}\widehat{(\vU_{\ep, M})}_k}_{L_T^\infty
\sp{
\ell^1
\sp{
\widetilde{\mathbb{Z}}^d\setminus\{0\}
}
}
}
 \n{\vU_{\ep, M}}_{
L_T^\infty (\underline{B}^{\f d 2}_{2, 2})
}
\\
&
\qquad
\qquad
\qquad
\qquad
\quad
\les
M^2
 \n{\vUe}^2_{
L_T^\infty (\underline{B}^{\f d 2-1}_{2, 2})
}
\les
M^2
\sp{E^{\ep}_T}^2,
\\[4pt]
&
\n{\widetilde{R}_{5, \ep, M} }_{
L_T^2 (B^{\f d 2-\theta}_{2, 2})
}
\les
M
\n{\vUe}_{
L_T^\infty (\underline{B}^{\f d 2-1}_{2, 2})
}
\n{\vUe }_{
L_T^2 (\underline{B}^{\f d 2}_{2, 2})
}
\les
M
[T]^{\f1 2}
\sp{E^{\ep}_T}^2,
\\[4pt]
&
\n{\widetilde{R}_{6, \ep, M} }_{
\widetilde{L}_T^\infty (B^{\f d 2-1-\theta}_{2, 2})
}
\les
M
\n{\vVe}_{
L_T^\infty (B^{\f d 2-1}_{2, 2})
}
\les
M
E^{\ep}_{T},
\\[4pt]
&
\n{\widetilde{R}_{6, \ep, M} }_{
L_T^2 (B^{\f d 2-\theta}_{2, 2})
}
\les
M
\n{\vVe}_{
L_T^2 (B^{\f d 2}_{2, 2})
}
\les
M
E^{\ep}_{T}.
\end{align*}
Collecting the above estimates and applying Young's inequality, we obtain
\begin{align}
&
\ep
\sum\limits_{j=1}^6
    \n{\widetilde{R}_{j, \ep, M} }_{
\widetilde{L}_T^\infty (B^{\f d 2-1-\theta}_{2, 2})
}
\les
\ep
C_M
\sp{
\vV^2_{T_0}
+
\vU^2_{T_0}
+
\sp{
E^{\ep}_{T}
}^2
+
E^{\ep}_{T}
},
\label{3.1.8}
\\
&
\ep
\sum\limits_{j=1}^6
\n{\widetilde{R}_{j, \ep, M} }_{
L_T^2 (B^{\f d 2-\theta}_{2, 2})}
\les
\ep
C_M
\sp{
\vV^2_{T_0}
+
\vU^2_{T_0}
+
[T]^{\f1 2}
\sp{E^{\ep}_{T}}^2
+
E^{\ep}_{T}
},
\label{3.1.9}
\end{align}
where we set $C_M:=\max\{\widetilde{C}_M M^{2+\theta}, \widecheck{C}_M M^{2+\theta}, M^2\}$. Here, $\theta$ is introduced to unify the notation with the coefficients appearing later. The same argument gives
\begin{align}
&
\n{\widetilde{\vZ}_{\ep, M}(0)}_{B^{\f d 2 -1-\theta}_{2, 2}}
\les
\ep
\sum\limits_{j=1}^6
\n{\widetilde{R}_{j, \ep, M} (0)}_{
B^{\f d 2-1-\theta}_{2, 2}
}
\les
\ep
C_M
\sp{
E_0
+
E_0^2
}
\label{3.1.10}.
\end{align}

\noindent $\bullet$ \textbf{Estimates for}  $\mathcal{Q}_1^\ep(\vV+\vVe, \widetilde{R}_{j, \ep, M} )$, $\mathcal{Q}_2^\ep(\underline{\vUe}, \widetilde{R}_{j, \ep, M} )$, $\mathcal{Q}_2^\ep(\mathbb{P}_0\vUe, \widetilde{R}_{j, \ep, M} )$ and $\overline{\mathcal{D}}(\widetilde{R}_{j, \ep, M} ))$.

From \eqref{minski2} and Lemma \ref{le4.2}  we get
\begin{align*}
&
\ep
\sum\limits_{j=1}^6
\n{\mathcal{Q}_1^\ep(\vV+\vVe, \widetilde{R}_{j, \ep, M} )}_{L_T^1(\underline{B}^{\f d 2 -1-\theta}_{2, 2})}
\les
\ep
\sum\limits_{j=1}^6
\n{\vV+\vVe}_{L_T^2(B^{\f d 2 }_{2, 1})}
\n{\widetilde{R}_{j, \ep, M} }_{
L_T^2 (B^{\f d 2-\theta}_{2, 2})} , 
\\
&
\ep
\sum\limits_{j=1}^6
\n{\mathcal{Q}_2^\ep(\underline{\vUe}, \widetilde{R}_{j, \ep, M} )}_{L_T^1(\underline{B}^{\f d 2 -1-\theta}_{2, 2})}
\les
\ep
\sum\limits_{j=1}^6
\n{\vUe}_{L_T^2(\underline{B}^{\f d 2 }_{2, 1})}
\n{\widetilde{R}_{j, \ep, M} }_{
L_T^2 (B^{\f d 2-\theta}_{2, 2})} , 
\\
&
\ep
\sum\limits_{j=1}^6
\n{\mathcal{Q}_2^\ep(\mathbb{P}_0 \vUe, \widetilde{R}_{j, \ep, M} )}_{L_T^2(\underline{B}^{\f d 2 -2-\theta}_{2, 2})}
\les
\ep
\sum\limits_{j=1}^6
\sp{|\widehat{(a_0)}_0|+|\widehat{(b_0)}_0|+\n{\widehat{(\vue)}_0}_{L^\infty(0, T)}}
\n{\widetilde{R}_{j, \ep, M} }_{
L_T^2 (B^{\f d 2-\theta}_{2, 2})} , 
\\
&
\ep
\sum\limits_{j=1}^6
\n{
\overline{\mathcal{D}}(\widetilde{R}_{j, \ep, M}) }_{L_T^2(\underline{B}^{\f d 2 -2-\theta}_{2, 2})}
\les
\ep
\sum\limits_{j=1}^6
\n{
\widetilde{R}_{j, \ep, M} 
}_{L_T^2(B^{\f d 2 -\theta}_{2, 2})}.
\end{align*}
Then, using Lemma \ref{cons} and \eqref{3.1.9}, we obtain
\begin{align}
&
\ep
\sum\limits_{j=1}^6
\left(
\n{\mathcal{Q}_1^\ep(\vV+\vVe, \widetilde{R}_{j, \ep, M} )}_{L_T^1(\underline{B}^{\f d 2 -1-\theta}_{2, 2})}  
+
\n{\mathcal{Q}_2^\ep(\underline{\vUe}, \widetilde{R}_{j, \ep, M} )}_{L_T^1(\underline{B}^{\f d 2 -1-\theta}_{2, 2})}
\right.
\nonumber
\\
&
\left.
\quad
+
\n{\mathcal{Q}_2^\ep(\mathbb{P}_0 \vUe, \widetilde{R}_{j, \ep, M} )}_{L_T^2(\underline{B}^{\f d 2 -2-\theta}_{2, 2})}
+
\n{
\overline{\mathcal{D}}(\widetilde{R}_{j, \ep, M}) }_{L_T^2(\underline{B}^{\f d 2 -2-\theta}_{2, 2})}
\right)
\nonumber
\\
& \qquad \qquad 
\les
\ep
C_M
\sp{
\vV_{T^*_0}
+
[T]^\f1 2
E^{\ep}_{T}
+1
}
\sp{
\vV^2_{T_0}
+
\vU^2_{T_0}
+
[T]^{\f1 2}
\sp{E^{\ep}_{T}}^2
+
E^{\ep}_{T}
} . 
\label{3.1.11}
\end{align}

\noindent $\bullet$ \textbf{Estimates for}  $\widetilde{R}^t_{j, \ep, M}$.

From \eqref{minski}, \eqref{minski2}, \eqref{3.1.6}, \eqref{3.1.7} and Lemma \ref{le4.4}, we get
\begin{align*}
&
\n{\widetilde{R}^t_{1, \ep, M}}_{L_T^1(\underline{B}^{\f d 2 -1-\theta}_{2, 2})}
\les
\widetilde{C}_M
\n{|k|^{-\theta}\widehat{(\p_t\vV_M)}_k}_{L_T^1
\sp{
\ell^1
\sp{
\widetilde{\mathbb{Z}}^d\setminus\{0\}
}
}
}
\n{\vV_M}_{L^\infty_T(B^{\f d 2}_{2, 2})}
\\
&
\qquad
\qquad
\qquad
\qquad
\qquad
+
\widetilde{C}_M
\n{|k|^{-\theta}\widehat{(\vV_M)}_k}_{L_T^\infty
\sp{
\ell^1
\sp{
\widetilde{\mathbb{Z}}^d\setminus\{0\}
}
}
}
\n{\p_t\vV_M}_{L^1_T(B^{\f d 2}_{2, 2})}
\\
&
\qquad
\qquad
\qquad
\qquad
\quad
\les
\widetilde{C}_M M^{2+\theta}
\n{\p_t\vV}_{L^1_T(B^{\f d 2-1-\theta}_{2, 2})}
\n{\vV}_{L^\infty_T(B^{\f d 2-1}_{2, 2})},
\\[4pt]
&
\n{\widetilde{R}^t_{2, \ep, M}}_{L_T^1(\underline{B}^{\f d 2 -1-\theta}_{2, 2})}
\les
\n{\widetilde{R}^t_{2, \ep, M}}_{L_T^1(\underline{B}^{\f d 2 -1}_{2, 2})}
\\
&
\qquad
\qquad
\qquad
\qquad
\quad
\les
\widetilde{C}_M
\n{|k|^{-\theta}\widehat{(\p_t\vV_M)}_k}_{L_T^1
\sp{
\ell^1
\sp{
\widetilde{\mathbb{Z}}^d\setminus\{0\}
}
}
}
\n{\vV_M}_{L^\infty_T(\underline{B}^{\f d 2+\theta}_{2, 2})}
\\
&
\qquad
\qquad
\qquad
\qquad
\qquad
+
\widetilde{C}_M
\n{|k|^{-\theta}\widehat{(\p_t\vU_M)}_k}_{L_T^1
\sp{
\ell^1
\sp{
\widetilde{\mathbb{Z}}^d\setminus\{0\}
}
}
}
\n{\vU_M}_{L^\infty_T(\underline{B}^{\f d 2+\theta}_{2, 2})}
\\
&
\qquad
\qquad
\qquad
\qquad
\quad
\les
\widecheck{C}_M
M^{2+2\theta}
\sp{
\n{\p_t\vV}_{L_T^1(B^{\f d 2 -1-\theta}_{2, 2})}
\n{\vU}_{L_T^\infty(\underline{B}^{\f d 2 -1}_{2, 2})}
+
\n{\p_t\vU}_{L_T^1(\underline{B}^{\f d 2 -1}_{2, 2})}
\n{\vV}_{L_T^\infty(B^{\f d 2 -1}_{2, 2})}
},
\\[4pt]
&
\n{\widetilde{R}^t_{3, \ep, M}}_{L_T^1(\underline{B}^{\f d 2 -1-\theta}_{2, 2})}
\les
\sp{|\widehat{(a_0)}_0|+ |\widehat{(b_0)}_0 |} 
\n{\p_t\vV}_{L_T^1(B^{\f d 2 -1-\theta}_{2, 2})},
\\[4pt]
&
\n{\widetilde{R}^t_{4, \ep, M}}_{L_T^1(\underline{B}^{\f d 2 -1-\theta}_{2, 2})}
\les
\sp{|\widehat{(a_0)}_0|+ |\widehat{(b_0)}_0 |} 
\n{\p_t\vUe}_{L_T^1(\underline{B}^{\f d 2-1}_{2, 1})},
\end{align*}
\begin{align*}
&
\n{\widetilde{R}^t_{5, \ep, M}}_{L_T^1(\underline{B}^{\f d 2 -1-\theta}_{2, 2})}
\les
M^2
\n{\p_t\vUe}_{L_T^1(\underline{B}^{\f d 2 -1}_{2, 1})}
\n{\vUe}_{\widetilde{L}_T^\infty(\underline{B}^{\f d 2 -1}_{2, 1})},
\\
&
\n{\widetilde{R}^t_{6, \ep, M}}_{L_T^1(\underline{B}^{\f d 2 -1-\theta}_{2, 2})}
\les
M
\n{\p_t\vVe}_{L_T^1(B^{\f d 2 -1}_{2, 1})}.
\end{align*}
We claim that
\begin{align}
&
\n{\p_t\vV}_{L_T^1(B^{\f d 2 -1-\theta}_{2, 2})}
\les
\vV_{T^*_0}
\sp{
\vV_{T^*_0}+[T]^{\f1 2}\vU_{T^*_0}+1
},
\nonumber
\\
&
\n{\p_t\vU}_{L_T^1(\underline{B}^{\f d 2 -1}_{2, 1})}
\les
\vU_{T^*_0}
\sp{
[T]^{\f1 2}
\vU_{T^*_0}
+
1
},
\nonumber
\\
&
\n{\p_t\vVe}_{L_T^1(B^{\f d 2 -1}_{2, 1})}
\les
[T]
E_{T}^{\ep}
\sp{
\sp{
E_{T}^{\ep}
}^2
+
E_{T}^{\ep}
+
1},
\label{tV}
\\
&
\n{\p_t\vUe}_{L_T^1(\underline{B}^{\f d 2 -1}_{2, 1})}
\les
[T]
E_{T}^{\ep}
\sp{
\sp{
E_{T}^{\ep}
}^2
+
E_{T}^{\ep}
+
1}.
\label{tU}
\end{align}
The proof of this claim will be given later. Combining the estimates at hand, we obtain
\begin{align}
\ep
\sum_{j=1}^6
\n{\widetilde{R}^t_{j, \ep, M}}_{L_T^1(\underline{B}^{\f d 2 -1-\theta}_{2, 2})}
\les
\ep
C_M
\left[
\sp{
\vV^2_{T_0}
+
\vU^2_{T_0}
}
\sp{
\vV_{T^*_0}+[T]^{\f1 2}\vU_{T^*_0}+1
}
+
[T]
\sp{
\sp{
E_{T}^{\ep}
}^2
+
E_{T}^{\ep}
+
1}^2
\right].
\label{3.1.17}
\end{align}

\noindent $\bullet$ \textbf{Estimates for}  $\p_t \vV, \p_t \vU$, $\p_t \vVe$ and $\p_t \vUe$.

Similar to the argument leading to \eqref{3.1.16}, we obtain
\begin{align*}
\n{
r_\ep
}_{L_T^1(B_{2, 1}^{\f d 2-1})}
\les
[T]
\sp{
E_{T}^{\ep}
}^2
\sp{
E_{T}^{\ep}+1
}.
\end{align*}
Using \eqref{3.1.3}, we get
\begin{align*}
\n{\mathcal{D}^\ep(\vVe)}_{L_T^1(B_{2, 1}^{\f d 2-1})}
=
\n{
\mathcal{L}\left(-\dfrac{t}{\ep}\right) 
\sp{
0, 0, \df{\nu}{R^0+ Q^0}\De \mathcal{P}^\perp \vue
}
}_{L_T^1(B_{2, 1}^{\f d 2-1})}
\les
\n{\mathcal{P}^\perp \vue}_{L_T^1(B_{2, 1}^{\f d 2+1})}.
\end{align*}
From Lemma \ref{cons}, Lemma \ref{le8.2}, Lemma \ref{4.2}, Lemma \ref{le6.1} and Lemma \ref{le6.2}, we get
\begin{align*}
&
\n{\p_t\vV}_{L_T^1(B^{\f d 2 -1-\theta}_{2, 2})}
\\
&
\leq
\n{\mathcal{Q}_1(\vV, \vV)}_{L_T^1(B^{\f d 2 -1-\theta}_{2, 2})}
+
\n{\mathcal{Q}_2(\underline{\vU}, \vV)}_{L_T^1(B^{\f d 2 -1-\theta}_{2, 2})}
+
\n{\mathcal{Q}_2(\mathbb{P}_0 \vU, \vV)}_{L_T^1(B^{\f d 2 -1-\theta}_{2, 2})}
+
\n{\overline{\mathcal{D}}(\vV)}_{L_T^1(B^{\f d 2 -1-\theta}_{2, 2})}
\\
&
\les
\n{\mathcal{Q}_1(\vV, \vV)}_{L_T^1(B^{\f d 2 -1}_{2, 1})}
+
\n{\mathcal{Q}_2(\underline{\vU}, \vV)}_{L_T^1(B^{\f d 2 -1-\theta}_{2, 2})}
+
\n{\mathcal{Q}_2(\mathbb{P}_0 \vU, \vV)}_{L_T^1(B^{\f d 2 -1}_{2, 1})}
+
\n{\overline{\mathcal{D}}(\vV)}_{L_T^1(B^{\f d 2 -1}_{2, 1})}
\\
&
\les
\n{\vV}^2_{L_T^2(B^{\f d 2 }_{2, 1})}
+
\n{\vV}_{L_T^2(B^{\f d 2 }_{2, 1})}
\n{\vU}_{L_T^2(\underline{B}^{\f d 2 }_{2, 1})}
+
\sp{|\widehat{(a_0)}_0|+|\widehat{(b_0)}_0|+|\widehat{(\vu_0)}_0|}
\n{\vV}_{L_T^1(B^{\f d 2 +1}_{2, 1})}
+
\n{\vV}_{L_T^1(B^{\f d 2 +1}_{2, 1})}
\\
&
\les
\vV_{T^*_0}
\sp{
\vV_{T^*_0}+[T]^{\f1 2}\vU_{T^*_0}+1
},
\end{align*}
\begin{align*}
&
\n{\p_t\vU}_{L_T^1(B^{\f d 2 -1}_{2, 1})}
\\
&
\leq
\n{\mathbb{P}\mathcal{Q}(\underline{\vU}, \underline{\vU})}_{L_T^1(B^{\f d 2 -1}_{2, 1})}
+
2
\n{\mathbb{P}\mathcal{Q}(\mathbb{P}_0\vU, \underline{\vU})}_{L_T^1(B^{\f d 2 -1}_{2, 1})}
+
\n{\mathcal{D}\vU}_{L_T^1(B^{\f d 2 -1}_{2, 1})}
\\
&
\les
\n{\vU}^2_{L_T^2(\underline{B}^{\f d 2 }_{2, 1})}
+
[T]^{\f1 2}
\sp{|\widehat{(a_0)}_0|+|\widehat{(b_0)}_0|+|\widehat{(\vu_0)}_0|}
\n{\vU}_{L_T^2(\underline{B}^{\f d 2 }_{2, 1})}
+
\n{\vw}_{L_T^1(B^{\f d 2 +1}_{2, 1})}
\\
&
\les
\vU_{T^*_0}
\sp{
[T]^{\f1 2}
\vU_{T^*_0}
+
1
},
\\[4pt]
&
\n{\p_t\vVe}_{L_T^1(B^{\f d 2 -1}_{2, 1})}
\\
&
\leq
\n{
\mathcal{Q}_1^\ep(\vVe, \vVe)
}_{L_T^1(B^{\f d 2 -1}_{2, 1})}
+
\n{
\mathcal{Q}_2^\ep(\underline{\vUe}, \vVe)
}_{L_T^1(B^{\f d 2 -1}_{2, 1})}
+
\n{
\mathcal{Q}_2^\ep(\mathbb{P}_0\vUe, \vVe)
}_{L_T^1(B^{\f d 2 -1}_{2, 1})}
\\
&
\quad
+
\n{
\mathcal{D}^\ep(\vVe)
}_{L_T^1(B^{\f d 2 -1}_{2, 1})}
+
\n{
\mathcal{L}\left(-\dfrac{t}{\ep}\right) \mathbb{P}^\perp r_\ep
}_{L_T^1(B^{\f d 2 -1}_{2, 1})}
+
\n{
\mathcal{L}\left(-\dfrac{t}{\ep}\right) \mathbb{P}^\perp\mathcal{Q}(\underline{\vUe}, \underline{\vUe})
}_{L_T^1(B^{\f d 2 -1}_{2, 1})}
\\
&
\quad
+
2
\n{\mathcal{L}\left(-\dfrac{t}{\ep}\right) \mathbb{P}^\perp\mathcal{Q}(\mathbb{P}_0 \vUe, \underline{\vUe})
}_{L_T^1(B^{\f d 2 -1}_{2, 1})}
\\
&
\les
\n{\vVe}^2_{L_T^2(B^{\f d 2}_{2, 1})}
+
\n{\vVe}_{L_T^2(B^{\f d 2}_{2, 1})}
\n{\vUe}_{L_T^2(\underline{B}^{\f d 2}_{2, 1})}
+
[T]^{\f1 2}
\sp{|\widehat{(a_0)}_0|+|\widehat{(b_0)}_0|+|\widehat{(\vue)}_0|_{L^\infty(0, T)}}
\n{\vVe}_{L_T^2(B^{\f d 2}_{2, 1})}
\\
&
\quad
+
\n{\mathcal{P}^\perp \vue}_{L_T^1(B^{\f d 2+1}_{2, 1})}
+
[T]
\sp{
E_{T}^{\ep}
}^2
\sp{
E_{T}^{\ep}+1
}
+
\n{\vUe}^2_{L_T^2(\underline{B}^{\f d 2}_{2, 1})}
\\
&
\quad
+
[T]^{\f1 2}
\sp{|\widehat{(a_0)}_0|+|\widehat{(b_0)}_0|+|\widehat{(\vue)}_0|_{L^\infty(0, T)}}
\n{\vUe}_{L_T^2(B^{\f d 2}_{2, 1})}
\\
&
\les
[T]
E_{T}^{\ep}
\sp{
\sp{
E_{T}^{\ep}
}^2
+
E_{T}^{\ep}
+
1},
\\[4pt]
&
\n{\p_t\vUe}_{L_T^1(\underline{B}^{\f d 2 -1}_{2, 1})}
\\
&
\leq
\n{\mathbb{P}\mathcal{Q}(\underline{\vUe}, \underline{\vUe})}_{L_T^1(B^{\f d 2 -1}_{2, 1})}
+
2
\n{\mathbb{P}\mathcal{Q}(\mathbb{P}_0\vUe, \underline{\vUe})}_{L_T^1(B^{\f d 2 -1}_{2, 1})}
+
\n{
\mathcal{D}(\vUe)}_{L_T^1(B^{\f d 2 -1}_{2, 1})}
+
\n{
\mathbb{P} r_\ep}_{L_T^1(B^{\f d 2 -1}_{2, 1})}
\\
&
\quad
+
2
\n{
\mathbb{P}\mathcal{Q}\left(\underline{\vUe}, \mathcal{L}\left(\df{t}{\ep}\right) \vVe\right)
}_{L_T^1(B^{\f d 2 -1}_{2, 1})}
+
2
\n{
\mathbb{P}\mathcal{Q}\left(\mathbb{P}_0\vUe, \mathcal{L}\left(\df{t}{\ep}\right) \vVe\right)
}_{L_T^1(B^{\f d 2 -1}_{2, 1})}
\\
&
\les
[T]
E_{T}^{\ep}
\sp{
\sp{
E_{T}^{\ep}
}^2
+
E_{T}^{\ep}
+
1}.
\end{align*}

\medskip

\noindent $\bullet$ \textbf{Estimates for}  $R^M_{j, \ep}$.

Observe that
\begin{align*}
&
R^M_{1, \ep}
=
\sp{
2\mathcal{Q}_1(\vV_M, \vV^M)
+
\mathcal{Q}_1(\vV^M, \vV^M)
}
-
\sp{
2\mathcal{Q}^\ep_1(\vV_M, \vV^M)
+
\mathcal{Q}^\ep_1(\vV^M, \vV^M)
},
\\
&
R^M_{2, \ep}
=
\sp{
\mathcal{Q}_2(\underline{\vU_M}, \vV^M)
+
\mathcal{Q}_2(\vU^M, \vV_M)
+
\mathcal{Q}_2(\vU^M, \vV_M)
}
\\
&
\qquad
\quad
-
\sp{
\mathcal{Q}^\ep_2(\underline{\vU_M}, \vV^M)
+
\mathcal{Q}^\ep_2(\vU^M, \vV_M)
+
\mathcal{Q}^\ep_2(\vU^M, \vV^M)
}.
\end{align*}
From \eqref{3.1.6}, \eqref{3.1.7}, Lemma \ref{cons} and Lemma \ref{le4.4}, we get
\begin{align}
&
\n{\mathcal{Q}_1(\vV_M, \vV^M)}_{L_T^1(\underline{B}^{\f d 2 -1-\theta}_{2, 2})}
\nonumber
\\
& \quad 
\les
\n{|k|^{-\f{\theta}{2}}\widehat{(\vV_M)}_k}_{L_T^2
\sp{
\ell^1
\sp{
\widetilde{\mathbb{Z}}^d\setminus\{0\}
}
}
}
\n{\vV^M}_{L_T^2(B^{\f d 2 -\f \theta 2}_{2, 2})}
+
\n{|k|^{-\theta}\widehat{(\vV^M)}_k}_{L_T^2
\sp{
\ell^1
\sp{
\widetilde{\mathbb{Z}}^d\setminus\{0\}
}
}
}
\n{\vV_M}_{L_T^2(B^{\f d 2}_{2, 2})}
\nonumber
\\
& \quad 
\les
\sp{
M^{-\f \theta 2}
+
M^{-\theta}
}
\vV^2_{T_0},
\nonumber
\\
&
\n{\mathcal{Q}_1(\vV^M, \vV^M)}_{L_T^1(\underline{B}^{\f d 2 -1-\theta}_{2, 2})}
\les
\n{|k|^{-\theta}\widehat{(\vV^M)}_k}_{L_T^2
\sp{
\ell^1
\sp{
\widetilde{\mathbb{Z}}^d\setminus\{0\}
}
}
}
\n{\vV^M}_{L_T^2(B^{\f d 2}_{2, 2})}
\les
M^{-\theta}
\vV^2_{T_0},
\nonumber
\\
&
\n{\mathcal{Q}_2(\underline{\vU_M}, \vV^M)}_{L_T^1(\underline{B}^{\f d 2 -1-\theta}_{2, 2})}
+
\n{\mathcal{Q}_2(\vU^M, \vV_M)}_{L_T^1(\underline{B}^{\f d 2 -1-\theta}_{2, 2})}
+
\n{\mathcal{Q}_2(\vU^M, \vV^M)}_{L_T^1(\underline{B}^{\f d 2 -1-\theta}_{2, 2})}
\nonumber
\\
&  \quad 
\les
\sp{
M^{-\f \theta 2}
+
M^{-\theta}
}
\n{\vV}_{L_T^2(B^{\f d 2}_{2, 2})}
\n{\vU}_{L_T^2(\underline{B}^{\f d 2 }_{2, 2})}
\les
\sp{
M^{-\f \theta 2}
+
M^{-\theta}
}
[T]^\f1 2
\vV_{T^*_0}
\vU_{T^*_0},
\label{3.1.14}
\end{align}
where we have used the relations 
$|k|=|m|$ and $|l|\leq 2|m|$
in the derivation of \eqref{3.1.14}. From \eqref{3.1.6}, Lemma \ref{cons} and Lemma \ref{le4.2}, we infer that 
\begin{align*}
&
\n{\mathcal{Q}^\ep_1(\vV_M, \vV^M)}_{L_T^1(\underline{B}^{\f d 2 -1-\theta}_{2, 2})}
+
\n{\mathcal{Q}^\ep_1(\vV^M, \vV^M)}_{L_T^1(\underline{B}^{\f d 2 -1-\theta}_{2, 2})}
\\
&  \quad 
\les
\sp{
\n{\vV^M}_{L_T^2(B^{\f d 2}_{2, 1})}
+
\n{\vV_M}_{L_T^2(B^{\f d 2}_{2, 1})}
}
\n{\vV^M}_{L_T^2(B^{\f d 2-\theta}_{2, 2})}
\\
&  \quad 
\les
\sp{
M^{-\f \theta 2}
+
M^{\theta}
}
\vV^2_{T_0},
\\[4pt]
&
\n{\mathcal{Q}^\ep_2(\underline{\vU_M}, \vV^M)}_{L_T^1(\underline{B}^{\f d 2 -1-\theta}_{2, 2})}
+
\n{\mathcal{Q}^\ep_2(\vU^M, \vV_M)}_{L_T^1(\underline{B}^{\f d 2 -1-\theta}_{2, 2})}
+
\n{\mathcal{Q}^\ep_2(\vU^M, \vV^M)}_{L_T^1(\underline{B}^{\f d 2 -1-\theta}_{2, 2})}
\\
&  \quad 
\les
\sp{
\n{\vU_M}_{L_T^2(\underline{B}^{\f d 2 }_{2, 1})}
+
\n{\vU^M}_{L_T^2(B^{\f d 2 }_{2, 1})}
+
\n{\vV_M}_{L_T^2(B^{\f d 2 }_{2, 1})}
}
\sp{
\n{\vV^M}_{L_T^2(B^{\f d 2 -\theta}_{2, 2})}
+
\n{\vU^M}_{L_T^2(B^{\f d 2 -\theta}_{2, 2})}
}
\\
&  \quad 
\les
M^{-\theta}
\sp{
[T]^{\f1 2}
\vU_{T^*_0}
+
\vV_{T^*_0}
}^2,
\\[4pt]
&
\n{R^M_{3, \ep}}_{L_T^1(\underline{B}^{\f d 2 -1-\theta}_{2, 2})}
\\
&  \quad 
\les
|\mathbb{P}_0 \vU|
\n{\vV^M}_{L_T^1(B^{\f d 2 -\theta}_{2, 2})}
\les
M^{-\theta}
|\mathbb{P}_0 \vU|
\n{\vV}_{L_T^1(B^{\f d 2}_{2, 2})}
\les
M^{-\theta}
|\mathbb{P}_0 \vU|
\n{\vV}_{L_T^1(B^{\f d 2+1}_{2, 2})}
\les
M^{-\theta}
\vU_{T^*_0}
\vV_{T^*_0},
\\[4pt]
&
\n{R^M_{4, \ep}}_{L_T^1(\underline{B}^{\f d 2 -1-\theta}_{2, 2})}
\\
&  \quad 
\les
\sp{|\widehat{(a_0)}_0|+|\widehat{(b_0)}_0|}
\n{\vUe^M}_{L_T^1(B^{\f d 2 -\theta}_{2, 2})}
\\
&  \quad 
\les
M^{-\theta}
\sp{|\widehat{(a_0)}_0|+|\widehat{(b_0)}_0|}
\n{\vUe}_{L_T^1(\underline{B}^{\f d 2}_{2, 2})}
\\
&  \quad 
\les
M^{-\theta}
[T]^{\f1 2}
\sp{|\widehat{(a_0)}_0|+|\widehat{(b_0)}_0|}
\n{\vUe}_{L_T^2(B^{\f d 2 }_{2, 2})}
\\
&  \quad 
\les
M^{-\theta}
[T]
\sp{E^{\ep}_T}^2,
\\
&
\n{R^M_{5, \ep}}_{L_T^1(\underline{B}^{\f d 2 -1-\theta}_{2, 2})}
\les
\sp{
\n{\vU_{\ep, M}}_{L_T^2(\underline{B}^{\f d 2 }_{2, 1})}
+
\n{\vUe^M}_{L_T^2(B^{\f d 2 }_{2, 1})}
}
\n{\vUe^M}_{L_T^2(B^{\f d 2-\theta }_{2, 2})}
\les
M^{-\theta}
[T]
\sp{E^{\ep}_T}^2,
\\[4pt]
&
\n{R^M_{6, \ep}}_{L_T^2(\underline{B}^{\f d 2 -2-\theta}_{2, 2})}
\les
\n{\vVe^M}_{L_T^2(B^{\f d 2-\theta }_{2, 2})}
\les
M^{-\theta}
E^{\ep}_{T}.
\end{align*}
Combining the above estimates yields
\begin{align}
\sum_{j=1}^{5}
\n{R^M_{j, \ep}}_{L_T^1(\underline{B}^{\f d 2 -1-\theta}_{2, 2})}
+
\n{R^M_{6, \ep}}_{L_T^2(\underline{B}^{\f d 2 -2-\theta}_{2, 2})}
\les
M^{-\f \theta 2}
\sp{
\sp{
[T]^{\f1 2}
\vU_{T^*_0}
+
\vV_{T^*_0}
}^2
+
[T]
\sp{E^{\ep}_T}^2
+
E^{\ep}_{T}
}.
\label{3.1.15}
\end{align}

\medskip

\noindent $\bullet$ \textbf{Estimates for}   $\mathcal{Q}^\ep_2(\mathbb{P}_0 \vWe, \vV)$ and $\mathcal{Q}^\ep_2(\underline{\vWe}, \vV)$

From Lemma \ref{le4.2} we get
\begin{align}
\n{\mathcal{Q}^\ep_2(\mathbb{P}_0\vWe, \vV)}_{L_T^2(\underline{B}^{\f d 2 -2-\theta}_{2, 2})}
&
\les
\n{\vV}_{L_T^2(B^{\f d 2 -1-\theta}_{2, 2})}
\n{\mathbb{P}_0 \vWe}_{L^\infty(0, T)}
\nonumber
\\
&
\les
\n{\vV}_{L_T^2(B^{\f d 2}_{2, 2})}
\n{\widehat{(\vue)}_0-\widehat{(\vu_0)}_0}_{L^\infty(0, T)}
\nonumber
\\
&
\les
\vV_{T^*_0}
\vW^\ep_{T, \theta},
\label{3.1.18}
\\[4pt]
\n{\mathcal{Q}^\ep_2(\underline{\vWe}, \vV)}_{L_T^1(\underline{B}^{\f d 2 -1-\theta}_{2, 2})}
&
\les
\n{\vWe}_{L_T^\infty(\underline{B}^{\f d 2-\theta}_{2, 2})}
\n{\vV}_{L_T^1(B^{\f d 2}_{2, 1})}
\nonumber
\\
&
\les
\n{\vWe}^{\f{\theta}{1+\theta}}_{L_T^\infty(\underline{B}^{\f d 2-1-\theta}_{2, 2})}
\n{\vWe}^{\f{1}{1+\theta}}_{L_T^\infty(\underline{B}^{\f d 2}_{2, 2})}
\n{\vV}_{L_T^1(B^{\f d 2+1}_{2, 1})}
\nonumber
\\
&
\les
\sp{
\vW^\ep_{T, \theta}
}^{\f{\theta}{1+\theta}}
\sp{
\vU_{T^*_0}+E^{\ep}_T
}^{\f{1}{1+\theta}}
\vV_{T^*_0}.
\label{3.1.19}
\end{align}
where we used 
$
\mathbb{P}_0\vWe
=
(0, 0, 
\df{\widehat{(\vue)}_0}{\sqrt{|\Td|}}-\df{\widehat{(\vu_0)}_0}{\sqrt{|\Td|}})^T.
$

\medskip

\noindent $\bullet$ \textbf{Estimates for}  $r_\ep$.\label{pressure}

From Lemma \ref{Le8.6} we see
\begin{align*}
\n{
\int_0^1
(\p_x \widetilde{p})(t\ep a_\ep, t\ep b_\ep)-(\p_x \widetilde{p})(0, 0)\dt
}_{B^{\f d 2-\theta}_{2,1}}
&
\leq
\int_0^1
\n{
(\p_x \widetilde{p})(t\ep a_\ep, t\ep b_\ep)-(\p_x \widetilde{p})(0, 0)
}_{B^{\f d 2-\theta}_{2,1}}\dt
\\
&
\les
\int_0^1 
t
\sp{
\ep
\n{
a_\ep
}_{B^{\f d 2-\theta}_{2,1}}
+
\ep
\n{
b_\ep
}_{B^{\f d 2-\theta}_{2,1}}
}
\dt
\\
&
\les
\ep\n{
a_\ep
}_{B^{\f d 2-\theta}_{2,1}}
+
\ep
\n{
b_\ep
}_{B^{\f d 2-\theta}_{2,1}}.
\end{align*}
Then, by Lemma \ref{cons}, we obtain
\begin{align*}
&
\n{
a_\ep \Grad a_\ep
\int_0^1
(\p_x \widetilde{p})(t\ep a_\ep, t\ep b_\ep)-(\p_x \widetilde{p})(0, 0)\dt
}_{L^1_T(B_{2,1}^{\f d 2 -1-\theta})}
\\
& \quad 
\les
\n{
\n{a_\ep \Grad a_\ep}_{B_{2,1}^{\f d 2 -1}}
\n{
\int_0^1
(\p_x \widetilde{p})(t\ep a_\ep, t\ep b_\ep)-(\p_x \widetilde{p})(0, 0)\dt
}_{B^{\f d 2-\theta}_{2,1}}
}_{L^1(0, T)}
\\
&\quad 
\les
\n{a_\ep}^2_{L_T^2(B^{\f d 2}_{2,1})}
\sp{
\ep
\n{a_\ep}_{L_T^\infty(B^{\f d 2-\theta}_{2,1})}
+
\ep
\n{b_\ep}_{L_T^\infty(B^{\f d 2-\theta}_{2,1})}
}
\\
&\quad 
\les
\ep^{\theta}
[T]
\sp{
E_{T,1}^{\ep, \zeta}
}^3.
\end{align*}
The same argument gives
\begin{align*}
&
\n{
b_\ep \Grad a_\ep
\int_0^1
(\p_y \widetilde{p})(t\ep a_\ep, t\ep b_\ep)-(\p_x \widetilde{p})(0, 0)\dt
}_{L^1_T(B_{2,1}^{\f d 2 -1-\theta})}    
\les
\ep^{\theta}
[T]
\sp{
E_{T}^{\ep}
}^3 , 
\\
&
\n{
a_\ep \Grad b_\ep
\int_0^1
(\p_x \widecheck{p})(t\ep a_\ep, t\ep b_\ep)-(\p_x \widecheck{p})(0, 0)\dt
}_{L^1_T(B_{2,1}^{\f d 2 -1-\theta})}    
\les
\ep^{\theta}
[T]
\sp{
E_{T}^{\ep}
}^3,
\\
&
\n{
b_\ep \Grad b_\ep
\int_0^1
(\p_y \widecheck{p})(t\ep a_\ep, t\ep b_\ep)-(\p_y \widecheck{p})(0, 0)\dt
}_{L^1_T(B_{2,1}^{\f d 2 -1-\theta})}    
\les
\ep^{\theta}
[T]
\sp{
E_{T}^{\ep}
}^3.
\end{align*}
Applying the Taylor expansion of pressure \eqref{Taylor}, we obtain
\begin{align*}
&
\n{
\dfrac{\nabla p(Z_\ep)}{\ep^2(R_\ep + Q_\ep)}
-
\left(
c_1 \dfrac{\Grad a_\ep}{\ep}+c_2\dfrac{\Grad b_\ep}{\ep}
+
c_3 
a_\ep\Grad a_\ep
+
c_4 
b_\ep\Grad a_\ep
+
c_5 
a_\ep\Grad b_\ep
+
c_6 
b_\ep\Grad b_\ep
\right)
}_{L^1_T(B_{2,1}^{\f d 2 -1-\theta})} 
\\
& \quad 
\les
\ep^{\theta}
[T]
\sp{
E_{T}^{\ep}
}^3.
\end{align*}
Using Lemma \ref{cons} and Lemma \ref{Le8.6} again, 
\begin{align*}
\n{
I(\ep a_\ep+ \ep b_\ep) \mathcal{A}\vue
}_{L^1_T(B_{2,1}^{\f d 2 -1-\theta})} 
&
\les
\sp{
\ep
\n{a_\ep}_{L_T^\infty(B^{\f d 2-\theta}_{2,1})}
+
\ep
\n{b_\ep}_{L_T^\infty(B^{\f d 2-\theta}_{2,1})}
}
\n{
\vue
}_{L^1_T(B_{2,1}^{\f d 2 +1})} 
\\
&
\les
\ep^{\theta}
\sp{
E_{T}^{\ep}
}^2 .
\end{align*}
Collecting the above estimates, we arrive at 
\begin{align}\label{3.1.16}
\n{
r_\ep
}_{L_T^1(\underline{B}_{2, 2}^{\f d 2-1-\theta})}
\les
\n{
r_\ep
}_{L_T^1(B_{2, 1}^{\f d 2-1-\theta})}
\les
\ep^\theta
[T]
\sp{
E_{T}^{\ep}
}^2
\sp{
E_{T}^{\ep}+1
}.
\end{align}
We define the function $\mathcal{X}(M):=C_M M^{\f{\theta}{2}}$, which is monotonically increasing in $M$, and set 
$\tau(\ep)
:=
\max\{
\sp{\mathcal{X}^{-1}(\f{1}{\ep})}^{-\f{\theta}{2}},
\ep^\theta
\}$.
By inserting \eqref{3.1.8}, \eqref{3.1.9}, \eqref{3.1.10}, \eqref{3.1.11}, \eqref{3.1.17}, \eqref{3.1.15}, \eqref{3.1.18}, \eqref{3.1.19} and \eqref{3.1.16} into \eqref{3.1.5}, one completes the proof. \ \ $\Box$

\subsection{Decay properties of \texorpdfstring{$\vWe$}{}}

This subsection is devoted to the decay properties of $\vWe$ with respect to the Mach number $\ep$. Recall that
\begin{align}\label{4.2.2}
\vWe
=
(\vWe^1, \vWe^2, \vWe^3)
=
\sp{
\rho_\ep-\rho, -\df{c_1}{c_2}(\rho_\ep-\rho), \mathcal{P}\underline{\vue}-\vw+(\df{\widehat{(\vue)}_0}{\sqrt{|\Td|}}-\df{\widehat{(\vu_0)}_0}{\sqrt{|\Td|}})
}^T.
\end{align}
\begin{Proposition}\label{Pro4.2}
Let $T^*_0$ be as in Theorem \ref{Th6.1},
$0<T< \infty$ with $T\leq T^*_0$, $0<\theta<1$ and $(a_0, b_0, \vu_0)\in B^{\f d 2}_{2, 1} \times B^{\f d 2}_{2, 1} \times B^{\f d 2-1}_{2, 1}$. Assume that $(a_\ep, b_\ep, \vue)$ is a regular solution to \eqref{1.1.3} on $[0, T]$,  satisfying
\begin{align*}
\ep 
\sp{
\n{\vae}_{L^\infty_T(L^\infty)}
+
\n{\vbe}_{L^\infty_T(L^\infty)}
}
\leq
\f{\min\{R^0, Q^0\}}{2}.
\end{align*}
Then 
\begin{align*}
\vW_{T, \theta}^\ep
\leq
C
[T]
\ep^{\f{\theta}{3+\theta}}
\text{e}^{C
\sp{ E^{\ep}_T
+
[T]
\sp{
\vU_{T^*_0}
+
\vU^2_{T^*_0}
}
}
}
\sp{
E_{T}^{\ep}
+
E^2_0
}
\sp{
\sp{
E_{T}^{\ep}
}^2
+
E_{T}^{\ep}
+
1}
\sp{
\vU_{T^*_0}
+
E_{T}^{\ep}
}.
\end{align*}
\end{Proposition}
\bProof
Subtracting \eqref{1.3} from \eqref{1.2}, we obtain
\begin{align}
\p_t\vWe
&
+
\mathbb{P}
\mathcal{Q}(\underline{\vWe}, \underline{\vWe})
+
2
\mathbb{P}
\mathcal{Q}(\underline{\vWe}, \underline{\vU})
+
2
\mathbb{P}
\mathcal{Q}(\mathbb{P}_0\vue, \underline{\vWe})
+
2
\mathbb{P}
\mathcal{Q}(\mathbb{P}_0\vWe, \underline{\vU})
-
\mathcal{D}\vWe
\nonumber
\\
&
=
\mathbb{P}r_\ep
-
2\mathbb{P}\mathcal{Q}\left(\underline{\vUe}, \mathcal{L}\left(\df{t}{\ep}\right) \vVe\right)
-
2\mathbb{P}\mathcal{Q}\left(\mathbb{P}_0\vU, \mathcal{L}\left(\df{t}{\ep}\right) \vVe\right),\label{4.2.1}
\end{align}
where we used 
$
\mathbb{P}\mathcal{Q}\left(\mathbb{P}_0\vWe, \mathcal{L}\left(\df{t}{\ep}\right) \vVe\right)=0.
$
Due to \eqref{4.2.2}, it suffices to consider the evolution of $\vWe^1$ and $\vWe^3$ only. Accordingly, we also express the terms
$\mathbb{P}\mathcal{Q}\left(\underline{\vUe}, \mathcal{L}\left(\df{t}{\ep}\right) \vVe\right)$
and
$\mathbb{P}\mathcal{Q}\left(\mathbb{P}_0\vU, \mathcal{L}\left(\df{t}{\ep}\right) \vVe\right)$
in their component forms: 
\begin{align*}
\mathbb{P}\mathcal{Q}\left(\underline{\vUe}, \mathcal{L}\left(\df{t}{\ep}\right) \vVe\right)=(R^{1}_{1, \ep}, R^{2}_{1, \ep}, R^{3}_{1, \ep}), \quad 
\mathbb{P}\mathcal{Q}\left(\mathbb{P}_0\vU, \mathcal{L}\left(\df{t}{\ep}\right) \vVe\right)=(R^{1}_{2, \ep}, R^{2}_{2, \ep}, R^{3}_{2, \ep}) . 
\end{align*}
It follows from \eqref{4.2.1} that 
\begin{equation}\label{4.2.3}
\left\{\begin{aligned}
&
\p_t
\vWe^1
+
\Div(\underline{\vWe^1}  \underline{\mathcal{P}\vue})
+
\Div(\underline{\vWe^1} \df{\widehat{(\vue)}_0}{\sqrt{|\Td|}})
=
-
\Div(\underline{\vWe^3}  \rho)
-
\Div(\rho\df{\widehat{(\vWe^3)}_0}{\sqrt{|\Td|}})
-
2R^{1}_{1,\ep}
-
2R^{1}_{2,\ep},
\\
&
\p_t
\vWe^3
+
\mathcal{P}
\sp{
\underline{\mathcal{P}\vue}\cdot \Grad\vWe^3
+
\df{\widehat{(\vue)}_0}{\sqrt{|\Td|}}\cdot \Grad\vWe^3
}
-
\df{\mu}{R^0+Q^0}
\De
\vWe^3
\\
&
\quad
=
\mathcal{P}
\left(
r^3_\ep
-
\underline{\vWe^3}\cdot \Grad\vw
-
\df{\widehat{(\vWe^3)}_0}{\sqrt{\Td}} 
\cdot
\Grad
\vw
-
2R^{3}_{1, \ep}
-
2R^{3}_{2, \ep}
\right).
\end{aligned}
\right.
\end{equation}
Similar to the analysis of $\vZe$, we apply the negative time regularity method. We set
\begin{align*}
&
\widetilde{\vV}_\ep
=
\sum\limits_{k,\alpha}
\df
{1}
{-ic_0\lambda_k^\alpha}
\widehat{\vV}^{\alpha}_{\ep, k} 
\Phi^\alpha_k,
\\
&
\mathbb{P}\mathcal{Q}\left(\underline{\vU_{\ep, M}}, \mathcal{L}\left(\df{t}{\ep}\right) \vV_{\ep, M}\right)
=
(R^{1}_{1, \ep, M}, R^{2}_{1, \ep, M}, R^{3}_{1, \ep, M}), 
\\
&
(R^{1, M}_{1, \ep}, R^{2, M}_{1, \ep}, R^{3, M}_{1, \ep})
=
\mathbb{P}\mathcal{Q}\left(\underline{\vUe}, \mathcal{L}\left(\df{t}{\ep}\right) \vVe\right)
-
\mathbb{P}\mathcal{Q}\left(\underline{\vU_{\ep, M}}, \mathcal{L}\left(\df{t}{\ep}\right) \vV_{\ep, M}\right),
\\
&
\mathbb{P}\mathcal{Q}\left(\underline{\vU_{\ep, M}}, \mathcal{L}\left(\df{t}{\ep}\right) \widetilde{\vV}_{\ep, M}\right)=(\widetilde{R}^{1}_{1, \ep, M}, \widetilde{R}^{2}_{1, \ep, M}, \widetilde{R}^{3}_{1, \ep, M}),
\\
&
\mathbb{P}\mathcal{Q}\left(\underline{\p_t\vU_{\ep, M}}, \mathcal{L}\left(\df{t}{\ep}\right) \widetilde{\vV}_{\ep, M}\right)
+
\mathbb{P}\mathcal{Q}\left(\underline{\vU_{\ep, M}}, \mathcal{L}\left(\df{t}{\ep}\right) \p_t\widetilde{\vV}_{\ep, M}\right)
=
(\widetilde{R}^{1, t}_{1, \ep, M}, \widetilde{R}^{2, t}_{1, \ep, M}, \widetilde{R}^{3, t}_{1, \ep, M}),
\\
&
\mathbb{P}\mathcal{Q}\left(\mathbb{P}_0\vU, \mathcal{L}\left(\df{t}{\ep}\right) \vV_{\ep, M}\right)
=
(R^{1}_{2, \ep, M}, R^{2}_{2, \ep, M}, R^{3}_{2, \ep, M}),
\\
&
(R^{1, M}_{2, \ep}, R^{2, M}_{2, \ep}, R^{3, M}_{2, \ep})
=
\mathbb{P}\mathcal{Q}\left(\mathbb{P}_0\vU, \mathcal{L}\left(\df{t}{\ep}\right) \vVe\right)
-
\mathbb{P}\mathcal{Q}\left(\mathbb{P}_0\vU, \mathcal{L}\left(\df{t}{\ep}\right) \vV_{\ep, M}\right),
\\
&
\mathbb{P}\mathcal{Q}\left(\mathbb{P}_0\vU, \mathcal{L}\left(\df{t}{\ep}\right) \widetilde{\vV}_{\ep, M}\right)
=
(\widetilde{R}^{1}_{2, \ep, M}, \widetilde{R}^{2}_{2, \ep, M}, \widetilde{R}^{3}_{2, \ep, M}),
\\
&
\mathbb{P}\mathcal{Q}\left(\mathbb{P}_0\vU, \mathcal{L}\left(\df{t}{\ep}\right) \p_t\widetilde{\vV}_{\ep, M}\right)
=
(\widetilde{R}^{1, t}_{2, \ep, M}, \widetilde{R}^{2, t}_{2, \ep, M}, \widetilde{R}^{3, t}_{2, \ep, M}).
\end{align*}
It can be readily verified that $R^{j}_{j',\ep, M}$, $\widetilde{R}^{j}_{j', \ep, M}$ and $\widetilde{R}^{j, t}_{j',\ep, M} $ satisfy
\begin{align}\label{4.2.4}
\p_t \ep\widetilde{R}^{j}_{j', \ep, M}
-
\ep\widetilde{R}^{j, t}_{j',\ep, M} 
=
R^{j}_{j',\ep, M}, 
\quad 
j=1, 2, 3,
\quad 
j'=1, 2.
\end{align}
Next, we set 
$
\widetilde{\vW}^1_{\ep, M}
=
\vWe^1+2\sum\limits_{j'=1}^2\ep\widetilde{R}^{1}_{j', \ep, M}
$
and 
$
\widetilde{\vW}^3_{\ep, M}
=
\vWe^3+2\sum\limits_{j'=1}^2\ep\widetilde{R}^{3}_{j', \ep, M}
$.
Adding twice of \eqref{4.2.3} to \eqref{4.2.1}, we obtain the following equations governing $\widetilde{\vW}^1_{\ep, M}$ and $\widetilde{\vW}^3_{\ep, M}$: 
\begin{equation}\label{4.2.5}
\left\{\begin{aligned}
&
\p_t
\widetilde{\vW}^1_{\ep, M}
+
\Div(\underline{\widetilde{\vW}^1_{\ep, M}}  \underline{\mathcal{P}\vue})
+
\Div(\underline{\widetilde{\vW}^1_{\ep, M}}\df{\widehat{(\vue)}_0}{\sqrt{|\Td|}})
\\
&
\quad
=
-
\Div(\underline{\widetilde{\vW}^3_{\ep, M}}  \rho)
-
\Div(\rho\df{\widehat{(\widetilde{\vW}^3_{\ep, M})}_0}{\sqrt{|\Td|}})
-
2R^{1, M}_{1,\ep}
-
2R^{1, M}_{2,\ep}
+
2
\ep
\widetilde{R}^{1, t}_{j, \ep, M}
+
F_{1, \ep},
\\
&
\p_t
\widetilde{\vW}^3_{\ep, M}
+
\mathcal{P}
\sp{
\underline{\mathcal{P}\vue}\cdot \Grad\widetilde{\vW}^3_{\ep, M}
+
\df{\widehat{(\vue)}_0}{\sqrt{|\Td|}}\cdot \Grad\widetilde{\vW}^3_{\ep, M}
}
-
\df{\mu}{R^0+Q^0}
\De
\widetilde{\vW}^3_{\ep, M}
\\
&
\quad
=
\mathcal{P}
\left(
r^3_\ep
-
\underline{\widetilde{\vW}^3_{\ep, M}}\cdot \Grad\vw
-
\df{\widehat{(\widetilde{\vW}^3_{\ep, M})}_0}{\sqrt{\Td}} 
\cdot
\Grad
\vw
-
2R^{3, M}_{1, \ep}
-
2R^{3, M}_{2, \ep}
+
2\ep
\widetilde{R}^{3, t}_{j, \ep, M}
+
F_{2, \ep}
\right),
\end{aligned}
\right.
\end{equation}
where
\begin{align*}
&
F_{1, \ep}
=
2
\ep
\sum\limits_{j=1}^2
\sp{
\Div(\underline{\widetilde{R}^{1}_{j, \ep, M}}  \underline{\mathcal{P}\vue})
+
\Div(\underline{\widetilde{R}^{1}_{j, \ep, M}} \df{\widehat{(\vue)}_0}{\sqrt{|\Td|}})
+
\Div(\underline{\widetilde{R}^{3}_{j, \ep, M}} \rho)
+
\Div(\widehat{(\widetilde{R}^{3}_{j, \ep, M})}_0 \rho)
},
\\
&
F_{2, \ep}
=
2
\ep
\sum\limits_{j=1}^2
\mathcal{P}
\sp{
\underline{\mathcal{P}\vue}
\cdot 
\Grad
\widetilde{R}^{3}_{j, \ep, M}
+
\df{\widehat{(\vue)}_0}{\sqrt{|\Td|}}
\cdot 
\Grad
\widetilde{R}^{3}_{j, \ep, M}
+
\underline{\widetilde{R}^{3}_{j, \ep, M}}\cdot \Grad\vw
+
\df{\widehat{(\widetilde{R}^{3}_{j, \ep, M})}_0}{\sqrt{\Td}} 
\cdot
\Grad
\vw
}
\\
&
\qquad
\quad
-2\ep\sum\limits_{j=1}^2\df{\mu}{R^0+Q^0}
\De\widetilde{R}^{3}_{j, \ep, M}.
\end{align*}
Applying Lemma \ref{le6.4} to \eqref{4.2.5} and using Lemma \ref{cons} yields
\begin{align}\label{4.2.6}
\vW_{T, \theta}^{\ep}
&
\leq
C
\exp{ \left[C
\sp{ E^{\ep}_T
+
[T]
\sp{
\vU_{T^*_0}
+
\vU^2_{T_0}
}
} \right]
}
\Bigg[ 
\n{\widetilde{\vW}^1_{\ep, M}(0)}_{B^{\f d 2 -1-\theta}_{2, 1}}
+
\n{\widetilde{\vW}^3_{\ep, M}(0)}_{B^{\f d 2 -1-\theta}_{2, 1}}
+
\n{r_\ep^3}_{L_T^1(B^{\f d 2 -1-\theta}_{2, 1})}
\nonumber
\\
&
\quad
+
\sum\limits_{j=1}^2
\sp{
\n{R^{1, M}_{j, \ep}}_{
L_T^1(B^{\f d 2 -1-\theta}_{2, 1})}
+
\n{R^{3, M}_{j, \ep}}_{
L_T^1(B^{\f d 2 -1-\theta}_{2, 1})}
+
\n{F_{j, \ep}}_{
L_T^1(B^{\f d 2 -1-\theta}_{2, 1})}
}
\nonumber
\\
&
\quad
+
\ep
\sum\limits_{j=1}^2
\left(
\n{\widetilde{R}^{1}_{j, \ep, M}}_{\widetilde{L}_T^\infty(B^{\f d 2 -1-\theta}_{2, 1})}
+
\n{\widetilde{R}^{3}_{j, \ep, M}}_{\widetilde{L}_T^\infty(B^{\f d 2 -1-\theta}_{2, 1})
\cap
L_T^1(\underline{B}^{\f d 2 +1-\theta}_{2, 1})}
+
\right.
\nonumber
\\
&
\left.
\qquad \qquad
+
\n{\widetilde{R}^{1, t}_{j, \ep, M}}_{
L_T^1(B^{\f d 2 -1-\theta}_{2, 1})}
+
\n{\widetilde{R}^{3, t}_{j, \ep, M}}_{
L_T^1(B^{\f d 2 -1-\theta}_{2, 1})}
\right)
\Bigg] . 
\end{align}
Next, we choose $M$ sufficiently large, which will be determined later, to estimate the terms on the right-hand side of \eqref{4.2.6}. To this end, we introduce some basic estimates for the truncations $\vA_M$
and $\vA^M$ of $\vA$. For $1\leq q, r \leq \infty$, $1\leq M$, $-\infty<s<\infty$ and $s_1\geq0$, we have
\begin{align}\label{4.2.7}
\n{\vA_M}_{\widetilde{L}_T^q(B^{s+s_1}_{2, r})}
\les
M^{s_1}
\n{\vA}_{\widetilde{L}_T^q(B^{s}_{2, r})},
\quad 
\n{\vA^M}_{\widetilde{L}_T^q(B^{s-s_1}_{2, r})}
\les
M^{-s_1}
\n{\vA}_{\widetilde{L}_T^q(\underline{B}^{s}_{2, r})}.
\end{align}

\noindent $\bullet$ \textbf{Estimates for}  $\widetilde{R}^{1}_{j, \ep, M}$, $\widetilde{R}^{3}_{j, \ep, M}$,  $\widetilde{\vW}^1_{\ep, M}(0)$ and $\widetilde{\vW}^3_{\ep, M}(0)$.

From \eqref{minski}, \eqref{minski2}, \eqref{3.1.3}, \eqref{4.2.7}, Lemma \ref{cons} and Lemma \ref{le8.2}, we obtain
\begin{align*}
&
\n{
\mathbb{P}
\mathcal{Q}
\left(\underline{\vU_{\ep, M}}, \mathcal{L}\left(\df{t}{\ep}\right) \widetilde{\vV}_{\ep, M}\right)
}_{\widetilde{L}_T^\infty(B^{\f d 2 -1-\theta}_{2, 1})}
\\
& \quad 
\les
\n{
\mathcal{Q}
\left(\underline{\vU_{\ep, M}}, \mathcal{L}\left(\df{t}{\ep}\right) \widetilde{\vV}_{\ep, M}\right)
}_{\widetilde{L}_T^\infty(B^{\f d 2 -1-\theta}_{2, 1})}
\les
\n{
\vU_{\ep, M}
}_{\widetilde{L}_T^\infty(\underline{B}^{\f d 2-\theta}_{2, 1})}
\n{
\widetilde{\vV}_{\ep, M}
}_{\widetilde{L}_T^\infty(B^{\f d 2}_{2, 1})}
\\
& \quad 
\les
M^{1-\theta}
\n{
\vUe
}_{\widetilde{L}_T^\infty(\underline{B}^{\f d 2-1}_{2, 1})}
\n{
\vVe
}_{\widetilde{L}_T^\infty(B^{\f d 2-1}_{2, 1})}
\les
M^{1-\theta}
\sp{E^{\ep}_T}^2,
\\[4pt]
&
\n{
\mathbb{P}
\mathcal{Q}
\left(\underline{\vU_{\ep, M}}, \mathcal{L}\left(\df{t}{\ep}\right) \widetilde{\vV}_{\ep, M}\right)
}_{L_T^1(B^{\f d 2 +1-\theta}_{2, 1})}
\\
& \quad 
\les
M^{2}
\n{
\mathcal{Q}
\left(\underline{\vU_{\ep, M}}, \mathcal{L}\left(\df{t}{\ep}\right) \widetilde{\vV}_{\ep, M}\right)
}_{L_T^1(B^{\f d 2 -1-\theta}_{2, 1})}
\les
M^{2}
\n{
\vU_{\ep, M}
}_{L_T^2(\underline{B}^{\f d 2-\theta}_{2, 1})}
\n{
\widetilde{\vV}_{\ep, M}
}_{L_T^2(B^{\f d 2}_{2, 1})}
\\
& \quad  
\les
M^{2}
\n{
\vUe
}_{L_T^2(\underline{B}^{\f d 2}_{2, 1})}
\n{
\vVe
}_{L_T^2(B^{\f d 2}_{2, 1})}
\les
M^{2}
[T]^{\f1 2}
\sp{E^{\ep}_T}^2,
\\[4pt]
&
\n{
\mathbb{P}
\mathcal{Q}
\left(\mathbb{P}_0 \vU, \mathcal{L}\left(\df{t}{\ep}\right) \widetilde{\vV}_{\ep, M}\right)
}_{\widetilde{L}_T^\infty(B^{\f d 2 -1-\theta}_{2, 1})}
\\
& \quad 
\les
|\mathbb{P}_0 \vU|
\n{
\widetilde{\vV}_{\ep, M}
}_{\widetilde{L}_T^\infty(B^{\f d 2 -\theta}_{2, 1})}
\les
|\mathbb{P}_0 \vU|
\n{
\vVe
}_{\widetilde{L}_T^\infty(B^{\f d 2 -1-\theta}_{2, 1})}
\les
\sp{E^{\ep}_T}^2,
\\[4pt]
&
\n{
\mathbb{P}
\mathcal{Q}
\left(\mathbb{P}_0 \vU, \mathcal{L}\left(\df{t}{\ep}\right) \widetilde{\vV}_{\ep, M}\right)
}_{L_T^1(B^{\f d 2 +1-\theta}_{2, 1})}
\\
& \quad  
\les
|\mathbb{P}_0 \vU|
\n{
\widetilde{\vV}_{\ep, M}
}_{L_T^1(B^{\f d 2+2 -\theta}_{2, 1})}
\les
M^{1-\theta}
|\mathbb{P}_0 \vU|
\n{
\vVe
}_{L_T^1(B^{\f d 2}_{2, 1})}
\les
M^{1-\theta}
[T]^{\f1 2}
\sp{E^{\ep}_T}^2.
\end{align*}
Combining the above estimates,
\begin{align}\label{4.2.8}
&
\ep
\sum\limits_{j=1}^2
\sp{
\n{\widetilde{R}^{1}_{j, \ep, M}}_{
\widetilde{L}_T^\infty(B^{\f d 2 -1-\theta}_{2, 1})
}
+
\n{\widetilde{R}^{3}_{j, \ep, M}}_{
\widetilde{L}_T^\infty(B^{\f d 2 -1-\theta}_{2, 1})
\cap
L_T^1(\underline{B}^{\f d 2 +1-\theta}_{2, 1})}
}
\nonumber
\\
& \quad 
\les
\ep
\n{
\mathbb{P}
\mathcal{Q}
\left(\underline{\vU_{\ep, M}}, \mathcal{L}\left(\df{t}{\ep}\right) \widetilde{\vV}_{\ep, M}\right)
}_{\widetilde{L}_T^\infty(B^{\f d 2 -1-\theta}_{2, 1})
\cap
L_T^1(B^{\f d 2 +1-\theta}_{2, 1})}
\nonumber
\\
&
\qquad
+
\ep
\n{
\mathbb{P}
\mathcal{Q}
\left(\mathbb{P}_0 \vU, \mathcal{L}\left(\df{t}{\ep}\right) \widetilde{\vV}_{\ep, M}\right)
}_{\widetilde{L}_T^\infty(B^{\f d 2 -1-\theta}_{2, 1})
\cap
L_T^1(B^{\f d 2 +1-\theta}_{2, 1})}
\nonumber
\\
& \quad 
\les
\ep
M^{2}
[T]^{\f1 2}
\sp{E^{\ep}_T}^2.
\end{align}
The same argument yields
\begin{align}\label{4.2.9}
\n{\widetilde{\vW}^1_{\ep, M}(0)}_{B^{\f d 2 -1-\theta}_{2, 1}}
+
\n{\widetilde{\vW}^3_{\ep, M}(0)}_{B^{\f d 2 -1-\theta}_{2, 1}}
&
\les
\ep
\sum\limits_{j=1}^2
\sp{
\n{\widetilde{R}^{1}_{j, \ep, M}(0)}_{
B^{\f d 2 -1-\theta}_{2, 1}
}
+
\n{\widetilde{R}^{3}_{j, \ep, M}(0)}_{(B^{\f d 2 -1-\theta}_{2, 1})
}
}
\nonumber
\\
&
\les
\ep
M^{1-\theta}
E_0^2.
\end{align}

\noindent $\bullet$ \textbf{Estimates for}  $R^{1, M}_{j, \ep}$, $R^{3, M}_{j, \ep}$.

From \eqref{minski}, \eqref{3.1.3}, \eqref{4.2.7}, Lemma \ref{cons} and Lemma \ref{le8.2}, we obtain
\begin{align}\label{4.2.10}
&
\sum\limits_{j=1}^2
\sp{
\n{R^{1, M}_{j, \ep}}_{
L_T^1(\underline{B}^{\f d 2 -1-\theta}_{2, 1})
}
+
\n{R^{3, M}_{j, \ep}}_{
L_T^1(\underline{B}^{\f d 2 -1-\theta}_{2, 1})}
}
\nonumber
\\
& \quad 
\les
\n{
\mathbb{P}
\mathcal{Q}
\left(\underline{\vU^M_{\ep}}, \mathcal{L}\left(\df{t}{\ep}\right) \vV^M_{\ep}\right)
}_{L_T^1(B^{\f d 2 -1-\theta}_{2, 1})}
+
\n{
\mathbb{P}
\mathcal{Q}
\left(\underline{\vU_{\ep, M}}, \mathcal{L}\left(\df{t}{\ep}\right) \vV^M_{\ep}\right)
}_{L_T^1(B^{\f d 2 -1-\theta}_{2, 1})}
\nonumber
\\
&
\qquad
+
\n{
\mathbb{P}
\mathcal{Q}
\left(\underline{\vU^M_{\ep}}, \mathcal{L}\left(\df{t}{\ep}\right) \vV_{\ep, M}\right)
}_{L_T^1(B^{\f d 2 -1-\theta}_{2, 1})}
+
\n{
\mathbb{P}
\mathcal{Q}
\left(\mathbb{P}_0 \vU, \mathcal{L}\left(\df{t}{\ep}\right) \vV^M_{\ep}\right)
}_{L_T^1(B^{\f d 2 -1-\theta}_{2, 1})}
\nonumber
\\
& \quad 
\les
\n{\vU^M_{\ep}}_{
L_T^2(\underline{B}^{\f d 2 }_{2, 1})
}
\n{\vV^M_{\ep}}_{
L_T^2(B^{\f d 2 -\theta}_{2, 1})
}
+
\n{\vU_{M, \ep}}_{
L_T^2(\underline{B}^{\f d 2 }_{2, 1})
}
\n{\vV^M_{\ep}}_{
L_T^2(B^{\f d 2 -\theta}_{2, 1})
}
\nonumber
\\
&
\qquad
+
\n{\vU^M_{\ep}}_{
L_T^2(\underline{B}^{\f d 2 -\theta}_{2, 1})
}
\n{\vV_{\ep, M}}_{
L_T^2(B^{\f d 2}_{2, 1})
}
+
[T]^{\f1 2}
|\mathbb{P}_0 \vU|
\n{\vV^M_{\ep}}_{
L_T^2(\underline{B}^{\f d 2 -\theta}_{2, 1})}
\nonumber
\\
&\quad 
\les
M^{-\theta}
[T]^{\f1 2}
\sp{E^{\ep}_T}^2.
\end{align}

\noindent $\bullet$ \textbf{Estimates for}  $F_{j, \ep}$.

From \eqref{minski}, \eqref{minski2}, \eqref{4.2.8} and Lemma \ref{le8.2}, we get
\begin{align}
&
 \n{
F_{1, \ep}
}_{L_T^1(B^{\f d 2 -1-\theta}_{2, 1})}   
\nonumber
\\
& \quad 
\les
\ep
\sum\limits_{j=1}^2
\left(
M
 \n{
\widetilde{R}^{1}_{j, \ep, M}
}_{L_T^\infty(\underline{B}^{\f d 2-1-\theta}_{2, 1})}  
\n{
\mathcal{P}\vue
}_{L_T^1(\underline{B}^{\f d 2}_{2, 1})}  
+
M
[T]
\n{
\widetilde{R}^{1}_{j, \ep, M}
}_{L_T^\infty(\underline{B}^{\f d 2-1-\theta}_{2, 1})} 
\n{\widehat{(\vue)}_0}_{L^\infty(0, T)}
\right.
\nonumber
\\
&
\left.
\qquad
+
\n{
\widetilde{R}^{3}_{j, \ep, M}
}_{L_T^1(\underline{B}^{\f d 2-\theta}_{2, 1})}  
\n{
\rho
}_{L_T^\infty(B^{\f d 2}_{2, 1})} 
+
[T]
\n{
\rho
}_{L_T^\infty(B^{\f d 2-\theta}_{2, 1})} 
\n{\widehat{(\widetilde{R}^{3}_{j, \ep, M})}_0}_{L^\infty(0, T)}
\right)
\nonumber
\\
& \quad 
\les
\ep
M^3
[T]^{\f 3 2}
\sp{E^{\ep}_T}^2
\sp{
E^{\ep}_T
+
\vU_{T^*_0}
},  
\label{4.2.11}    
\\[4pt]
&
\n{
F_{2, \ep}
}_{L_T^1(B^{\f d 2 -1-\theta}_{2, 1})}  
\nonumber
\\
&\quad 
\les  
\ep
\sum\limits_{j=1}^2
\left(
\n{
\widetilde{R}^{3}_{j, \ep, M}
}_{L_T^2(\underline{B}^{\f d 2-\theta}_{2, 1})}  
\n{
\mathcal{P}\vue
}_{L_T^2(\underline{B}^{\f d 2}_{2, 1})} 
+
\n{
\widetilde{R}^{3}_{j, \ep, M}
}_{L_T^1(\underline{B}^{\f d 2-\theta}_{2, 1})} 
\n{\widehat{(\vue)}_0}_{L^\infty(0, T)}
\right.
\nonumber
\\
&
\left.
\qquad
+
\n{
\widetilde{R}^{3}_{j, \ep, M}
}_{L_T^2(\underline{B}^{\f d 2-\theta}_{2, 1})}  
\n{
\vw
}_{L_T^2(B^{\f d 2}_{2, 1})} 
+
\n{
\vw
}_{L_T^1(B^{\f d 2-\theta}_{2, 1})} 
\n{\widehat{(\widetilde{R}^{(3)}_{j, \ep, M})}_0}_{L^\infty(0, T)}
+
\n{
\widetilde{R}^{3}_{j, \ep, M}
}_{L_T^1(\underline{B}^{\f d 2+1-\theta}_{2, 1})} 
\nonumber
\right)
\\
& \quad 
\les
\ep
M^2
[T]^{\f 1 2}
\sp{E^{\ep}_T}^2
\sp{
E^{\ep}_T
+
\vU_{T^*_0}
+
1
}. 
\label{4.2.12}
\end{align}

\noindent $\bullet$ \textbf{Estimates for}  $\widetilde{R}^{1, t}_{j, \ep, M}$ and  $\widetilde{R}^{3, t}_{j, \ep, M}$.

From \eqref{tV}, \eqref{tU}, \eqref{4.2.7} and Lemma \ref{le8.2}, we get
\begin{align}
&
\ep
\sum\limits_{j=1}^2
\sp{
\n{\widetilde{R}^{1, t}_{j, \ep, M}}_{
L_T^1(B^{\f d 2 -1-\theta}_{2, 1})}
+
\n{\widetilde{R}^{3, t}_{j, \ep, M}}_{
L_T^1(B^{\f d 2 -1-\theta}_{2, 1})}
}
\nonumber
\\
&\quad 
\les
\ep
\n{
\mathbb{P}\mathcal{Q}\left(\underline{\p_t\vU_{\ep, M}}, \mathcal{L}\left(\df{t}{\ep}\right) \widetilde{\vV}_{\ep, M}\right)
}_{L_T^1(B^{\f d 2 -1-\theta}_{2, 1})}
+
\ep
\n{
\mathbb{P}\mathcal{Q}\left(\underline{\vU_{\ep, M}}, \mathcal{L}\left(\df{t}{\ep}\right) \p_t\widetilde{\vV}_{\ep, M}\right)
}_{L_T^1(B^{\f d 2 -1-\theta}_{2, 1})}
\nonumber
\\
&
\qquad
+
\ep
\n{
\mathbb{P}\mathcal{Q}\left(\mathbb{P}_0 \vU, \mathcal{L}\left(\df{t}{\ep}\right) \p_t\widetilde{\vV}_{\ep, M}\right)
}_{L_T^1(B^{\f d 2 -1-\theta}_{2, 1})}
\nonumber
\\
&\quad 
\les
\ep
M^{1-\theta}
\sp{
\n{
\p_t\vUe
}_{L_T^1(\underline{B}^{\f d 2 -1}_{2, 1})}
\n{
\vVe
}_{L_T^\infty(B^{\f d 2 -1}_{2, 1})}
+
\sp{
\n{
\vUe
}_{L_T^\infty(\underline{B}^{\f d 2 -1}_{2, 1})}
+
|\mathbb{P}_0 \vU|
}
\n{
\p_t \vVe
}_{L_T^1(B^{\f d 2 -1}_{2, 1})}
}
\nonumber
\\
&\quad 
\les
\ep
M^{1-\theta}
[T]
E_{T}^{\ep}
\sp{
\sp{
E_{T}^{\ep}
}^2
+
E_{T}^{\ep}
+
1}
\sp{
\vU_{T^*_0}
+
E_{T}^{\ep}
}.
\label{4.2.14}
\end{align}

\noindent $\bullet$ \textbf{Estimates for}  $r^3_\ep$.

We see from \eqref{3.1.16} that
\begin{align}\label{4.2.13}
\n{
r_\ep^3
}_{L^1_T(B^{\f d 2-1-\theta}_{2, 1})}
\leq
\n{
r_\ep
}_{L^1_T(B^{\f d 2-1-\theta}_{2, 1})}
\les
\ep^\theta
[T]
\sp{
E_{T}^{\ep}
}^2
\sp{
E_{T}^{\ep}+1
}.
\end{align}
Choosing $M=\ep^{-\f{1}{3+\theta}}$ and inserting \eqref{4.2.8}-\eqref{4.2.13} into \eqref{4.2.6}, we complete the proof.
\ \ $\Box$

\section{Linear and nonlinear estimates of equations \texorpdfstring{\eqref{1.1}}{} and \texorpdfstring{\eqref{1.2}}{}}

\subsection{Auxiliary lemmas}
In this section, we establish linear and nonlinear estimates for \eqref{1.1} and \eqref{1.2}. To this end, we prepare some auxiliary lemmas. 
\begin{Lemma}\text{(High-frequency estimates).}\label{lee6.1}
Let $0<\zeta$, $-\infty<s, s_1, s_2, s_3, s_4, s_5, s_6<\infty$ and $k\in\{1, 2, \cdots, d\}$. Then 
\begin{align}
&
\n{T_{f}g}^{h;\zeta}_{B_{2,1}^s}
\les
\n{f}_{B_{2,1}^{s_1}}
\n{g}^{h;\f\zeta 4}_{B_{2,1}^{s_2+\f d 2}},
\ \
\text{if }
s=s_1+s_2, \ \
s_1\leq \f d 2,
\nonumber
\\
&
\n{R(f, g)}^{h;\zeta}_{B_{2,1}^s}
\les
\n{f}^{h;\f\zeta {16}}_{\dot{B}_{2,1}^{s_1}}
\n{g}^{h;\f\zeta {64}}_{\dot{B}_{2,1}^{s_2+\f d 2}},
\ \ 
\text{if }
-\f d 2<s=s_1+s_2,
\nonumber
\\
&
\n{fg}^{h;\zeta}_{B_{2,1}^s}
\les
\n{f}_{B_{2,1}^{s_1}}
\n{g}^{h;\f\zeta 4}_{B_{2,1}^{s_2+\f d 2}}
+
\n{g}_{B_{2,1}^{s_3}}
\n{f}^{h;\f\zeta 4}_{B_{2,1}^{s_4+\f d 2}}
+
\n{f}^{h;\f\zeta {16}}_{B_{2,1}^{s_5}}
\n{g}^{h;\f\zeta {64}}_{B_{2,1}^{s_6+\f d 2}},
\nonumber
\\
&
\qquad
\qquad
\qquad
\text{if }
-\f d 2<s=s_1+s_2=s_3+s_4=s_5+s_6, \ \ s_1\leq \f d2,\ \ s_3\leq \f d 2,
\nonumber
\\
&
\sum\limits_{ 2^j \geq \zeta}
2^{js}
\n{
S_{j-2} f \De_j g
}_{L^2}
\les
\n{f}_{B^{\f d 2}_{2, 1}}
\n{g}^{h; \zeta}_{B^{s}_{2, 1}},
\label{6.1.0}
\\
&
\sum\limits_{ 2^j \geq \zeta}
2^{js}
\n{
\De_j T_f g-S_{j-2} f \De_j g
}_{L^2}
\les
\n{f}_{\underline{B}^{\f d 2+1}_{2, 1}}
\n{g}^{h; \f{\zeta}{4}}_{B^{s-1}_{2, 1}},
\label{6.1.1}
\\
&
\sum\limits_{ 2^j \geq \zeta}
2^{js}
\n{
\De_j \mathcal{P} (T_{\vu} \Grad \vv)-S_{j-2} \vu \Grad\De_j \mathcal{P} \vv
}_{L^2}
\les
\n{\vu}_{\underline{B}^{\f d 2+1}_{2, 1}}
\n{\vv}^{h; \f{\zeta}{4}}_{B^{s-1}_{2, 1}},
\label{6.1.2}
\\
&
\sum\limits_{ 2^j \geq \zeta}
2^{js}
\n{
\De_j \mathcal{P}^\perp (T_{\vu} \Grad \vv)-S_{j-2} \vu \Grad\De_j \mathcal{P}^\perp \vv
}_{L^2}
\les
\n{\vu}_{\underline{B}^{\f d 2+1}_{2, 1}}
\n{\vv}^{h; \f{\zeta}{4}}_{B^{s-1}_{2, 1}},
\label{6.1.3}
\\
&
\sum\limits_{ 2^j \geq \zeta}
2^{js}
\n{
\p_k \De_j \sp{T_f g}-S_{j-2} f \p_k\De_j g
}_{L^2}
\les
\n{f}_{\underline{B}^{\f d 2+1}_{2, 1}}
\n{g}^{h; \f{\zeta}{4}}_{B^{s}_{2, 1}},
\label{6.1.-1}
\end{align}
provided that the right-hand sides are finite.
\end{Lemma}
\bProof
From \eqref{ao1} and \eqref{ao2} , we have
\begin{align*}
&
\n{T_{f}g}^{h;\zeta}_{B_{2,1}^s}
\les
\sum\limits_{ 2^j \geq \zeta} 
\sum\limits_{|j-j'|\leq 2}
2^{js}
\n{S_{j'-2}f}_{L^\infty} 
\n{
\De_{j'} g
}_{L^2}
\les
\n{f}_{B_{2,1}^{s_1}}
\n{g}^{h;\f\zeta 4}_{B_{2,1}^{s_2+\f d 2}},
\\
&
\n{R(f, g)}^{h;\zeta}_{B_{2,1}^s}
\les
\n{R(f, g)}^{h;\zeta}_{B_{1,1}^{s+\f d 2}}
\les
\sum\limits_{2^j \geq \zeta} 
\sum\limits_{j-j'\leq 4}
\sum\limits_{|j'-j''|\leq 2}
2^{j(s+\f d 2)}
\n{\De_{j'} f}_{L^2}
\n{\De_{j''} g}_{L^2}
\les
\n{f}^{h;\f\zeta {16}}_{\dot{B}_{2,1}^{s_1}}
\n{g}^{h;\f\zeta {64}}_{\dot{B}_{2,1}^{s_2+\f d 2}}.
\end{align*}
Hence we get
\begin{align*}
\n{fg}^{h;\zeta}_{B_{2,1}^s}  
&
\leq
\n{T_f g}^{h;\zeta}_{B_{2,1}^s}  
+
\n{T_g f}^{h;\zeta}_{B_{2,1}^s}  
+
\n{R(f, g)}^{h;\zeta}_{B_{2,1}^s}  
\\
&
\les
\n{f}_{B_{2,1}^{s_1}}
\n{g}^{h;\f\zeta 4}_{B_{2,1}^{s_2+\f d 2}}
+
\n{g}_{B_{2,1}^{s_3}}
\n{f}^{h;\f\zeta 4}_{B_{2,1}^{s_4+\f d 2}}
+
\n{f}^{h;\f\zeta {16}}_{\dot{B}_{2,1}^{s_5}}
\n{g}^{h;\f\zeta {64}}_{\dot{B}_{2,1}^{s_6+\f d 2}}.
\end{align*}
To show \eqref{6.1.0}, we have
\begin{align*}
\sum\limits_{ 2^j \geq \zeta}
2^{js}
\n{
S_{j-2} f \De_j g
}_{L^2}
\les
\n{f}_{L^\infty}
\n{g}^{h; \f{\zeta}{4}}_{B^{s-1}_{2, 1}}
\les
\n{f}_{B^{\f d 2}_{2, 1}}
\n{g}^{h; \f{\zeta}{4}}_{B^{s}_{2, 1}}.
\end{align*}

Next, we prove \eqref{6.1.1}-\eqref{6.1.3}. We only give the proof of \eqref{6.1.2}, since the others follow similarly. Observe that 
\begin{align*}
&
\sum\limits_{ 2^j \geq \zeta}
2^{js}
\n{
\De_j \mathcal{P} (T_{\vu} \Grad \vv)-S_{j-2} \vu \Grad\De_j \mathcal{P} \vv
}_{L^2} 
\\
& \quad 
=
\sum\limits_{ 2^j \geq \zeta}
2^{js}
\n{
\sum\limits_{|j-j'|\leq 2}
\De_j \mathcal{P} (S_{j'-2}\vu \Grad \De_{j'} \vv)
-
\sum\limits_{|j-j'|\leq 2}
S_{j-2} \vu \De_j\De_{j'}\mathcal{P} \vv
}_{L^2}
\\
& \quad 
\leq
\sum\limits_{ 2^j \geq \zeta}
\sum\limits_{|j-j'|\leq 2}
\left(
2^{js}
\n{
\De_j \mathcal{P} (S_{j'-2}\vu \Grad \De_{j'} \vv)
-
S_{j'-2}\vu \Grad \De_{j'} \De_j \mathcal{P}\vv
}_{L^2}
\right.
\\
&
\qquad
\qquad
\qquad
\left.
+
2^{js}
\n{(S_{j'-2}f-S_{j-2}\vu)\De_j\De_{j'} \mathcal{P} \vv}_{L^2}
\right)
\\
& \quad 
\les
\sum\limits_{ 2^j \geq \zeta}
\sum\limits_{|j-j'|\leq 2}
\sp{
2^{(j-1)s}
\n{\Grad S_{j'-2}\vu}_{L^\infty}
\n{\De_{j'} \vv}_{L^2}
+
2^{js}
\n{ S_{j'-2}\vu- S_{j-2}\vu}_{L^\infty}
\n{\De_{j'} \vv}_{L^2}
}
\\
& \quad 
\les
\n{\vu}_{\underline{B}^{\f d 2+1}_{2, 1}}
\n{\vv}^{h; \f{\zeta}{4}}_{B^{s-1}_{2, 1}},
\end{align*}
where we used the following commutator estimate
\begin{align*}
\n{
\De_j \mathcal{P} (S_{j'-2}\vu \Grad \De_{j'} \vv)
-
S_{j'-2}\vu \Grad \De_{j'} \De_j \mathcal{P}\vv
}_{L^2}
\les
2^{-j}
\n{\Grad S_{j'-2}\vu}_{L^\infty}
\n{\De_{j'} \vv}_{L^2},
\end{align*}
see Lemma 2.97 in \cite{BCD11}. For \eqref{6.1.-1}, it suffices to invoke the following commutator estimate
\begin{align*}
\n{
\p_k\De_j (S_{j'-2}f  \De_{j'} g)
-
S_{j'-2}f \p_k \De_j \De_{j'} g
}_{L^2}
\les
\n{\Grad S_{j'-2}f}_{L^\infty}
\n{\De_{j'} g}_{L^2},
\end{align*}
and perform similar calculations as \eqref{6.1.2}.   \ \ $\Box$

\begin{Lemma}\text{(Middle-frequency estimates).}\label{lee6.2}
Let $0<\zeta<\eta<\infty$, $-\infty<s, s_1, s_2, s_3, s_4, s_5, s_6<\infty$ and $k\in\{1, 2, \cdots, d\}$. Then 
\begin{align*}
&
\n{T_{f}g}^{m;\zeta, \eta}_{B_{2,1}^s}
\les
\n{f}_{B_{2,1}^{s_1}}
\n{g}^{m;\f\zeta 4, 4\eta}_{B_{2,1}^{s_2+\f d 2}},
\ \
\text{if }
s=s_1+s_2, \ \
s_1\leq \f d 2,
\\
&
\n{R(f, g)}^{m;\zeta, \eta}_{B_{2,1}^s}
\les
\n{f}^{h;\f\zeta {16}}_{\dot{B}_{2,1}^{s_1}}
\n{g}^{h;\f\zeta {64}}_{\dot{B}_{2,1}^{s_2+\f d 2}},
\ \ 
\text{if }
-\f d 2<s=s_1+s_2,
\\
&
\n{fg}^{m;\zeta, \eta}_{B_{2,1}^s}
\les
\n{f}_{B_{2,1}^{s_1}}
\n{g}^{m;\f\zeta 4, 4\eta}_{B_{2,1}^{s_2+\f d 2}}
+
\n{g}_{B_{2,1}^{s_3}}
\n{f}^{h;\f\zeta 4, 4\eta}_{B_{2,1}^{s_4+\f d 2}}
+
\n{f}^{h;\f\zeta {16}}_{B_{2,1}^{s_5}}
\n{g}^{h;\f\zeta {64}}_{B_{2,1}^{s_6+\f d 2}},
\\
&
\qquad
\qquad
\qquad
\text{if }
-\f d 2<s=s_1+s_2=s_3+s_4=s_5+s_6, \ \ s_1\leq \f d2,\ \ s_3\leq \f d 2,
\\
&
\sum\limits_{ \zeta \leq 2^j < \eta}
2^{js}
\n{
S_{j-2} f \De_j g
}_{L^2}
\les
\n{f}_{B^{\f d 2}_{2, 1}}
\n{g}^{m; \zeta, \eta}_{B^{s}_{2, 1}},
\\
&
\sum\limits_{ \zeta \leq 2^j < \eta}
2^{js}
\n{
\De_j T_f g-S_{j-2} f \De_j g
}_{L^2}
\les
\n{f}_{\underline{B}^{\f d 2+1}_{2, 1}}
\n{g}^{m; \f{\zeta}{4}, 4\eta}_{B^{s-1}_{2, 1}},
\\
&
\sum\limits_{ \zeta \leq 2^j < \eta}
2^{js}
\n{
\De_j \mathcal{P} (T_{\vu} \Grad \vv)-S_{j-2} \vu \Grad\De_j \mathcal{P} \vv
}_{L^2}
\les
\n{\vu}_{\underline{B}^{\f d 2+1}_{2, 1}}
\n{\vv}^{m; \f{\zeta}{4}, 4\eta}_{B^{s-1}_{2, 1}},
\\
&
\sum\limits_{ \zeta \leq 2^j < \eta}
2^{js}
\n{
\De_j \mathcal{P}^\perp (T_{\vu} \Grad \vv)-S_{j-2} \vu \Grad\De_j \mathcal{P}^\perp \vv
}_{L^2}
\les
\n{\vu}_{\underline{B}^{\f d 2+1}_{2, 1}}
\n{\vv}^{m; \f{\zeta}{4}, 4\eta}_{B^{s-1}_{2, 1}},
\\
&
\sum\limits_{ \zeta \leq 2^j < \eta}
2^{js}
\n{
\p_k \De_j \sp{T_f g}-S_{j-2} f \p_k\De_j g
}_{L^2}
\les
\n{f}_{\underline{B}^{\f d 2+1}_{2, 1}}
\n{g}^{m; \f{\zeta}{4}, 4\eta}_{B^{s}_{2, 1}},
\end{align*}
provided that the right-hand sides are finite.
\end{Lemma}
The proof is similar to Lemma \ref{lee6.1}, and the details are omitted.

\subsection{Linear estimates of equations \texorpdfstring{\eqref{1.1}}{} and \texorpdfstring{\eqref{1.2}}{}}
Since $\mathbb{P}(a_\ep, b_\ep, \vue)^T$ and $\mathbb{P}^\perp(a_\ep, b_\ep, \vue)^T$ admit the following decompositions
\begin{align*}
&
\mathbb{P}(a_\ep, b_\ep, \vue)^T
=
\sp{\rho_\ep, -\df{c_1}{c_2}\rho_\ep, \mathcal{P}\underline{\vue}}^T
+
\sp{ \df{\widehat{(a_0)}_0}{\sqrt{|\Td|}}, \df{\widehat{(b_0)}_0}{\sqrt{|\Td|}}, \df{\widehat{(\vue)}_0}{\sqrt{|\Td|}} }^T,
\\
&
\mathbb{P}^\perp(a_\ep, b_\ep, \vue)^T
=
\sp{ \vthe, \df{Q^0}{R^0}\vthe, \mathcal{P}^\perp \vue }^T,
\end{align*}
it suffices to consider the evolution of $\rho_\ep$, $\vthe$, $\mathcal{P}\vue$ and $\mathcal{P}^\perp \vue$. Therefore, we consider the linearized equations
\begin{equation}\label{linea1}
\left\{\begin{aligned}
&
\p_t \rho_\ep 
+ 
T_{\vv}\Grad \rho_\ep  
+
c_{01}
\Div \mathcal{P}^\perp \vue
+
T_{\Grad \rho^*}\vue
+
T_{\rho^*}\Div \mathcal{P}^\perp \vue
=
F_1,
\\
&
\p_t
\mathcal{P}\vue
+
\mathcal{P}
\sp{
T_{\vv}\Grad \vue
}
-
\df{\mu}{R^0+Q^0}
\De \mathcal{P}\vue
=
G_1,
\\
&
(\rho_\ep (0), \mathcal{P}\vue (0))
=
(\rho_0, \mathcal{P}\vu_0),
\end{aligned}
\right.
\end{equation}
and 
\begin{equation}\label{linea2}
\left\{\begin{aligned}
&
\p_t 
\vthe 
+
\df{R^0}{\ep}
\Div \mathcal{P}^\perp \vue
+
T_{\vv} \Grad \vthe 
+
c_{02}
\Div \mathcal{P}^\perp \vue
=
F_2,
\\
&
\p_t
\mathcal{P}^\perp \vue
+
\df{c^2_0}{R^0}
\df{\Grad \vthe}{\ep}
+
\mathcal{P}^\perp
\sp{
T_{\vv}
\Grad \vue
}
-
\df{\nu}{R^0+Q^0}
\De
\mathcal{P}^\perp \vue
=
G_2,
\\
&
(\vthe (0), \mathcal{P}^\perp \vue (0))
=
(\vartheta_0, \mathcal{P}^\perp \vu_0),
\end{aligned}
\right.
\end{equation}
where 
\[
c_{01}:=\dfrac{Q^0\widehat{(a_0)}_0-R^0\widehat{(b_0)}_0}{(Q^0+R^0\f{c_1}{c_2})\sqrt{\T}}, \qquad 
c_{02}:=\dfrac{c_1\widehat{(a_0)}_0+c_2\widehat{(b_0)}_0}{(c_1+c_2\f{Q^0}{R^0})\sqrt{\T}} . 
\]

The terms $T_{\Grad \rho^*} \vue$ and $T_{\rho^*}\Div \mathcal{P}^\perp \vue$  are inserted in \eqref{linea1}$_1$ to facilitate the subsequent treatment of the terms $T_{\Grad \rho_\ep} \vue$ and
$T_{\rho_\ep}\Div \mathcal{P}^\perp \vue$. If $\n{\rho^* }_{L^\infty_T(B^{\f d 2}_{2,1})}<\infty$, the terms $T_{\Grad \rho^*} \vue$ and $T_{\rho^*}\Div \mathcal{P}^\perp \vue$  can be decomposed into three terms, one of which can be absorbed by the left-hand side of \eqref{6.2.9} and \eqref{6.2.-2}, see Remark \ref{Re6.1}, while the remaining two are small, see \eqref{6.3.4}. We then set $\rho^*=\rho$, so that $T_{\Grad \rho_\ep} \vue$ and $T_{\rho_\ep}\Div \mathcal P^\perp \vue$ are transformed into 
$T_{\Grad (\rho_\ep-\rho)} \vue$ and 
$T_{(\rho_\ep - \rho)}\Div \mathcal P^\perp \vue$, respectively, which can be handled successfully.

We now state the linear estimates as follows.

\begin{Proposition}\label{Pro6.1}
Let $1\leq \zeta<\infty$, $0<T\leq \infty$, $-\infty<s<\infty$ and $(\rho_0, \mathbb{P}\vu_0)\in B^{s}_{2, 1} \times B^{s-1}_{2, 1}$. Assume that $\n{\rho^* }_{L^\infty_T(B^{\f d 2}_{2,1})}<\infty$ and 
$(\rho_\ep, \vue)$ is a regular solution to \eqref{linea1} on $[0, T]$. Then, for any $t\in[0, T]$, we have  
\begin{align}
&
\f{1}{8C\sp{1 + \n{\rho^*}_{L^\infty_T(B^{\f d 2}_{2,1})}}}
\n{\rho_\ep}^{h;\zeta}_{\widetilde{L}^\infty_{t}(B^{s}_{2, 1})}
+
\n{ \mathcal{P}\vue}^{h; \zeta}_{\widetilde{L}^\infty_{t}(B^{s-1}_{2, 1})}
+
\f{7}{8}
\n{\mathcal{P}\vue}^{h; \zeta}_{L^1_{t}(B^{s+1}_{2, 1})}
\nonumber
\\
& \quad 
\leq
C
\left(
\n{\rho_0}^{h;\zeta}_{B^{s}_{2, 1}}
+\n{\mathcal{P}\vu_0}^{h; \zeta}_{B^{s-1}_{2, 1}}
+
\n{G_1}^{h; \zeta}_{L^1_{t}(B^{s-1}_{2, 1})}
+
\n{F_1}^{h; \zeta}_{L^1_{t}(B^{s}_{2, 1})}
+
\int_0^t 
\n{\vv}_{L^1_{t}(\underline{B}^{\f d 2+1}_{2, 1})}
\n{\rho_\ep}^{h; \f{\zeta}{4}}_{B^{s}_{2, 1}}   \dta
\right.
\nonumber
\\
&
\qquad
\left.
+
\int_0^t 
\n{\vv}_{L^1_{t}(\underline{B}^{\f d 2+1}_{2, 1})}
\n{\mathcal{P}\vue}^{h; \f{\zeta}{4}}_{B^{s-1}_{2, 1}}
\dta
+
\n{\mathcal{P} \vue}^{m;\f{\zeta}{4}, \zeta}_{L^1_{t}(B^{s+1}_{2, 1})}
+
\n{\mathcal{P}^\perp \vue}^{m;\f{\zeta}{4}, \zeta}_{L^1_{t}(B^{s+1}_{2, 1})}
\right)
+
\f 1 4
\n{\mathcal{P}^\perp\vue}^{h; \zeta}_{L^1_{t}(B^{s+1}_{2, 1})} . 
\label{6.2.-3}
\end{align}
\end{Proposition}
\begin{Remark}\label{Re6.1}
The term 
$
\f 1 4
\n{\mathcal{P}^\perp\vue}^{h; \zeta}_{L^1_{t}(B^{s+1}_{2, 1})}
$ will be absorbed by the left-hand sides of \eqref{6.2.9} and \eqref{6.2.-2}.
\end{Remark}

\begin{Proposition}\label{Pro6.2}
Let $1\leq \zeta<\infty$, 
$0<\ep \leq  \min\{1, \f{R^0}{8|c_{02}|}\}$\footnote{
When $c_{02}=0$, we define $\f{R^0}{8|c_{02}|}=\infty$, which implies $\min\{1, \f{R^0}{8|c_{02}|}\}=1$. The same convention applies throughout the rest of the paper.
}, 
$0<T\leq \infty$,
$-\infty<s<\infty$ and $(\vartheta_0, \mathcal{P}^\perp \vu_0)\in B^{s}_{2, 1}  \times B^{s-1}_{2, 1}$. Assume that 
$(\vthe, \vue)$ is a regular solution to \eqref{linea2} on $[0, T]$. Then, for any $t\in[0, T]$, we have 
\begin{align}
&
\ep
\n{\vthe}^{h;\f{\beta_0}{\ep}}_{\widetilde{L}^\infty_{t}(B^{s}_{2, 1})}
+
\n{ \mathcal{P}^\perp\vue}^{h; \f{\beta_0}{\ep}}_{\widetilde{L}^\infty_{t}(B^{s-1}_{2, 1})}
+
\f{1}{\ep}
\n{\vthe}^{h; \f{\beta_0}{\ep}}_{L^1_{t}(B^{s}_{2, 1})}
+
\n{\mathcal{P}^\perp\vue}^{h; \f{\beta_0}{\ep}}_{L^1_{t}(B^{s+1}_{2, 1})}
\nonumber
\\
& \quad 
\leq
C
\left(
\ep
\n{\vartheta_0}^{h;\f{\beta_0}{\ep}}_{B^{s}_{2, 1}}
+
\n{\mathcal{P}^\perp\vu_0}^{h;\f{\beta_0}{\ep}}_{B^{s-1}_{2, 1}}
+
\ep
\n{F_2}^{h; \f{\beta_0}{\ep}}_{L^1_{t}(B^{s}_{2, 1})}
+
\n{G_2}^{h; \f{\beta_0}{\ep}}_{L^1_{t}(B^{s-1}_{2, 1})}
+
\ep
\int_0^t
\n{\vv}_{\underline{B}^{\f d 2+1}_{2, 1}}
\n{\vthe}_{B^{s}_{2, 1}}
\dta
\right.
\nonumber
\\
&
\qquad
\left.
+
\int_0^t
\n{\vv}_{B^{\f d 2}_{2, 1}}
\n{\vthe}^{h; \f{\beta_0}{4\ep}}_{B^{s}_{2, 1}}
\dta
+
\int_0^t
\n{\vv}_{B^{\f d 2}_{2, 1}}
\n{\mathcal{P}^\perp\vue}^{h; \f{\beta_0}{4\ep}}_{B^{s}_{2, 1}}
\dta
\right), 
\label{6.2.9}
\end{align}
where $\beta_0$ is a positive constant.
\end{Proposition}

\begin{Proposition}\label{Pro6.3}
Let $1\leq \zeta<\df{\beta_0}{\ep}$, $0<\ep \leq \min\{1, \f{R^0}{8|c_{02}|}\}$, $0<T\leq \infty$, $-\infty<s<\infty$ and $(\vartheta_0, \mathcal{P}^\perp \vu_0)\in B^{s-1}_{2, 1}  \times B^{s-1}_{2, 1}$. Assume that 
$(\vthe, \vue)$ is a regular solution to \eqref{linea2} on $[0, T]$. Then, for any $t\in[0, T]$, we have 
\begin{align}
&
\n{\vthe}^{m;\zeta, \f{\beta_0}{\ep}}_{\widetilde{L}^\infty_{t}(B^{s-1}_{2, 1})}
+
\n{ \mathcal{P}^\perp\vue}^{m;\zeta, \f{\beta_0}{\ep}}_{\widetilde{L}^\infty_{t}(B^{s-1}_{2, 1})}
+
\n{\vthe}^{m;\zeta, \f{\beta_0}{\ep}}_{L^1_{t}(B^{s+1}_{2, 1})}
+
\n{\mathcal{P}^\perp\vue}^{m;\zeta, \f{\beta_0}{\ep}}_{L^1_{t}(B^{s+1}_{2, 1})}
\nonumber
\\
& \quad 
\leq
C
\left(
\n{\vartheta_0}^{m;\zeta, \f{\beta_0}{\ep}}_{B^{s-1}_{2, 1}}
+
\n{\mathcal{P}^\perp\vu_0}^{m;\zeta, \f{\beta_0}{\ep}}_{B^{s-1}_{2, 1}}
+
\n{F_2}^{m;\zeta, \f{\beta_0}{\ep}}_{L^1_{t}(B^{s-1}_{2, 1})}
+
\n{G_2}^{m;\zeta, \f{\beta_0}{\ep}}_{L^1_{t}(B^{s-1}_{2, 1})}
\right.
\nonumber
\\
&
\qquad
\left.
+ 
\int_0^t
\n{\vv}_{\underline{B}^{\f d 2+1}_{2, 1}}
\n{\vthe}^{m; \f{\zeta}{4}, \f{4\beta_0}{\ep}}_{B^{s-1}_{2, 1}}
\dta
+
\int_0^t
\n{\vv}_{\underline{B}^{\f d 2+1}_{2, 1}}
\n{\mathcal{P}^\perp\vue}^{m; \f{\zeta}{4}, \f{4\beta_0}{\ep}}_{B^{s-1}_{2, 1}}
\dta
\right), 
\label{6.2.-2}
\end{align}
where $\beta_0$ is a positive constant.

\end{Proposition}

{\bf Proof of Proposition \ref{Pro6.1}:} Applying $\De_j$ to \eqref{linea1}, multiplying the resulting equation by $(\De_j \rho_\ep, \De_j \mathcal{P}\vue)$, and
integrating over $\Td$, we obtain
\begin{align*}
&
\f{1}{2}
\f{d}{\dt}
\n{\De_j \rho_\ep}^2_{L^2}
\\
&\quad 
=
-
\int_{\Td} 
\sp{
\De_j T_{\vv} \Grad \rho_\ep
-
S_{j-2} \vv  \Grad  \De_j \rho_\ep
}
\De_j \rho_\ep
\dx
-
\int_{\Td} 
S_{j-2} \vv  \Grad  \De_j \rho_\ep
\De_j \rho_\ep
\dx
+
\int_{\Td}
\De_j F_1
\De_j \rho_\ep
\dx
\\
&
\qquad
-
c_{01}
\int_{\Td}
\Div \De_j  \mathcal{P}^\perp \vue
\De_j \rho_\ep
\dx
+
\int_{\Td}
\De_j(T_{\Grad \rho^*}\vue)
\dx
+
\int_{\Td}
\De_j(T_{\rho^*}\Div \mathcal{P}^\perp \vue)
\De_j \rho_\ep
\dx,
\\
&
\f{1}{2}
\f{d}{\dt}
\n{\De_j \mathcal{P}\vue}^2_{L^2}
+
\f{\mu}{R^0+Q^0}
\n{\Grad \De_j \mathcal{P}\vue}^2_{L^2}
\\
& \quad 
=
-
\int_{\Td} 
\sp{
\De_j \mathcal{P}(T_{\vv} \Grad \vue)
-
S_{j-2} \vv \Grad \De_j \mathcal{P}\vue
}
\De_j \mathcal{P}\vue
\dx
-
\int_{\Td} 
S_{j-2} \vv  \Grad \De_j \mathcal{P}\vue
\De_j \mathcal{P}\vue
\dx
\\
&
\qquad
+
\int_{\Td}
\De_j G_1
\De_j \mathcal{P}\vue
\dx.
\end{align*}
Integration by parts yields
\begin{align*}
&\int_{\Td} 
S_{j-2} \vv  \Grad \De_j \rho_\ep
\De_j \rho_\ep
\dx
=
-
\f1 2
\int_{\Td}
\Div S_{j-2} \vv
\De_j \rho_\ep
\De_j \rho_\ep
\dx,
\\
&
\int_{\Td} 
S_{j-2} \vv \Grad \De_j \mathcal{P}\vue
\De_j \mathcal{P}\vue
\dx
=
-
\f1 2
\int_{\Td}
\Div S_{j-2} \vv
\De_j \mathcal{P}\vue
\De_j \mathcal{P}\vue
\dx
.
\end{align*}
Hence,
\begin{align}
&
\n{\De_j \rho_\ep}_{L^\infty_t(L^2)}
\nonumber
\leq
C
\left(
\n{\De_j \rho_0}_{L^2}
+
\n{\De_j T_{\vv} \Grad \rho_\ep
-
S_{j-2} \vv  \Grad \De_j \rho_\ep}_{L_t^1(L^2)}
+
\n{\De_j F_1}_{L_t^1(L^2)}
+
2^j \n{\De_j \mathcal{P}^\perp\vue}_{L_t^1(L^2)}
\right.
\nonumber
\\
&
\qquad  \qquad \qquad  \qquad 
\left.
+
\n{\De_j (T_{\Grad \rho^*}\vue)}_{L_t^1(L^2)}
+
\n{\De_j (T_{\rho^*}\Div \mathcal{P}^\perp\vue)}_{L_t^1(L^2)}
\right), 
\label{6.2.1}
\\
&
\n{\De_j \mathcal{P}\vue}_{L^\infty_t(L^2)}
+
2^{2j}
\n{\De_j \mathcal{P}\vue}_{L^1_t(L^2)}
\nonumber
\\
&
\quad
\leq
C
\sp{
\n{\De_j \mathcal{P}\vu_0}_{L^2}
+
\n{
\De_j \mathcal{P} \sp{T_{\vv} \Grad \vue}
-
S_{j-2} \vv  \Grad \De_j \mathcal{P}\vue
}_{L_t^1(L^2)}
+
\n{\De_j G_1}_{L_t^1(L^2)}
}.
\label{6.2.2}
\end{align}
Multiplying both sides of \eqref{6.2.1} by $2^{sj}$
and both sides of \eqref{6.2.2} by $2^{(s-1)j}$,  then summing over 
$j$ with $2^j\geq \zeta$ and applying Lemma \ref{lee6.1}, we obtain
\begin{align}
&
\n{\rho_\ep}^{h;\zeta}_{\widetilde{L}^\infty_{t}(B^{s}_{2, 1})}
\leq
C
\Bigg[
\n{\rho_0}^{h;\zeta}_{B^{s}_{2, 1}}
+
\n{F_1}^{h; \zeta}_{L^1_{t}(B^{s}_{2, 1})}
+
\int_0^t 
\n{\vv}_{L^1_{t}(\underline{B}^{\f d 2+1}_{2, 1})}
\n{\rho_\ep}^{h; \f{\zeta}{4}}_{B^{s}_{2, 1}}   \dta
+
\n{\mathcal{P}^\perp\vue}^{h;\zeta}_{L^1_{t}(B^{s+1}_{2, 1})}
\nonumber
\\
&
\qquad
\qquad
\qquad
\quad
+
\n{\rho^*}_{L^\infty_{t}(B^{\f d 2}_{2, 1})}
\sp{
\n{\mathcal{P}^\perp \vue}^{h;\zeta}_{L^1_{t}(B^{s+1}_{2, 1})}
+
\n{\mathcal{P}^\perp \vue}^{m;\f{\zeta}{4}, \zeta}_{L^1_{t}(B^{s+1}_{2, 1})}
+
\n{\mathcal{P}\vue}^{h;\f{\zeta}{4}}_{L^1_{t}(B^{s+1}_{2, 1})}
+
\n{\mathcal{P}\vue}^{m;\f{\zeta}{4}, \zeta}_{L^1_{t}(B^{s+1}_{2, 1})}
}
\Bigg],
\label{6.2.-4}
\\
&
\n{ \mathcal{P}\vue}^{h; \zeta}_{\widetilde{L}^\infty_{t}(B^{s-1}_{2, 1})}
+
\n{\mathcal{P}\vue}^{h; \zeta}_{L^1_{t}(B^{s+1}_{2, 1})}
\leq
C
\sp{
\n{\mathcal{P}\vu_0}^{h; \zeta}_{B^{s-1}_{2, 1}}
+
\n{G_1}^{h; \zeta}_{L^1_{t}(B^{s-1}_{2, 1})}
+
\int_0^t 
\n{\vv}_{L^1_{t}(\underline{B}^{\f d 2+1}_{2, 1})}
\n{\mathcal{P}\vue}^{h; \f{\zeta}{4}}_{B^{s-1}_{2, 1}}
\dta
}.
\label{6.2.-5}
\end{align}
Multiplying both sides of $\eqref{6.2.-4}$ by $
(8C)^{-1} (1 + \n{\rho^*}_{L^\infty_T(B^{\f d 2}_{2,1})} )^{-1}
$ and adding the resulting inequality to \eqref{6.2.-5} yields
\begin{align*}
&
\f{1}{8C\sp{1 + \n{\rho^*}_{L^\infty_T(B^{\f d 2}_{2,1})}}}
\n{\rho_\ep}^{h;\zeta}_{\widetilde{L}^\infty_{t}(B^{s}_{2, 1})}
+
\n{ \mathcal{P}\vue}^{h; \zeta}_{\widetilde{L}^\infty_{t}(B^{s-1}_{2, 1})}
+
\f{7}{8}
\n{\mathcal{P}\vue}^{h; \zeta}_{L^1_{t}(B^{s+1}_{2, 1})}
\\
& \quad 
\leq
C
\left(
\n{\rho_0}^{h;\zeta}_{B^{s}_{2, 1}}
+\n{\mathcal{P}\vu_0}^{h; \zeta}_{B^{s-1}_{2, 1}}
+
\n{G_1}^{h; \zeta}_{L^1_{t}(B^{s-1}_{2, 1})}
+
\n{F_1}^{h; \zeta}_{L^1_{t}(B^{s}_{2, 1})}
+
\int_0^t 
\n{\vv}_{L^1_{t}(\underline{B}^{\f d 2+1}_{2, 1})}
\n{\rho_\ep}^{h; \f{\zeta}{4}}_{B^{s}_{2, 1}}  \dta
\right.
\\
&
\qquad
\left.
+
\int_0^t 
\n{\vv}_{L^1_{t}(\underline{B}^{\f d 2+1}_{2, 1})}
\n{\mathcal{P}\vue}^{h; \f{\zeta}{4}}_{B^{s-1}_{2, 1}}
\dta
+
\n{\mathcal{P} \vue}^{m;\f{\zeta}{4}, \zeta}_{L^1_{t}(B^{s+1}_{2, 1})}
+
\n{\mathcal{P}^\perp \vue}^{m;\f{\zeta}{4}, \zeta}_{L^1_{t}(B^{s+1}_{2, 1})}
\right)
+
\f 1 4
\n{\mathcal{P}^\perp\vue}^{h; \f{\zeta}{4}}_{L^1_{t}(B^{s+1}_{2, 1})}.
\end{align*}
This completes the proof.    \ \ $\Box$

{\bf Proof of Proposition \ref{Pro6.2}: }Similar to \cite{DH16, F24, H11}, we introduce the effective velocity  
$
\vc{\varpi}_\ep
:=
\mathcal{P}^\perp \vue
+
\f{(R^0+Q^0)c^2_0} {\ep\nu R^0}
\Grad 
(-\Delta)^{-1}
\vthe.
$
Then, we see 
\begin{equation}\label{linea3}
\left\{
\begin{aligned}
&
\ep
\p_t 
\vthe 
+
\df{(R^0+\ep c_{02})(R^0+Q^0)c^2_0}{\ep \nu R^0}
\vthe
+
\ep
T_{\vv}
\Grad
\vthe
=
\ep
F_2
-
\sp{
R^0+\ep c_{02}
}
\Div
\vc{\varpi}_\ep,
\\
&
\p_t 
\vc{\varpi}_\ep
-
\df{\nu}{R^0+Q^0}
\De
\vc{\varpi}_\ep
\\
&
\quad
=
G_2
-
\mathcal{P}^\perp
\sp{
T_{\vv}\Grad \vue
}
+
\df{(R^0+\ep c_{02})(R^0+Q^0)c^2_0}{\ep^2\nu R^0}
\vc{\varpi}_\ep
-
\df{(R^0+\ep c_{02})(R^0+Q^0)^2c^4_0}{\ep^3\nu^2(R^0)^2}
\Grad
(-\De)^{-1}
\vthe
\\
&
\qquad
+
\df{(R^0+Q^0)c^2_0}{\ep\nu R^0}
\Grad (-\De)^{-1}
\sp{
F_2
-
T_{\vv}\Grad \vthe
},
\\
&
(\vthe (0), \vc{\varpi}_\ep (0))
=
(\vartheta_0, 
\mathcal{P}^\perp \vu_0
+
\f{(R^0+Q^0)c^2_0} {\ep\nu R^0}
\Grad 
(-\Delta)^{-1}
\vartheta_0). 
\end{aligned}
\right.
\end{equation}
Applying $\De_j$ to \eqref{linea3}, multiplying the resulting equation by $(\De_j \vthe, \De_j \vc{\varpi}_\ep)$ and
integrating over $\Td$, we obtain
\begin{align*}
&
\f{\ep}{2}
\f{d}{\dt}
\n{\De_j \vthe}^2_{L^2}
+
\df{(R^0+\ep c_{02})(R^0+Q^0)c^2_0}{\ep \nu R^0}
\n{\De_j \vthe}^2_{L^2}
\\
& \quad 
=
-
\ep
\int_{\Td} 
\sp{
\De_j T_{\vv} \Grad \vthe
-
S_{j-2} \vv  \Grad \De_j \vthe
}
\De_j \vthe
\dx
-
\ep
\int_{\Td} 
S_{j-2} \vv \Grad \De_j \vthe
\De_j \vthe
\dx
\\
&
\qquad
+
\ep
\int_{\Td}
\De_j F_2
\De_j \vthe
\dx
-
(R^0+\ep c_{02})
\int_{\Td}
\Div \De_j  \vc{\varpi}_\ep
\De_j \vthe
\dx,
\\
&
\f{1}{2}
\f{d}{\dt}
\n{\De_j \vc{\varpi}_\ep}^2_{L^2}
+
\f{\nu}{R^0+Q^0}
\n{\Grad \De_j \vc{\varpi}_\ep}^2_{L^2}
\\
& \quad 
=
\int_{\Td} 
\De_j G_2
\De_j \vc{\varpi}_\ep
\dx
-
\int_{\Td}
\De_j
\mathcal{P}^\perp
\sp{T_{\vv}\Grad \vue}
\De_j \vc{\varpi}_\ep
\dx
+
\df{(R^0+\ep c_{02})(R^0+Q^0)c^2_0}{\ep^2\nu R^0}
\int_{\Td}
\De_j \vc{\varpi}_\ep
\De_j \vc{\varpi}_\ep
\dx
\\
&
\qquad
+
\df{(R^0+\ep c_{02})(R^0+Q^0)^2c^4_0}{\ep^3\nu^2 (R^0)^2}
\int_{\Td}
\Grad (-\De)^{-1}
\De_j \vthe
\De_j \vc{\varpi}_\ep
\dx,
\\
&
\qquad
+
\df{(R^0+Q^0)c^2_0}{\ep\nu R^0}
\int_{\Td}
\Grad (-\De)^{-1}
\De_j
\sp{
F_2
-
T_{\vv}\Grad \vthe
}
\De_j \vc{\varpi}_\ep
\dx.
\end{align*}
Integration by parts gives 
\begin{align*}
\int_{\Td} 
S_{j-2} \vv \De_j \Grad \vthe
\De_j \vthe
\dx
=
-
\f1 2
\int_{\Td}
\Div S_{j-2} \vv
\De_j \vthe
\De_j \vthe
\dx.
\end{align*}
Hence
\begin{align}
&
\ep
2^j
\n{\De_j \vthe}_{L^\infty_t(L^2)}
+
\f{2^j}{\ep}
\n{\De_j \vthe}_{L^1_t(L^2)}
\nonumber
\\
& \quad 
\leq
C
\ep
2^j
\left(
\n{\De_j \vartheta_0 }_{L^2} 
+
\n{\De_j T_{\vv} \Grad \vthe
-
S_{j-2} \vv \De_j \Grad \vthe}_{L_t^1(L^2)}
+
\n{
\Div S_{j-2} \vv \De_j \vthe}_{L_t^1(L^2)}
\right.
\nonumber
\\
&
\left.
\qquad
+
\n{
 \De_j F_2}_{L_t^1(L^2)}
\right)
+
C
2^{2j}
\n{
 \De_j \vc{\varpi}_\ep}_{L_t^1(L^2)},
\label{6.2.5}
\\[4pt]
&
\n{\De_j \vc{\varpi}_\ep}_{L^\infty_t(L^2)}
+
2^{2j}
\n{\De_j \vc{\varpi}_\ep}_{L^1_t(L^2)}
\nonumber
\\
& \quad 
\leq
C
\sp{
\n{\De_j \vc{\varpi}_0 }_{L^2} 
+
\n{\De_j G_2}_{L_t^1(L^2)}
+
\n{
\De_j 
\sp{T_{\vv} \Grad \mathcal{P}^\perp \vue}
}_{L_t^1(L^2)}
+
\df{2^{2j}}{(2^j \ep)^2}
\n{\De_j \vc{\varpi}_\ep}_{L_t^1(L^2)}}
\nonumber
\\
&
\qquad
+
\df{C}{(2^j \ep)^2}
\df{2^j}{\ep}
\n{\De_j \vthe}_{L_t^1(L^2)}
+
\df{C}{2^j \ep}
\sp{
\n{\De_j F_2}_{L_t^1(L^2)}
+
\n{
\De_j 
\sp{T_{\vv} \Grad \vthe}
}_{L_t^1(L^2)}
}.
\label{6.2.6}
\end{align}
We set $\beta_0:=2\sqrt{\max\{2C^2, C\}}$ and consider $j$ such that $\beta_0 \leq 2^j \ep $. Multiplying both sides of \eqref{6.2.5} by $(4C)^{-1}$
and adding the resulting inequality to \eqref{6.2.6} yields
\begin{align}\label{6.2.7}
&
\ep
2^j
\n{\De_j \vthe}_{L^\infty_t(L^2)}
+
\n{\De_j \vc{\varpi}_\ep}_{L^\infty_t(L^2)}
+
\f{2^j}{\ep}
\n{\De_j \vthe}_{L^1_t(L^2)}
+
2^{2j}
\n{\De_j \vc{\varpi}_\ep}_{L^1_t(L^2)}
\nonumber
\\
& \quad 
\les
\ep
2^j
\left(
\n{\De_j \vartheta_0 }_{L^2} 
+
\n{\De_j T_{\vv} \Grad \vthe
-
S_{j-2} \vv \De_j \Grad \vthe}_{L_t^1(L^2)}
+
\n{
\Div S_{j-2} \vv \De_j \vthe}_{L_t^1(L^2)}
\right.
\nonumber
\\
&
\left.
\qquad
+
\n{
 \De_j F_2}_{L_t^1(L^2)}
\right)
+
\n{\De_j \vc{\varpi}_0 }_{L^2} 
+
\n{\De_j G_2}_{L_t^1(L^2)}
+
\n{
\De_j 
\sp{T_{\vv} \Grad \mathcal{P}^\perp \vue}
}_{L_t^1(L^2)}
\nonumber
\\
&
\qquad
+
\n{\De_j F_2}_{L_t^1(L^2)}
+
\n{
\De_j 
\sp{T_{\vv} \Grad \vthe}
}_{L_t^1(L^2)}.
\end{align}
For such $j$, we have
\begin{align}\label{6.2.8}
&
\ep
2^j
\n{\De_j \vthe}_{L^\infty_t(L^2)}
+
\n{\De_j \vc{\varpi}_\ep}_{L^\infty_t(L^2)}
\approx
\ep
2^j
\n{\De_j \vthe}_{L^\infty_t(L^2)}
+
\n{\De_j \mathcal{P}^\perp \vue}_{L^\infty_t(L^2)},
\nonumber
\\
&
\f{2^j}{\ep}
\n{\De_j \vthe}_{L^1_t(L^2)}
+
2^{2j}
\n{\De_j \vc{\varpi}_\ep}_{L^1_t(L^2)}
\approx
\f{2^j}{\ep}
\n{\De_j \vthe}_{L^1_t(L^2)}
+
2^{2j}
\n{\De_j \mathcal{P}^\perp \vue}_{L^1_t(L^2)} . 
\end{align}
Inserting \eqref{6.2.8} into \eqref{6.2.7}, multiplying both sides of the resulting inequality by $2^{j(s-1)}$, and summing over $j$ satisfying 
$2^j\ep \geq \beta_0$, and then applying Lemma \ref{lee6.1}, we obtain
\begin{align*}
&
\ep
\n{\vthe}^{h;\f{\beta_0}{\ep}}_{\widetilde{L}^\infty_{t}(B^{s}_{2, 1})}
+
\n{ \mathcal{P}^\perp\vue}^{h; \f{\beta_0}{\ep}}_{\widetilde{L}^\infty_{t}(B^{s-1}_{2, 1})}
+
\f{1}{\ep}
\n{\vthe}^{h; \f{\beta_0}{\ep}}_{L^1_{t}(B^{s}_{2, 1})}
+
\n{\mathcal{P}^\perp\vue}^{h; \f{\beta_0}{\ep}}_{L^1_{t}(B^{s+1}_{2, 1})}
\\
& \quad 
\les
\ep
\n{\vartheta_0}^{h;\f{\beta_0}{\ep}}_{B^{s}_{2, 1}}
+
\n{\mathcal{P}^\perp\vu_0}^{h;\f{\beta_0}{\ep}}_{B^{s-1}_{2, 1}}
+
\ep
\n{F_2}^{h; \f{\beta_0}{\ep}}_{L^1_{t}(B^{s}_{2, 1})}
+
\n{G_2}^{h; \f{\beta_0}{\ep}}_{L^1_{t}(B^{s-1}_{2, 1})}
+
\ep
\int_0^t
\n{\vv}_{\underline{B}^{\f d 2+1}_{2, 1}}
\n{\vthe}_{B^{s}_{2, 1}}
\dta
\\
&
\qquad
+
\int_0^t
\n{\vv}_{B^{\f d 2}_{2, 1}}
\n{\vthe}^{h; \f{\beta_0}{4\ep}}_{B^{s}_{2, 1}}
\dta
+
\int_0^t
\n{\vv}_{B^{\f d 2}_{2, 1}}
\n{\mathcal{P}^\perp\vue}^{h; \f{\beta_0}{4\ep}}_{B^{s}_{2, 1}}
\dta.
\end{align*}
The proof is finished.    \ \ $\Box$

{\bf Proof of Proposition \ref{Pro6.3}: }Applying $\De_j$ to \eqref{linea2}, multiplying the resulting equation by $(\De_j \vthe, \De_j \mathcal{P}^\perp\vue)$, and
integrating over $\Td$, we obtain
\begin{align}
&
\f{1}{2}
\f{d}{\dt}
\n{\De_j \vthe}^2_{L^2}
\nonumber
\\
&\quad 
=
-
\int_{\Td} 
\sp{
\De_j T_{\vv} \Grad \vthe
-
S_{j-2} \vv  \Grad \De_j \vthe
}
\De_j \vthe
\dx
+
\f1 2
\int_{\Td} 
\Div S_{j-2} \vv  \De_j \vthe
\De_j \vthe
\dx
\nonumber
\\
&
\qquad
+
\int_{\Td}
\De_j F_2
\De_j \vthe
\dx
-
\f{R^0+c_{0}^{2} \ep}{\ep}
\int_{\Td}
\Div \De_j \mathcal{P}^\perp\vue
\De_j \vthe
\dx,
\label{6.2.14}
\\
&
\f{1}{2}
\f{d}{\dt}
\n{\De_j \mathcal{P}^\perp\vue}^2_{L^2}
+
\f{\nu}{R^0+Q^0}
\n{\Grad \De_j \mathcal{P}^\perp\vue}^2_{L^2}
\nonumber
\\
& \quad 
=
-
\int_{\Td} 
\sp{
\De_j \mathcal{P}^\perp(T_{\vv} \Grad \vue)
-
S_{j-2} \vv \Grad \De_j \mathcal{P}^\perp\vue
}
\De_j \mathcal{P}^\perp\vue
\dx
+
\f1 2
\int_{\Td} 
\Div S_{j-2} \vv   \De_j \mathcal{P}^\perp\vue
\De_j \mathcal{P}^\perp\vue
\dx
\nonumber
\\
&
\qquad
-
\df{c^2_0}{R^0\ep}
\int_{\Td}
\De_j \Grad \vthe
\De_j \mathcal{P}^\perp\vue
\dx
+
\int_{\Td}
\De_j G_2
\De_j \mathcal{P}^\perp\vue
\dx.
\label{6.2.15}
\end{align}
Multiplying both sides of \eqref{6.2.14} by $\f{2c^2_0}{R^0}$
and both sides of \eqref{6.2.15} by $2(R^0+\ep c_{0}^2)$, then adding the resulting two equations yields
\begin{align}
&
\f{c^2_0}{R^0}
\f{d}{\dt}
\n{\De_j \vthe}^2_{L^2}
+
(R^0+\ep c_{0}^2)
\f{d}{\dt}
\n{\De_j \mathcal{P}^\perp\vue}^2_{L^2}
+
\f{\nu(R^0+\ep c_{02})}{R^0+Q^0}
\n{\Grad \De_j \mathcal{P}^\perp\vue}^2_{L^2}
\nonumber
\\
& \quad 
=
\f{2c^2_0}{R^0}
\left(
\f1 2
\int_{\Td} 
\Div S_{j-2} \vv  \De_j \vthe
\De_j \vthe
\dx
-
\int_{\Td} 
\sp{
\De_j T_{\vv} \Grad \vthe
-
S_{j-2} \vv  \Grad \De_j \vthe
}
\De_j \vthe
\dx
\right.
\nonumber
\\
&
\left.\qquad 
+
\int_{\Td}
\De_j F_2
\De_j \vthe
\dx
\right)
+
2(R^0+\ep c_{0}^2)
\Bigg[
\f1 2
\int_{\Td} 
\Div S_{j-2} \vv   \De_j \mathcal{P}^\perp\vue
\De_j \mathcal{P}^\perp\vue
\dx
\nonumber
\\
&
 \qquad 
-
\int_{\Td} 
\sp{
\De_j \mathcal{P}^\perp(T_{\vv} \Grad \vue)
-
S_{j-2} \vv \Grad \De_j \mathcal{P}^\perp\vue
}
\De_j \mathcal{P}^\perp\vue
\dx
+
\int_{\Td}
\De_j G_2
\De_j \mathcal{P}^\perp\vue
\dx
\Bigg],
\label{6.2.16}
\end{align}
where we used
\begin{align*}
\int_{\Td}
\Div \De_j \mathcal{P}^\perp\vue
\De_j \vthe
\dx
+
\int_{\Td}
\De_j \Grad \vthe
\De_j \mathcal{P}^\perp\vue
\dx
=
0.
\end{align*}
Applying $\De_j \Grad $ and $\De_j$ to equations \eqref{linea2}$_1$
and \eqref{linea2}$_2$, respectively, then multiplying the resulting equations by $\ep\De_j \mathcal{P}^\perp\vue$ and  $\ep\De_j \Grad \vthe$, respectively, and integrating over $\Td$, we obtain
\begin{align}
&
\f{d}{\dt}
\int_{\Td} \ep\Grad\vthe \mathcal{P}^\perp\vue\dx
+
\df{c^2_0}{R^0}
\n{\Grad \De_j \vthe}^2_{L^2}
\nonumber
\\
&
=
-
\ep
\int_{\Td} 
\sp{
\De_j\Grad (T_{\vv} \Grad \vthe)
-
S_{j-2} \vv \Grad^2 \De_j \vthe
}
\De_j \mathcal{P}^\perp\vue
\dx
\nonumber
\\
&
\quad
-
\ep
\int_{\Td} 
\sp{
\De_j \mathcal{P}^\perp(T_{\vv} \Grad \vue)
-
S_{j-2} \vv \Grad \De_j \mathcal{P}^\perp\vue
}
\Grad \De_j \vthe
\dx
+
\ep
\int_{\Td}
\Div S_{j-2} \vv 
\Grad \De_j \vthe
\De_j \mathcal{P}^\perp\vue
\dx
\nonumber
\\
&
\quad
+
\ep
\int_{\Td}
 \Grad \De_j F_2
\De_j \mathcal{P}^\perp\vue
\dx
-
(R^0+c_{02}\ep)
\int_{\Td}
\Grad \Div \De_j \mathcal{P}^\perp\vue
\De_j \mathcal{P}^\perp\vue
\dx
\nonumber
\\
&
\quad
+
\df{\nu \ep}{R^0+Q^0}
\int_{\Td}
\De \De_j \mathcal{P}^\perp\vue
\Grad \De_j \vthe
\dx
+
\ep
\int_{\Td}
\De_j G_2
\Grad \De_j \vthe
\dx,
\label{6.2.17}
\end{align}
where we used
\begin{align*}
&
\int_{\Td} 
S_{j-2} \vv  \Grad^2 \De_j \vthe
\De_j \mathcal{P}^\perp\vue
\dx
+
\int_{\Td} 
S_{j-2} \vv \Grad \De_j \mathcal{P}^\perp\vue
\Grad \De_j \vthe
\dx
\\
&
\qquad
=
-
\int_{\Td}
\Div S_{j-2} \vv 
\Grad \De_j \vthe
\De_j \mathcal{P}^\perp\vue
\dx.
\end{align*}
Next, we focus on 
\begin{align*}
(R^0+c_{02}\ep)
\int_{\Td}
\Grad \Div \De_j \mathcal{P}^\perp\vue
\De_j \mathcal{P}^\perp\vue
\dx,
\qquad 
\df{\nu \ep}{R^0+Q^0}
\int_{\Td}
\De \De_j \mathcal{P}^\perp\vue
\Grad \De_j \vthe
\dx,
\end{align*}
which satisfy
\begin{align}
&
\left|
(R^0+c_{02}\ep)
\int_{\Td}
\Grad \Div \De_j \mathcal{P}^\perp\vue
\De_j \mathcal{P}^\perp\vue
\dx
\right|
=
(R^0+c_{02}\ep)
\n{\Div \De_j \mathcal{P}^\perp\vue}^2_{L^2}
\leq
(R^0+c_{02}\ep)
\n{\Grad \De_j\mathcal{P}^\perp\vue}^2_{L^2}, 
\label{6.2.19}
\\
& 
\df{\nu \ep}{R^0+Q^0}
\left|
\int_{\Td}
\De \De_j \mathcal{P}^\perp\vue
\Grad \De_j \vthe
\dx 
\right|
\leq
C
\n{\Grad \De_j \mathcal{P}^\perp\vue}_{L^2}
\n{\Grad \De_j \vthe}_{L^2}
\nonumber
\\
& \qquad \qquad 
\leq
\df{C^2 R^0}{2c^2_0}
\n{\Grad \De_j \mathcal{P}^\perp\vue}^2_{L^2}
+
\df{c^2_0}{2R^0}
\n{\Grad \De_j \vthe}^2_{L^2},
\label{6.2.18}
\end{align}
where we used $2^j \ep < \beta_0$ and $\n{\De \De_j f}_{L^2}\approx 2^j \n{ \Grad \De_j f}_{L^2} $. Our idea is to multiply both sides of \eqref{6.2.19} and \eqref{6.2.18} by a small positive constant $\delta$, and then absorb it by the left-hand sides of \eqref{6.2.16} and \eqref{6.2.17}. On the other hand, for such $\delta$, we require that
\begin{align*}
&
\f{c^2_0}{R^0}
\n{\De_j \vthe}^2_{L^2}
+
(R^0+\ep c_{02})
\n{\De_j \mathcal{P}^\perp\vue}^2_{L^2}
+
\delta
\int_{\Td} \ep\Grad\vthe \mathcal{P}^\perp\vue\dx
\\
& \quad 
\approx
\f{c^2_0}{R^0}
\n{\De_j \vthe}^2_{L^2}
+
(R^0+\ep c_{02})
\n{\De_j \mathcal{P}^\perp\vue}^2_{L^2}.
\end{align*}
Observe that
\begin{align*}
\delta
\int_{\Td} \ep\Grad\vthe \mathcal{P}^\perp\vue\dx   
\leq
C
\delta
\sp{
\n{\De_j \vthe}^2_{L^2}
+
\n{\De_j \mathcal{P}^\perp\vue}^2_{L^2}
}.
\end{align*}
Based on the above considerations, we set
\begin{align*}
\delta=\min\left\{\f{\nu}{4(R^0+Q^0)}, \f{\nu c^2_0}{2C^2(R^0+Q^0)}, \f{c^2_0}{2CR^0}, \f{R^0}{2C} \right\}.
\end{align*}
Multiplying both sides of \eqref{6.2.17} by this $\delta$ and adding the resulting inequality to \eqref{6.2.16} yields
\begin{align*}
&
\n{\De_j \vthe}_{L^\infty_t(L^2)}
+
\n{\De_j \mathcal{P}^\perp\vue}_{L^\infty_t(L^2)}
+
2^{2j}
\n{\De_j \vthe}_{L^1_t(L^2)}
+
2^{2j}
\n{\De_j \mathcal{P}^\perp\vue}_{L^1_t(L^2)}
\\
& \quad 
\les
\n{\De_j \vartheta_0}_{L^2}
+
\n{\De_j \mathcal{P}^\perp\vu_0}_{L^2}
+
\n{
\Div S_{j-2} \vv  \De_j \vthe
}_{L^2}
+
\n{
\Div S_{j-2} \vv  \De_j \mathcal{P}^\perp\vue
}_{L^2}
\\
&
\qquad
+
\n{\De_j T_{\vv} \Grad \vthe
-
S_{j-2} \vv  \Grad \De_j \vthe}_{L^2}
+
\n{\De_j \mathcal{P}^\perp(T_{\vv} \Grad \vue)
-
S_{j-2} \vv \Grad \De_j \mathcal{P}^\perp\vue}_{L^2}
\\
&
\qquad
+
\ep
\n{\De_j\Grad (T_{\vv} \Grad \vthe)
-
S_{j-2} \vv \Grad^2 \De_j \vthe}_{L^2}
+
\n{\De_j F_2}_{L^2}
+
\n{\De_j G_2}_{L^2}.
\end{align*}
Multiplying both sides of the above inequality by $2^{j(s-1)}$ and summing over 
$j$ satisfying $\zeta<2^j \leq \f{\beta_0}{\ep}$, then applying Lemma \ref{lee6.2}, we get
\begin{align*}
&
\n{\vthe}^{m;\zeta, \f{\beta_0}{\ep}}_{\widetilde{L}^\infty_{t}(B^{s-1}_{2, 1})}
+
\n{ \mathcal{P}^\perp\vue}^{m;\zeta, \f{\beta_0}{\ep}}_{\widetilde{L}^\infty_{t}(B^{s-1}_{2, 1})}
+
\n{\vthe}^{m;\zeta, \f{\beta_0}{\ep}}_{L^1_{t}(B^{s+1}_{2, 1})}
+
\n{\mathcal{P}^\perp\vue}^{m;\zeta, \f{\beta_0}{\ep}}_{L^1_{t}(B^{s+1}_{2, 1})}
\\
& \quad 
\leq
C
\left(
\n{\vartheta_0}^{m;\zeta, \f{\beta_0}{\ep}}_{B^{s-1}_{2, 1}}
+
\n{\mathcal{P}^\perp\vu_0}^{m;\zeta, \f{\beta_0}{\ep}}_{B^{s-1}_{2, 1}}
+
\n{F_2}^{m;\zeta, \f{\beta_0}{\ep}}_{L^1_{t}(B^{s-1}_{2, 1})}
+
\n{G_2}^{m;\zeta, \f{\beta_0}{\ep}}_{L^1_{t}(B^{s-1}_{2, 1})}
\right.
\\
&
\qquad
\left.
+
\int_0^t
\n{\vv}_{\underline{B}^{\f d 2+1}_{2, 1}}
\n{\vthe}^{m; \f{\zeta}{4}, \f{4\beta_0}{\ep}}_{B^{s-1}_{2, 1}}
\dta
+
\int_0^t
\n{\vv}_{\underline{B}^{\f d 2+1}_{2, 1}}
\n{\mathcal{P}^\perp\vue}^{m; \f{\zeta}{4}, \f{4\beta_0}{\ep}}_{B^{s-1}_{2, 1}}
\dta
\right).
\end{align*}
This completes the proof.      \ \ $\Box$

\subsection{Nonlinear estimates of equations \texorpdfstring{\eqref{1.1}}{} and \texorpdfstring{\eqref{1.2}}{}}\label{nolin}

In this subsection, we establish nonlinear estimates for equations \eqref{1.1} and \eqref{1.2}. Our strategy is to exploit the dissipative terms on the left-hand sides of \eqref{6.2.-3}-\eqref{6.2.-2} to absorb the corresponding right-hand side terms, and then apply Gr\"{o}nwall's inequality to derive uniform estimates. In this process, two main difficulties arise. The first concerns the low-frequency coupling terms, which we decompose into five components: one can be reformulated into a form suitable for Gr\"{o}nwall's inequality, the second can be absorbed by the left-hand sides of \eqref{6.2.-3}-\eqref{6.2.-2}, and the remaining three can be made arbitrarily small by exploiting the smallness of the coefficient $\ep\zeta$ together with the decay properties of $\vZe$ and $\vWe$ , see Propositions \ref{Pro4.1} and \ref{Pro4.2}, as well as the high-frequency vanishing property of $\vV$ and $\vU$, i.e., 
\begin{align}\label{6.3.0}
\lim\limits_{\zeta \to \infty}
\sp{
\n{\vV}^{h;\zeta}_{\widetilde{L}^\infty_{T^*_0}(B^{\f d 2-1}_{2,1})
\cap
L^1_{T^*_0}(B^{\f d 2+1}_{2,1})
}
+
\n{\vw}^{h;\zeta}_{\widetilde{L}^\infty_{T^*_0}(B^{\f d 2-1}_{2,1})
\cap
L^1_{T^*_0}(B^{\f d 2+1}_{2,1})
}
+
\n{\rho}^{h;\zeta}_{\widetilde{L}^\infty_{T^*_0}(B^{\f d 2}_{2,1})}
}
=
0.
\end{align}
More details will be given in Lemma \ref{lee6.3}. The second difficulty comes from the pressure term, where the composition of functions makes Lemmas \ref{lee6.1} and \ref{lee6.2} inapplicable. To overcome this, we invoke \eqref{cons}: 
\begin{align}\label{6.3.1}
\ep
\sp{
\n{a_\ep}_{\widetilde{L}^\infty_t(B^{\f d 2}_{2,1})}
+
\n{b_\ep}_{\widetilde{L}^\infty_t(B^{\f d 2}_{2,1})}
}
\leq
C
\sp{
\ep
\n{\vthe}^{h;\f{\beta_0}{\ep}}_{\widetilde{L}^\infty_t(B^{\f d 2}_{2, 1})}
+
\n{\vthe}^{m; \zeta, \f{\beta_0}{\ep}}_{\widetilde{L}^\infty_t(B^{\f d 2-1}_{2, 1})}
+
\zeta\ep E^{\ep}_{t}
} , 
\end{align}
which enables the application of Gr\"{o}nwall's inequality.

We set
\begin{align*}
&
\vv=\vue,
\ \
\rho^*=\rho_\ep,
\ \
F_1
=
-
\sp{
T_{\Grad (\rho_\ep-\rho)}\vue
+
T_{(\rho_\ep-\rho)}\Div \mathcal{P}^\perp\vue
+
R(\Grad \rho_\ep, \vue)
+
T_{\Div \mathcal{P}^\perp\vue}\rho_\ep
+
R(\Div \mathcal{P}^\perp\vue, \rho_\ep)
},
\\
&
F_2
=
-
\sp{
T_{\Grad \vthe}\vue
+
R(\Grad \vthe, \vue)
+
\vthe 
\Div \mathcal{P}^\perp\vue
},
\\
&
G_1
=
-
\mathcal{P}
\sp{
T_{\Grad \vue} \vue
+
\sp{c_3+\f{c_5Q^0}{R^0}-\f{c_1c_4}{c_2}-\f{c_1c_6}{c_2}}
\rho_\ep
\Grad \vthe
+
\sp{c_3+\f{c_4Q^0}{R^0}-\f{c_1c_5}{c_2}-\f{c_6Q^0}{R^0}}
\vthe
\Grad \rho_\ep
}
+
\mathcal{P} r^3_\ep,
\\
&
G_2
=
-
\mathcal{P}^\perp
\sp{
T_{\Grad \vue} \vue
+
\sp{c_3+\f{c_5Q^0}{R^0}-\f{c_1c_4}{c_2}-\f{c_1c_6}{c_2}}
\rho_\ep
\Grad \vthe
+
\sp{c_3+\f{c_4Q^0}{R^0}-\f{c_1c_5}{c_2}-\f{c_6Q^0}{R^0}}
\vthe
\Grad \rho_\ep
}
\\
&
\qquad
\quad
-
\sp{
c_3+\f{c_4Q^0}{R^0}+\f{c_5Q^0}{R^0}+\f{c_6 (Q^0)^2}{(R^0)^2}
}
\vthe
\Grad \vthe
-
\sp{
c_3+\f{c_6 c^2_1}{c^2_2}-\f{c_1c_4}{c_2}-\f{c_1c_5}{c_2}
}
\rho_\ep
\Grad \rho_\ep
\\
&
\qquad
\quad
-
\sp{
\f{c_3 \widehat{(a_0)}_0}{\sqrt{\T}}
+
\f{c_4 \widehat{(b_0)}_0}{\sqrt{\T}}
+
\f{c_5 Q^0 \widehat{(a_0)}_0}{R^0\sqrt{\T}}
+
\f{c_6 Q^0\widehat{(b_0)}_0}{R^0\sqrt{\T}}
}
\sp{
\Grad \vthe +\Grad \rho_\ep
}
+
\mathcal{P}^\perp r^3_\ep
\end{align*}

\begin{Lemma}\label{lee6.3}
Let $T^*_0$, $\vU$ and $\vV$ be as in Theorem \ref{Th6.1} and Theorem \ref{Th7.2}. Let  $0<T<\infty$ with $T\leq T^*_0$, $0<\theta<1$,  $1\leq \zeta<\df{\beta_0}{\ep}$ and $0<\ep \leq \min\{1, \f{R^0}{8|c_{02}|}\}$. 
If $(\vae, \vbe, \vue)$ is a regular solution to \eqref{1.1.3} on $[0, T]$, then for any $0\leq t\leq T$ and any $0<\delta<1$, we have
\begin{align}
&
\n{\mathcal{P} \vue}^{m;\f{\zeta}{4}, \zeta}_{L^1_{t}(B^{s+1}_{2, 1})}
+
\n{\mathcal{P}^\perp \vue}^{m;\f{\zeta}{4}, \zeta}_{L^1_{t}(B^{s+1}_{2, 1})}
\leq
C
\sp{
\zeta^{1+2\theta}
[t]^{\f1 2}
\sp{
\vW^\ep_{t, \theta}
+
\vZ^\ep_{t, \theta}
}
+
\n{\vw}^{h; \f{\zeta}{4}}_{L^1_{T^*_0}(B^{\f d 2+1}_{2, 1})}
+
\n{\vV}^{h; \f{\zeta}{4}}_{L^1_{T^*_0}(B^{\f d 2+1}_{2, 1})}
}, \label{6.3.4}
\\[5pt]
&
\ep
\int_0^t
\n{\vue}_{\underline{B}^{\f d 2+1}_{2, 1}}
\n{\vthe}_{B^{\f d 2}_{2, 1}}
\dta
+
\int_0^t
\n{\vue}_{\underline{B}^{\f d 2+1}_{2, 1}}
\n{\vthe}^{h; \f{\beta_0}{\ep}}_{B^{\f d 2-1}_{2, 1}}
\dta
\leq
C
\int_0^t
\n{\vue}_{\underline{B}^{\f d 2+1}_{2, 1}}
\ep
\sp{
\n{a_\ep}_{\widetilde{L}^\infty_\tau(B^{\f d 2}_{2,1})}
+
\n{b_\ep}_{\widetilde{L}^\infty_\tau(B^{\f d 2}_{2,1})}
}
\dta,
\label{6.3.7}
\\[5pt]
&
\int_0^t
\sp{
\n{\vthe}_{B^{\f d 2}_{2, 1}}
+
\n{\rho_\ep}_{B^{\f d 2}_{2, 1}}
}
\n{\vthe}^{h; \f{\beta_0}{\ep}}_{B^{\f d 2}_{2, 1}}
\dta
\leq
C
\int_0^t
\ep
\sp{
\n{a_\ep}_{\widetilde{L}^\infty_\tau(B^{\f d 2}_{2,1})}
+
\n{b_\ep}_{\widetilde{L}^\infty_\tau(B^{\f d 2}_{2,1})}
}
\f{1}{\ep}
\n{\vthe}^{h; \f{\beta_0}{\ep}}_{B^{\f d 2-1}_{2, 1}}
\dta
\label{6.3.19},
\\
&
\int_0^t
\sp{
\n{\vthe}^{h;\zeta}_{B^{\f d 2}_{2, 1}}
+
\n{\rho_\ep}^{h;\zeta}_{B^{\f d 2}_{2, 1}}
}
\dta
\leq
\f{\delta}{2}
\n{\vthe}^{m; \zeta, \f{\beta_0}{\ep}}_{L^1_t(B^{\f d 2+1}_{2,1})}
+
\sp{
\f{1}{2\delta}
+1
}
\int_0^t
\n{\vthe}^{m; \zeta, \f{\beta_0}{\ep}}_{\widetilde{L}^\infty_\tau (B^{\f d 2-1}_{2, 1})}
+
\n{\rho_\ep}^{h;\zeta}_{\widetilde{L}^\infty_\tau (B^{\f d 2-1}_{2, 1})}
\dta
+
\ep
E_{t}^{\ep},
\label{6.3.22}
\\[4pt]
&
\int_0^t
\n{\vue}_{B^{\f d 2}_{2, 1}}
\sp{
\n{\vthe}^{m; \f{\zeta}{16}, \zeta}_{B^{\f d 2}_{2, 1}}
+
\n{\mathcal{P}^\perp\vue}^{m; \f{\zeta}{16}, \zeta}_{B^{\f d 2}_{2, 1}}
+
\n{\mathcal{P}\vue}^{m; \f{\zeta}{16}, \zeta}_{B^{\f d 2}_{2, 1}}
}
\dta
+
\int_0^t
\n{\vthe}_{B^{\f d 2}_{2, 1}}
\n{\vue}^{m; \f{\zeta}{16}, \zeta}_{B^{\f d 2}_{2, 1}}
\dta
\nonumber
\\
&
\quad
+
\int_0^t
\sp{
\n{\vthe}_{B^{\f d 2}_{2, 1}}
+
\n{\rho_\ep}_{B^{\f d 2}_{2, 1}}
}
\sp{
\n{\vthe}^{m; \f{\zeta}{16}, \zeta}_{B^{\f d 2}_{2, 1}}
+
\n{\rho_\ep}^{m; \f{\zeta}{16}, \zeta}_{B^{\f d 2}_{2, 1}}
}
\dta
\nonumber
\\
&
\leq
C
[t]
E_{t}^{\ep}
\sp{
\zeta^{2\theta}
\sp{
\vW^{\ep}_{t, \theta}
+
\vZ^{\ep}_{t, \theta}
}
+
\n{\vV}^{h;\f{\zeta}{16}}_{\widetilde{L}_{T^*_0}^\infty(B^{\f d 2}_{2, 1})
\cap
L_{T^*_0}^1(B^{\f d 2}_{2, 1})}
+
\n{\vw}^{h;\f{\zeta}{16}}_{\widetilde{L}_{T^*_0}^\infty(B^{\f d 2}_{2, 1})
\cap
L_{T^*_0}^1(B^{\f d 2}_{2, 1})}
+
\n{\rho}^{h;\f{\zeta}{16}}_{\widetilde{L}_{T^*_0}^\infty(B^{\f d 2}_{2, 1})
}
},
\label{6.3.8}
\\[5pt]
&
\int_0^t
\n{\vue}_{\underline{B}^{\f d 2+1}_{2, 1}}
\sp{
\n{\mathcal{P}^\perp\vue}^{m; \f{\zeta}{16}, \zeta}_{B^{\f d 2-1}_{2, 1}}
+
\n{\mathcal{P}\vue}^{m; \f{\zeta}{16}, \zeta}_{B^{\f d 2-1}_{2, 1}}
+
\n{\vthe}^{m; \f{\zeta}{16}, \zeta}_{B^{\f d 2-1}_{2, 1}}
+
\n{\rho_\ep}^{m; \f{\zeta}{16}, \zeta}_{B^{\f d 2}_{2, 1}}
}
\dta
\nonumber
\\
&
\leq
C
E^{\ep}_{T}
\sp{
\zeta^{1+\theta}
\sp{
\vW^{\ep}_{t, \theta}
+
\vZ^{\ep}_{t, \theta}
}
+
\n{\vV}^{h;\f{\zeta}{16}}_{\widetilde{L}_{T^*_0}^\infty(B^{\f d 2-1}_{2, 1})
\cap
L_{T^*_0}^1(B^{\f d 2+1}_{2, 1})}
+
\n{\vw}^{h;\f{\zeta}{16}}_{\widetilde{L}_{T^*_0}^\infty(B^{\f d 2-1}_{2, 1})
\cap
L_{T^*_0}^1(B^{\f d 2+1}_{2, 1})}
+
\n{\rho}^{h;\f{\zeta}{16}}_{\widetilde{L}_{T^*_0}^\infty(B^{\f d 2}_{2, 1})
}
},
\label{6.3.13}
\\[5pt]
&
\int_0^t
\n{\vue}_{B^{\f d 2}_{2, 1}}
\sp{
\n{\vthe}^{m; \zeta, \f{\beta_0}{\ep}}_{B^{\f d 2}_{2, 1}}
+
\n{\mathcal{P}\vue}^{h; \zeta}_{B^{\f d 2}_{2, 1}}
+
\n{\mathcal{P}^\perp\vue}^{h; \zeta}_{B^{\f d 2}_{2, 1}}
}
\dta
+
\int_0^t
\n{\vthe}_{B^{\f d 2}_{2, 1}}
\n{\vue}^{h; \zeta}_{B^{\f d 2}_{2, 1}}
\dta
\nonumber
\\
&
\quad
+
\int_0^t
\sp{
\n{\vthe}_{B^{\f d 2}_{2, 1}}
+
\n{\rho_\ep}_{B^{\f d 2}_{2, 1}}
}
\n{\vthe}^{m; \zeta, \f{\beta_0}{\ep}}_{B^{\f d 2}_{2, 1}}
\dta
\nonumber
\\
&
\leq
\f{1}{\delta}
\int_0^t
\sp{
\n{\vue}^2_{B^{\f d 2}_{2, 1}}
+
\n{\vthe}^2_{B^{\f d 2}_{2, 1}}
+
\n{\rho_\ep}^2_{B^{\f d 2}_{2, 1}}
}
\sp{
\n{\vthe}^{m; \zeta, \f{\beta_0}{\ep}}_{\widetilde{L}^\infty_\tau(B^{\f d 2-1}_{2, 1})}
+
\n{\mathcal{P}\vue}^{h; \zeta}_{\widetilde{L}^\infty_\tau(B^{\f d 2-1}_{2, 1})}
+
\n{\mathcal{P}^\perp\vue}^{h; \zeta}_{\widetilde{L}^\infty_\tau(B^{\f d 2-1}_{2, 1})}
}
\dta
\nonumber
\\
&
\quad
+
\f{3\delta}{2}
\sp{
\n{\vthe}^{m;\zeta, \f{\beta_0}{\zeta}}_{L^1_t(B^{\f d 2+1}_{2, 1})}
+
\n{\mathcal{P}^\perp\vue}^{h;\zeta}_{L^1_t(B^{\f d 2+1}_{2, 1})}
+
\n{\mathcal{P}\vue}^{h;\zeta}_{L^1_t(B^{\f d 2+1}_{2, 1})}
},
\label{6.3.9}
\end{align}

\begin{align}
&
\int_0^t
\n{\vue}^{l;\zeta}_{B^{\f d 2}_{2, 1}}
\n{\vthe}^{h; \f{\beta_0}{\ep}}_{B^{\f d 2}_{2, 1}}
\dta
\leq
\zeta
\ep
\sp{
E_{t}^{\ep}
}^2,
\label{6.3.10}
\\[5pt]
&
\int_0^t
\n{\vue}^{h;\zeta}_{B^{\f d 2}_{2, 1}}
\n{\vthe}^{h; \f{\beta_0}{\ep}}_{B^{\f d 2}_{2, 1}}
\dta
\nonumber
\\
&
\leq
\f{1}{2\delta}
\int_0^t
\sp{
\n{\mathcal{P}\vue}^{h;\zeta}_{\widetilde{L}^\infty_\tau(B^{\f d 2-1}_{2, 1})}
+
\n{\mathcal{P}^\perp \vue}^{h;\zeta}_{\widetilde{L}^\infty_\tau(B^{\f d 2-1}_{2, 1})}
}
\sp{
\n{\vthe}^{h; \f{\beta_0}{\ep}}_{B^{\f d 2}_{2, 1}}
}^2
\dta
+
\f{\delta}{2}
\sp{
\n{\mathcal{P}\vue}^{h;\zeta}_{L^1_t(B^{\f d 2+1}_{2, 1})}
+
\n{\mathcal{P}^\perp \vue}^{h;\zeta}_{L^1_t(B^{\f d 2+1}_{2, 1})}
}
\label{6.3.11},
\\
&
\int_0^t
\n{\vue}_{B^{\f d 2}_{2, 1}}
\sp{
\n{\vthe}^{h; \f{\zeta}{16}}_{B^{\f d 2}_{2, 1}}
+
\n{\mathcal{P}^\perp\vue}^{h; \f{\zeta}{16}}_{B^{\f d 2}_{2, 1}}
+
\n{\mathcal{P}\vue}^{h; \f{\zeta}{16}}_{B^{\f d 2}_{2, 1}}
}
\dta
+
\int_0^t
\n{\vthe}_{B^{\f d 2}_{2, 1}}
\n{\vue}^{h; \f{\zeta}{16}}_{B^{\f d 2}_{2, 1}}
\dta
\nonumber
\\
&
\quad
+
\int_0^t
\sp{
\n{\vthe}_{B^{\f d 2}_{2, 1}}
+
\n{\rho_\ep}_{B^{\f d 2}_{2, 1}}
}
\sp{
\n{\vthe}^{h; \f{\zeta}{16}}_{B^{\f d 2}_{2, 1}}
+
\n{\rho_\ep}^{h; \f{\zeta}{16}}_{B^{\f d 2}_{2, 1}}
}
\dta
\nonumber
\\
&
\leq
\sp{
\f{3}{2\delta}
+
C
+
1
}
\int_0^t
\sp{
\n{\mathcal{P}\vue}^{h;\zeta}_{\widetilde{L}^\infty_\tau(B^{\f d 2-1}_{2, 1})}
+
\n{\mathcal{P}^\perp \vue}^{h;\zeta}_{\widetilde{L}^\infty_\tau(B^{\f d 2-1}_{2, 1})}
+
\n{\vthe}^{m; \zeta, \f{\beta_0}{\ep}}_{\widetilde{L}^\infty_\tau(B^{\f d 2-1}_{2, 1})}
+
\n{\rho_\ep}^{h; \zeta}_{\widetilde{L}^\infty_\tau(B^{\f d 2}_{2, 1})}
+
\ep
\sp{
\n{a_\ep}_{\widetilde{L}^\infty_\tau(B^{\f d 2}_{2,1})}
+
\n{b_\ep}_{\widetilde{L}^\infty_\tau(B^{\f d 2}_{2,1})}
}
}
\nonumber
\\
&
\quad
\cdot
\sp{
\n{\vthe}_{B^{\f d 2}_{2, 1}}
+
\n{\rho}_{B^{\f d 2}_{2, 1}}
+
\n{\vthe}^2_{B^{\f d 2}_{2, 1}}
+
\n{\vue}^2_{B^{\f d 2}_{2, 1}}
+
\n{\rho_\ep}^2_{B^{\f d 2}_{2, 1}}
+
\f{1}{\ep}
\n{\vthe}^{h;\f{\beta_0}{\ep}}_{B^{\f d 2}_{2, 1}}
}
\dta
\nonumber
\\
&
\quad
+
2\delta
\sp{
\n{\vthe}^{m;\zeta, \f{\beta_0}{\zeta}}_{L^1_t(B^{\f d 2+1}_{2, 1})}
+
\n{\mathcal{P}^\perp\vue}^{h;\zeta}_{L^1_t(B^{\f d 2+1}_{2, 1})}
+
\n{\mathcal{P}\vue}^{h;\zeta}_{L^1_t(B^{\f d 2+1}_{2, 1})}
}
+
\zeta
\ep
\sp{
E_{t}^{\ep}
}^2
\nonumber
\\
&
\quad
+
C
[t]
E_{t}^{\ep}
\sp{
\zeta^{2\theta}
\sp{
\vW^{\ep}_{t, \theta}
+
\vZ^{\ep}_{t, \theta}
}
+
\n{\vV}^{h;\f{\zeta}{16}}_{\widetilde{L}_{T^*_0}^\infty(B^{\f d 2}_{2, 1})
\cap
L_{T^*_0}^1(B^{\f d 2}_{2, 1})}
+
\n{\vw}^{h;\f{\zeta}{16}}_{\widetilde{L}_{T^*_0}^\infty(B^{\f d 2}_{2, 1})
\cap
L_{T^*_0}^1(B^{\f d 2}_{2, 1})}
+
\n{\rho}^{h;\f{\zeta}{16}}_{\widetilde{L}_{T^*_0}^\infty(B^{\f d 2}_{2, 1})}
},
\label{6.3.12}
\\[5pt]
&
\int_0^t 
\n{(\rho_\ep-\rho)}_{B^{\f d 2}_{2, 1}}
\n{\vue}_{\underline{B}^{\f d 2+1}_{2, 1}}
\dta
+
\int_0^t 
\n{\rho_\ep}^{h; \f{\zeta}{16}}_{B^{\f d 2}_{2, 1}}
\n{\vue}_{\underline{B}^{\f d 2+1}_{2, 1}}
\dta
\nonumber
\\
&
\leq
C
\int_0^t
\n{\rho}^{h; \zeta}_{\widetilde{L}^\infty_\tau(B^{\f d 2}_{2, 1})}
\n{\vue}_{\underline{B}^{\f d 2+1}_{2, 1}}
\dta
\nonumber
\\
&
\quad
+
C
E^{\ep, \zeta}_{t, 1}
\sp{
\zeta^{1+\theta}
\sp{
\vW^{\ep}_{t, \theta}
+
\vZ^{\ep}_{t, \theta}
}
+
\n{\vV}^{h;\f{\zeta}{16}}_{\widetilde{L}_{T^*_0}^\infty(B^{\f d 2-1}_{2, 1})
\cap
L_{T^*_0}^1(B^{\f d 2+1}_{2, 1})}
+
\n{\vw}^{h;\f{\zeta}{16}}_{\widetilde{L}_{T^*_0}^\infty(B^{\f d 2-1}_{2, 1})
\cap
L_{T^*_0}^1(B^{\f d 2+1}_{2, 1})}
+
\n{\rho}^{h;\f{\zeta}{16}}_{\widetilde{L}_{T^*_0}^\infty(B^{\f d 2}_{2, 1})
}
}.
\label{6.3.14}
\end{align}
\end{Lemma}

We postpone the proof of Lemma \ref{lee6.3} to the end of this subsection. The nonlinear estimates read as follows.

\begin{Proposition}\label{Pro6.4}
Let $T^*_0$, $\vU$ and $\vV$ be as in Theorem \ref{Th6.1} and Theorem \ref{Th7.2}. Let  $0<T<\infty$ with $T\leq T^*_0$, $0<\theta<1$,  $1\leq \zeta<\df{\beta_0}{\ep}$ and $0<\ep \leq \min\{1, \f{R^0}{8|c_{02}|}\}$. 
If $(\vae, \vbe, \vue)$ is a regular solution to \eqref{1.1.1} on $[0, T]$ satisfying
\begin{align*}
\ep 
\sp{
\n{\vae}_{L^\infty_T(L^\infty)}
+
\n{\vbe}_{L^\infty_T(L^\infty)}
}
\leq
\f{\min\{R^0, Q^0\}}{2}.
\end{align*}
Then 
\begin{align*}
&
\ep
\sp{
\n{a_\ep}_{\widetilde{L}^\infty_T(B^{\f d 2}_{2,1})}
+
\n{b_\ep}_{\widetilde{L}^\infty_T(B^{\f d 2}_{2,1})}
}
+
\ep
\n{\vthe}^{h;\f{\beta_0}{\ep}}_{\widetilde{L}^\infty_{T}(B^{\f d 2}_{2, 1})}
+
\f{1}{\ep}
\n{\vthe}^{h;\f{\beta_0}{\ep}}_{L^1_{T}(B^{\f d 2}_{2, 1})}
+
\n{\vthe}^{m;\zeta, \f{\beta_0}{\ep}}_{L^1_{T}(B^{\f d 2+1}_{2, 1})}
+
\n{ \mathcal{P}^\perp\vue}^{h; \zeta}_{\widetilde{L}^\infty_{T}(B^{\f d 2-1}_{2, 1})
\cap
L^1_{T}(B^{\f d 2+1}_{2, 1})
} 
\\
&
\quad
+
\n{\rho_\ep}^{h;\zeta}_{\widetilde{L}^\infty_{T}(B^{\f d 2}_{2, 1})}
+
\n{ \mathcal{P}\vue}^{h; \zeta}_{\widetilde{L}^\infty_{T}(B^{\f d 2-1}_{2, 1})
\cap
L^1_{T}(B^{\f d 2+1}_{2, 1})
} 
\\
&
\leq
C
\text{e}^{
C
[T]
\sp{
1
+
E_{T}^{\ep}
+
\sp{
E_{T}^{\ep}
}^2
}
}
\Bigg[
\n{a_0}^{h;\zeta}_{B^{\f d 2}_{2, 1}}
+
\n{b_0}^{h;\zeta}_{B^{\f d 2}_{2, 1}}
+
\n{\vu_0}^{h; \zeta}_{B^{\f d 2-1}_{2, 1}}
+
\zeta
\ep
\sp{
\sp{
E_{T}^{\ep}
}^2
+
E_{T}^{\ep}
}
\\
&
\quad
+
[T]
E_{T}^{\ep}
\sp{
\zeta^{1+2\theta}
\sp{
\vW^{\ep}_{T, \theta}
+
\vZ^{\ep}_{T, \theta}
}
+
\n{\vV}^{h;\f{\zeta}{16}}_{\widetilde{L}_{T^*_0}^\infty(B^{\f d 2-1}_{2, 1})
\cap
L_{T^*_0}^1(B^{\f d 2+1}_{2, 1})}
+
\n{\vw}^{h;\f{\zeta}{16}}_{\widetilde{L}_{T^*_0}^\infty(B^{\f d 2-1}_{2, 1})
\cap
L_{T^*_0}^1(B^{\f d 2+1}_{2, 1})}
+
\n{\rho}^{h;\f{\zeta}{16}}_{\widetilde{L}_{T^*_0}^\infty(B^{\f d 2}_{2, 1})
}
}
\Bigg].
\end{align*}

\end{Proposition}
{\bf Proof of Proposition \ref{Pro6.4}: }Notice that
\begin{align}\label{6.3.5}
0
<
\f{1}{8C\sp{1 + \n{\rho_\ep}_{\widetilde{L}^\infty_{T^*_0}(B^{\f d 2}_{2,1})}}}
\leq 
\f{1}{8C\sp{1 + \n{\rho_\ep}_{L^\infty_T(B^{\f d 2}_{2,1})}}}.
\end{align}
By taking $s=\f d 2$ and setting $\vv$, $\rho^*$, $F_1$, $F_2$, $G_1$, $G_2$ as above, and substituting them into Propositions \ref{Pro6.1}-\ref{Pro6.3}, and combining with \eqref{6.3.1} and  \eqref{6.3.5}, we obtain
\begin{align}
&
\ep
\sp{
\n{a_\ep}_{\widetilde{L}^\infty_T(B^{\f d 2}_{2,1})}
+
\n{b_\ep}_{\widetilde{L}^\infty_T(B^{\f d 2}_{2,1})}
}
+
\ep
\n{\vthe}^{h;\f{\beta_0}{\ep}}_{\widetilde{L}^\infty_{T}(B^{\f d 2}_{2, 1})}
+
\f{1}{\ep}
\n{\vthe}^{h;\f{\beta_0}{\ep}}_{L^1_{T}(B^{\f d 2}_{2, 1})}
+
\n{\vthe}^{m;\zeta, \f{\beta_0}{\ep}}_{L^1_{t}(B^{\f d 2+1}_{2, 1})}
+
\n{ \mathcal{P}^\perp\vue}^{h; \zeta}_{\widetilde{L}^\infty_{T}(B^{\f d 2-1}_{2, 1})
\cap
L^1_{T}(B^{\f d 2+1}_{2, 1})
} 
\nonumber
\\
&
\quad
+
\n{\rho_\ep}^{h;\zeta}_{\widetilde{L}^\infty_{T}(B^{\f d 2}_{2, 1})}
+
\n{ \mathcal{P}\vue}^{h; \zeta}_{\widetilde{L}^\infty_{T}(B^{\f d 2-1}_{2, 1})
\cap
L^1_{T}(B^{\f d 2+1}_{2, 1})
} 
\nonumber
\\
&
\leq
C
\left(
\ep
\n{\vartheta_0}^{h;\f{\beta_0}{\ep}}_{B^{\f d 2}_{2, 1}}
+
\n{\vartheta_0}^{m;\zeta, \f{\beta_0}{\ep}}_{B^{\f d 2-1}_{2, 1}}
+
\n{\rho_0}^{h;\zeta}_{B^{\f d 2}_{2, 1}}
+
\n{\vu_0}^{h; \zeta}_{B^{\f d 2-1}_{2, 1}}
+
\n{F_1}^{h; \zeta}_{L^1_{t}(B^{\f d 2}_{2, 1})}
+
\n{G_1}^{h; \zeta}_{L^1_{t}(B^{\f d 2-1}_{2, 1})}
+
\n{G_2}^{h; \zeta}_{L^1_{t}(B^{\f d 2-1}_{2, 1})}
+
\ep
\n{F_2}^{h; \f{\beta_0}{\ep}}_{L^1_{t}(B^{\f d 2}_{2, 1})}
\right.
\nonumber
\\
&
\quad
\left.
+
\n{F_2}^{m; \zeta, \f{\beta_0}{\ep}}_{L^1_{t}(B^{\f d 2-1}_{2, 1})}
+
\int_0^t 
\n{\vue}_{L^1_{t}(\underline{B}^{\f d 2+1}_{2, 1})}
\n{\rho_\ep}^{h; \f{\zeta}{4}}_{B^{\f d 2}_{2, 1}}
\dta
+
\int_0^t 
\n{\vue}_{L^1_{t}(\underline{B}^{\f d 2+1}_{2, 1})}
\n{\mathcal{P}\vue}^{h; \f{\zeta}{4}}_{B^{\f d 2-1}_{2, 1}}
\dta
+
\ep
\int_0^t
\n{\vue}_{\underline{B}^{\f d 2+1}_{2, 1}}
\n{\vthe}_{B^{\f d 2}_{2, 1}}
\dta
\right.
\nonumber
\\
&
\quad
+
\int_0^t
\n{\vue}_{B^{\f d 2}_{2, 1}}
\n{\vthe}^{h; \f{\beta_0}{4\ep}}_{B^{\f d 2}_{2, 1}}
\dta
+
\int_0^t
\n{\vue}_{B^{\f d 2}_{2, 1}}
\n{\mathcal{P}^\perp\vue}^{h; \f{\beta_0}{4\ep}}_{B^{\f d 2}_{2, 1}}
\dta
+
\int_0^t
\n{\vue}_{\underline{B}^{\f d 2+1}_{2, 1}}
\n{\vthe}^{m; \f{\zeta}{4}, \f{4\beta_0}{\ep}}_{B^{\f d 2-1}_{2, 1}}
\dta
\nonumber
\\
&
\quad
\left.
+
\int_0^t
\n{\vue}_{\underline{B}^{\f d 2+1}_{2, 1}}
\n{\mathcal{P}^\perp\vue}^{m; \f{\zeta}{4}, \f{4\beta_0}{\ep}}_{B^{\f d 2-1}_{2, 1}}
\dta
+
\zeta\ep E^{\ep}_{t}
+
\n{\mathcal{P}\vue}^{m; \f{\zeta}{4}, \zeta}_{L^1_{t}(B^{\f d 2+1}_{2, 1})}
+
\n{\mathcal{P}^\perp\vue}^{m; \f{\zeta}{4}, \zeta}_{L^1_{t}(B^{\f d 2+1}_{2, 1})}
\right).
\label{6.3.6}
\end{align}
We proceed to estimate each term on the right-hand side of \eqref{6.3.6}.  To this end, we shall repeatedly invoke \eqref{minski}, \eqref{minski2}, together with the following elementary inequalities, which hold for all $0<\zeta<\eta<\infty$, $1\leq p, r\leq \infty$, $-\infty<s<\infty$, $0<s_1<\infty$,
\begin{align*}
&
\n{f}^{l;\zeta}_{B^{s}_{p, r}}
\leq
\n{f}_{B^{s}_{p, r}},
\quad  
\n{f}^{h;\zeta}_{B^{s}_{p, r}}
\leq 
\n{f}_{B^{s}_{p, r}},
\quad 
\n{f}^{m;\zeta, \eta}_{B^{s}_{p, r}}
\leq
\min\{\n{f}^{l;\eta}_{\underline{B}^{s}_{p, r}}, \n{f}^{h;\zeta}_{B^{s}_{p, r}}\},
\\
&
\n{f}^{l;\zeta}_{B^{s+s_1}_{p, r}}
\leq
\zeta^{s_1}
\n{f}^{l;\zeta}_{B^{s}_{p, r}},
\quad 
\n{f}^{h;\zeta}_{B^{s-s_1}_{p, r}}
\leq
\zeta^{-s_1}
\n{f}^{h;\zeta}_{B^{s}_{p, r}}.
\end{align*}

\noindent $\bullet$ \textbf{Estimates for the initial data.}  

Seeing that
\begin{align*}
&
\ep
\n{\vartheta_0}^{h;\f{\beta_0}{\ep}}_{B^{\f d 2}_{2, 1}}
\les
\ep
\sp{
\n{a_0}^{h;\f{\beta_0}{\ep}}_{B^{\f d 2}_{2, 1}}
+
\n{b_0}^{h;\f{\beta_0}{\ep}}_{B^{\f d 2}_{2, 1}}
}
,
\ \
\n{\vartheta_0}^{m;\zeta, \f{\beta_0}{\ep}}_{B^{\f d 2-1}_{2, 1}}
\les
\sp{
\n{a_0}^{m;\zeta, \f{\beta_0}{\ep}}_{B^{\f d 2-1}_{2, 1}}
+
\n{b_0}^{m;\zeta, \f{\beta_0}{\ep}}_{B^{\f d 2-1}_{2, 1}}
}
,
\ \
\n{\rho_0}^{h;\zeta}_{B^{\f d 2}_{2, 1}}
\les
\sp{
\n{a_0}^{h;\zeta}_{B^{\f d 2}_{2, 1}}
+
\n{b_0}^{h;\zeta}_{B^{\f d 2}_{2, 1}}
},
\end{align*}
which implies
\begin{align}\label{6.3.15}
C
\sp{
\ep
\n{\vartheta_0}^{h;\f{\beta_0}{\ep}}_{B^{\f d 2}_{2, 1}}
+
\n{\vartheta_0}^{m;\zeta, \f{\beta_0}{\ep}}_{B^{\f d 2-1}_{2, 1}}
+
\n{\rho_0}^{h;\zeta}_{B^{\f d 2}_{2, 1}}
+
\n{\vu_0}^{h; \zeta}_{B^{\f d 2-1}_{2, 1}}
}
\leq
C^2
\sp{
\n{a_0}^{h;\zeta}_{B^{\f d 2}_{2, 1}}
+
\n{b_0}^{h;\zeta}_{B^{\f d 2}_{2, 1}}
+
\n{\vu_0}^{h; \zeta}_{B^{\f d 2-1}_{2, 1}}
}.
\end{align}

\noindent $\bullet$ \textbf{Estimates for the integral terms.}  

From Lemma \ref{lee6.3}, we get
\begin{align}\label{6.3.16}
&
C
\left(
\int_0^t 
\n{\vue}_{L^1_{t}(\underline{B}^{\f d 2+1}_{2, 1})}
\n{\rho_\ep}^{h; \f{\zeta}{4}}_{B^{\f d 2}_{2, 1}}
\dta
+
\int_0^t 
\n{\vue}_{L^1_{t}(\underline{B}^{\f d 2+1}_{2, 1})}
\n{\mathcal{P}\vue}^{h; \f{\zeta}{4}}_{B^{\f d 2-1}_{2, 1}}
\dta
+
\ep
\int_0^t
\n{\vue}_{\underline{B}^{\f d 2+1}_{2, 1}}
\n{\vthe}_{B^{\f d 2}_{2, 1}}
\dta
+
\int_0^t
\n{\vue}_{B^{\f d 2}_{2, 1}}
\n{\vthe}^{h; \f{\beta_0}{4\ep}}_{B^{\f d 2}_{2, 1}}
\dta
\right.
\nonumber
\\
&
\left.
\quad
+
\int_0^t
\n{\vue}_{B^{\f d 2}_{2, 1}}
\n{\mathcal{P}^\perp\vue}^{h; \f{\beta_0}{4\ep}}_{B^{\f d 2}_{2, 1}}
\dta
+
\int_0^t
\n{\vue}_{\underline{B}^{\f d 2+1}_{2, 1}}
\n{\vthe}^{m; \f{\zeta}{4}, \f{4\beta_0}{\ep}}_{B^{\f d 2-1}_{2, 1}}
\dta
+
\int_0^t
\n{\vue}_{\underline{B}^{\f d 2+1}_{2, 1}}
\n{\mathcal{P}^\perp\vue}^{m; \f{\zeta}{4}, \f{4\beta_0}{\ep}}_{B^{\f d 2-1}_{2, 1}}
\dta
\right)
\nonumber
\\
&
\leq
C
\int_0^t
\n{\vue}_{\underline{B}^{\f d 2+1}_{2, 1}}
\left(
\n{\rho_\ep}^{h; \zeta}_{\widetilde{L}^\infty_\tau(B^{\f d 2}_{2, 1})}
+
\n{\mathcal{P}^\perp\vue}^{h; \zeta}_{\widetilde{L}^\infty_\tau(B^{\f d 2-1}_{2, 1})}
+
\n{\mathcal{P}\vue}^{h; \zeta}_{\widetilde{L}^\infty_\tau(B^{\f d 2-1}_{2, 1})}
+
\n{\vthe}^{m; \zeta, \f{\beta_0}{\ep}}_{\widetilde{L}^\infty_\tau(B^{\f d 2-1}_{2, 1})}
\right)
\dta
\nonumber
\\
&
\quad
+
C
\int_0^t
\n{\vue}_{\underline{B}^{\f d 2+1}_{2, 1}}
\sp{
\ep
\n{\vthe}_{B^{\f d 2}_{2, 1}}
+
\n{\vthe}^{h; \f{\beta_0}{\ep}}_{B^{\f d 2}_{2, 1}}
}
\dta
+
C
\int_0^t
\n{\vue}_{B^{\f d 2}_{2, 1}}
\sp{
\n{\mathcal{P}^\perp \vue}^{h; \f{\zeta}{16}}_{B^{\f d 2}_{2, 1}}
+
\n{\vthe}^{h; \f{\zeta}{16}}_{B^{\f d 2}_{2, 1}}
}
\dta
\nonumber
\\
&
\quad
+
C
\int_0^t
\n{\vue}_{\underline{B}^{\f d 2+1}_{2, 1}}
\sp{
\n{\mathcal{P}^\perp\vue}^{m; \f{\zeta}{16}, \zeta}_{B^{\f d 2-1}_{2, 1}}
+
\n{\mathcal{P}^\perp\vue}^{m; \f{\zeta}{16}, \zeta}_{B^{\f d 2-1}_{2, 1}}
+
\n{\vthe}^{m; \f{\zeta}{16}, \zeta}_{B^{\f d 2-1}_{2, 1}}
+
\n{\rho_\ep}^{m; \f{\zeta}{16}, \zeta}_{B^{\f d 2}_{2, 1}}
}
\dta
\nonumber
\\
&
\leq
2
C
\sp{
\f{3}{2\delta}
+
C
+
1
}
\int_0^t
\sp{
\n{\vthe}_{B^{\f d 2}_{2, 1}}
+
\n{\rho_\ep}_{B^{\f d 2}_{2, 1}}
+
\n{\vue}_{\underline{B}^{\f d 2+1}_{2, 1}}
+
\n{\vue}^2_{B^{\f d 2}_{2, 1}}
+
\n{\vthe}^2_{B^{\f d 2}_{2, 1}}
+
\n{\rho_\ep}^2_{B^{\f d 2}_{2, 1}}
+
\f{1}{\ep}
\n{\vthe}^{h;\f{\beta_0}{\ep}}_{B^{\f d 2}_{2, 1}}
}
\nonumber
\\
&
\quad
\cdot
\left(
\n{\rho_\ep}^{h; \zeta}_{\widetilde{L}^\infty_\tau(B^{\f d 2}_{2, 1})}
+
\n{\mathcal{P}^\perp\vue}^{h; \zeta}_{\widetilde{L}^\infty_\tau(B^{\f d 2-1}_{2, 1})}
+
\n{\mathcal{P}\vue}^{h; \zeta}_{\widetilde{L}^\infty_\tau(B^{\f d 2-1}_{2, 1})}
+
\n{\vthe}^{m; \zeta, \f{\beta_0}{\ep}}_{\widetilde{L}^\infty_\tau(B^{\f d 2-1}_{2, 1})}
+
\ep
\sp{
\n{a_\ep}_{\widetilde{L}^\infty_\tau(B^{\f d 2}_{2,1})}
+
\n{b_\ep}_{\widetilde{L}^\infty_\tau(B^{\f d 2}_{2,1})}
}
\right)
\dta
\nonumber
\\
&
\quad
+
2C^2[t]
E^{\ep, \zeta}_{t, 1}
\sp{
\zeta^{1+\theta}
\sp{
\vW^{\ep}_{t, \theta}
+
\vZ^{\ep}_{t, \theta}
}
+
\n{\vV}^{h;\f{\zeta}{16}}_{\widetilde{L}_{T^*_0}^\infty(B^{\f d 2-1}_{2, 1})
\cap
L_{T^*_0}^1(B^{\f d 2+1}_{2, 1})}
+
\n{\vw}^{h;\f{\zeta}{16}}_{\widetilde{L}_{T^*_0}^\infty(B^{\f d 2-1}_{2, 1})
\cap
L_{T^*_0}^1(B^{\f d 2+1}_{2, 1})}
+
\n{\rho}^{h;\f{\zeta}{16}}_{\widetilde{L}_{T^*_0}^\infty(B^{\f d 2}_{2, 1})
}
}
\nonumber
\\
&
\quad
+
C
\zeta
\ep
\sp{
E_{t}^{\ep}
}^2
+
2
C
\delta
\sp{
\n{\vthe}^{m;\zeta, \f{\beta_0}{\zeta}}_{L^1_t(B^{\f d 2+1}_{2, 1})}
+
\n{\mathcal{P}^\perp\vue}^{h;\zeta}_{L^1_t(B^{\f d 2+1}_{2, 1})}
+
\n{\mathcal{P}\vue}^{h;\zeta}_{L^1_t(B^{\f d 2+1}_{2, 1})}
},
\end{align}
where $\delta$ will be chosen sufficiently small so that this term can be absorbed by the left-hand side of \eqref{6.3.6}. 

\noindent $\bullet$ \textbf{Estimates for $F_1$} 

We infer from Lemma \ref{lee6.1} that 
\begin{align*}
\n{F_1}^{h;\zeta}_{L^1_{t}(B^{\f d 2}_{2, 1})}
\leq
C
\sp{
\int_0^t 
\n{(\rho_\ep-\rho)}_{B^{\f d 2}_{2, 1}}
\n{\vue}_{\underline{B}^{\f d 2+1}_{2, 1}}
\dta
+
\int_0^t 
\n{\rho_\ep}^{h; \f{\zeta}{16}}_{B^{\f d 2}_{2, 1}}
\n{\vue}_{\underline{B}^{\f d 2+1}_{2, 1}}
\dta
},
\end{align*}
which combined with Lemma \ref{lee6.3} yields
\begin{align}\label{6.3.17}
&
 C
\n{F_1}^{h;\zeta}_{L^1_{t}(B^{\f d 2}_{2, 1})}  
\nonumber
\\
&
\leq
C^3
\int_0^t
\n{\rho}^{h; \zeta}_{\widetilde{L}^\infty_\tau(B^{\f d 2}_{2, 1})}
\n{\vue}_{\underline{B}^{\f d 2+1}_{2, 1}}
\dta
\nonumber
\\
&
\quad
+
C^3
E^{\ep, \zeta}_{t, 1}
\sp{
\zeta^{1+\theta}
\sp{
\vW^{\ep}_{t, \theta}
+
\vZ^{\ep}_{t, \theta}
}
+
\n{\vV}^{h;\f{\zeta}{16}}_{\widetilde{L}_{T^*_0}^\infty(B^{\f d 2-1}_{2, 1})
\cap
L_{T^*_0}^1(B^{\f d 2+1}_{2, 1})}
+
\n{\vw}^{h;\f{\zeta}{16}}_{\widetilde{L}_{T^*_0}^\infty(B^{\f d 2-1}_{2, 1})
\cap
L_{T^*_0}^1(B^{\f d 2+1}_{2, 1})}
+
\n{\rho}^{h;\f{\zeta}{16}}_{\widetilde{L}_{T^*_0}^\infty(B^{\f d 2}_{2, 1})
}
}. 
\end{align}

\noindent $\bullet$ \textbf{Estimates for $F_2$} 

On the one hand, we see from Lemma \ref{lee6.1} that 
\begin{align*}
\ep
\n{F_2}^{h; \f{\beta_0}{\ep}}_{L^1_{t}(B^{\f d 2}_{2, 1})}
\leq
C
\ep
\int_0^t
\n{\vue}_{\underline{B}^{\f d 2+1}_{2, 1}}
\n{\vthe}_{B^{\f d 2}_{2, 1}}
\dta.
\end{align*}
On the other hand, Lemma \ref{lee6.2} ensures 
\begin{align*}
\n{F_2}^{m; \zeta, \f{\beta_0}{\ep}}_{L^1_{t}(B^{\f d 2-1}_{2, 1})}
\leq
C
\sp{
\int_0^t
\n{\vue}_{B^{\f d 2}_{2, 1}}
\n{\vthe}^{h; \f{\zeta}{16}}_{B^{\f d 2}_{2, 1}}
\dta
+
\int_0^t
\n{\vthe}_{B^{\f d 2}_{2, 1}}
\n{\vue}^{h; \f{\zeta}{16}}_{B^{\f d 2}_{2, 1}}
\dta
}.
\end{align*}
With the help of Lemma \ref{lee6.3}, we conclude 
\begin{align}\label{6.3.18}
&
C
\sp{
\ep
\n{F_2}^{h; \f{\beta_0}{\ep}}_{L^1_{t}(B^{\f d 2}_{2, 1})}
+
\n{F_2}^{m; \zeta, \f{\beta_0}{\ep}}_{L^1_{t}(B^{\f d 2-1}_{2, 1})}
}
\nonumber
\\
&
\leq
2
C^2
\sp{
\f{3}{2\delta}
+
C
+
1
}
\int_0^t
\sp{
\n{\vthe}_{B^{\f d 2}_{2, 1}}
+
\n{\rho_\ep}_{B^{\f d 2}_{2, 1}}
+
\n{\vue}_{\underline{B}^{\f d 2+1}_{2, 1}}
+
\n{\vue}^2_{B^{\f d 2}_{2, 1}}
+
\n{\vthe}^2_{B^{\f d 2}_{2, 1}}
+
\n{\rho_\ep}^2_{B^{\f d 2}_{2, 1}}
+
\f{1}{\ep}
\n{\vthe}^{h;\f{\beta_0}{\ep}}_{B^{\f d 2}_{2, 1}}
}
\nonumber
\\
&
\quad
\cdot
\left(
\n{\rho_\ep}^{h; \zeta}_{\widetilde{L}^\infty_\tau(B^{\f d 2}_{2, 1})}
+
\n{\mathcal{P}^\perp\vue}^{h; \zeta}_{\widetilde{L}^\infty_\tau(B^{\f d 2-1}_{2, 1})}
+
\n{\mathcal{P}\vue}^{h; \zeta}_{\widetilde{L}^\infty_\tau(B^{\f d 2-1}_{2, 1})}
+
\n{\vthe}^{m; \zeta, \f{\beta_0}{\ep}}_{\widetilde{L}^\infty_\tau(B^{\f d 2-1}_{2, 1})}
+
\ep
\sp{
\n{a_\ep}_{\widetilde{L}^\infty_\tau(B^{\f d 2}_{2,1})}
+
\n{b_\ep}_{\widetilde{L}^\infty_\tau(B^{\f d 2}_{2,1})}
}
\right)
\dta
\nonumber
\\
&
\quad
+
C^3
[t]
E_{t}^{\ep}
\sp{
\zeta^{2\theta}
\sp{
\vW^{\ep}_{t, \theta}
+
\vZ^{\ep}_{t, \theta}
}
+
\n{\vV}^{h;\f{\zeta}{16}}_{\widetilde{L}_{T^*_0}^\infty(B^{\f d 2}_{2, 1})
\cap
L_{T^*_0}^1(B^{\f d 2}_{2, 1})}
+
\n{\vw}^{h;\f{\zeta}{16}}_{\widetilde{L}_{T^*_0}^\infty(B^{\f d 2}_{2, 1})
\cap
L_{T^*_0}^1(B^{\f d 2}_{2, 1})}
+
\n{\rho}^{h;\f{\zeta}{16}}_{\widetilde{L}_{T^*_0}^\infty(B^{\f d 2}_{2, 1})
}
}
\nonumber
\\
&
\quad
+
C^2
\zeta
\ep
\sp{
E_{t}^{\ep}
}^2
+
2
C^2
\delta
\sp{
\n{\vthe}^{m;\zeta, \f{\beta_0}{\zeta}}_{L^1_t(B^{\f d 2+1}_{2, 1})}
+
\n{\mathcal{P}^\perp\vue}^{h;\zeta}_{L^1_t(B^{\f d 2+1}_{2, 1})}
+
\n{\mathcal{P}\vue}^{h;\zeta}_{L^1_t(B^{\f d 2+1}_{2, 1})}
}.
\end{align}

\noindent $\bullet$ \textbf{Estimates for $G_1$} 

We have already carried out a detailed analysis of the pressure $r_\ep$, see subsection \ref{pressure}. In particular, 
\begin{align*}
\n{r^3_\ep}_{L^1_t(B^{\f d 2-1}_{2, 1})}   
&
\les
\int_0^t
\ep
\sp{
\n{\vae}_{B^{\f d 2}_{2, 1}}
+
\n{\vbe}_{B^{\f d 2}_{2, 1}}
}
\sp{
\n{\vue}_{\underline{B}^{\f d 2+1}_{2, 1}}
+
\n{\vae}^2_{B^{\f d 2}_{2, 1}}
+
\n{\vbe}^2_{B^{\f d 2}_{2, 1}}
}
\dta
\\
&
\les
\int_0^t
\ep
\sp{
\n{\vae}_{\widetilde{L}^\infty_\tau(B^{\f d 2}_{2, 1})}
+
\n{\vbe}_{L^\infty_\tau(B^{\f d 2}_{2, 1})}
}
\sp{
\n{\vue}_{\underline{B}^{\f d 2+1}_{2, 1}}
+
\n{\vthe}^2_{B^{\f d 2}_{2, 1}}
+
\n{\rho_\ep}^2_{B^{\f d 2}_{2, 1}}
+
|\widehat{(a_0)}_0|^2
+
|\widehat{(b_0)}_0|^2
}
\dta . 
\end{align*}
Hence
\begin{align}\label{6.3.21}
& \n{\mathcal{P}r^3_\ep}^{h;\zeta}_{L^1_t(B^{\f d 2-1}_{2, 1})}
\leq C 
\n{r^3_\ep}_{L^1_t(B^{\f d 2-1}_{2, 1})}
\nonumber
\\
& \quad 
\leq
C
\int_0^t
\ep
\sp{
\n{\vae}_{\widetilde{L}^\infty_\tau(B^{\f d 2}_{2, 1})}
+
\n{\vbe}_{L^\infty_\tau(B^{\f d 2}_{2, 1})}
}
\sp{
\n{\vue}_{\underline{B}^{\f d 2+1}_{2, 1}}
+
\n{\vthe}^2_{B^{\f d 2}_{2, 1}}
+
\n{\rho_\ep}^2_{B^{\f d 2}_{2, 1}}
+
|\widehat{(a_0)}_0|^2
+
|\widehat{(b_0)}_0|^2
}
\dta.
\end{align}
In view of Lemma \ref{lee6.1}, 
\begin{align*}
&
\n{G_1-\mathcal{P}r^3_\ep}^{h; \zeta}_{L^1_t(B^{\f d 2-1}_{2, 1})}
\\
&  \quad 
\leq
C
\left[ 
\int_0^t
\n{\vue}_{B^{\f d 2}_{2, 1}}
\sp{
\n{\mathcal{P}^\perp\vue}^{h; \f{\zeta}{16}}_{B^{\f d 2}_{2, 1}}
+
\n{\mathcal{P}\vue}^{h; \f{\zeta}{16}}_{B^{\f d 2}_{2, 1}}
}
\dta
+
\int_0^t
\n{\vthe}_{B^{\f d 2}_{2, 1}}
\n{\rho_\ep}^{h; \f{\zeta}{16}}_{B^{\f d 2}_{2, 1}} \dta 
+
\int_0^t
\n{\rho_\ep}_{B^{\f d 2}_{2, 1}}
\n{\vthe}^{h; \f{\zeta}{16}}_{B^{\f d 2}_{2, 1}}
\dta
\right]. 
\end{align*}
By applying Lemma \ref{lee6.3},  we obtain
\begin{align}\label{6.3.20}
&
C
\n{G_1}^{h;\zeta}_{L^1_t(B^{\f d 2-1}_{2, 1})}
\nonumber
\\
& \quad 
\leq
2
C^2
\sp{
\f{3}{2\delta}
+
C
+
1
}
\int_0^t
\Bigg[
\n{\vthe}_{B^{\f d 2}_{2, 1}}
+
\n{\rho_\ep}_{B^{\f d 2}_{2, 1}}
+
\n{\vue}_{\underline{B}^{\f d 2+1}_{2, 1}}
+
\n{\vue}^2_{B^{\f d 2}_{2, 1}}
+
\n{\vthe}^2_{B^{\f d 2}_{2, 1}}
\nonumber
\\
&
\qquad
+
\n{\rho_\ep}^2_{B^{\f d 2}_{2, 1}}
+
\f{1}{\ep}
\n{\vthe}^{h;\f{\beta_0}{\ep}}_{B^{\f d 2}_{2, 1}}
+
|\widehat{(a_0)}_0|^2
+
|\widehat{(b_0)}_0|^2
\Bigg]
\Bigg[
\n{\rho_\ep}^{h; \zeta}_{\widetilde{L}^\infty_\tau(B^{\f d 2}_{2, 1})}
+
\n{\mathcal{P}^\perp\vue}^{h; \zeta}_{\widetilde{L}^\infty_\tau(B^{\f d 2-1}_{2, 1})}
+
\n{\mathcal{P}\vue}^{h; \zeta}_{\widetilde{L}^\infty_\tau(B^{\f d 2-1}_{2, 1})}
\nonumber
\\
&
\qquad
+
\n{\vthe}^{m; \zeta, \f{\beta_0}{\ep}}_{\widetilde{L}^\infty_\tau(B^{\f d 2-1}_{2, 1})}
+
\ep
\sp{
\n{a_\ep}_{\widetilde{L}^\infty_\tau(B^{\f d 2}_{2,1})}
+
\n{b_\ep}_{\widetilde{L}^\infty_\tau(B^{\f d 2}_{2,1})}
}
\Bigg]
\dta
\nonumber
\\
&
\qquad
+
C^3
[t]
E_{t}^{\ep}
\sp{
\zeta^{2\theta}
\sp{
\vW^{\ep}_{t, \theta}
+
\vZ^{\ep}_{t, \theta}
}
+
\n{\vV}^{h;\f{\zeta}{16}}_{\widetilde{L}_{T^*_0}^\infty(B^{\f d 2}_{2, 1})
\cap
L_{T^*_0}^1(B^{\f d 2}_{2, 1})}
+
\n{\vw}^{h;\f{\zeta}{16}}_{\widetilde{L}_{T^*_0}^\infty(B^{\f d 2}_{2, 1})
\cap
L_{T^*_0}^1(B^{\f d 2}_{2, 1})}
+
\n{\rho}^{h;\f{\zeta}{16}}_{\widetilde{L}_{T^*_0}^\infty(B^{\f d 2}_{2, 1})
}
}
\nonumber
\\
&
\qquad
+
C^2
\zeta
\ep
\sp{
E_{t}^{\ep}
}^2
+
2
C^2
\delta
\sp{
\n{\vthe}^{m;\zeta, \f{\beta_0}{\zeta}}_{L^1_t(B^{\f d 2+1}_{2, 1})}
+
\n{\mathcal{P}^\perp\vue}^{h;\zeta}_{L^1_t(B^{\f d 2+1}_{2, 1})}
+
\n{\mathcal{P}\vue}^{h;\zeta}_{L^1_t(B^{\f d 2+1}_{2, 1})}
}.
\end{align}

\noindent $\bullet$ \textbf{Estimates for $G_2$} 

Similar to \eqref{6.3.21}, it holds 
\begin{align*}
\n{\mathcal{P}^\perp r^3_\ep}^{h;\zeta}_{L^1_t(B^{\f d 2-1}_{2, 1})}
\leq
C
\int_0^t
\ep
\sp{
\n{\vae}_{\widetilde{L}^\infty_\tau(B^{\f d 2}_{2, 1})}
+
\n{\vbe}_{L^\infty_\tau(B^{\f d 2}_{2, 1})}
}
\sp{
\n{\vue}_{\underline{B}^{\f d 2+1}_{2, 1}}
+
\n{\vthe}^2_{B^{\f d 2}_{2, 1}}
+
\n{\rho_\ep}^2_{B^{\f d 2}_{2, 1}}
+
|\widehat{(a_0)}_0|^2
+
|\widehat{(b_0)}_0|^2
}
\dta.
\end{align*}
Due to Lemma \ref{lee6.1}, 
\begin{align*}
&
\n{G_2-\mathcal{P}^\perp r^3_\ep}^{h; \zeta}_{L^1_t(B^{\f d 2-1}_{2, 1})}
\\
&\quad 
\leq
C
\left[
\int_0^t
\n{\vue}_{B^{\f d 2}_{2, 1}}
\sp{
\n{\mathcal{P}^\perp\vue}^{h; \f{\zeta}{16}}_{B^{\f d 2}_{2, 1}}
+
\n{\mathcal{P}\vue}^{h; \f{\zeta}{16}}_{B^{\f d 2}_{2, 1}}
}
\dta
+
\int_0^t
\sp{
\n{\vthe}_{B^{\f d 2}_{2, 1}}
+
\n{\rho_\ep}_{B^{\f d 2}_{2, 1}}
}
\sp{
\n{\rho_\ep}^{h; \f{\zeta}{16}}_{B^{\f d 2}_{2, 1}}
+
\n{\vthe}^{h; \f{\zeta}{16}}_{B^{\f d 2}_{2, 1}}
}
\dta
\right]
\\
&
\qquad
+
C
\int_0^t
\sp{
\n{\vthe}^{h;\zeta}_{B^{\f d 2}_{2, 1}}
+
\n{\rho_\ep}^{h;\zeta}_{B^{\f d 2}_{2, 1}}
}
\dta.
\end{align*}
We further deduce from Lemma \ref{lee6.3} that 
\begin{align}\label{6.3.23}
&
C
\n{G_2}^{h; \zeta}_{L^1_t(B^{\f d 2-1}_{2, 1})} 
\nonumber
\\
&
\leq
3
C^2
\sp{
\f{3}{2\delta}
+
C
+
1
}
\int_0^t
\left(
1
+
\n{\vthe}_{B^{\f d 2}_{2, 1}}
+
\n{\rho_\ep}_{B^{\f d 2}_{2, 1}}
+
\n{\vue}_{\underline{B}^{\f d 2+1}_{2, 1}}
+
\n{\vue}^2_{B^{\f d 2}_{2, 1}}
+
\n{\vthe}^2_{B^{\f d 2}_{2, 1}}
\right.
\nonumber
\\
&
\quad
\left.
+
\n{\rho_\ep}^2_{B^{\f d 2}_{2, 1}}
+
\f{1}{\ep}
\n{\vthe}^{h;\f{\beta_0}{\ep}}_{B^{\f d 2}_{2, 1}}
+
|\widehat{(a_0)}_0|^2
+
|\widehat{(b_0)}_0|^2
\right)
\left(
\n{\rho_\ep}^{h; \zeta}_{\widetilde{L}^\infty_\tau(B^{\f d 2}_{2, 1})}
+
\n{\mathcal{P}^\perp\vue}^{h; \zeta}_{\widetilde{L}^\infty_\tau(B^{\f d 2-1}_{2, 1})}
+
\n{\mathcal{P}\vue}^{h; \zeta}_{\widetilde{L}^\infty_\tau(B^{\f d 2-1}_{2, 1})}
\right.
\nonumber
\\
&
\left.
\quad
+
\n{\vthe}^{m; \zeta, \f{\beta_0}{\ep}}_{\widetilde{L}^\infty_\tau(B^{\f d 2-1}_{2, 1})}
+
\ep
\sp{
\n{a_\ep}_{\widetilde{L}^\infty_\tau(B^{\f d 2}_{2,1})}
+
\n{b_\ep}_{\widetilde{L}^\infty_\tau(B^{\f d 2}_{2,1})}
}
\right)
\dta
\nonumber
\\
&
\quad
+
C^3
[t]
E_{t}^{\ep}
\sp{
\zeta^{2\theta}
\sp{
\vW^{\ep}_{t, \theta}
+
\vZ^{\ep}_{t, \theta}
}
+
\n{\vV}^{h;\f{\zeta}{16}}_{\widetilde{L}_{T^*_0}^\infty(B^{\f d 2}_{2, 1})
\cap
L_{T^*_0}^1(B^{\f d 2}_{2, 1})}
+
\n{\vw}^{h;\f{\zeta}{16}}_{\widetilde{L}_{T^*_0}^\infty(B^{\f d 2}_{2, 1})
\cap
L_{T^*_0}^1(B^{\f d 2}_{2, 1})}
+
\n{\rho}^{h;\f{\zeta}{16}}_{\widetilde{L}_{T^*_0}^\infty(B^{\f d 2}_{2, 1})
}
}
\nonumber
\\
&
\quad
+
C^2
\zeta
\ep
\sp{
\sp{
E_{t}^{\ep}
}^2
+
E_{t}^{\ep}
}
+
\f{5C^2}{2}
\delta
\sp{
\n{\vthe}^{m;\zeta, \f{\beta_0}{\zeta}}_{L^1_t(B^{\f d 2+1}_{2, 1})}
+
\n{\mathcal{P}^\perp\vue}^{h;\zeta}_{L^1_t(B^{\f d 2+1}_{2, 1})}
+
\n{\mathcal{P}\vue}^{h;\zeta}_{L^1_t(B^{\f d 2+1}_{2, 1})}
}.
\end{align}

\medskip
\medskip

Inserting \eqref{6.3.4}, \eqref{6.3.15}, \eqref{6.3.16}, \eqref{6.3.17}, \eqref{6.3.18}, 
\eqref{6.3.20} and \eqref{6.3.23} into \eqref{6.3.6} yields
\begin{align}\label{6.3.24}
&
\ep
\sp{
\n{a_\ep}_{\widetilde{L}^\infty_T(B^{\f d 2}_{2,1})}
+
\n{b_\ep}_{\widetilde{L}^\infty_T(B^{\f d 2}_{2,1})}
}
+
\ep
\n{\vthe}^{h;\f{\beta_0}{\ep}}_{\widetilde{L}^\infty_{T}(B^{\f d 2}_{2, 1})}
+
\f{1}{\ep}
\n{\vthe}^{h;\f{\beta_0}{\ep}}_{L^1_{T}(B^{\f d 2}_{2, 1})}
+
\n{\vthe}^{m;\zeta, \f{\beta_0}{\ep}}_{L^1_{t}(B^{s+1}_{2, 1})}
+
\n{ \mathcal{P}^\perp\vue}^{h; \zeta}_{\widetilde{L}^\infty_{T}(B^{\f d 2-1}_{2, 1})
\cap
L^1_{T}(B^{\f d 2+1}_{2, 1})
} 
\nonumber
\\
&
\quad
+
\n{\rho_\ep}^{h;\zeta}_{\widetilde{L}^\infty_{T}(B^{\f d 2}_{2, 1})}
+
\n{ \mathcal{P}\vue}^{h; \zeta}_{\widetilde{L}^\infty_{T}(B^{\f d 2-1}_{2, 1})
\cap
L^1_{T}(B^{\f d 2+1}_{2, 1})
} 
\nonumber
\\
&
\leq
C^2
\sp{
\n{a_0}^{h;\zeta}_{B^{\f d 2}_{2, 1}}
+
\n{b_0}^{h;\zeta}_{B^{\f d 2}_{2, 1}}
+
\n{\vu_0}^{h; \zeta}_{B^{\f d 2-1}_{2, 1}}
}
\nonumber
\\
&
+
\sp{
8
C^2
+
2
C
}
\sp{
\f{3}{2\delta}
+
C
+
1
}
\int_0^t
\left(
1
+
\n{\vthe}_{B^{\f d 2}_{2, 1}}
+
\n{\rho_\ep}_{B^{\f d 2}_{2, 1}}
+
\n{\vue}_{\underline{B}^{\f d 2+1}_{2, 1}}
+
\n{\vue}^2_{B^{\f d 2}_{2, 1}}
+
\n{\vthe}^2_{B^{\f d 2}_{2, 1}}
\right.
\nonumber
\\
&
\quad
\left.
+
\n{\rho_\ep}^2_{B^{\f d 2}_{2, 1}}
+
\f{1}{\ep}
\n{\vthe}^{h;\f{\beta_0}{\ep}}_{B^{\f d 2}_{2, 1}}
+
|\widehat{(a_0)}_0|^2
+
|\widehat{(b_0)}_0|^2
\right)
\left(
\n{\rho_\ep}^{h; \zeta}_{\widetilde{L}^\infty_\tau(B^{\f d 2}_{2, 1})}
+
\n{\mathcal{P}^\perp\vue}^{h; \zeta}_{\widetilde{L}^\infty_\tau(B^{\f d 2-1}_{2, 1})}
+
\n{\mathcal{P}\vue}^{h; \zeta}_{\widetilde{L}^\infty_\tau(B^{\f d 2-1}_{2, 1})}
\right.
\nonumber
\\
&
\left.
\quad
+
\n{\vthe}^{m; \zeta, \f{\beta_0}{\ep}}_{\widetilde{L}^\infty_\tau(B^{\f d 2-1}_{2, 1})}
+
\ep
\sp{
\n{a_\ep}_{\widetilde{L}^\infty_\tau(B^{\f d 2}_{2,1})}
+
\n{b_\ep}_{\widetilde{L}^\infty_\tau(B^{\f d 2}_{2,1})}
}
\right)
\dta
\nonumber
\\
&
\quad
+
\sp{3C^2+4C^3}
[t]
E^{\ep, \zeta}_{t, 1}
\sp{
\zeta^{1+2\theta}
\sp{
\vW^{\ep}_{t, \theta}
+
\vZ^{\ep}_{t, \theta}
}
+
\n{\vV}^{h;\f{\zeta}{16}}_{\widetilde{L}_{T^*_0}^\infty(B^{\f d 2-1}_{2, 1})
\cap
L_{T^*_0}^1(B^{\f d 2+1}_{2, 1})}
+
\n{\vw}^{h;\f{\zeta}{16}}_{\widetilde{L}_{T^*_0}^\infty(B^{\f d 2-1}_{2, 1})
\cap
L_{T^*_0}^1(B^{\f d 2+1}_{2, 1})}
+
\n{\rho}^{h;\f{\zeta}{16}}_{\widetilde{L}_{T^*_0}^\infty(B^{\f d 2}_{2, 1})
}
}
\nonumber
\\
&
\quad
+
\sp{
4
C^2
+
C
}
\zeta
\ep
\sp{
\sp{
E_{t}^{\ep}
}^2
+
E_{t}^{\ep}
}
+
\sp{
2
C
+
\f{13C^2}{2}
}
\delta
\sp{
\n{\vthe}^{m;\zeta, \f{\beta_0}{\zeta}}_{L^1_t(B^{\f d 2+1}_{2, 1})}
+
\n{\mathcal{P}^\perp\vue}^{h;\zeta}_{L^1_t(B^{\f d 2+1}_{2, 1})}
+
\n{\mathcal{P}\vue}^{h;\zeta}_{L^1_t(B^{\f d 2+1}_{2, 1})}
}.
\end{align}
By choosing $\delta=(4C+13 C^2)^{-1}$, the last term above is absorbed into the left-hand side of \eqref{6.3.24}. Hence, we arrive at 
\begin{align*}
&
\ep
\sp{
\n{a_\ep}_{\widetilde{L}^\infty_T(B^{\f d 2}_{2,1})}
+
\n{b_\ep}_{\widetilde{L}^\infty_T(B^{\f d 2}_{2,1})}
}
+
\ep
\n{\vthe}^{h;\f{\beta_0}{\ep}}_{\widetilde{L}^\infty_{T}(B^{\f d 2}_{2, 1})}
+
\f{1}{\ep}
\n{\vthe}^{h;\f{\beta_0}{\ep}}_{L^1_{T}(B^{\f d 2}_{2, 1})}
+
\n{ \mathcal{P}^\perp\vue}^{h; \zeta}_{\widetilde{L}^\infty_{T}(B^{\f d 2-1}_{2, 1})
\cap
L^1_{T}(B^{\f d 2+1}_{2, 1})
} 
\nonumber
\\
&
\quad
+
\n{\rho_\ep}^{h;\zeta}_{\widetilde{L}^\infty_{T}(B^{\f d 2}_{2, 1})}
+
\n{ \mathcal{P}\vue}^{h; \zeta}_{\widetilde{L}^\infty_{T}(B^{\f d 2-1}_{2, 1})
\cap
L^1_{T}(B^{\f d 2+1}_{2, 1})
}  
\\
&
\leq
C
\sp{
\n{a_0}^{h;\zeta}_{B^{\f d 2}_{2, 1}}
+
\n{b_0}^{h;\zeta}_{B^{\f d 2}_{2, 1}}
+
\n{\vu_0}^{h; \zeta}_{B^{\f d 2-1}_{2, 1}}
}
+
C
\int_0^t
\left(
1
+
\n{\vthe}_{B^{\f d 2}_{2, 1}}
+
\n{\rho_\ep}_{B^{\f d 2}_{2, 1}}
+
\n{\vue}_{\underline{B}^{\f d 2+1}_{2, 1}}
+
\n{\vue}^2_{B^{\f d 2}_{2, 1}}
+
\n{\vthe}^2_{B^{\f d 2}_{2, 1}}
\right.
\nonumber
\\
&
\quad
\left.
+
\n{\rho_\ep}^2_{B^{\f d 2}_{2, 1}}
+
\f{1}{\ep}
\n{\vthe}^{h;\f{\beta_0}{\ep}}_{B^{\f d 2}_{2, 1}}
+
|\widehat{(a_0)}_0|^2
+
|\widehat{(b_0)}_0|^2
\right)
\left(
\n{\rho_\ep}^{h; \zeta}_{\widetilde{L}^\infty_\tau(B^{\f d 2}_{2, 1})}
+
\n{\mathcal{P}^\perp\vue}^{h; \zeta}_{\widetilde{L}^\infty_\tau(B^{\f d 2-1}_{2, 1})}
+
\n{\mathcal{P}\vue}^{h; \zeta}_{\widetilde{L}^\infty_\tau(B^{\f d 2-1}_{2, 1})}
\right.
\nonumber
\\
&
\left.
\quad
+
\n{\vthe}^{m; \zeta, \f{\beta_0}{\ep}}_{\widetilde{L}^\infty_\tau(B^{\f d 2-1}_{2, 1})}
+
\ep
\sp{
\n{a_\ep}_{\widetilde{L}^\infty_\tau(B^{\f d 2}_{2,1})}
+
\n{b_\ep}_{\widetilde{L}^\infty_\tau(B^{\f d 2}_{2,1})}
}
\right)
\dta
+
C
\zeta
\ep
\sp{
\sp{
E_{t}^{\ep}
}^2
+
E_{t}^{\ep}
}
\nonumber
\\
&
\quad
+
[t]
E^{\ep, \zeta}_{t, 1}
\sp{
\zeta^{1+\theta}
\sp{
\vW^{\ep}_{t, \theta}
+
\vZ^{\ep}_{t, \theta}
}
+
\n{\vV}^{h;\f{\zeta}{16}}_{\widetilde{L}_{T^*_0}^\infty(B^{\f d 2-1}_{2, 1})
\cap
L_{T^*_0}^1(B^{\f d 2+1}_{2, 1})}
+
\n{\vw}^{h;\f{\zeta}{16}}_{\widetilde{L}_{T^*_0}^\infty(B^{\f d 2-1}_{2, 1})
\cap
L_{T^*_0}^1(B^{\f d 2+1}_{2, 1})}
+
\n{\rho}^{h;\f{\zeta}{16}}_{\widetilde{L}_{T^*_0}^\infty(B^{\f d 2}_{2, 1})
}
}.
\end{align*}
Application of Gr\"{o}nwall's inequality shows 
\begin{align*}
&
\ep
\sp{
\n{a_\ep}_{\widetilde{L}^\infty_T(B^{\f d 2}_{2,1})}
+
\n{b_\ep}_{\widetilde{L}^\infty_T(B^{\f d 2}_{2,1})}
}
+
\ep
\n{\vthe}^{h;\f{\beta_0}{\ep}}_{\widetilde{L}^\infty_{T}(B^{\f d 2}_{2, 1})}
+
\f{1}{\ep}
\n{\vthe}^{h;\f{\beta_0}{\ep}}_{L^1_{T}(B^{\f d 2}_{2, 1})}
+
\n{ \mathcal{P}^\perp\vue}^{h; \zeta}_{\widetilde{L}^\infty_{T}(B^{\f d 2-1}_{2, 1})
\cap
L^1_{T}(B^{\f d 2+1}_{2, 1})
} 
\\
&
\quad
+
\n{\rho_\ep}^{h;\zeta}_{\widetilde{L}^\infty_{T}(B^{\f d 2}_{2, 1})}
+
\n{ \mathcal{P}\vue}^{h; \zeta}_{\widetilde{L}^\infty_{T}(B^{\f d 2-1}_{2, 1})
\cap
L^1_{T}(B^{\f d 2+1}_{2, 1})
} 
\\
&
\leq
C
\text{e}^{
C
[T]
\sp{
1
+
E_{T}^{\ep}
+
\sp{
E_{T}^{\ep}
}^2
}
}
\Bigg[
\n{a_0}^{h;\zeta}_{B^{\f d 2}_{2, 1}}
+
\n{b_0}^{h;\zeta}_{B^{\f d 2}_{2, 1}}
+
\n{\vu_0}^{h; \zeta}_{B^{\f d 2-1}_{2, 1}}
+
\zeta
\ep
\sp{
\sp{
E_{T}^{\ep}
}^2
+
E_{T}^{\ep}
}
\\
&
\quad
+
[T]
E_{T}^{\ep}
\sp{
\zeta^{1+\theta}
\sp{
\vW^{\ep}_{T, \theta}
+
\vZ^{\ep}_{T, \theta}
}
+
\n{\vV}^{h;\f{\zeta}{16}}_{\widetilde{L}_{T^*_0}^\infty(B^{\f d 2-1}_{2, 1})
\cap
L_{T^*_0}^1(B^{\f d 2+1}_{2, 1})}
+
\n{\vw}^{h;\f{\zeta}{16}}_{\widetilde{L}_{T^*_0}^\infty(B^{\f d 2-1}_{2, 1})
\cap
L_{T^*_0}^1(B^{\f d 2+1}_{2, 1})}
+
\n{\rho}^{h;\f{\zeta}{16}}_{\widetilde{L}_{T^*_0}^\infty(B^{\f d 2}_{2, 1})
}
}
\Bigg].
\end{align*}
The completes the proof. \ \ $\Box$

{\bf Proof of Lemma \ref{lee6.3}: }To show \eqref{6.3.4}, we observe 
\begin{align*}
&
\n{\mathcal{P}\vue}^{m; \f{\zeta}{4}, \zeta}_{L^1_{t}(B^{\f d 2+1}_{2, 1})}
+
\n{\mathcal{P}^\perp\vue}^{m; \f{\zeta}{4}, \zeta}_{L^1_{t}(B^{\f d 2+1}_{2, 1})}
\nonumber
\\
& \quad
\leq
\n{\mathcal{P}\vue-\vw}^{m; \f{\zeta}{4}, \zeta}_{L^1_{t}(B^{\f d 2+1}_{2, 1})}
+
\n{\vw}^{m; \f{\zeta}{4}, \zeta}_{L^1_{t}(B^{\f d 2+1}_{2, 1})}
+
\n{
\mathcal{L}\left(\dfrac{t}{\ep}\right)\vVe
-
\mathcal{L}\left(\dfrac{t}{\ep}\right)\vV
}^{m; \f{\zeta}{4}, \zeta}_{L^1_{t}(B^{\f d 2+1}_{2, 1})}
+
\n{
\mathcal{L}\left(\dfrac{t}{\ep}\right)\vV
}^{m; \f{\zeta}{4}, \zeta}_{L^1_{t}(B^{\f d 2+1}_{2, 1})}
\\
&\quad
\les
\zeta^{\theta}
\n{\mathcal{P}\vue-\vw}^{l; \zeta}_{L^1_{t}(B^{\f d 2+1-\theta}_{2, 1})}
+
\n{\vw}^{h; \f{\zeta}{4}}_{L^1_{t}(B^{\f d 2+1}_{2, 1})}
+
\zeta^{1+2\theta}
\n{
\vVe
-
\vV
}^{l; \zeta}_{L^1_{t}(B^{\f d 2-2\theta}_{2, 1})}
+
\n{\vV}^{h; \f{\zeta}{4}}_{L^1_{t}(B^{\f d 2+1}_{2, 1})}
\\
&\quad
\les
\zeta^{\theta}
\n{\mathcal{P}\vue-\vw}^{l; \zeta}_{L^1_{t}(B^{\f d 2+1-\theta}_{2, 1})}
+
\n{\vw}^{h; \f{\zeta}{4}}_{L^1_{t}(B^{\f d 2+1}_{2, 1})}
+
\zeta^{1+2\theta}
[t]^{\f 1 2}
\n{
\vVe
-
\vV
}^{l; \zeta}_{L^2_{t}(B^{\f d 2-2\theta}_{2, 1})}
+
\n{\vV}^{h; \f{\zeta}{4}}_{L^1_{t}(B^{\f d 2+1}_{2, 1})}
\\
&\quad
\les
\zeta^{1+2\theta}
[t]^{\f1 2}
\sp{
\vW^\ep_{t, \theta}
+
\vZ^\ep_{t, \theta}
}
+
\n{\vw}^{h; \f{\zeta}{4}}_{L^1_{T^*_0}(B^{\f d 2+1}_{2, 1})}
+
\n{\vV}^{h; \f{\zeta}{4}}_{L^1_{T^*_0}(B^{\f d 2+1}_{2, 1})}.
\end{align*}
To obtain \eqref{6.3.7}, we have
\begin{align*}
 \ep
\int_0^t
\n{\vue}_{\underline{B}^{\f d 2+1}_{2, 1}}
\n{\vthe}_{B^{\f d 2}_{2, 1}}
\dta
+
\int_0^t
\n{\vue}_{\underline{B}^{\f d 2+1}_{2, 1}}
\n{\vthe}^{h; \f{\beta_0}{\ep}}_{B^{\f d 2-1}_{2, 1}}
\dta
&
\les
 \ep
\int_0^t
\n{\vue}_{\underline{B}^{\f d 2+1}_{2, 1}}
\n{\vthe}_{B^{\f d 2}_{2, 1}}
\dta
\\
&\quad
\les
\int_0^t
\n{\vue}_{\underline{B}^{\f d 2+1}_{2, 1}}
\ep
\sp{
\n{\vae}_{B^{\f d 2}_{2, 1}}
+
\n{\vbe}_{B^{\f d 2}_{2, 1}}
}
\dta
\\
&\quad
\les
\int_0^t
\n{\vue}_{\underline{B}^{\f d 2+1}_{2, 1}}
\ep
\sp{
\n{a_\ep}_{\widetilde{L}^\infty_\tau(B^{\f d 2}_{2,1})}
+
\n{b_\ep}_{\widetilde{L}^\infty_\tau(B^{\f d 2}_{2,1})}
}
\dx.
\end{align*}
To obtain \eqref{6.3.19}, we notice that 
\begin{align*}
\int_0^t
\sp{
\n{\vthe}_{B^{\f d 2}_{2, 1}}
+
\n{\rho_\ep}_{B^{\f d 2}_{2, 1}}
}
\n{\vthe}^{h; \f{\beta_0}{\ep}}_{B^{\f d 2}_{2, 1}}
\dta
&
\les
\int_0^t
\ep
\sp{
\n{\vthe}_{B^{\f d 2}_{2, 1}}
+
\n{\rho_\ep}_{B^{\f d 2}_{2, 1}}
}
\f{1}{\ep}
\n{\vthe}^{h; \f{\beta_0}{\ep}}_{B^{\f d 2-1}_{2, 1}}
\dta
\\
&
\les
\int_0^t
\ep
\sp{
\n{\vae}_{B^{\f d 2}_{2, 1}}
+
\n{\vbe}_{B^{\f d 2}_{2, 1}}
}
\f{1}{\ep}
\n{\vthe}^{h; \f{\beta_0}{\ep}}_{B^{\f d 2-1}_{2, 1}}
\dta
\\
&
\les
\int_0^t
\ep
\sp{
\n{a_\ep}_{\widetilde{L}^\infty_\tau(B^{\f d 2}_{2,1})}
+
\n{b_\ep}_{\widetilde{L}^\infty_\tau(B^{\f d 2}_{2,1})}
}
\f{1}{\ep}
\n{\vthe}^{h; \f{\beta_0}{\ep}}_{B^{\f d 2-1}_{2, 1}}
\dx.
\end{align*}
For \eqref{6.3.22}, we have
\begin{align*}
&
\int_0^t
\sp{
\n{\vthe}^{h;\zeta}_{B^{\f d 2}_{2, 1}}
+
\n{\rho_\ep}^{h;\zeta}_{B^{\f d 2}_{2, 1}}
}
\dta
\\
& \quad
\leq
\int_0^t
\n{\vthe}^{m;\zeta, \f{\beta_0}{\ep}}_{B^{\f d 2}_{2, 1}}
\dta
+
\int_0^t
\n{\vthe}^{h;\f{\beta_0}{\ep}}_{B^{\f d 2}_{2, 1}}
\dta
+ 
\int_0^t
\n{\rho_\ep}^{h;\zeta}_{B^{\f d 2}_{2, 1}}
\dta
\\
&\quad
\leq
\int_0^t
\sp{
\n{\vthe}^{m;\zeta, \f{\beta_0}{\ep}}_{B^{\f d 2-1}_{2, 1}}
}^{\f1 2}
\sp{
\n{\vthe}^{m;\zeta, \f{\beta_0}{\ep}}_{B^{\f d 2+1}_{2, 1}}
}^{\f1 2}
\dta
+
\ep
\int_0^t
\f{1}{\ep}
\n{\vthe}^{h;\f{\beta_0}{\ep}}_{B^{\f d 2}_{2, 1}}
\dta
+
\int_0^t
\n{\rho_\ep}^{h;\zeta}_{B^{\f d 2}_{2, 1}}
\dta
\\
&\quad
\leq
\f{\delta}{2}
\n{\vthe}^{m; \zeta, \f{\beta_0}{\ep}}_{L^1_t(B^{\f d 2+1}_{2,1})}
+
\f{1}{2\delta}
\int_0^t
\n{\vthe}^{m; \zeta, \f{\beta_0}{\ep}}_{\widetilde{L}^\infty_\tau (B^{\f d 2-1}_{2, 1})}
\dta
+
\int_0^t
\n{\rho_\ep}^{h;\zeta}_{\widetilde{L}^\infty_\tau (B^{\f d 2-1}_{2, 1})}
\dta
+
\ep
E^\ep_{T} . 
\end{align*}
To derive \eqref{6.3.8},  
\begin{align*}
&
\int_0^t
\n{\vue}_{B^{\f d 2}_{2, 1}}
\sp{
\n{\vthe}^{m; \f{\zeta}{16}, \zeta}_{B^{\f d 2}_{2, 1}}
+
\n{\mathcal{P}^\perp\vue}^{m; \f{\zeta}{16}, \zeta}_{B^{\f d 2}_{2, 1}}
+
\n{\mathcal{P}\vue}^{m; \f{\zeta}{16}, \zeta}_{B^{\f d 2}_{2, 1}}
}
\dta
+
\int_0^t
\n{\vthe}_{B^{\f d 2}_{2, 1}}
\n{\vue}^{m; \f{\zeta}{16}, \zeta}_{B^{\f d 2}_{2, 1}}
\dta
\\
&
\quad
+
\int_0^t
\sp{
\n{\vthe}_{B^{\f d 2}_{2, 1}}
+
\n{\rho_\ep}_{B^{\f d 2}_{2, 1}}
}
\sp{
\n{\vthe}^{m; \f{\zeta}{16}, \zeta}_{B^{\f d 2}_{2, 1}}
+
\n{\rho_\ep}^{m; \f{\zeta}{16}, \zeta}_{B^{\f d 2}_{2, 1}}
}
\dta
\\
&
\leq
\int_0^t
\sp{
\n{\vue}_{B^{\f d 2}_{2, 1}}
+
2
\n{\vthe}_{B^{\f d 2}_{2, 1}}
+
\n{\rho_\ep}_{B^{\f d 2}_{2, 1}}
}
\n{\mathcal{L}\left(\dfrac{t}{\ep}\right)\vVe
-
\mathcal{L}\left(\dfrac{t}{\ep}\right)\vV}^{m; \f{\zeta}{16}, \zeta}_{B^{\f d 2}_{2, 1}}
\dta
\\
&
\quad
+
\int_0^t
\sp{
\n{\vue}_{B^{\f d 2}_{2, 1}}
+
2
\n{\vthe}_{B^{\f d 2}_{2, 1}}
+
\n{\rho_\ep}_{B^{\f d 2}_{2, 1}}
}
\n{
\mathcal{L}\left(\dfrac{t}{\ep}\right)\vV}^{m; \f{\zeta}{16}, \zeta}_{B^{\f d 2}_{2, 1}}
\dta
\\
&
\quad
+
\int_0^t
\sp{
\n{\vue}_{B^{\f d 2}_{2, 1}}
+
\n{\vthe}_{B^{\f d 2}_{2, 1}}
}
\n{\mathcal{P}\vue-\vw}^{m;\f{\zeta}{16}, \zeta}_{B^{\f d 2}_{2, 1}}
\dta
+
\int_0^t
\sp{
\n{\rho_\ep}_{B^{\f d 2}_{2, 1}}
+
\n{\vthe}_{B^{\f d 2}_{2, 1}}
}
\n{\rho_\ep-\rho}^{m;\f{\zeta}{16}, \zeta}_{B^{\f d 2}_{2, 1}}
\dta
\\
&
\quad
+
\int_0^t
\sp{
\n{\vue}_{B^{\f d 2}_{2, 1}}
+
\n{\vthe}_{B^{\f d 2}_{2, 1}}
}
\n{\vw}^{m;\f{\zeta}{16}, \zeta}_{B^{\f d 2}_{2, 1}}
\dta
+
\int_0^t
\sp{
\n{\rho_\ep}_{B^{\f d 2}_{2, 1}}
+
\n{\vthe}_{B^{\f d 2}_{2, 1}}
}
\n{\rho}^{m;\f{\zeta}{16}, \zeta}_{B^{\f d 2}_{2, 1}}
\dta
\\
&
\les
\zeta^{2\theta}
\int_0^t
\sp{
\n{\vue}_{B^{\f d 2}_{2, 1}}
+
\n{\vthe}_{B^{\f d 2}_{2, 1}}
+
\n{\rho_\ep}_{B^{\f d 2}_{2, 1}}
}
\n{\vVe
-
\vV}^{l; \zeta}_{B^{\f d 2-2\theta}_{2, 1}}
\dx
\\
&
\quad
+
\zeta^{\theta}
\int_0^t
\sp{
\n{\vue}_{B^{\f d 2}_{2, 1}}
+
\n{\vthe}_{B^{\f d 2}_{2, 1}}
}
\n{\mathcal{P}\vue-\vw}^{l;\zeta}_{\underline{B}^{\f d 2-\theta}_{2, 1}}
\dta
+
\zeta^{1+\theta}
\int_0^t
\sp{
\n{\rho_\ep}_{B^{\f d 2}_{2, 1}}
+
\n{\vthe}_{B^{\f d 2}_{2, 1}}
}
\n{\rho_\ep-\rho}^{l;\zeta}_{B^{\f d 2-1-\theta}_{2, 1}}
\dta
\\
&
\quad
+
\int_0^t
\sp{
\n{\vue}_{B^{\f d 2}_{2, 1}}
+
\n{\vthe}_{B^{\f d 2}_{2, 1}}
+
\n{\rho_\ep}_{B^{\f d 2}_{2, 1}}
}
\sp{
\n{\vV}^{h; \f{\zeta}{16}}_{B^{\f d 2}_{2, 1}}
+
\n{\vw}^{h;\f{\zeta}{16}}_{B^{\f d 2}_{2, 1}}
+
\n{\rho_\ep}^{h;\f{\zeta}{16}}_{B^{\f d 2}_{2, 1}}
}
\dta
\\
&
\les
[t]
E_{t}^{\ep}
\sp{
\zeta^{2\theta}
\sp{
\vW^{\ep}_{t, \theta}
+
\vZ^{\ep}_{t, \theta}
}
+
\n{\vV}^{h;\f{\zeta}{16}}_{\widetilde{L}_{T^*_0}^\infty(B^{\f d 2}_{2, 1})
\cap
L_{T^*_0}^1(B^{\f d 2}_{2, 1})}
+
\n{\vw}^{h;\f{\zeta}{16}}_{\widetilde{L}_{T^*_0}^\infty(B^{\f d 2}_{2, 1})
\cap
L_{T^*_0}^1(B^{\f d 2}_{2, 1})}
+
\n{\rho}^{h;\f{\zeta}{16}}_{\widetilde{L}_{T^*_0}^\infty(B^{\f d 2}_{2, 1})
}
}
\end{align*}

The same argument used for \eqref{6.3.8} yields \eqref{6.3.13}. For \eqref{6.3.9}, we only treat the term $\int_0^t
\n{\vue}_{B^{\f d 2}_{2, 1}}
\n{\vthe}^{m; \zeta, \f{\beta_0}{\ep}}_{B^{\f d 2}_{2, 1}}
\dta$, since the remaining two terms follow similarly. It holds 
\begin{align*}
\int_0^t
\n{\vue}_{B^{\f d 2}_{2, 1}}
\n{\vthe}^{m; \zeta, \f{\beta_0}{\ep}}_{B^{\f d 2}_{2, 1}}
\dta
&
\leq
\int_0^t
\n{\vue}_{B^{\f d 2}_{2, 1}}
\sp{
\n{\vthe}^{m; \zeta, \f{\beta_0}{\ep}}_{B^{\f d 2-1}_{2, 1}}
}^{\f1 2}
\sp{
\n{\vthe}^{m; \zeta, \f{\beta_0}{\ep}}_{B^{\f d 2+1}_{2, 1}}
}^{\f1 2}
\dta
\\
&
\leq
\f{1}{2\delta}
\int_0^t
\n{\vue}^2_{B^{\f d 2}_{2, 1}}
\n{\vthe}^{m; \zeta, \f{\beta_0}{\ep}}_{\widetilde{L}^\infty_\tau(B^{\f d 2-1}_{2, 1})}
\dta
+
\f{\delta}{2}
\n{\vthe}^{m;\zeta, \f{\beta_0}{\zeta}}_{L^1_t(B^{\f d 2+1}_{2, 1})}.
\end{align*}
For \eqref{6.3.10}, we have
\begin{align*}
\int_0^t
\n{\vue}^{l;\zeta}_{B^{\f d 2}_{2, 1}}
\n{\vthe}^{h; \f{\beta_0}{\ep}}_{B^{\f d 2}_{2, 1}}
\dta   
\leq
\zeta
\ep
\int_0^t
\n{\vue}^{l;\zeta}_{B^{\f d 2-1}_{2, 1}}
\f{1}{\ep}
\n{\vthe}^{h; \f{\beta_0}{\ep}}_{B^{\f d 2}_{2, 1}}
\dta 
\leq
\zeta
\ep
\sp{
E_{t}^{\ep}
}^2.
\end{align*}
For \eqref{6.3.11}, we notice 
\begin{align*}
&
\int_0^t
\n{\vue}^{h;\zeta}_{B^{\f d 2}_{2, 1}}
\n{\vthe}^{h; \f{\beta_0}{\ep}}_{B^{\f d 2}_{2, 1}}
\dta
\\
&
\leq
\int_0^t
\sp{
\n{\vue}^{h;\zeta}_{B^{\f d 2-1}_{2, 1}}
}^{\f1 2}
\sp{
\n{\vue}^{h;\zeta}_{B^{\f d 2+1}_{2, 1}}
}^{\f1 2}
\n{\vthe}^{h; \f{\beta_0}{\ep}}_{B^{\f d 2}_{2, 1}}
\dta
\\
&
\leq
\f{1}{2\delta}
\int_0^t
\sp{
\n{\mathcal{P}\vue}^{h;\zeta}_{\widetilde{L}^\infty_\tau(B^{\f d 2-1}_{2, 1})}
+
\n{\mathcal{P}^\perp \vue}^{h;\zeta}_{\widetilde{L}^\infty_\tau(B^{\f d 2-1}_{2, 1})}
}
\sp{
\n{\vthe}^{h; \f{\beta_0}{\ep}}_{B^{\f d 2}_{2, 1}}
}^2
\dta
+
\f{\delta}{2}
\sp{
\n{\mathcal{P}\vue}^{h;\zeta}_{L^1_t(B^{\f d 2+1}_{2, 1})}
+
\n{\mathcal{P}^\perp \vue}^{h;\zeta}_{L^1_t(B^{\f d 2+1}_{2, 1})}
}.
\end{align*}
To obtain \eqref{6.3.12}, we invoke \eqref{6.3.19}, \eqref{6.3.8}, \eqref{6.3.9}, \eqref{6.3.10}, \eqref{6.3.11} and use  
\begin{align*}
&
\int_0^t
\n{\vue}_{B^{\f d 2}_{2, 1}}
\sp{
\n{\vthe}^{h; \f{\zeta}{16}}_{B^{\f d 2}_{2, 1}}
+
\n{\mathcal{P}^\perp \vue}^{h; \f{\zeta}{16}}_{B^{\f d 2}_{2, 1}}
+
\n{\mathcal{P}\vue}^{h; \f{\zeta}{16}}_{B^{\f d 2}_{2, 1}}
} \dta 
+
\int_0^t
\n{\vthe}_{B^{\f d 2}_{2, 1}}
\n{\vue}^{h; \f{\zeta}{16}}_{B^{\f d 2}_{2, 1}}
\dta
\\
&
\quad
+
\int_0^t
\sp{
\n{\vthe}_{B^{\f d 2}_{2, 1}}
+
\n{\rho_\ep}_{B^{\f d 2}_{2, 1}}
}
\sp{
\n{\vthe}^{h; \f{\zeta}{16}}_{B^{\f d 2}_{2, 1}}
+
\n{\rho_\ep}^{h; \f{\zeta}{16}}_{B^{\f d 2}_{2, 1}}
}
\dta
\\
&
\leq
\int_0^t
\n{\vue}_{B^{\f d 2}_{2, 1}}
\n{\vthe}^{m; \f{\zeta}{16}, \zeta}_{B^{\f d 2}_{2, 1}}
\dta
+
\int_0^t
\n{\vue}_{B^{\f d 2}_{2, 1}}
\n{\vthe}^{m; \zeta, \f{\beta_0}{\ep}}_{B^{\f d 2}_{2, 1}}
\dta
+
\int_0^t
\n{\vue}^{l;\zeta}_{B^{\f d 2}_{2, 1}}
\n{\vthe}^{h; \f{\beta_0}{\ep}}_{B^{\f d 2}_{2, 1}}
\dta
+
\int_0^t
\n{\vue}^{h;\zeta}_{B^{\f d 2}_{2, 1}}
\n{\vthe}^{h; \f{\beta_0}{\ep}}_{B^{\f d 2}_{2, 1}}
\dta
\\
&
\quad
+
\int_0^t
\n{\vue}_{B^{\f d 2}_{2, 1}}
\sp{
\n{\mathcal{P}\vue}^{m; \f{\zeta}{16}, \zeta}_{B^{\f d 2}_{2, 1}}
+
\n{\vthe}^{m; \f{\zeta}{16}, \zeta}_{B^{\f d 2}_{2, 1}}
}
\dta
+
\int_0^t
\n{\vue}_{B^{\f d 2}_{2, 1}}
\sp{
\n{\mathcal{P}\vue}^{h; \zeta}_{B^{\f d 2}_{2, 1}}
+
\n{\vthe}^{h; \zeta}_{B^{\f d 2}_{2, 1}}
}
\dta
\\
&
\quad
+
\int_0^t
\n{\vthe}_{B^{\f d 2}_{2, 1}}
\n{\vue}^{m; \f{\zeta}{16}, \zeta}_{B^{\f d 2}_{2, 1}}
\dta
+
\int_0^t
\n{\vthe}_{B^{\f d 2}_{2, 1}}
\n{\vue}^{h; \zeta}_{B^{\f d 2}_{2, 1}}
\dta
+
\int_0^t
\sp{
\n{\vthe}_{B^{\f d 2}_{2, 1}}
+
\n{\rho_\ep}_{B^{\f d 2}_{2, 1}}
}
\sp{
\n{\vthe}^{m; \f{\zeta}{16}, \zeta}_{B^{\f d 2}_{2, 1}}
+
\n{\rho_\ep}^{m; \f{\zeta}{16}, \zeta}_{B^{\f d 2}_{2, 1}}
}
\dta
\\
&
\quad
+
\int_0^t
\sp{
\n{\vthe}_{B^{\f d 2}_{2, 1}}
+
\n{\rho_\ep}_{B^{\f d 2}_{2, 1}}
}
\n{\vthe}^{m; \zeta, \f{\beta_0}{\ep}}_{B^{\f d 2}_{2, 1}}
\dta
+
\int_0^t
\sp{
\n{\vthe}_{B^{\f d 2}_{2, 1}}
+
\n{\rho_\ep}_{B^{\f d 2}_{2, 1}}
}
\n{\rho_\ep}^{h; \zeta}_{B^{\f d 2}_{2, 1}}
\dta
\\
&
\quad
+
\int_0^t
\sp{
\n{\vthe}_{B^{\f d 2}_{2, 1}}
+
\n{\rho_\ep}_{B^{\f d 2}_{2, 1}}
}
\n{\vthe}^{h; \f{\beta_0}{\ep}}_{B^{\f d 2}_{2, 1}}
\dta.
\end{align*}
It remains to consider \eqref{6.3.14}. Indeed, 
\begin{align*}
&
\int_0^t 
\n{(\rho_\ep-\rho)}_{B^{\f d 2}_{2, 1}}
\n{\vue}_{\underline{B}^{\f d 2+1}_{2, 1}}
\dta
+
\int_0^t 
\n{\rho_\ep}^{h; \f{\zeta}{16}}_{B^{\f d 2}_{2, 1}}
\n{\vue}_{\underline{B}^{\f d 2+1}_{2, 1}}
\dta
\\
&
\leq
\int_0^t 
\n{(\rho_\ep-\rho)}^{l; \f{\zeta}{16}}_{B^{\f d 2}_{2, 1}}
\n{\vue}_{\underline{B}^{\f d 2+1}_{2, 1}}
\dta
+
2
\int_0^t 
\n{\rho_\ep}^{m; \f{\zeta}{16}, \zeta}_{B^{\f d 2}_{2, 1}}
\n{\vue}_{\underline{B}^{\f d 2+1}_{2, 1}}
\dta
+
2
\int_0^t 
\n{\rho_\ep}^{h; \zeta}_{B^{\f d 2}_{2, 1}}
\n{\vue}_{\underline{B}^{\f d 2+1}_{2, 1}}
\dta
+
\int_0^t 
\n{\rho}^{h; \f{\zeta}{16}, \zeta}_{B^{\f d 2}_{2, 1}}
\n{\vue}_{\underline{B}^{\f d 2+1}_{2, 1}}
\dta
\\
&
\leq
E_{t}^{\ep}
\sp{
\zeta^{1+\theta}
\vW^{\ep}_{t, \theta}
+
\n{\rho}^{h;\f{\zeta}{16}}_{\widetilde{L}_{T^*_0}^\infty(B^{\f d 2}_{2, 1})
}
}
+
2
\int_0^t
\n{\rho}^{h; \zeta}_{\widetilde{L}^\infty_\tau(B^{\f d 2}_{2, 1})}
\n{\vue}_{\underline{B}^{\f d 2+1}_{2, 1}}
\dta
+
2
\int_0^t 
\n{\rho_\ep}^{m; \f{\zeta}{16}, \zeta}_{B^{\f d 2}_{2, 1}}
\n{\vue}_{\underline{B}^{\f d 2+1}_{2, 1}}
\dta,
\end{align*}
which together with \eqref{6.3.13} completes the proof.     \ \ $\Box$

\section{Well-posedness of the limiting equations \texorpdfstring{\eqref{1.3}}{} and \texorpdfstring{\eqref{1.4}}{}}

\subsection{Well-posedness of the limiting equation \texorpdfstring{\eqref{1.3}}{}}

Integrating \eqref{1.3} over $\Td$
and dividing by $\sqrt{|\Td|}$ yields $\p_t \widehat{\vU}_0\equiv 0$, i.e., 
\begin{align}\label{5.1}
\widehat{\vU}_0 
\equiv 
\widehat{(\vU(0))}_0
=
(\widehat{(a_0)}_0, \widehat{(b_0)}_0, \widehat{(\vu_0)}_0).
\end{align}
By $\vU\in \text{Ker}\, L$ and \eqref{5.1}, there exist $\rho$ and $\vw$ satisfying $\widehat{\rho}_0=\widehat{\vw}_0=0$ and $\Div \vw=0$
and 
\begin{align}\label{5.2}
\vU
=
(\vU^1, \vU^2, \vU^3)
=
\sp{ \rho, -\df{c_1}{c_2}\rho, \vw}^T
+
\sp{
\df{\widehat{(a_0)}_0}{\sqrt{|\Td|}},
\df{\widehat{(b_0)}_0}{\sqrt{|\Td|}},
\df{\widehat{(\vu_0)}_0}{\sqrt{|\Td|}}
}^T.
\end{align}
Writing \eqref{1.3} componentwise as $(\vU^1, \vU^2, \vU^3)$, we see 
\begin{equation}\label{5.3}
\left\{\begin{aligned}
&
\p_t \vU^1
+ 
\Grad\vU^1\cdot \vU^3
=0,
\\
&
\p_t \vU^2
+ 
\Grad\vU^2\cdot \vU^3
=0,
\\
&
\p_t \vU^3
+ 
\mathcal{P}
\sp{
\vU^3\cdot \Grad \vU^3
}
-
\df{\mu}{R^0+Q^0}
\De
\vU^3
=0,
\quad 
\Div \vU^3=0,
\\
&
(\vU^1(0), \vU^2(0), \vU^3(0))^T
=
\mathbb{P}
(a_0, b_0, \vu_0)^T.
\end{aligned}
\right.
\end{equation}
Observe that \eqref{5.3}$_1$ and \eqref{5.3}$_2$ are the transport equations, and \eqref{5.3}$_3$ is the incompressible Navier--Stokes system. By the classical theory of transport equations and the incompressible Navier--Stokes equations, we have 
\begin{Theorem}\label{Th6.1}
If $(a_0, b_0, \vu_0)\in B^{\f d 2}_{2,1}\times B^{\f d 2}_{2,1} \times B^{\f d 2-1}_{2,1}$, then there exists $0<T^*_0\leq \infty$ such that equations \eqref{5.3} admits a unique solution
\begin{align*}
\vU=
(\vU^1, \vU^2, \vU^3)
\in 
\widetilde{C}_{T^*_0}(B^{\f d 2}_{2, 1})
\times
\widetilde{C}_{T^*_0}(B^{\f d 2}_{2, 1})
\times
\sp{
\widetilde{C}_{T^*_0}(B^{\f d 2-1}_{2, 1})
\cap
L_{T^*_0}^1(\underline{B}^{\f d 2+1}_{2, 1})
},
\end{align*}
where
\begin{align*}
\widetilde{C}_{T^*_0}(B^{s}_{2, 1})
:=
\widetilde{L}_{T^*_0}^\infty(B^{s}_{2, 1})
\cap
C([0, T^*_0]; B^{s}_{2, 1}).
\end{align*}
\end{Theorem}

\subsection{Well-posedness of the limiting equation \texorpdfstring{\eqref{1.4}}{}}

Masmoudi \cite{M01} studied the low Mach number limits of the compressible Navier--Stokes equations and derived the limiting system \eqref{1.4}, where the expressions of the operators $\mathcal{Q}_1$ and $\mathcal{Q}_2$
differ slightly from those in the present paper. He was the first to prove the well-posedness of \eqref{1.4} for Sobolev initial data. Subsequently, Danchin \cite{D02} extended this result to Besov initial data. Following their approach, we investigate the well-posedness of \eqref{1.4} on $[0, T^*_0]$, where $\vU$ and $T^*_0$ are as in Theorem \ref{Th6.1}. 

\begin{Theorem}\label{Th7.2}
Let $0\leq s$ and $\vV(0) \in \text{Ker} \, L \cap  B^{s}_{2, 1}$ be real-valued. Then \eqref{1.4} admits a unique real-valued solution 
$\vV\in \text{Ker}\, L \cap \widetilde{C}_{T^*_0}(B^{s}_{2, 1})\cap L^1_{T^*_0}(B^{s+2}_{2, 1})$. Moreover, 
\begin{align*}
&
\n{\vV}_{L^\infty_{T^*_0}(L^2)}   
+
\n{\vV}_{L_{T^*_0}^2(H^1)}   
\leq
C
\n{\vV(0)}_{L^2},
\\
&
\n{\vV}_{\widetilde{L}^\infty_{T^*_0}(B^s_{2, 1})}
+
\n{\vV}_{L^1_{T^*_0}(B^{s+2}_{2, 1})}
\leq
C
\text{e}^{C \n{\vV(0)}^2_{L^2}}
\n{\vV(0)}_{B^{s}_{2, 1}}.
\end{align*}

\end{Theorem}
\bProof
Since the proof is similar to that of Theorem 8.1 and Theorem 8.2 in \cite{D02}, we only present a sketch here.
Taking the $\mathbb{H}$-inner product of \eqref{1.4} with $\vV$ and applying Lemma \ref{le6.1} and Lemma \ref{le6.2}, we obtain
\begin{align*}
\df{1}{2}
\df{d}{dt}
\n{\vV}^2_{L^2}
+
\df{\nu}{2(R^0+Q^0)}
\n{\Grad\vV}^2_{L^2}
=0.
\end{align*}
Hence 
\begin{align}\label{5.19}
\n{\vV}_{L^\infty_{T^*_0}(L^2)}   
+
\n{\vV}_{L_{T^*_0}^2(H^1)}   
\les
\n{\vV(0)}_{L^2}.
\end{align}
Applying $\De_j$ to \eqref{1.4}, taking the $\mathbb{H}$-inner product with $\De_j \vV$, and using Lemma \ref{le6.1}, we obtain
\begin{align*}
\df{d}{dt}
\n{\De_j\vV}^2_{L^2}
+
2^{2j}
\n{\De_j\vV}^2_{L^2}
\les
\n{\De_j \mathcal{Q}_1(\vV, \vV)}_{L^2}
\n{\De_j\vV}_{L^2},
\end{align*}
which gives 
\begin{align*}
\n{\De_j \vV}_{L_{t}^\infty(L^2)}
+
2^{2j}\n{\De_j \vV}_{L_{t}^1(L^2)}
\les
\n{\De_j \vV (0)}_{L^2}
+
\n{\De_j \mathcal{Q}_1(\vV, \vV)}_{L_{t}^1(L^2)}.
\end{align*}
Multiplying both sides above by $2^{js}$
and summing over $j$, we see 
\begin{align}\label{5.22}
\n{\vV}_{\widetilde{L}^\infty_t(B^s_{2, 1})}
+
\n{\vV}_{L^1_t(B^{s+2}_{2, 1})}
\leq
C
\n{\vV(0)}_{B^s_{2, 1}}
+
C
\n{\mathcal{Q}_1(\vV, \vV)}_{L^1_t(B^s_{2, 1})} . 
\end{align}
Using \eqref{minski2} and Lemma \ref{le6.2}, 
\begin{align}\label{5.23}
C
\n{\mathcal{Q}_1(\vV, \vV)}_{L^1_t(B^s_{2, 1})}
&
\leq
C
\int_0^t
\n{\vV}_{B^{s+1}_{2,1}}
\n{\vV}_{H^1}
\dta
\leq
C
\int_0^t
\n{\vV}^{\f1 2}_{B^{s}_{2,1}}
\n{\vV}^{\f1 2}_{B^{s+2}_{2,1}}
\n{\vV}_{H^1}
\dta
\nonumber
\\
&
\leq
\f1 2
\n{\vV}_{L^1_t(B^{s+2}_{2, 1})}
+
\f{C^2}{2}
\int_0^t
\n{\vV}_{\widetilde{L}^\infty_\tau(B^s_{2, 1})}
\n{\vV}^2_{H^1}
\dta.
\end{align}
Inserting \eqref{5.23} into \eqref{5.22}, applying Gr\"{o}nwall's inequality, and combining with \eqref{5.19}, we conclude 
\begin{align*}
\n{\vV}_{\widetilde{L}^\infty_{T^*_0}(B^s_{2, 1})}
+
\n{\vV}_{L^1_{T^*_0}(B^{s+2}_{2, 1})}
\leq
C
\text{e}^{C \n{\vV(0)}^2_{L^2}}
\n{\vV(0)}_{B^{s}_{2, 1}}.
\end{align*}
With these a priori estimates at hand, a standard argument yields the existence of a solution in the corresponding space.

Next, we prove the uniqueness of the solution in the space 
$\widetilde{C}_{T^*_0}(B^{s}_{2, 1})\cap L^1_{T^*_0}(B^{s+2}_{2, 1})$.
Indeed, uniqueness holds in the larger space 
$C([0, T^*_0];L^2) \cap L^2_{T^*_0}(H^1) \supseteq \widetilde{C}_{T^*_0}(B^{s}_{2, 1})\cap L^1_{T^*_0}(B^{s+2}_{2, 1})$.
Let $\vV_1$, $\vV_2 \in C([0, T^*_0];L^2) \cap L^2_{T^*_0}(H^1)$
be two real-valued solutions to \eqref{1.4} emanating from the same initial data. Hence, $\delta \vV:=\vV_1-\vV_2$ satisfies 
\begin{align}\label{5.24}
\p_t \delta\vV 
+
\mathcal{Q}_1(\delta\vV, \delta\vV)
+
2\mathcal{Q}_1(\vV_2, \delta\vV)
+
\mathcal{Q}_2(\underline{\vU}, \delta\vV)
+
\mathcal{Q}_2(\mathbb{P}_0 \vU, \delta\vV)
- 
\overline{\mathcal{D}}(\delta\vV)=0
\end{align}
Taking the $\mathbb{H}$-inner product of \eqref{5.24} with $\delta\vV$, and applying Lemma \ref{le6.1} and Lemma \ref{le6.2}, we obtain
\begin{align*}
\df{d}{dt}
\n{\delta\vV}^2_{L^2}
+
\n{\Grad\delta\vV}^2_{L^2}
\leq
C
\n{\vV_2}_{H^1}
\n{\delta\vV}_{L^2}
\n{\delta\vV}_{H^1}
\leq
\f{C^2}{2}
\n{\vV_2}^2_{H^1}
\n{\delta\vV}^2_{L^2}
+
\f1 2
\n{\delta\vV}^2_{H^1}.
\end{align*}
This gives $\delta \vV \equiv 0$ by applying Gronwall's inequality.  \ \ $\Box$

We now establish the estimates for $\mathcal{Q}_1$ and $\mathcal{Q}_2$
used in the proof above.

\begin{Lemma}\label{le6.1}
Let $0<\theta<1$, $1\leq r \leq \infty$, $s\in \mathbb{R}$, $j\in \mathbb{Z}$, $\vE \in \text{Ker } L$, $\vA\in \text{Im} \,L$ be real-valued and $\vB\in \text{Im} \, L$. Then
\begin{align}
&
\langle 
\mathcal{Q}_2(\underline{\vE}, \vA), \vA
\rangle_{\mathbb{H} }
=
0,\label{5.12}
\\
&
\langle 
\De_j \mathcal{Q}_2(\underline{\vE}, \vA), \De_j \vA
\rangle_{\mathbb{H} }
=
0,\label{5.5}
\\
&
\langle 
\mathcal{Q}_2(\mathbb{P}_0 \vE, \vA), \vA
\rangle_{\mathbb{H}}
=
0, \label{5.13}
\\
&
\langle 
\De_j \mathcal{Q}_2(\mathbb{P}_0 \vE, \vA), \De_j \vA
\rangle_{\mathbb{H}}
=
0,
\label{5.7}
\\
&
\n{\mathcal{Q}_2(\underline{\vE}, \vB)}_{H^{\f d 2-1-\theta}}
\les
\n{\vB}_{H^{\f d 2}}\n{\vE}_{\underline{H}^{\f d 2}}
\label{5.25},
\\
&
\n{\mathcal{Q}_2(\mathbb{P}_0 \vE, \vB)}_{B^{s}_{2, r}}
\les
|\mathbb{P}_0 \vE|
\n{\vB}_{B^{s+1}_{2, r}}.
\end{align}
\end{Lemma}
\bProof 
Observe first that, since $\vA$ is real-valued, 
\begin{align}\label{5.11}
\widehat{\vA}_{k}^\alpha=\overline{\widehat{\vA}_{-k}^\alpha} . 
\end{align}
Then, we make use of \eqref{Q_2} to obtain 
\begin{align*}
&
\langle 
\mathcal{Q}_2(\underline{E}, \vA), \vA
\rangle_{\mathbb{H} }
\\
& \quad 
=
\sum_{
\substack{k, m, l, \alpha, \gamma
\\ 
k=l+m
\\ \lambda_k^\alpha=\lambda_m^\gamma
}
}
\Bigg(
\df{2(\widehat{\vw}^*_l\cdot k)(m\cdot k)}{|m||k|}
+
\df{\widehat{\rho}^*_l\alpha\tsgn(k)}{c_0|k|}
\left(
|k|^2\left(c_3R^0-\df{c_6c_1Q^0}{c_2}\right)
\right.
\nonumber
\\
&\qquad\left.
+
Q^0(m\cdot k)(c_4+c_5)
+
\df{c_1R^0}{c_2}(m\cdot k)(c_4+c_5)
-
\df{c_1R^0}{c_2}(k\cdot k)(c_4+c_5)
\right)
\Bigg) 
\cdot 
\df{
i \widehat{\vA}_m^\gamma
\widehat{\vA}_{-k}^\alpha
}
{2\sqrt{|\Td|}}
.
\end{align*}
Replacing $k$ by $-m$, $m$ by $-k$, $\alpha$ by $\gamma$, and $\gamma$ by $\alpha$, we get
\begin{align*}
\langle 
\mathcal{Q}_2(\underline{E}, \vA), \vA
\rangle_{\mathbb{H}^s }
=
\f 1 2
\sum_{
\substack{k, m, l, \alpha, \gamma
\\ 
k=l+m
\\ \lambda_k^\alpha=\lambda_m^\gamma
}
} 
\df{(\widehat{\vw}^*_l\cdot l)(m\cdot k)}{|m||k|}
\df{
i  \widehat{\vA}_m^\gamma
\widehat{\vA}_{-k}^\alpha
}
{\sqrt{|\Td|}}
=
0.
\end{align*}
In the same spirit, we obtain
\begin{align*}
 \langle 
\De_j \mathcal{Q}_2(\underline{E}, \vA), \De_j \vA
\rangle_{\mathbb{H} }   
=
\f 1 2
\sum_{
\substack{k, m, l, \alpha, \gamma
\\ 
k=l+m
\\ \lambda_k^\alpha=\lambda_m^\gamma
}
} 
\df{(\widehat{\vw}^*_l\cdot l)(m\cdot k)}{|m||k|}
\df{
i \sp{\varphi(2^{-j}|k|)}^2 \widehat{\vA}_m^\gamma
\widehat{\vA}_{-k}^\alpha
}
{\sqrt{|\Td|}}
=
0.
\end{align*}
It follows from \eqref{Q_21} and \eqref{5.11} that 
\begin{align*}
& \langle 
\mathcal{Q}_2(\mathbb{P}_0 \vE, \vA), \vA
\rangle_{\mathbb{H}}  
=
\sum\limits_{k,\alpha}
\left(
\widehat{(\vE^1)}_0
\left(c_1+c_3R^0+c_5Q^0\right)
+
\widehat{(\vE^2)}_0
\left(c_2+c_4R^0+c_6Q^0\right)
\right)
\df{i\lambda_k^\alpha\widehat{\vA}^\alpha_k \widehat{\vA}^\alpha_{-k}}
{2\sqrt{|\Td|}c_0}
\nonumber
\\
&
\qquad \qquad \qquad \qquad \qquad
+
\sum\limits_{k,\alpha} 
k\cdot \widehat{(\vE^3)}_0
\df{i \widehat{\vA}_k^\alpha \widehat{\vA}_{-k}^\alpha}{\sqrt{|\Td|}}.
\end{align*}
Replacing $k$ by $-k$, we get
\begin{align*}
 \langle 
\mathcal{Q}_2(\mathbb{P}_0 \vE, \vA), \vA
\rangle_{\mathbb{H}^s} 
=
0.
\end{align*}
By the same token, we get \eqref{5.7}. 

To proceed, we show \eqref{5.25}. By \eqref{Q_2}, we have $|k=|m|$ and $|l|\leq 2|k|$. Lemma \ref{le4.4} implies 
\begin{align*}
\n{\mathcal{Q}_2(\underline{\vE}, \vB)}_{H^{\f d 2-1-\theta}}
&
\les  
\n{\vE}_{\underline{H}^{\f d 2}}
\n{|k|^{-\theta}\widehat{\vB}^{\alpha}_k}_{\ell^1
\sp{
\widetilde{\mathbb{Z}}^d\setminus\{0\}}}
+
\n{\vB}_{H^{\f d 2}}
\n{|k|^{-\theta}\widehat{\vE}_k}_{\ell^1
\sp{
\widetilde{\mathbb{Z}}^d\setminus\{0\}}}
\\
&
\les
\n{\vB}_{H^{\f d 2}}\n{\vE}_{\underline{H}^{\f d 2}}.
\end{align*}
The estimate for $\n{\mathcal{Q}_2(\mathbb{P}_0 \vE, \vB)}_{B^{s}_{2, r}}$ is straightforward.         \ \ $\Box$

\begin{Lemma}\label{le6.2}
Let $s>-\f 3 2$,
$
\vA
=
\sum\limits_{k,\alpha} 
\widehat{\vA}^{\alpha}_k 
\Phi^\alpha_k,
\vB
=
\sum\limits_{k,\alpha} 
\widehat{\vB}^{\alpha}_k 
\Phi^\alpha_k
\in \text{Im}\, L$. Then 
\begin{align*}
&
\langle 
\mathcal{Q}_1(\vA, \vA), \vA
\rangle_\mathbb{H} 
=
0,
\ \
\text{if } \vA \text{ is real-valued},
\\
&
|\langle 
\mathcal{Q}_1(\vA, \vB), \vB
\rangle_\mathbb{H}|
\les
\n{\vA}_{H^1}
\n{\vB}_{L^2}
\n{\vB}_{B^{\f1 2}_{2, 1}},
\ \
\text{if } \vB \text{ is real-valued},
\\
&
|\langle 
\mathcal{Q}_1(\vA, \vB), \vB
\rangle_\mathbb{H}|
\les
\sp{
\n{\vA}_{H^1}
\n{\vB}_{L^2}
+
\n{\vB}_{H^1}
\n{\vA}_{L^2}
}
\n{\vB}_{B^{\f1 2}_{2, 1}},
\\
&
\n{
\mathcal{Q}_1(\vA, \vB)
}_{B^{s}_{2, 1}}
\les
\n{\vA}_{B^{s+1}_{2, 1}}
\n{\vB}_{B^{\f1 2}_{2, 1}}
+
\n{\vB}_{B^{s+1}_{2, 1}}
\n{\vA}_{B^{\f1 2}_{2, 1}}.
\end{align*}
\end{Lemma}
\bProof
Similarly to \cite{D02, M01}, we introduce
\begin{align*}
&
\mathbf{P}
:=
\left\{
\mathbf{p}=(\mathbf{p}^1, \mathbf{p}^2, \cdots, \mathbf{p}^d) \in \widetilde{\mathbb{Z}}^d\setminus\{0\}: 
\text{the integers }
\mathbf{p}^1\mathbf{a}^1, 
\cdots
\mathbf{p}^d\mathbf{a}^d
\text{ have greatest common divisor }
1
\right\},
\\
&
\vA^{\mathbf{p}}
:=
\sum\limits_{j\in \mathbb{Z}\setminus\{0\}}
\widehat{\vA}^{\text{sgn}(\mathbf{p})}_{j\mathbf{p}} 
\Phi^{\text{sgn}(\mathbf{p})}_{j\mathbf{p}},
\quad 
a^{\mathbf{p}}
:=
\sum\limits_{j\in \mathbb{Z}\setminus\{0\}}
\widehat{\vA}^{\text{sgn}(\mathbf{p})}_{j\mathbf{p}} 
\df{\text{e}^{ijz}}{\sqrt{2\pi}},
\end{align*}
where $a^{\mathbf{p}}$ is a periodic function on $\mathbb{T}$.
A direct calculation shows that
$\mathcal{Q}_1(\vA^{\mathbf{p}}, \vB^{\mathbf{q}})=0$, if $\mathbf{p}\neq \mathbf{q}$. Therefore, we deduce from \eqref{Q_1} that 
\begin{align}\label{5.14}
\mathcal{Q}_1(\vA, \vB)
&
=
\sum\limits_{\mathbf{p}, \mathbf{q} \in \mathbf{P}}
\mathcal{Q}_1(\vA^{\mathbf{p}}, \vB^{\mathbf{q}})
=
\sum\limits_{\mathbf{p} \in \mathbf{P}}
\mathcal{Q}_1(\vA^{\mathbf{p}}, \vB^{\mathbf{p}})
\nonumber
\\
&
=
\sqrt{2\pi}
c_{d, \mathbf{a}}
\sum\limits_{\mathbf{p} \in \mathbf{P}}
\sum\limits_{j\in \mathbb{Z}\setminus\{0\}}
|\mathbf{p} |
\Phi^{\text{sgn}(\mathbf{p})}_{j\mathbf{p}}
\sum_{
\substack{
j', j'' \in \mathbb{Z}\setminus\{0\}
\\ 
j=j'+j''
}
}
\df{ij}{\sqrt{2\pi}}
\widehat{\vA}^{\text{sgn}(\mathbf{p})}_{j'\mathbf{p}} 
\widehat{\vB}^{\text{sgn}(\mathbf{p})}_{j''\mathbf{p}} ,
\end{align}
where 
\begin{align*}
   c_{d, \mathbf{a}}
:=
\df{
\df{3c_0}{R^0}
+
\df{c_3R^0}{c_0}
+
\df{(c_4+c_5)Q^0}{c_0}
+
\df{c_6(Q^0)^2}{c_0R^0}
}
{2c_{\mathbf{a}}}. 
\end{align*}
Hence, 
\begin{align}\label{5.15}
\langle 
\mathcal{Q}_1(\vA, \vB), \vB
\rangle_\mathbb{H}  
=
\sqrt{2\pi}
c_{d, \mathbf{a}}
\sum\limits_{\mathbf{p} \in \mathbf{P}}
|\mathbf{p}|
\int_{\mathbb{T}}
\p_z
\sp{
a^{\mathbf{p}}
b^{\mathbf{p}}
}
\overline{b^{\mathbf{p}}}
\,
dz.
\end{align}
In particular, when $\vA = \vB$, \eqref{5.15} gives
\begin{align*}
\langle 
\mathcal{Q}_1(\vA, \vA), \vA
\rangle_\mathbb{H}  
=
0.
\end{align*}
If $\vB$ is real-valued, integration by parts yields
\begin{align*}
\int_{\mathbb{T}}
\p_z
\sp{
a^{\mathbf{p}}
b^{\mathbf{p}}
}
\overline{b^{\mathbf{p}}}
\,
dz
=
\int_{\mathbb{T}}
\p_z
\sp{
a^{\mathbf{p}}
b^{\mathbf{p}}
}
b^{\mathbf{p}}
\,
dz
=
-
\f1 2
\int_{\mathbb{T}}
\p_z
a^{\mathbf{p}}
b^{\mathbf{p}}
b^{\mathbf{p}}
\,
dz.
\end{align*}
In light of Lemma \ref{le6.3}, we get 
\begin{align*}
|\langle 
\mathcal{Q}_1(\vA, \vB), \vB
\rangle_\mathbb{H}|
&
\les
\sum\limits_{\mathbf{p} \in \mathbf{P}}
|\mathbf{p}|
\n{\p_z a^{\mathbf{p}}}_{L^2(\mathbb{T})}
\n{ b^{\mathbf{p}}}_{L^2(\mathbb{T})}
\n{ b^{\mathbf{p}}}_{L^\infty(\mathbb{T})}  
\\
&
\les
\sp{
\sum\limits_{\mathbf{p} \in \mathbf{P}} 
|\mathbf{p}|^2
\n{a^{\mathbf{p}}}_{H^1(\mathbb{T})}^2
}^{\f1 2}
\sp{
\sum\limits_{\mathbf{p} \in \mathbf{P}} 
\n{b^{\mathbf{p}}}_{L^2(\mathbb{T})}^2
}^{\f1 2}
\sup_{\mathbf{p} \in \mathbf{P}}
\n{b^{\mathbf{p}}}_{B^{\f1 2}_{2,1}(\mathbb{T})}
\\
&
\les
\n{\vA}_{H^1}
\n{\vB}_{L^2}
\n{\vB}_{B^{\f1 2}_{2, 1}}.
\end{align*}
For the general case, we have
\begin{align*}
|\langle 
\mathcal{Q}_1(\vA, \vB), \vB
\rangle_\mathbb{H}|
&
\les
\sum\limits_{\mathbf{p} \in \mathbf{P}}
|\mathbf{p}|
\sp{
\n{\p_z a^{\mathbf{p}}}_{L^2(\mathbb{T})}
\n{ b^{\mathbf{p}}}_{L^2(\mathbb{T})}
+
\n{\p_z b^{\mathbf{p}}}_{L^2(\mathbb{T})}
\n{ a^{\mathbf{p}}}_{L^2(\mathbb{T})}
}
\n{ b^{\mathbf{p}}}_{L^\infty(\mathbb{T})}
\\
&
\les
\sp{
\sum\limits_{\mathbf{p} \in \mathbf{P}} 
|\mathbf{p}|^2
\n{a^{\mathbf{p}}}_{H^1(\mathbb{T})}^2
}^{\f1 2}
\sp{
\sum\limits_{\mathbf{p} \in \mathbf{P}} 
\n{b^{\mathbf{p}}}_{L^2(\mathbb{T})}^2
}^{\f1 2}
\sup_{\mathbf{p} \in \mathbf{P}}
\n{b^{\mathbf{p}}}_{B^{\f1 2}_{2,1}(\mathbb{T})}
\\
&
\quad
+
\sp{
\sum\limits_{\mathbf{p} \in \mathbf{P}} 
|\mathbf{p}|^2
\n{b^{\mathbf{p}}}_{H^1(\mathbb{T})}^2
}^{\f1 2}
\sp{
\sum\limits_{\mathbf{p} \in \mathbf{P}} 
\n{a^{\mathbf{p}}}_{L^2(\mathbb{T})}^2
}^{\f1 2}
\sup_{\mathbf{p} \in \mathbf{P}}
\n{b^{\mathbf{p}}}_{B^{\f1 2}_{2,1}(\mathbb{T})}
\\
&
\les
\sp{
\n{\vA}_{H^1}
\n{\vB}_{L^2}
+
\n{\vB}_{H^1}
\n{\vA}_{L^2}
}
\n{\vB}_{B^{\f1 2}_{2, 1}}.
\end{align*}

Finally, we prove the estimates for  
$\n{\mathcal{Q}_{1}(\vA, \vA)}_{B^{s}_{2, 1}}$. 
To this end, we use a different cutoff function $\varphi_{\mathbf{p}}$.  For each $\mathbf{p} \in \mathbf{P}$, there exist a unique $j_\mathbf{p} \in \mathbb{Z}$ such that $2^{j_\mathbf{p}} \leq |\mathbf{p}| < 2^{j_\mathbf{p}+1} $. We define 
\begin{align}\label{5.18}
\varphi_{\mathbf{p}}(\zeta):=\varphi(|\mathbf{p}|^{-1}2^{j_{\mathbf{p}}}\zeta), \quad  \zeta \in \mathbb{R}.
\end{align}
The change of cutoff functions only introduces harmless positive constants, i.e.,
\begin{align*}
\n{g}_{B^{s}_{p, r}} 
\approx 
\left(
\sum\limits_{j\in\mathbb{Z}}2^{jsr}\|\Delta_{j}^{\mathbf{p} }g\|^{r}_{L^p(\Td)}
+
\left\|\f{\widehat{g}_0}{\sqrt{\T}}\right\|^{r}_{L^p(\Td)}
\right)^{\f1 r},
\qquad 
\Delta_j^{\mathbf{p}} g  :
= 
\sum\limits_{k \in \widetilde{\mathbb{Z}}^d}
\varphi_{\mathbf{p}}(2^{-j}|k|)
\widehat{g}_k
\f{\text{e}^{i k \cdot x}} {\sqrt{|\Td|}}.
\end{align*}
We conclude from \eqref{5.14} and Lemma \ref{le6.3} that 
\begin{align*}
&
\n{
\mathcal{Q}_1(\vA, \vB)
}_{B^s_{2,1}}   
\\
& \quad 
\les
\sum\limits_{\mathbf{p}\in \mathbf{P}}
\n{
\sum\limits_{j'\in \mathbb{Z}\setminus\{0\}}
\df{j'}{\sqrt{2\pi}}
|\mathbf{p} |
\Phi^{\text{sgn}(\mathbf{p})}_{j'\mathbf{p}}
\sum_{
\substack{
j'', j''' \in \mathbb{Z}\setminus\{0\}
\\ 
j'=j''+j'''
}
}
\widehat{\vA}^{\text{sgn}(\mathbf{p})}_{j'\mathbf{p}} 
\widehat{\vB}^{\text{sgn}(\mathbf{p})}_{j''\mathbf{p}}
}_{B^s_{2,1}}
\\
& \quad 
\les
\sum\limits_{\mathbf{p}\in \mathbf{P}}
\sum\limits_{j\in \mathbb{Z}}
2^{js}
\sp{
\sum\limits_{j'\in \mathbb{Z}\setminus\{0\}}
|j'\mathbf{p}|^2
\sp{
\varphi_{\mathbf{p}}(2^{-j}|j'\mathbf{p}|)
}^2
|
\sum_{
\substack{
j', j'' \in \mathbb{Z}\setminus\{0\}
\\ 
j=j'+j''
}
}
\df{1}{\sqrt{2\pi}}
\widehat{\vA}^{\text{sgn}(\mathbf{p})}_{j'\mathbf{p}} 
\widehat{\vB}^{\text{sgn}(\mathbf{p})}_{j''\mathbf{p}}
|^2
}^{\f1 2}
\\
&  \quad 
\les
\sum\limits_{\mathbf{p}\in \mathbf{P}}
\sum\limits_{j\in \mathbb{Z}}
2^{j(s+1)}
\n{\De_{j-j_{\mathbf{p}}}(a^{\mathbf{p}}b^{\mathbf{p}})}_{L^2(\mathbb{T})}
\\
& \quad 
\les
\sum\limits_{\mathbf{p}\in \mathbf{P}}
|\mathbf{p}|^{s+1}
\n{
a^{\mathbf{p}}b^{\mathbf{p}}
}_{B^{s+1}_{2,1}(\mathbb{T})}
\\
& \quad 
\les
\sum\limits_{\mathbf{p}\in \mathbf{P}}
|\mathbf{p}|^{s+1}
\sp{
\n{a^{\mathbf{p}}}_{B^{s+1}_{2,1}(\mathbb{T})}
\n{b^{\mathbf{p}}}_{B^{\f 1 2}_{2, 1}(\mathbb{T})}
+
\n{b^{\mathbf{p}}}_{B^{s+1}_{2, 1}(\mathbb{T})}
\n{a^{\mathbf{p}}}_{B^{\f 1 2}_{2, 1}(\mathbb{T})}
}
\\
& \quad 
\les
\n{\vA}_{B^{s+1}_{2,1}}
\n{\vB}_{B^{\f 1 2}_{2, 1}}
+
\n{\vB}_{B^{s+1}_{2,1}}
\n{\vA}_{B^{\f 1 2}_{2, 1}}.
\end{align*}
The proof is finished.       \ \ $\Box$

\begin{Lemma}\label{le6.3}
Let $s\in \mathbb{R}$ and 
$
\vA
=
\sum\limits_{k,\alpha} 
\widehat{\vA}^{\alpha}_k 
\Phi^\alpha_k
\in \text{Im } L
$. Then 
\begin{align}
&
\sum\limits_{\mathbf{p}\in \mathbf{P}}
|\mathbf{p}|^{2s}
\n{a^{\mathbf{p}}}^2_{H^{s}(\mathbb{T})}
\approx
\n{\vA}_{H^s}^2,
\label{5.16}
\\
&
\sp{
\sum\limits_{\mathbf{p}\in \mathbf{P}}
|\mathbf{p}|^{2s}
\n{a^{\mathbf{p}}}^2_{B^{s}_{2,1}(\mathbb{T})}
}^{\f1 2}
\les
\n{\vA}_{B^{s}_{2, 1}}.
\label{5.17}
\end{align}
\end{Lemma}
\bProof
Since the proof is similar to that of Lemma 9.3 in \cite{D02}, we only give a brief outline. By Parseval's identity, we obtain \eqref{5.16}. Like \eqref{5.18}, we define 
\begin{align*}
\varphi_{\mathbf{p}^*}(\zeta):=\varphi(|\mathbf{p}|2^{-j_{\mathbf{p}}}\zeta), 
\quad  \zeta \in \mathbb{R}.
\end{align*}
Thanks to the Minkowski's inequality, 
\begin{align*}
\sp{
\sum\limits_{\mathbf{p}\in \mathbf{P}}
|\mathbf{p}|^{2s}
\n{a^{\mathbf{p}}}_{B^{s}_{2,1}(\mathbb{T})}^2  
}^{\f1 2}
&
\approx
\sp{
\sum\limits_{\mathbf{p}\in \mathbf{P}}
|\mathbf{p}|^{2s}
\sp{
\sum\limits_{j\in \mathbb{Z}}
2^{js}
\sp{
\sum\limits_{j'\in \mathbb{Z}\setminus\{0\}}
\sp{
\varphi_{\mathbf{p}^*}(2^{-j}|j'|)
}^2
|\widehat{\vA}^{\text{sgn}(\mathbf{p})}_{j'\mathbf{p}}|^2
}^{\f1 2}
}^2
}^{\f1 2}
\\
&
\approx
\sp{
\sum\limits_{\mathbf{p}\in \mathbf{P}}
\sp{
\sum\limits_{j\in \mathbb{Z}}
2^{(j+j_\mathbf{p})s}
\sp{
\sum\limits_{j'\in \mathbb{Z}\setminus\{0\}}
\sp{
\varphi(2^{-(j+j_{\mathbf{p}})}|j'\mathbf{p}|)
}^2
|\widehat{\vA}^{\text{sgn}(\mathbf{p})}_{j'\mathbf{p}}|^2
}^{\f1 2}
}^2
}^{\f1 2}
\\
&
\approx
\sp{
\sum\limits_{\mathbf{p}\in \mathbf{p}}
\sp{
\sum\limits_{j\in \mathbb{Z}}
2^{js}
\n{\De_j \vA^p}_{L^2}
}^2
}^{\f1 2}
\\
&
\les
\sum\limits_{j\in \mathbb{Z}}
2^{js}
\sp{
\sum\limits_{\mathbf{p}\in \mathbf{p}}
\n{\De_j \vA^p}^2_{L^2}
}^{\f1 2}
\\
&
\approx
\n{\vA}_{B^{s}_{2, 1}},
\end{align*}
which yields \eqref{5.17}.     \ \ $\Box$

\section{Proof of Theorem \ref{main}}
In this section, we complete the proof of Theorem \ref{main}. We define
\begin{align*}
T^*_\ep
:=
\sup
&
\biggl\{
0<T<\infty: \eqref{1.1.1} \text{ admits a solution } (\vae, \vbe, \vue) \text{ on } [0, T] \text{ such that } 
\\
& \quad 
(\vae, \vbe, \vue) \in
\widetilde{C}_{T}(B^{\f d 2}_{2, 1})
\times
\widetilde{C}_{T}(B^{\f d 2}_{2, 1})
\times
\sp{
\widetilde{C}_{T}(B^{\f d 2-1}_{2, 1})
\cap
L_{T}^1(B^{\f d 2+1}_{2, 1})
}
\biggr\}.
\end{align*}
It follows from the local well-posedness theory of two-phase flows that
$T^*_\ep>0$. We now prove by contradiction that $T_0< T^*_\ep$  for any $0<\ep\leq \ep_0$,  where $0<\ep_0\leq \min\{1, \f{R^0}{8|c_{02}|}\}$ will be chosen later. We assume that there exists a $0<\ep<\ep_0$ such that $T^*_\ep\leq T_0$. Then we define
\begin{align*}
T^{**}_\ep
:=
\sup
\left\{
0<T<T^{*}_\ep: 
E^\ep_{T}\leq 2 C\sp{\vV_{T^*_0}+\vU_{T^*_0}},
\quad 
E^{\ep, \zeta_0}_{T}\leq 2\delta_0
\right\},
\end{align*}
where $\zeta_0$ is sufficiently large and $\delta_0$ is sufficiently small, satisfying $\zeta_0< \beta_0 /\ep$. Now we let $0<T<T^{**}_\ep$.
There exists $C_2>0$ such that
\begin{align*}
\ep 
\sp{
\n{\vae}_{L^\infty_T(L^\infty)}
+
\n{\vbe}_{L^\infty_T(L^\infty)}
}
\leq 
C_2
\ep
\sp{
\n{a_\ep}_{\widetilde{L}^\infty_T(B^{\f d 2}_{2,1})}
+
\n{b_\ep}_{\widetilde{L}^\infty_T(B^{\f d 2}_{2,1})}
}.
\end{align*}
From \eqref{cons3}, we get
\begin{align*} 
&
\ep 
\sp{
\n{\vae}_{L^\infty_T(L^\infty)}
+
\n{\vbe}_{L^\infty_T(L^\infty)}
}
\leq
C
C_2
(E^{\ep, \zeta_0}_{T}+\zeta_0\ep E^{\ep}_{T})
\leq
C
C_2
(E^{\ep, \zeta_0}_{T}+\zeta_0\ep_0 E^{\ep}_{T})
\\
& \quad 
\leq
C
C_2
\sp{
2\delta_0
+
2 \zeta_0\ep_0 C\sp{\vV_{T^*_0}+\vU_{T^*_0}}
}.
\end{align*}
We temporarily assume that $\delta_0$, $\zeta_0$ and $\ep_0$ satisfy 
\begin{align}\label{assu1}
C
C_2
\sp{
2\delta_0
+
2 \zeta_0\ep_0 C\sp{\vV_{T^*_0}+\vU_{T^*_0}}
}
\leq
\f{\min\{R^0, Q^0\}}{2} .
\end{align}
From \eqref{cons20}, \eqref{cons5},  Proposition \ref{Pro4.1}, Proposition \ref{Pro4.2} and Proposition \ref{Pro6.4},  we get
\begin{align*}
&
\vW^{\ep}_{T, \theta}\leq C_4 {\ep_0}^{\f{\theta}{3+\theta}},
\quad 
\vZ^{\ep}_{T, \theta}
\leq
C_5
\tau^*(\ep_0),
\\
&
E_{T}^{\ep, \zeta_0}
\leq
C
[T_0]^{\f1 2}
\zeta_0^{1+2\theta}
\sp{C_4+C_5}\tau^*(\ep_0)
+
C_6
\sp{
\n{a_0}^{h;\zeta_0}_{B^{\f d 2}_{2, 1}}
+
\n{b_0}^{h;\zeta_0}_{B^{\f d 2}_{2, 1}}
+
\n{\vu_0}^{h; \zeta_0}_{B^{\f d 2-1}_{2, 1}}
+
(C^2_3+C_3)
\ep_0\zeta_0
}
\\
&
\qquad \quad 
+
C_7
\sp{
{\zeta_0}^{1+2\theta}
\sp{C_4+C_5}\tau^*(\ep_0)
+
\n{\vV}^{h;\f{\zeta}{16}}_{\widetilde{L}_{T^*_0}^\infty(B^{\f d 2-1}_{2, 1})
\cap
L_{T^*_0}^1(B^{\f d 2+1}_{2, 1})}
+
\n{\vw}^{h;\f{\zeta_0}{16}}_{\widetilde{L}_{T^*_0}^\infty(B^{\f d 2-1}_{2, 1})
\cap
L_{T^*_0}^1(B^{\f d 2+1}_{2, 1})}
+
\n{\rho}^{h;\f{\zeta_0}{16}}_{\widetilde{L}_{T^*_0}^\infty(B^{\f d 2}_{2, 1})
}
},
\\
&
E^{\ep}_{T}
\leq
C
\sp{
E^{\ep, \zeta_0}_T
+
[T_0]^{\f1 2}
{\zeta_0}^{1+2\theta}
\sp{C_4+C_5}\tau^*(\ep_0)
+
\vV_{T^*_0}
+
\vU_{T^*_0}
},
\end{align*}
with
\begin{align*}
&
\tau^*(\ep)=\max\left\{\ep^{\f{\theta^2}{(1+\theta)(3+\theta)}},  \tau(\ep)\right\},
\quad 
C_3
=
2 C\sp{\vV_{T^*_0}+\vU_{T^*_0}},
\\
&
C_4
=
C
[T_0]
\exp{ \left[ C
\sp{
C_3
+
[T_0]
}
\sp{
\vU_{T^*_0}
+
\vU^2_{T^*_0}
} \right]
}
\sp{
C_3
+
E^2_0
}
\sp{
C^2_3
+
C_3
+
1}
\sp{
\vU_{T^*_0}
+
C_3
},
\\
&
C_5
=
C
\exp{ \left[
C\sp{\vV^2_{T^*_0}+[T_0]C^2_3}
\right]
}
\\
&
\qquad
\cdot
\Big[
\sp{\vU_{T^*_0}+C_3}^{\f{1}{1+\theta}}\vV_{T^*_0}
+
\vV_{T^*_0}
+
[T_0]^{\f{3}{2}}
\sp{E_0+E^2_0+\vV_{T^*_0}+\vU_{T^*_0}+C_3+1}
\sp{\vV^2_{T^*_0}+\vU_{T^*_0}+\sp{C^2_3+C_3+1}^2}
\Big], 
\\
&
C_6
=
C^2
\exp{\left[
C
[T_0]
\sp{
1
+
C_3
+
C^2_3
} \right] 
},
\quad 
C_7
=
C^2
C_3
[T_0]
\exp{\left[
C
[T_0]
\sp{
1
+
C_3
+
C^2_3
} \right] 
}.
\end{align*}
Now we set $\delta_0=\min\{\f{2\min\{R^0, Q^0\}}{16CC_2}, \f{C_3}{4C}, \f{C_3}{4}\}$ and choose $\zeta_0 \geq 1$ suitably large such that
\begin{align}\label{assu2}
&
C_6
\sp{
\n{a_0}^{h;\zeta_0}_{B^{\f d 2}_{2, 1}}
+
\n{b_0}^{h;\zeta_0}_{B^{\f d 2}_{2, 1}}
+
\n{\vu_0}^{h; \zeta_0}_{B^{\f d 2-1}_{2, 1}}
}
\nonumber
\\
& \quad 
+
C_7
\sp{
\n{\vV}^{h;\f{\zeta}{16}}_{\widetilde{L}_{T^*_0}^\infty(B^{\f d 2-1}_{2, 1})
\cap
L_{T^*_0}^1(B^{\f d 2+1}_{2, 1})}
+
\n{\vw}^{h;\f{\zeta_0}{16}}_{\widetilde{L}_{T^*_0}^\infty(B^{\f d 2-1}_{2, 1})
\cap
L_{T^*_0}^1(B^{\f d 2+1}_{2, 1})}
+
\n{\rho}^{h;\f{\zeta_0}{16}}_{\widetilde{L}_{T^*_0}^\infty(B^{\f d 2}_{2, 1})
}
}
\leq
\f{\delta_0}{2} .
\end{align}
For such $\delta_0$ and $\zeta_0$,  we choose $0<\ep_0<\min\{1, \f{\beta_0}{\zeta_0}, \f{R^0}{8|c_{02}|}\}$ such that
\begin{align}\label{assu3}
&
2
C
C_2
\zeta_0\ep_0 C\sp{\vV_{T^*_0}+\vU_{T^*_0}}
\leq
\f{\min\{R^0, Q^0\}}{16},
\nonumber
\\
&
C
[T_0]^{\f1 2}
\zeta_0^{1+2\theta}
\sp{C_4+C_5}\tau^*(\ep_0)
+
C_7
{\zeta_0}^{1+2\theta}
\sp{C_4+C_5}\tau^*(\ep_0)
+
C_6
(C^2_3+C_3)
\zeta_0
\ep_0
\leq
\f{\delta_0}{2} .
\end{align}
For these fixed $\delta_0$, $\zeta_0$ and $\ep_0$, we see that \eqref{assu1} holds and 
\begin{align*}
E^{\ep}_T\leq \f{14C}{8}\sp{\vV_{T^*_0}+\vU_{T^*_0}}, \quad 
E_{T}^{\ep, \zeta_0}\leq \delta_0 .
\end{align*}
Hence, $T^{**}_{\ep}=T^{*}_{\ep}$ and for any $0<T<T^{*}_{\ep}$ we have
\begin{align*}
E^\ep_{T}\leq 2 C\sp{\vV_{T^*_0}+\vU_{T^*_0}},
\quad  
E^{\ep, \zeta_0}_{T}\leq 2\delta_0,
\end{align*}
which implies
\begin{align*}
\n{\vue}_{L^1_{T^*_\ep}(\underline{B}^{\f d 2+1}_{2, 1})}<\infty,
\ \ 
\f{R^0}{2}\leq R^0+\ep\vae(t, x) \leq \f{3R^0}{2},
\quad 
\f{Q^0}{2}\leq Q^0+\ep\vbe(t, x) \leq \f{3Q^0}{2},
\quad 
\text{ for all } (t,x)\in[0,T^*_\ep) \times \Td.
\end{align*}
Therefore, the solution can be continuously extended beyond $T^*_\ep$, which contradicts the definition  of $T^*_\ep$. Therefore, for any $0<\ep \leq \ep_0$,  we obtain  $T_0< T^*_\ep$. 

Finally, we derive the critical convergence. For any $0<\delta \leq \delta_0$, we can choose $\zeta_\delta \geq \zeta_0$ such that \eqref{assu2} holds with
$(\zeta_0, \delta_0)$ replaced by $(\zeta_\delta, \delta)$ such that
\begin{align*}
\n{\vV}^{h;\zeta_\delta}_{
\widetilde{L}^\infty_{T^*_0}(B^{\f d 2-1}_{2, 1})
\cap
L^1_{T^*_0}(B^{\f d 2+1}_{2, 1})
}
+
\n{\rho}^{h;\zeta_\delta}_{
\widetilde{L}^\infty_{T^*_0}(B^{\f d 2}_{2, 1})}
+
\n{\vw}^{h;\zeta_\delta}_{
\widetilde{L}^\infty_{T^*_0}(B^{\f d 2-1}_{2, 1})
\cap
L^1_{T^*_0}(B^{\f d 2+1}_{2, 1})
}
\leq 
\delta.
\end{align*}
Then we choose $0<\ep_\delta\leq \ep_0$ such that \eqref{assu1} and \eqref{assu3} are satisfied with 
$(\zeta_0, \delta_0, \ep_0)$ with $(\zeta_\delta, \delta, \ep_\delta)$ such that 
\begin{align*}
C
\zeta_\delta 
\ep_\delta 
C_3 
\leq 
\delta  
\end{align*}
Consequently, for any $0<\ep\leq \ep_\delta$ we have $E^{\ep, \zeta_\delta}_{T_0}\leq \delta$. Moreover, we infer from \eqref{cons3} and \eqref{cons20} that for any $\ep\leq \ep_\delta$
\begin{align*}
&
\ep
\sp{
\n{a_\ep}_{\widetilde{L}^\infty_T(B^{\f d 2}_{2,1})}
+
\n{b_\ep}_{\widetilde{L}^\infty_T(B^{\f d 2}_{2,1})}
}
+
\n{\vVe-\vV}_{
\widetilde{L}^\infty_{T_0}(B^{\f d 2-1}_{2, 1})
\cap
\widetilde{L}^2_{T_0}(B^{\f d 2}_{2, 1})
}
+
\n{\rho_\ep-\rho}_{
\widetilde{L}^\infty_{T_0}(B^{\f d 2}_{2, 1})}
\\
&
\qquad
+
\n{\mathcal{P}\vue-\vw}_{
\widetilde{L}^\infty_{T_0}(\underline{B}^{\f d 2-1}_{2, 1})
\cap
L^1_{T_0}(\underline{B}^{\f d 2+1}_{2, 1})
}
+
\n{\widehat{(\vue)}_0-\widehat{(\vu_0)}_0}_{L^\infty(0, T_0)}
\\
& \quad 
\leq 
C
\sp{
E^{\ep, \zeta_\delta}_{T_0}
+
\n{\vV}^{h;\zeta_\delta}_{
\widetilde{L}^\infty_{T^*_0}(B^{\f d 2-1}_{2, 1})
\cap
L^1_{T^*_0}(B^{\f d 2+1}_{2, 1})
}
+
\n{\rho}^{h;\zeta_\delta}_{
\widetilde{L}^\infty_{T^*_0}(B^{\f d 2}_{2, 1})}
+
\n{\vw}^{h;\zeta_\delta}_{
\widetilde{L}^\infty_{T^*_0}(B^{\f d 2-1}_{2, 1})
\cap
L^1_{T^*_0}(B^{\f d 2+1}_{2, 1})
}
}
+
C
\sp{
E^{\ep, \zeta_\delta}_{T_0}
+
\zeta_\delta
\ep_\delta
E^{\ep}_{T_0}
}
\\
& \quad 
\leq
4C\delta.
\end{align*}
This completes the proof.        \ \ $\Box$

\bigskip

\noindent{\bf{Acknowledgement.}} The work of Y.L. was supported by National Natural Science Foundation of China (12571228), Natural Science Foundation of Anhui Province (2408085MA018). The work of S.L. was supported by the Tianyuan Fund for Mathematics of the National Natural Science Foundation of China (12526541).

\vspace{4mm}
\noindent{\bf{Data Availability.}} Data sharing is not applicable to this article as no datasets were generated or analyzed
during the current study.
\vspace{4mm}

\section*{Declarations}

\noindent{\bf{Conflicts of interest.}} All authors certify that there are no conflicts of interest for this work.

\vspace{4mm}

\noindent{\bf{Ethics approval.}} This article does not contain any studies involving humans or animals.

\appendix
\section{Some useful lemmas}

In this appendix, we present some useful lemmas.
\begin{Lemma}\label{le4.3}
Let $0<\theta<1$, $\vA(0)\in B^{\f d 2-1-\theta}_{2, 2} \cap \text{Im}\, L$, $\vB\in \text{Im} \, L$, and $\vE\in \text{Ker}\, L$. If $\vA\in \text{Im}\, L$ is a solution to the following equation on $[0, T]$: 
\begin{equation}\label{eq4.1}
\left\{\begin{aligned}
&\p_t \vA
+
\mathcal{Q}^\ep_1(\vB, \vA)
+
\mathcal{Q}^\ep_2(\underline{\vE}, \vA)
+
\mathcal{Q}^\ep_2(\mathbb{P}_0\vE, \vA)-\mu\De \vA
=F+G, \\[4pt]
&\vA(0)=\vA_0.
\end{aligned} \right.
\end{equation}
Then it holds 
\begin{align*}
\n{\vA}_{\widetilde{L}_T^\infty(B^{\f d 2 -1-\theta}_{2, 2})\cap L_T^2(B^{\f d 2-\theta}_{2, 2})}
&
\lesssim 
\exp{ \left[ C
\sp{
\n{\vB}^2_{L_T^2(B^{\f d 2}_{2, 1})}
+
\n{\vE}^2_{L_T^2(\underline{B}^{\f d 2}_{2, 1})}
+
\n{\mathbb{P}_0 \vE}^2_{L^2(0, T)}
} \right] 
}
\\
&
\quad
\cdot
\sp{
\n{\vA_0}_{B^{\f d 2 -1-\theta}_{2, 2}}
+
\n{F}_{L_T^2(\underline{B}^{\f d 2 -2-\theta}_{2, 2})}
+
\n{G}_{L_T^1(\underline{B}^{\f d 2 -1-\theta}_{2, 2})}
}. 
\end{align*}

\end{Lemma}
\bProof 
The proof is similar to that of Proposition 4.1 in \cite{D02}; whence we only provide a sketch. Since $\vA\in \text{Im}\, L$, we have $\widehat{\vA}_0=0$, and consequently, for $-\infty<s<\infty$ and $1\leq q\leq \infty$
\begin{align}\label{eq4.2}
    \n{\vA}_{\widetilde{L}_T^q(B^s_{2, 2})}
=
\n{\vA}_{\widetilde{L}_T^q(\underline{B}^s_{2, 2})}.
\end{align}
Applying $\De_j$ to \eqref{eq4.1}$_1$, multiplying the resulting equation by $\De_j \vA$, integrating over $(0, t)\times \Td$, then multiplying by $2^{2j(\f d 2 -1-\theta)}$, we obtain by summation over $j$ that 
\begin{align}\label{eq4.6}
&
\n{\vA}^2_{\widetilde{L}^\infty_t(B^{\f d 2 -1-\theta}_{2, 2})}
+
\n{\vA}^2_{L^2_t(B^{\f d 2-\theta}_{2, 2})}\nonumber
\\
\quad
 &
\leq 
C
\Bigg[
\n{\vA_0}^2_{B^{\f d 2 -1-\theta}_{2, 2}}
+
\int_{0}^t
\n{F}_{\underline{B}^{\f d 2 -2-\theta}_{2, 2}}
\n{\vA}_{\underline{B}^{\f d 2 -\theta}_{2, 2} }\dta
+
\int_{0}^t
\n{G}_{\underline{B}^{\f d 2 -1-\theta}_{2, 2}}
\n{\vA}_{\underline{B}^{\f d 2 -1-\theta}_{2, 2}}\dta
\\
&
\quad
+
\int_{0}^t
\sp{
\n{\mathcal{Q}^\ep_1(\vB, \vA)}_{\underline{B}^{\f d 2 -1-\theta}_{2, 2}}
+
\n{\mathcal{Q}^\ep_2(\underline{\vE}, \vA)}_{\underline{B}^{\f d 2 -1-\theta}_{2, 2}}
+
\n{\mathcal{Q}^\ep_2(\mathbb{P}_0\vE, \vA)}_{\underline{B}^{\f d 2 -1-\theta}_{2, 2}}
}
\n{\vA}_{\underline{B}^{\f d 2 -1-\theta}_{2, 2}}\dta
\Bigg]    \nonumber,
\end{align}
where we used \eqref{eq4.2}. Hence 
\begin{align}
C \int_{0}^t
\n{F}_{\underline{B}^{\f d 2 -2-\theta}_{2, 2}}
\n{\vA}_{\underline{B}^{\f d 2-\theta}_{2, 2}} \dta
&
\leq
C
\n{F}_{L^2_t(\underline{B}^{\f d 2 -2-\theta}_{2, 2})}
\n{\vA}_{L^2_t(\underline{B}^{\f d 2 -\theta}_{2, 2})}\nonumber
\\
&
\leq
C^2
\n{F}^2_{L^2_t(\underline{B}^{\f d 2 -2-\theta}_{2, 2})}
+
\f 1 4 
\n{\vA}^2_{L^2_t(B^{\f d 2 -\theta}_{2, 2})}, \label{eq4.3} \\
C
\int_{0}^t
\n{G}_{\underline{B}^{{\f d 2 -1-\theta}}_{2, 2}}
\n{\vA}_{\underline{B}^{\f d 2 -1-\theta}_{2, 2}}\dta
&
\leq
C
\n{G}_{L^1_t(\underline{B}^{\f d 2 -1-\theta}_{2, 2})}
\n{\vA}_{L^\infty_t(\underline{B}^{\f d 2 -1-\theta}_{2, 2})}\nonumber
\\
&
\leq
C^2
\n{G}^2_{L^1_t(\underline{B}^{\f d 2 -1-\theta}_{2, 2})}
+
\f 1 4 
\n{\vA}^2_{\widetilde{L}^\infty_t(B^{\f d 2 -1-\theta}_{2, 2})}. \label{eq4.4}
\end{align} 
It follows from Lemma \ref{le4.2} that 
\begin{align}\label{eq4.5}
&
C
\int_{0}^t
\sp{
\n{\mathcal{Q}^\ep_1(\vB, \vA)}_{\underline{B}^{\f d 2 -1-\theta}_{2, 2}}
+
\n{\mathcal{Q}^\ep_2(\underline{\vE}, \vA)}_{\underline{B}^{\f d 2 -1-\theta}_{2, 2}}
+
\n{\mathcal{Q}^\ep_2(\mathbb{P}_0\vE, \vA)}_{\underline{B}^{\f d 2 -1-\theta}_{2, 2}}
}
\n{\vA}_{\underline{B}^{\f d 2 -1-\theta}_{2, 2}}\dta\nonumber
\\
& \quad 
\leq
C
\int_{0}^t
\sp{
\n{\vB}_{B^{\f d 2 -1}_{2, 1}}
+
\n{\vE}_{\underline{B}^{\f d 2 -1}_{2, 1}}
+
|\mathbb{P}_0\vE|
}
\n{\vA}_{B^{\f d 2 -1-\theta}_{2, 2}}
\n{\vA}_{B^{\f d 2 -\theta}_{2, 2}}
\dta\nonumber
\\
& \quad 
\leq
C^2
\int_{0}^t
\sp{
\n{\vB}_{B^{\f d 2 -1}_{2, 1}}
+
\n{\vE}_{\underline{B}^{\f d 2 -1}_{2, 1}}
+
|\mathbb{P}_0\vE|
}^2
\n{\vA}^2_{\widetilde{L}_\tau^\infty (B^{\f d 2 -1-\theta}_{2, 2})}\dta
+
\f1 4
\n{\vA}^2_{L^2_t(B^{\f d 2-\theta}_{2, 2})} .
\end{align}

Inserting \eqref{eq4.3}-\eqref{eq4.5} into \eqref{eq4.6}, and applying Gr\"{o}nwall's inequality, followed by taking the square root, completes the proof. \ \ $\Box$

\begin{Lemma}\label{le6.4} Assume that 
$0\leq\theta<1$, 
$(a_0, \vv_0)\in B^{\f d 2-1-\theta}_{2, 1}\times B^{\f d 2-1-\theta}_{2, 1}$ with $\Div \vv_0=0$, 
and
$\Div \vu=\Div \vw^*=0$ .
Let $(a, \vv)$
be a solution to the following equation on $[0, T]$,
\begin{equation}\label{eq4.7}
\left\{\begin{aligned}
&
\p_t a
+
\Div(\underline{a}\underline{\vu})
+
\Div(\underline{a}\df{\widehat{\vu}_0}{\sqrt{|\Td|}})
=
-
\Div(\underline{\vv} b)
-
\Div(\df{\widehat{\vv}_0}{\sqrt{|\Td|}}b)
+
F,
\\
&
\p_t
\vv
+
\mathcal{P}
\sp{
\underline{\vu}\cdot \Grad\vv
+
\df{\widehat{\vu}_0}{\sqrt{|\Td|}}\cdot \Grad\vv
}
-
\mu
\De
\vv
=
\mathcal{P}
\left(
-
\underline{\vv}\cdot \Grad\vw^*
-
\df{\widehat{\vv}_0}{\sqrt{\Td}} 
\cdot
\Grad
\vw^*
+G
\right),
\\
&
\Div\vv=0,
\\
&
a(0)= a_0, \ \ \vv(0)=\vv_0.
\end{aligned}
\right.
\end{equation}
Then it holds
\begin{align*}
&
\n{a}_{\widetilde{L}_T^\infty(B^{\f d 2 -1-\theta}_{2, 1})}
+
\n{\vv}_{\widetilde{L}_T^\infty(B^{\f d 2 -1-\theta}_{2, 1})
\cap
L_T^1(\underline{B}^{\f d 2 +1-\theta}_{2, 1})}
\\
&
\quad
\leq
C
\exp{  \left[
C
\int_{0}^T 
\sp{
\n{\vu}_{\underline{B}^{\f d 2+1}_{2, 1}} 
+
\n{\vw^*}_{\underline{B}^{\f d 2+1}_{2, 1}} 
+
\n{b}^2_{\underline{B}^{\f d 2}_{2, 1}}
+
\n{b}_{\underline{B}^{\f d 2}_{2, 1}}
}
\dt \right]
}
\\
&
\qquad
\cdot
\sp{
\n{a_0}_{\underline{B}^{\f d 2 -1-\theta}_{2, 1}}
+
\n{\vv_0}_{\underline{B}^{\f d 2 -1-\theta}_{2, 1}}
+
\n{F}_{L_T^1(B^{\f d 2 -1-\theta}_{2, 1})}
+
\n{G}_{L_T^1(B^{\f d 2 -1-\theta}_{2, 1})}
}.
\end{align*}
\end{Lemma}
\bProof
Applying $\De_j$ to \eqref{eq4.7} and taking $L^2$ inner product gives 
\begin{align*}
&
\f1 2 
\df{d}{\dt}
\n{\De_j a}^2_{L^2} 
=
-
\int_{\Td}
[\De_j, \underline{\vv}]\Grad a
\De_j a
\dx
-
\int_{\Td}
\sp{
\Div\De_j(\underline{\vv} b)
+
\Div\De_j(\df{\widehat{\vv}_0}{\sqrt{|\Td|}}b)
-
\De_j F
}
\De_j a\dx,
\\
&
\f1 2 
\df{d}{\dt}
\n{\De_j \vv}^2_{L^2} 
+
\mu
\n{\Grad\De_j \vv}^2_{L^2}
\\
&
\quad
=
-
\int_{\Td}
[\De_j, \underline{\vu}]\Grad \vv
\De_j \vv
\dx
-
\int_{\Td}
\sp{
\De_j
(\underline{\vv}\cdot \Grad\vw^*)
+
\De_j
(\df{\widehat{\vv}_0}{\sqrt{\Td}}
\cdot
\Grad
\vw^*)
-
\De_j G
}
\De_j \vv \dx.
\end{align*}
Hence, we obtain
\begin{align*}
&
\n{\De_j a}_{L^\infty_t(L^2)}
+
\n{\De_j \vv}_{L^\infty_t(L^2)}
+
2^{2j}
\n{\De_j \vv}_{L^1_t(L^2)}
\\
&
\les
\n{\De_j a_0}_{L^2}
+
\n{\De_j \vv_0}_{L^2}
+
\int_{0}^t 
\sp{
\n{[\De_j, \underline{\vv}]\Grad a}_{L^2}
+
\n{[\De_j, \underline{\vu}]\Grad \vv}_{L^2}}
\dta
+
\int_{0}^t 
\sp{
\n{[\De_j F}_{L^2}
+
\n{[\De_j G}_{L^2}}
\dta
\\
& \quad  
+
\int_0^t
\n{\Div\De_j(\underline{\vv} b)}_{L^2}
+
\n{\Div\De_j(\df{\widehat{\vv}_0}{\sqrt{|\Td|}}b)}_{L^2}
+
\n{\De_j
(\underline{\vv}\cdot \Grad\vw^*)}_{L^2}
+
\n{\De_j
(\df{\widehat{\vv}_0}{\sqrt{\Td}}
\cdot
\Grad
\vw^*)}_{L^2}
\dta,
\end{align*}
Multiplying the above inequality by $2^{j(\f d 2 -1-\theta)}$
and summing over $j$, and applying Lemma \ref{le4.1} together with the following commutator estimate (see Lemma A.2 in \cite{D01}):
\begin{align*}
\sum\limits_{j\in \mathbb{Z}}
2^{j(\f d 2-1-\theta)}
\n{[\De_j, \underline{g}]\Grad f}_{L^2}
\les
\n{g}_{\underline{B}^{\f d 2+1}_{2, 1}}
\n{f}_{\underline{B}^{\f d 2-1-\theta}_{2, 1}},
\end{align*}
we obtain
\begin{align*}
&
\n{a}_{\widetilde{L}_t^\infty(\underline{B}^{\f d 2 -1-\theta}_{2, 1})}
+
\n{\vv}_{\widetilde{L}_t^\infty(\underline{B}^{\f d 2 -1-\theta}_{2, 1})
\cap
L_T^1(\underline{B}^{\f d 2 +1-\theta}_{2, 1})}
\\
&
\leq
C
\left[
\n{a_0}_{\underline{B}^{\f d 2 -1-\theta}_{2, 1}}
+
\n{\vv_0}_{\underline{B}^{\f d 2 -1-\theta}_{2, 1}}
+
\int_{0}^t \n{\vu}_{\underline{B}^{\f d 2+1}_{2, 1}}  \n{a}_{\underline{B}^{\f d 2-1-\theta}_{2, 1}}\dta
+
\int_{0}^t 
|\widehat{\vv}_0| 
\sp{
\n{b}_{\underline{B}^{\f d 2-\theta}_{2, 1}}
+
\n{\vw^*}_{\underline{B}^{\f d 2-\theta}_{2, 1}}
} \dta 
\right]  
\\
&
\qquad
+
C
\sp{
\int_{0}^t 
\sp{
\n{\vu}_{\underline{B}^{\f d 2+1}_{2, 1}}
+
\n{\vw^*}_{\underline{B}^{\f d 2+1}_{2, 1}}
}
\n{\vv}_{\underline{B}^{\f d 2-1-\theta}_{2, 1}}\dta
+
\n{F}_{L_t^1(\underline{B}^{\f d 2 -1-\theta}_{2, 1})}
+
\n{G}_{L_t^1(\underline{B}^{\f d 2 -1-\theta}_{2, 1})}
}
\\
&
\qquad
+
C\int_{0}^t \n{\vv}_{\underline{B}^{\f d 2-\theta}_{2, 1}}  \n{b}_{\underline{B}^{\f d 2}_{2, 1}}\dta.
\end{align*}
By the standard interpolation inequality, 
\begin{align*}
C\int_{0}^t \n{\vv}_{\underline{B}^{\f d 2-\theta}_{2, 1}}  \n{b}_{\underline{B}^{\f d 2}_{2, 1}}\dta
&
\leq
C\int_{0}^t \n{\vv}^{\f1 2}_{\underline{B}^{\f d 2-1-\theta}_{2, 1}}  
\n{\vv}^{\f1 2}_{\underline{B}^{\f d 2+1-\theta}_{2, 1}} \n{b}_{\underline{B}^{\f d 2}_{2, 1}}\dta
\\
&
\leq 
C^2
\int_{0}^t \n{\vv}_{\underline{B}^{\f d 2-1-\theta}_{2, 1}}  \n{b}^2_{\underline{B}^{\f d 2}_{2, 1}}\dta
+
\f1 4
\n{\vv}_{
L_t^1(\underline{B}^{\f d 2 +1-\theta}_{2, 1})}.
\end{align*}
Integrating \eqref{eq4.7} over $(0, t)\times \Td$ and dividing by $\sqrt{|\Td|}$ yields
\begin{align*}
|\widehat{a}_0|
+
|\widehat{\vv}_0|
\leq
|\widehat{(a_0)}_0|
+
|\widehat{(\vv_0)}_0|
+
\int_0^t \sp{ |\widehat{F}_0|+|\widehat{G}_0| }  \dta.
\end{align*}
Collecting the above estimates, 
\begin{align*}
&
\n{a}_{\widetilde{L}_t^\infty(B^{\f d 2 -1-\theta}_{2, 1})}
+
\n{\vv}_{\widetilde{L}_t^\infty(B^{\f d 2 -1-\theta}_{2, 1})
\cap
L_T^1(\underline{B}^{\f d 2 +1-\theta}_{2, 1})}
\\
&
\quad
\leq
C
\int_{0}^t 
\sp{
\n{\vu}_{\underline{B}^{\f d 2+1}_{2, 1}} 
+
\n{\vw^*}_{\underline{B}^{\f d 2+1}_{2, 1}} 
+
\n{b}^2_{\underline{B}^{\f d 2}_{2, 1}}
+
\n{b}_{\underline{B}^{\f d 2}_{2, 1}}
}
\sp{
\n{a}_{\widetilde{L}_\tau^\infty(B^{\f d 2 -1-\theta}_{2, 1})}
+
\n{\vv}_{\widetilde{L}_\tau^\infty(B^{\f d 2 -1-\theta}_{2, 1})}
}
\dta
\\
&
\quad
\quad
+
C
\sp{
\n{a_0}_{\underline{B}^{\f d 2 -1-\theta}_{2, 1}}
+
\n{\vv_0}_{\underline{B}^{\f d 2 -1-\theta}_{2, 1}}
+
\n{F}_{L_t^1(B^{\f d 2 -1-\theta}_{2, 1})}
+
\n{G}_{L_t^1(B^{\f d 2 -1-\theta}_{2, 1})}
},
\end{align*}
which completes the proof by Gr\"{o}nwall's inequality.      \ \ $\Box$

\end{document}